\documentclass[12pt]{amsart}

\usepackage[left=2.2cm,tmargin=2.5cm,bmargin=2.5cm,right=2.2cm]{geometry}

\usepackage{amssymb,enumerate}
\usepackage{color,xcolor}
\usepackage{tikz-cd}
\usepackage{imakeidx}
\makeindex[title={Index of notation},columnsep=2pt]

\usepackage{microtype}

\usepackage{url}
\usepackage{hyperref}

\usepackage{thmtools}
\newtheorem{theorem}{Theorem}[section]
\newtheorem{corollary}[theorem]{Corollary}
\newtheorem{lemma}[theorem]{Lemma}
\newtheorem{proposition}[theorem]{Proposition}

\newtheorem{conjecture}[theorem]{Conjecture}
\newtheorem*{condition}{Condition (BC)}
\newcommand{\condbc}{Condition~\hyperlink{BC}{BC}}

\theoremstyle{definition}
\newtheorem{defi}[theorem]{Definition}
\newtheorem{definition}[theorem]{Definition}
\newtheorem{example}[theorem]{Example}
\newtheorem{notation}[theorem]{Notation}
\newtheorem{notation*}{Notation*}

\theoremstyle{remark}
\newtheorem{remark}[theorem]{Remark}
\numberwithin{equation}{section}

\def \Hom {\operatorname{Hom}}
\def \mHom {\mathcal{H}om}

\def \tr {\operatorname{tr}}
\def \Frob {\operatorname{Frob}}

\def \End {\operatorname{End}}
\def \Aut {\operatorname{Aut}}
\def \rank {\operatorname{rank}}
\newcommand{\im}{\operatorname{Im}}
\def \Spec {\operatorname{Spec}}

\def \Ql {\overline{\mathbb Q}_\ell}

\def \Bun {\operatorname{Bun}}
\def \BunGD {\operatorname{Bun}_{G(D)}}
\def \Hecke {\mathcal{H}k}

\def \lWR {\langle}
\def \rWR {\rangle}
\def \charr {\operatorname{char}}
\def \weights {W}

\def \FrK {\mathsf{K}}
\def \KK { \mathbf{K}(D)}

\def \kkF {F}
\def \kkL {L}
\def \kkE {E}
\def \kottK {J}

\newcommand{\Gr}{\operatorname{Gr}}

\newcommand{\mapsfrom}{\mathrel{\reflectbox{\ensuremath{\mapsto}}}}
\newcommand{\bG}{\mathbf{G}}
\providecommand{\bseries}[1]{ [\hspace{-0,5mm}[ {#1} ]\hspace{-0,5mm}] } 
\providecommand{\pseries}[1]{ (\hspace{-0,7mm}( {#1} )\hspace{-0,7mm}) }  
\newcommand{\A}{\mathbb{A}}

\newcommand{\C}{\mathbb{C}}

\newcommand{\F}{\mathbb{F}}
\newcommand{\G}{\mathbb{G}}

\newcommand{\PP}{\mathbb{P}}
\newcommand{\Q}{\mathbb{Q}}

\newcommand{\Z}{\mathbb{Z}}

\DeclareMathOperator {\SL} {SL}
\DeclareMathOperator {\PGL} {PGL}
\DeclareMathOperator {\Sp} {Sp}
\DeclareMathOperator {\Gal}  {Gal}
\newcommand{\ko}{\mathfrak{o}}

\providecommand{\cB}{\mathcal{B}}
\providecommand{\cC}{\mathcal{C}}

\providecommand{\cG}{\mathcal{G}}
\providecommand{\cH}{\mathcal{H}}

\providecommand{\cL}{\mathcal{L}}

\providecommand{\cN}{\mathcal{N}}
\providecommand{\cO}{\mathcal{O}}
\providecommand{\cP}{\mathcal{P}}

\providecommand{\cS}{\mathcal{S}}

\providecommand{\cV}{\mathcal{V}}

\newcommand{\lH}{\mathcal{H}}

\newcommand{\lL}{\mathcal{L}}

\newcommand{\lV}{\mathcal{V}}

\DeclareMathOperator {\GL} {GL}
\DeclareMathOperator {\vol}  {vol}

\providecommand{\set}[1]{\left\{#1\right\}}
\DeclareMathOperator {\Res}  {res}
\newcommand{\isom}{\stackrel{\sim}{\rightarrow}}

\usepackage{scalerel}

\DeclareMathOperator{\cind}{c\--Ind}

\title[On the Ramanujan conjecture over function fields]{On the Ramanujan conjecture for 
automorphic forms over function fields I. Geometry}
\author{Will Sawin}
\address{ETH Institute for Theoretical Studies \\ ETH Zurich \\ 8092 Z\"{u}rich, 
Switzerland}
\author{Nicolas Templier}
\address{Department of Mathematics \\ Cornell University \\ Ithaca, NY 14853, USA}

\begin{document}

\begin{abstract} Let $G$ be a split semisimple group over a function field.
We prove the temperedness at unramified places of automorphic representations of 
$G$, subject to a local assumption at one place, stronger than supercuspidality, and 
assuming the existence of cyclic base change with good properties. 
Our method relies on the geometry of $\Bun_G$. 
It is independent of the work of Lafforgue on the global Langlands correspondence.
\end{abstract}

\maketitle
\thispagestyle{empty}
\tableofcontents

\section{Main result}

Let $F$ be the function field of a smooth projective curve over a finite field $k$.
\index{$F=k(X)$, global function field}
The Ramanujan conjecture that every cuspidal automorphic representation of $\GL(r)$ with 
unitary central character is 
tempered is
established by L.~Lafforgue~\cite{Lafforgue:chtoucas}.
For general reductive groups, cuspidal 
automorphic representations that are known to be 
tempered arise in the works of Lomeli~\cite{Lomeli:functoriality-classical} for 
generic representations of split classical groups, and of 
Heinloth--Ng\^o--Yun~\cite{HNY:Kloosterman} and 
Yun~\cite{Yun:epipelagic,Yun:motives-Serre} for rigid representations.

For a reductive group $G$, it is well-known that the cuspidality condition is not 
sufficient to imply temperedness, which
led to the formulation of Arthur's 
conjectures~\cite{Arthur:unipotent-II}.  
For example, there are two classical constructions of cuspidal non-tempered automorphic 
representations for $\Sp_4$ by Saito--Kurokawa and
Howe--Piatetskii-Shapiro~\cite{Howe-PS:corvallis}.

Thus, if we want to prove that $\pi$ is tempered, we need a condition on $\pi$ stronger than cuspidality. We shall impose that $\pi_u$ 
is supercuspidal for one place $u$.
This is still not sufficient as the above examples~\cite{Howe-PS:corvallis} show, and
Arthur's conjecture points towards the condition that $\pi_u$ belongs to a supercuspidal 
$L$-packet.
We shall introduce a further condition that $\pi_u$ is \emph{monomial 
geometric 
supercuspidal}, and establish the Ramanujan bound in this case.
The concept will be discussed in detail below. In brief it means that $\pi_u$ is 
compactly induced 
	from a character on a ``nice enough" open subgroup of $G(F_u)$.
We also need another \condbc{} from Section~\ref{s:BC} below, on the existence of 
an automorphic base change for constant field extensions.

\begin{theorem}\label{t:intro:main}
Assume that $G$ is split semisimple, and that $\operatorname{char}(F)>2$. 
Suppose that
\begin{itemize}
\item for at least one place $u$, the representation $\pi_u$ is monomial geometric 
supercuspidal; 
\item $\pi$ is base-changeable in the sense of \condbc{}. 
\end{itemize}
Then $\pi$ is tempered at every unramified place.
\end{theorem}

Langlands theorem on the analytic continuation of Eisenstein series implies that CAP 
representations are non-tempered at every unramified place. Combined with 
Theorem~\ref{t:intro:main}, it follows that $\pi$ is not CAP.

\begin{remark}\label{rem:unramified}
Recently, V.~Lafforgue~\cite{Lafforgue:reductifs-chtoucas} constructed 
global parameters using shtukas and excursion operators. 
An automorphic consequence is that $\pi$ is tempered at one unramified place if and only 
if it is tempered at every unramified place (Theorem \ref{temperedOneplace} below), which 
was~\cite[Conj.4(1)]{Clozel:park-city}.
\end{remark}

The present paper focuses on establishing a Ramanujan bound on 
average, see~\eqref{intro:average} below, and deducing~Theorem~\ref{t:intro:main}. 
It is part of a series of two articles, and  
the next~\cite{Sawin-Templier:II-bc} will focus on providing examples of 
representations that satisfy 
\condbc{}, and on establishing the functorial image between inner-forms which will enable 
us to reduce cases of the Ramanujan bounds for general reductive groups to the split 
semisimple case. 

\subsection{Monomial geometric supercuspidal representations 
(mgs)}\label{sub:intro-mgs}
The definition of monomial geometric supercuspidal is motivated by features of the 
problem and our method to attack it.

We rely on studying families defined by \emph{local prescribed behavior}, which means in 
our context a set of automorphic representations of $G(\A_F)$ that satisfy some given 
conditions at a fixed finite set of places and are unramified 
outside. If we can show temperedness for one member of the family by our method, the 
same argument applies to every member of the family. So we must impose strong 
enough local conditions. 
At minimum, we 
should avoid Eisenstein series, and, for at least one place $u$, requiring that $\pi_u$ be supercuspidal is the easiest way to 
achieve this.

Our method is geometric, and requires a geometric way to check the local condition. We 
focus on \emph{monomial local conditions}. These are the conditions defined by fixing a 
subgroup $J$ of $G(F_u)$ and a character $\chi: J \to \mathbb C^\times$, and demanding that the 
local representation $\pi_u$ of $G(F_u)$ contains a vector that transforms according to 
$\chi$ under the action of $J$. There is a natural geometric description of the set of 
automorphic forms satisfying a monomial local condition as long as $J$ is the group of 
$k$-rational points of a pro-algebraic subgroup of the loop group $G\bseries{t}$ and 
$\chi$ is the trace function of a character sheaf. This is certainly not the most general 
possible way to construct a geometric object that defines a local condition on 
automorphic representations --- in fact 
the 
geometric Langlands program suggests that there should be geometric objects corresponding 
to all automorphic  representations, in a suitable sense --- but it is easy to work with 
and contains many important examples. A general formalism of monomial local conditions 
for automorphic representations was already used by Yun~\cite[\S2.6.2]{YunRigidity}. Our 
setup (Section~\ref{s:geometric-setup}) is essentially Yun's formalism restricted to a 
special case for both geometric and notational simplicity (and for this reason we use 
somewhat different notation).

Geometric objects behave similarly over different fields. In our case, the relevant 
geometric objects are defined over the constant field $k$, and so it is possible to base 
change them along a constant field extension. If we use any geometric property to prove 
temperedness, this property will be maintained over constant field extensions, and so 
temperedness must hold not only for all members of the family, but also for all members 
of 
the analogous family after extension of the constant field $k$. In particular, these 
representations must not be Eisenstein series. Again, the easiest way to ensure this is 
to ensure that our character $(J,\chi)$ still prescribes a supercuspidal representation 
after a constant 
field extension.
This yields the notion of \emph{monomial geometric supercuspidal datum}
(Definition~\ref{def:geometric-ind-data}).

Another advantage of adding the monomial and geometric modifiers to the supercuspidal
local condition is that it allows us to sidestep the unipotent supercuspidal 
representations. The usual construction of these is not by a monomial representation but 
rather from representations of finite groups of Lie type.
We expect that no monomial geometric construction of unipotent representations exists.
For example in Deligne--Lusztig theory, irreducible representations are induced from 
characters on elliptic tori, but this fails to work uniformly after finite field extensions, 
since every torus eventually splits.

The local conditions we define are \emph{geometric} in precisely the sense of the geometric 
Langlands program. However, there is one major difference in our approach. Progress in 
the geometric Langlands program has mainly focused on first studying automorphic forms 
that are 
everywhere unramified, and then generalizing to unipotent or tame ramification, 
before beginning to tackle the general case. In our problem, we find it is convenient to 
study highly ramified automorphic forms  --- in particular, including 
local factors with wildly ramified Langlands parameters --- which necessitates working 
in a 
more general setup. We do this because when one of the local factors is supercuspidal, 
the Hecke kernels in the family will correspond to pure perverse sheaves 
(Theorem~\ref{mainduality-semisimple}), although we also believe the more general setup 
is interesting on its own terms. 

More formally, let $G$ be a quasi-split reductive group over a field 
$k$. We start with the datum of a pro-algebraic subgroup $H$ of $G\bseries{t}$ 
containing the subgroup of elements congruent to $1$ modulo $t^m$ for some $m$, and a 
character sheaf $\cL$ on $H$ which is trivial on that subgroup. We say this datum is 
\emph{geometrically supercuspidal} if for every parabolic subgroup 
$P\subset G_{\overline k}$ with radical $N$, and every $g\in G_{\overline 
k}\bseries{t}$, the restriction of 
$\cL_{\overline k}$ to the identity component of
$gN_{\overline{k} } \bseries{t} g^{-1}\cap H_{\overline k}$ is non-trivial. (The intersection takes place 
in $G_{\overline k}\bseries{t}$.)

If $k=\F_q$ is a finite field, this occurs if and only if
$\cind^{G(\mathbb F_{q^n} \pseries{t})}_{J_n} \chi_n$ is admissible supercuspidal for 
every 
integer $n\ge 
1$, where $J_n:=H(\mathbb F_{q^n} )$ and $\chi_n$ is the trace function of $\mathcal L$ over 
$\mathbb F_{q^n}$ (Lemma~\ref{checking-supercuspidality}).

\subsection{Ramanujan bound for $\GL(r)$}\label{sub:intro:gln}
For the general linear group, the Ramanujan bound is the statement that a cuspidal 
automorphic representation of $\GL(r)$ with unitary central character is tempered at 
every 
place.
One can distinguish two main approaches: 
\begin{itemize}
\item Laumon~\cite{Laumon:book:cohomology-Drinfeld} under a 
cohomological condition at one place, 
extending Drinfeld's first proof~\cite{Drinfeld:elliptic-II} for $\GL(2)$, using elliptic 
modules. 
\item L.~Lafforgue~\cite{Lafforgue:chtoucas} in general, extending Drinfeld's second 
proof~\cite{Drinfeld:Petersson} for $\GL(2)$, using shtukas.
\end{itemize}

Our approach is yet different, even in the case of $\GL(r)$, under the mgs (monomial geometric supercuspidal) condition. 
Rather than using moduli 
spaces of elliptic modules or shtukas, we study moduli spaces $\Bun_{\GL(r)}$ of 
vector bundles, as in the 
geometric Langlands program. Functions on these moduli spaces give rise to families of 
automorphic forms satisfying certain local prescribed conditions. We will prove 
temperedness using 
estimates for an entire family at once, rather than working with 
individual automorphic forms in the family.

\subsection{Outline of the proof}

We embed $\pi$ in a suitable \emph{automorphic family} $(\lV_n)_{n\ge 1}$. We let $\lV_1$ 
consist of the 
multi-set 
of 
automorphic 
representations $\Pi$ of $G(\A_F)$, counted with multiplicities, such that $\Pi_u$ has a 
non-zero $(J,\chi)$-invariant 
vector, $\Pi$ has bounded ramification at a fixed finite set of places, and $\Pi$ is unramified 
elsewhere. The ramification bound is chosen compatibly with the original representation $\pi$ in 
such a way that $\pi\in \lV_1$. Since $(J,\chi)$ arises from a supercuspidal datum, all $\pi\in 
\lV_1$ are supercuspidal.

For every integer $n\ge 1$, consider the constant field extension $F_n:= F 
\otimes_{\F_q} \F_{q^n}$, assuming $k=\F_q$. We let
$\mathcal V_n$ consist of
automorphic representations of $G(\A_{F_n})$ with similar bounded ramification and with 
mgs prescribed 
behavior at the places of $F_n$ above $u$, namely with a non-zero $(J_n,\chi_n)$-invariant 
vector. Again all $\Pi\in \mathcal V_n$ are cuspidal.

Let $v\in X(k)$ be a $k$-rational point such that $\pi_v$ is unramified. To study the 
temperedness of $\pi_v$, we shall consider 
the local components $\Pi_v$ for $\Pi\in \lV_n$. More precisely, for a coweight 
$\lambda$ of $G$, we shall consider the collection of all traces of Hecke operators 
$\tr_\lambda(\Pi_v)$ for $\Pi\in \lV_n$.

We express the kernel of this Hecke operator as the trace function of a complex of 
sheaves, which we will show, as consequence of our mgs local prescribed behavior, is a 
pure perverse sheaf 
(Theorem~\ref{mainduality-semisimple}). This will 
imply, by standard estimates for the trace functions of perverse sheaves, a bound for the 
trace of a Hecke operator in the family (Theorem~\ref{t:average-Ramanujan}), which takes 
the form
\begin{equation}\label{intro:average}
\sum_{\Pi \in \lV_n} |\tr_{\lambda}(\Pi_v)|^2
\lesssim C_\lambda \cdot q^{nd}
\end{equation}
Here $d$ depends on the underlying group $G$ and the prescribed conditions, and $C_{\lambda}$ 
is 
the 
dimension of some cohomology groups and it is essential for us 
that it is independent of $n$ (it only depends on the underlying group $G$, the fixed local 
prescribed conditions, and the chosen unramified place $v$).

If we first examine the $\lambda=0$ case, we see that the number of automorphic 
representations in the family is at most $C_0 \cdot q^{nd}$. This bound should be close to 
the truth --- one expects that the sum on the geometric side of the trace formula for 
the 
number of automorphic forms in the family $\lV_n$ is dominated by the contribution of the 
trivial conjugacy class, which is an adelic volume, and one can show this adelic volume $\approx 
C \cdot q^{nd}$ for another explicit constant $C$.

Furthermore, the Ramanujan bound would imply  $|\tr_{\lambda}(\Pi_v)| \leq \dim(V_\lambda)$, 
so conditionally on the Ramanujan bound for all representations of $\lV_n$, we obtain
\[ \sum_{\Pi \in \lV_n} |\tr_{\lambda}(\Pi_v)|^2
\lesssim  C \cdot \dim (V_\lambda)^2 \cdot  q^{nd} .\]
Thus, \eqref{intro:average} is as strong as the Ramanujan bound on average over the family 
$\mathcal V_n$, except that the constant $C_\lambda$ has unknown dependence on $\lambda$, 
whereas in the Ramanujan bound on average the constant $\dim (V_\lambda)^2 $ has explicit, 
mild dependence on $\lambda$. 

This suggests that we are on the right track, but that the constant $C_\lambda$ is 
problematic.

Here comes the final step. Because $C_\lambda$ is constant in $n$ while every other term is 
exponential in $n$, the 
 quality of the estimate~\eqref{intro:average} improves as 
$n$ goes to infinity. To take advantage of this, we will use automorphic base change for 
constant field 
extensions $F_n/F$ to amplify 
the estimate, and deduce $|\tr_{\lambda}(\pi_v)| \le \dim(V_{\lambda}) \cdot 
q^{\frac{d}{2}}$ for our 
original representation $\pi$.
Varying $\lambda$,  we can further bootstrap this estimate to 
\[
|\tr_{\lambda}(\pi_v)| \le 
\dim(V_{\lambda}),
\]
 which is the temperedness 
of the unramified representation $\pi_v$.

\begin{remark}
Recall from~\cite{Clozel:park-city} the following conjecture:
$\pi$ should be tempered at every 
unramified place as soon as $\pi_u$ is the 
Steinberg representation for some place $u$. Compared to this, our situation consists in
replacing the Steinberg condition by 
a more ramified condition.
Our method of proof doesn't extend to the case of the Steinberg representation because 
the Euler--Poincar\'e function is an alternating sum, which we do not know how to 
geometrize globally to a pure 
sheaf on $\Bun_G$. 
\end{remark}

\subsection{Contrasting Drinfeld's modular varieties and 
$\Bun_G$}\label{sub:Drinfeld-var}

This subsection does not directly describe our argument, but we hope it provides some 
intuition that will be helpful to the reader.

The moduli spaces of shtukas and $\Bun_G$ are both stacks whose geometries carry
information about automorphic forms over function fields, but they carry it in different 
ways and have different properties.

Each moduli space of shtukas can be related to a particular family of automorphic forms 
with a particular set of Hecke operators acting on it. For example, the moduli space of 
shtukas $\operatorname{Cht}_{D, I,W}^{(I)}$ defined in 
\cite[Def.0.2]{Lafforgue:reductifs-chtoucas} can be related to the family of automorphic 
forms of 
level $D$ on $G(\A_F)$, with the set of Hecke operators determined by the representations 
$W$. 

 The geometry of the moduli space 
casts light on this family. More precisely, the cohomology of the moduli space $\operatorname{Cht}_{D, I,W}^{(I)}$ relative 
to the base is expected to be a sum over automorphic forms of level $D$ of 
local systems constructed from their Langlands parameters 
\cite[Rem.0.30]{Lafforgue:reductifs-chtoucas}. 
The arithmetic structure on the moduli space carries additional information about the 
automorphic forms in this family. For instance, the Galois action on the cohomology of a 
moduli space of $\GL(r)$-shtukas with level structure determines the Galois action on the 
Langlands parameters of the cusp forms of that level \cite[Lem.VI.26 and 
Thm.VI.27]{Lafforgue:chtoucas}. 

On the other hand, $\Bun_G$ is related to a 
sequence $\lV_n$ of spaces of automorphic forms, one over each finite field extension 
$\mathbb F_{q^n}$ of the base 
field $\mathbb F_q$. In fact, the set of rational points $\Bun_G(\mathbb F_{q^n}) $ is the quotient of $G(\mathbb F_{q^n}(X) ) \backslash G(\mathbb A_{ \mathbb F_{q^n}(X)})  $ by a maximal compact subgroup, so the space of functions on $\Bun_G(\mathbb F_{q^n}) $ is the space of automorphic forms of level $1$ on $G_{ \mathbb F_{q^n}}(X)$. Thus, the space $\Bun_G$ contains information about automorphic forms of level $1$ on $G_{ \mathbb F_{q^n}}(X)$ for all $n$.  (Variants of $\Bun_G$ with level structure hold the same relationship to spaces of automorphic forms of higher level.) 

Because geometry is insensitive to base change, the geometry of $\Bun_G$ is only 
related to asymptotic information about these spaces of automorphic forms as $q^n\to \infty$ 
(or possibly other subtler sorts of information that are invariant on passing to 
subsequences). For instance, by the Lefschetz fixed point formula, the dimension of the 
space of automorphic forms of level $1$ on $G_{ \mathbb F_{q^n}}(X)$  equals the number of $\mathbb F_{q^n}$-points of $\Bun_G$ which equals the 
supertrace of Frobenius on the cohomology of $\Bun_G$ (Lemma~\ref{arithmetic-to-geometry} and Proposition~\ref{p:spectral-trace}), so the cohomology of $\Bun_G$ gives 
information about the dimension of all the 
spaces of automorphic forms in the sequence. (However, for any nontrivial $G$, there 
exists some $n$ such that $\Bun_G$ will have infinitely many $\mathbb F_{q^n}$-points. To 
rigorously relate cohomology to counting automorphic forms we must make this count 
finite, 
which requires us to fix a central character, and, in addition, do something to remove 
Eisenstein series. In our paper the supercuspidal local prescribed
conditions discussed in \S\ref{sub:intro-mgs} are used to remove the Eisenstein series.)

This fundamental difference can explain many of the more basic differences between the 
geometry of the moduli space of shtukas and $\Bun_G$ --- for instance, their dimensions. 

The dimension of the moduli space of shtukas $\operatorname{Cht}_{N, I,W}^{(I)}$ depends on 
the group $G$ and on the representations $W_i$ of the Langlands dual group occuring at 
the 
legs $i\in I$, but does not depend on the level $N$ --- in fact, moduli spaces of shtukas of 
higher 
level are finite \'{e}tale covers of moduli spaces of shtukas of lower level.  On the other 
hand, the dimension of the moduli space $\Bun_{G(N)} $ of $G$-bundles with level $N$ 
structure depends on both the group $G$ and the level $N$, while the representations $W$ 
do not appear in the definition.

We can explain this discrepancy between the dimensions of $\operatorname{Cht}_{N, I,W}^{(I)}$ and $\Bun_{G(N)}$ by 
looking at how the dimension is reflected in the associated spaces of automorphic forms. 
Recall here that the dimension of a space determines the largest possible size of 
Frobenius eigenvalues on its compactly supported cohomology. (Of course, in each case it 
is possible to calculate the dimensions much more directly than this. The point of this 
argument is to see why the simple concrete properties of these two spaces are 
necessary for their respective applications.)

We expect the cohomology of the moduli space of shtukas $\operatorname{Cht}_{N, I,W}^{(I)}$  to be a sum of contributions associated to different automorphic forms, with each contribution the tensor product over legs $i$ of the representation $W_i$ composed with the Langlands parameter. The size of the Frobenius eigenvalues acting on $W_i$ depends on the weights of the representation $W_i$. On the other hand, there is no reason for highly ramified Langlands parameters to have different Frobenius eigenvalues from less ramified parameters. (For instance, because Langlands parameters can become more or less ramified under pullback, without changing their Frobenius weights.)  Thus, it is reasonable to expect that the dimension depends on the choice of $W_i$, but not on the level. 

On the other hand, the Frobenius eigenvalues on the cohomology of $\Bun_{G(N)}$ are 
relevant because they give a formula for the dimension of the spaces of automorphic 
forms of level $N$ on $G(\A_{ \mathbb F_{q^n}(X)})$. In particular, as $n$ goes to $\infty$, the 
largest Frobenius eigenvalue should dominate, and so the largest Frobenius eigenvalue 
should match the asymptotic growth rate in $n$ of the dimension of this space of 
automorphic forms. We can calculate the dimension of this space of automorphic forms 
by the trace formula, where the main term is one over the volume of the level $N$ 
subgroup of $G (\mathbb A_{ \mathbb F_{q^n}}(X)$. This inverse volume grows with both the 
degree $n$ and level $N$ --- in fact, it is approximately $q^{  n (\dim G) (g + |N|-1 ) } $, 
where 
$|N|$ is the degree of the divisor $N$. Thus, it is reasonable to expect the dimension of 
$\Bun_{G(N)}$ is $ (\dim G) (g+ |N|-1)$, as indeed it is.

Similarly, the number of forms of level $N$ on $G(\A_{ \mathbb F_{q^n}(X)})$ with a nonzero 
$(J,\chi)$-equivariant vector, is approximately $q^{ n ( (\dim G) ( g+ |D|- 1) - \dim 
H)}$(see~\S\ref{sub:sums-Weil}).

This also suggests differences in their potential arithmetic applications. The moduli 
spaces of shtukas are well-suited to prove the automorphic-to-Galois direction of the 
Langlands correspondence because each automorphic form, and its associated Langlands 
parameter, appears in their cohomology. Of course this is exactly why 
Drinfeld~\cite{Drinfeld:elliptic-II} introduced 
them and how L.~Lafforgue~\cite{Lafforgue:chtoucas} and 
V.~Lafforgue~\cite{Lafforgue:reductifs-chtoucas} used 
them, and it seems likely that researchers 
will continue to deduce 
information about the Langlands correspondence from study of these 
moduli spaces in the future. But $\Bun_G$ is not well-suited for this purpose, as with 
the number of 
automorphic forms going to infinity as $q^n \to \infty$, it is harder to pick out a 
single one. 
Though an analogue of the automorphic-to-Galois direction of the Langlands correspondence 
is part of the geometric 
Langlands program over the complex numbers, it is not clear what, if any, the finite 
field analogue might be.

On the other hand, $\Bun_G$ does seem well-suited to answer 
asymptotic questions about how analytic quantities, such as averages of Hecke operators, 
behave when $q^n \to \infty$, as we demonstrate in the present paper. The 
Ramanujan bound and Arthur's conjectures seem to lie in the intersection of these two 
domains --- it can be 
attacked using Langlands parameters, but also can be viewed as a question of the $q^n \to 
\infty$ limit. Thus there is potential to use both approaches to prove new cases of Arthur's 
conjectures.

\subsection{Results on families} 

Because our method to prove the main theorem relies on families of automorphic forms 
defined by geometric monomial local conditions, along the way we obtain some new results 
about these families. We expect further results can be obtained this way using our work 
in the future. For this reason we discuss the strengths and weaknesses of restricting to 
monomial representations from the point of view of families (rather than with regards to 
proving the Ramanujan bound for individual automorphic forms). Given a family of 
automorphic forms unramified away from some finite set of places, and defined by some 
local conditions at the remaining places, questions such as the following have been considered:

\begin{enumerate}

\item Can the number of forms in the family be expressed as a finite sum of Weil numbers?

\item What about the trace of a Hecke operator on this space of forms?

\item Can the Weil numbers that appear in these sums be calculated explicitly?

\item Can these sums be approximated, or can the largest Weil numbers appearing in them 
be estimated?

\end{enumerate}

Question (1) and question (3) were answered affirmatively by 
Drinfeld~\cite{Drinfeld:count} in the case of everywhere unramified automorphic forms on 
$\GL(2)$, by Flicker for forms on $\GL(2)$ that are Steinberg at 
one place and unramified everywhere else~\cite{Flicker-GL2}, by Deligne and 
Flicker~\cite{Deligne-Flicker} for forms on $\GL(r)$ that are Steinberg at at least two 
places, and unramified everywhere else, and by Yu~\cite{YuNumber} for forms on $\GL(r)$ 
that are unramified everywhere. Of course answering (3) is sufficient to answer 
question (4).

In this paper we answer question (1) in the case of monomial geometric conditions, 
supercuspidal at at least one place, and unramified elsewhere 
(Proposition~\ref{p:count-Weil}).
And most importantly we answer question (2), in the form 
that $\sum_{\Pi\in \cV_n} q^{n\langle \lambda, \rho \rangle} |\tr_\lambda(\Pi_v)|^2$ is 
a signed 
sum of length $C_\lambda$ of $n$th powers of $q$-Weil integers of weight $\le 2d+\langle 
\lambda,2\rho\rangle$. This 
 is actually how we establish the main estimate~\eqref{intro:average}.
See~Theorem~\ref{weil-numbers-semisimple} and \S\ref{sub:sums-Weil} for details.

\subsection{Local prescribed behavior}
There are many different kinds of local conditions that appear in the theory of 
automorphic forms.  As mentioned before, we work with local conditions that demand the 
representation contain an eigenvector of a compact open subgroup $J$ with eigenvalue 
$\chi$, where $J$ and $\chi$ arise from geometric objects ---  an algebraic subgroup of 
$G(\kappa[t]/t^m)$ for some $m$ and a character sheaf on that algebraic subgroup. The 
theory of inertial types produces many examples where this condition, for a suitable 
choice of $(J,\chi)$, characterizes the representation up to an unramified twist (e.g. the 
twist-minimal supercuspidal representations of $\GL(2)$ with conductor not congruent to 
$2$ 
modulo $4$). However, not all representations can be characterized up to an unramified 
twist this way (e.g., the twist-minimal supercuspidal representations of $\GL(2)$ with 
conductor congruent to $2$ mod $4$). But it may still be possible to characterize the 
representation up to a tamely ramified twist or other mild variant. 

Choosing $(J,\chi)$ whose associated local condition uniquely picks out a given 
representation is very similar to the problem of constructing the representation as an 
induced representation (but slightly easier as one is allowed to produce the 
representation with multiplicity). Yu has shown how to construct a wide class of 
supercuspidal representations using Deligne--Lusztig representations of algebraic groups 
over finite fields and Heisenberg--Weil representations. (For instance, in the $\GL(2)$ 
twist-minimal case with conductor congruent to $2$ mod $4$, Deligne--Lusztig theory is 
needed for conductor $2$ and Heisenberg--Weil representations are needed for higher 
conductor). 

 The matrix coefficients of the Weil representation were 
 expressed as the trace function of a perverse sheaf in a 1982 letter of Deligne, and the 
 same was done in \cite{GurevichHadani} to the coefficients in a basis consisting of the 
 matrices appearing in the Heisenberg representation. It is likely that much of what we 
 do can be generalized using this geometrization. Sheaves whose trace functions are the 
 traces of discrete series representations were constructed \cite{Lusztig85} but we do 
 not know if there is any way to do the same for matrix coefficients (it is not clear 
 what basis to use). It could also be possible to replicate our methods using just the 
 trace and not all the matrix coefficients, but we are less certain of it.  
 
 Using these tools to make these representations geometric would follow the strategy of 
 \cite{CR17}. Note, however, some differences with their work. Their goal was to 
 geometrize the trace of the automorphic representation, while our construction has the 
 effect of geometrizing a test function, and they handled $p$-adic groups while we work 
 in the equal characteristic case.
 
 For our problem, new difficulties appear when adding Heisenberg--Weil and 
 Deligne--Lusztig representations and their 
 more complicated sheaves. Because restricting to one-dimensional characters, and their 
 associated character sheaves, will simplify things at several points, we leave the full 
 theory to a later date.

\section{Preliminaries}

\subsection{Unramified groups}
Let $k$ be a finite field.
We say a connected reductive group over $k\pseries{t}$ is \emph{unramified} if it is quasi-split 
and splits over $\overline{k}\pseries{t}$. 
 The following is well-known. Since we couldn't locate the result in the literature, we 
 provide a quick proof.
\begin{lemma}\label{l:unramified}
	An unramified group over $k\pseries{t}$ is the base change 
	$G_{k\pseries{t}}$ of a reductive group $G$ over $k$.
\end{lemma}

\begin{proof}
Bruhat--Tits~\cite[\S4.6.10]{BT84}, and 
\cite[Chap.II]{Landvogt:LNM:compactification-BT}, establish the existence of a model 
$\cG$ that is a smooth 
affine group 
scheme over 
$k\bseries{t}$, with reductive special fiber. Let $G:=\cG_\kappa$ be this special fiber.
According to~\cite[Rem.7.2.4]{Conrad:reductive-gp-schemes}, the classification of 
forms 
of a 
reductive group over a Henselian local field with finite residue field is the same as the 
classification over the residue field.
Indeed let $\mathcal{G}$, and $\mathcal{G}'$ be two connected reductive group schemes 
over $k\bseries{t}$. Suppose their special fibers over $k$ are isomorphic. The scheme of 
isomorphisms
from $\mathcal{G}$ to $\mathcal{G}'$ is smooth, and has a point over $k$, so has a 
section over $k\bseries{t}$. In particular if we take $\mathcal{G}'$ to be a
constant group scheme $G$, we get that $\mathcal{G}$ is constant as well.
\end{proof}

\begin{remark}
The same notion of unramified group arises in mixed characteristic, that is over a 
finite extension $K$ of $\Q_p$.
In that context, it is standard that there is a smooth model $\mathcal{G}$ over the local 
ring $\mathfrak{o}_K$, and that $\mathcal{G}( \mathfrak{o}_K )$
is a hyperspecial maximal subgroup.  
This is analogous to Lemma~\ref{l:unramified}, where the model is given by 
$G_{k\bseries{t}}$, and the hyperspecial maximal subgroup by
$G(k\bseries{t})$, only that in equal characteristic the statement is
simpler, and it is not necessary to introduce the group scheme $\mathcal{G}$. 
In mixed characteristic, the lifting argument still works, but there is no notion of 
constant group scheme over $\mathfrak{o}_K$ (though an analogue could likely be 
constructed using Witt vectors).
\end{remark}

\begin{lemma}\label{trivialization-generic} Let $G$ be a reductive group over a finite field $k$. Let $X$ be a smooth connected algebraic curve over $k$. Then every $G$-torsor on $X$ admits a trivialization over the generic point. \end{lemma}

\begin{proof} Let $F = \mathbb F_q(X)$. By \cite[Lem.1.1]{Ngo06}, it is sufficient to 
check that the kernel $\ker^1(F,G)$ of the natural map from $H^1(F,G)$ to the product 
over all places $x$ of $H^1(F_x,G)$ is trivial. By \cite[Thm.2.6(1)]{thang11}, the 
kernel $\ker^1(F,G)$ is Pontryagin dual to $\ker^1(F, Z(\widehat{G}))$. To show that $\ker^1(F, 
Z(\widehat{G}))$ is trivial, it suffices to fix a nontrivial $F$-torsor $\mathcal T$ of $Z(\widehat{G})$ 
and show it remains nontrivial upon restriction to some place.

We can describe an $F$-torsor $\mathcal T$ by the action of  $\operatorname{Gal}(F)$  on $\mathcal 
T_{\overline{F}}$, where  $\mathcal T_{\overline{F}}$ is a $Z(\widehat{G})_{\overline{F}}$-torsor in the sense 
of algebraic groups. Because torsors are by definition trivial over some \'{e}tale open set, 
this $\operatorname{Gal}(F)$-action must factor through a finite group $H$.

If the $\operatorname{Gal}(F)$-action on $\mathcal T_{\overline{F}}$ factors through $\operatorname{Gal}(k)$, 
then we can take $H$ to be a finite quotient of $\operatorname{Gal}(k)$, necessarily cyclic. Thus 
the Frobenius element at any place of degree prime to $|H|$ generates $H$, and so $\mathcal 
T$ is nontrivial if and only if it is nontrivial at one of these places.

If the $\operatorname{Gal}(F)$-action on $\mathcal T_{\overline{F}}$ does not factor through 
$\operatorname{Gal}(k)$, then because the $\operatorname{Gal}(F)$-action on $Z(\widehat{G})_{\overline{F}}$ 
does factor through $\operatorname{Gal}(k)$, we may find some conjugacy class $\sigma \in H$ 
which acts trivially on $Z(\widehat{G})_{\overline{F}}$ but nontrivially on $\mathcal T_{\overline{F}}$. By 
the Chebotarev density theorem, there exists some place $v$ such that the image of 
$\operatorname{Frob}_v$ in $H$ is conjugate to $\sigma$. The restriction $\mathcal T_v$  of $\mathcal T$ 
to $v$ must be nontrivial because, since $\operatorname{Frob}_v$ acts trivially on 
$Z(\widehat{G})_{\overline{F}}$, it acts trivially on the trivial torsor over $Z(\widehat{G})_{\overline{F}}$, 
so $\mathcal T_v$ cannot be isomorphic to the trivial torsor over $Z(\widehat{G})$ as a set with 
Frobenius action.

(In the case when $G$ is split simply-connected semisimple, this result could instead be deduced from a 
result of Harder~\cite[Thm.2.4.1]{Harder:Chevalley-fn-fields} that if $G$ is split and 
simply-connected semisimple, then $H^1(F,G)$ is trivial.) 
\end{proof}

\subsection{Satake isomorphism}\label{sub:Satake}
In this subsection, let $G$ be a split connected reductive group over a finite field $k$. 
Let 
$F=k\pseries{t}$, $\ko=k\bseries{t}$, $K=G(\ko)$, and 
consider the \emph{unramified Hecke algebra}
\[
\cH(G) = \cH(G(F),K) = \cC_c(K \backslash G(F) / K,\C).
\]
We are fixing the Haar measure on $G(F)$ to give $K$ volume one.
The below results hold more generally over the base ring $\Z[q^{\frac12},q^{-\frac12}]$ 
rather than $\C$.
Let $T\subset G$ be a maximal torus.
There is an identication of the lattice $\Lambda:=X_*(T)$ of coweights of $G$ with the 
quotient group $T(F)/T(\ko)$, 
where a cocharacter $\mu:\G_m \to T$ corresponds to the element $\mu(t)\in T(F)$ modulo 
multiplication by $T(\ko)$.
This induces an algebra isomorphism 
$
\cC_c(T(F)/T(\ko))\simeq \C[X_*(T)]$.
The Weyl group $W = N_G(T) / Z_G(T)$ acts on both sides of this isomorphism, in 
particular we can form the subalgebras 
of $W$-invariant functions
\[
\cC_c(T(F)/T(\ko),\C)^W
\simeq \C[X_*(T)]^W.
\]

Choose a Borel subgroup $B=TU$ and let $\Lambda^+\subset \Lambda$ be the positive Weyl 
chamber. Let 
$\delta:B(F)\to q^\Z$ be the modulus character, where 
$q$ is the size of $k$.
Denote by $\delta^{\frac12}:B(F)\to q^{\frac12\Z}$, the positive square-root.
For every $\mu\in \Lambda$, we have 
$\delta^{\frac12}(\mu(t)) = q^{-\langle 
\rho,\mu 
\rangle}$, 
where 
$\rho\in X^*(T)\otimes_\Z \C$ 
is the half-sum of the positive roots. 
The \emph{Satake transform} $\cS(f)$ of a function $f\in \cH(G)=\cH(G(F),K)$ is defined by
\[
\cS(f)(s) := 
\delta^{\frac12}(s) 
\int_{U(F)} 
f(su) du
,\quad s\in T(F) \subseteq B(F),
\]
where $du$ is the Haar measure on $U(F)$ that gives $U(F)\cap K= U(\ko)$ volume one.
The value of the integral depends only on $s$ modulo $T(\ko)$.
It induces an algebra isomorphism~\cite{Gross:satake} 
\[
\cS:\cH(G) \xrightarrow[]{\sim} \cC_c(T(F)/T(\ko),\C)^W,
\] 

Denote by $V_\lambda$ the irreducible representation of 
$\widehat G(\C)$ with 
highest weight 
\[
\lambda\in \Lambda^+ \subset \Lambda=  X_*(T)=X^*(\widehat T).
\]
The trace of a finite dimensional representation $V$ of $\widehat G(\C)$ can be viewed as an 
element of 
$\C[X^*(\widehat T)]^W$ by recording its weight spaces multiplicities $\dim 
\operatorname{Hom}_{\widehat T(\C)}(\mu, V)$ for all $\mu\in X^*(\widehat T)$, hence it corresponds 
to 
an element of the 
Hecke algebra $\cH(G)$ 
under the Satake isomorphism.
In particular, the trace of the representation $V_\lambda$ 
is of the form 
\begin{equation}\label{Salambda}
\operatorname{tr}(V_\lambda) = \sum_{\mu \in X^*(\widehat T)} \dim \operatorname{Hom}_{\widehat 
T(\C)}(\mu,V_\lambda)\cdot [\mu]  = 
\cS(a_\lambda)
\end{equation}
 for a \emph{unique element} $a_\lambda\in 
\cH(G)$.
As we vary $\lambda \in \Lambda^+$, the elements $a_\lambda$ form a linear basis of $\cH(G)$ since 
$\tr(V_\lambda)$ form a linear basis of $\C[X^*(\widehat T)]^W$ in view of highest weight theory.

\begin{proposition}[Satake] \label{p:Satake}
There is a bijection between isomorphism classes of irreducible $K$-unramified 
representations $\pi$, algebra homomorphisms $\tr(\pi):\cH(G) \to \C$, $W$-conjugacy 
classes of unramified characters 
$\chi:T(F)/T(\ko)\to 
\C^\times$, and semisimple conjugacy classes $t_\pi$ in $\widehat G(\C)$, characterized as 
follows:

(i) The bijection $\pi\mapsto t_\pi$ coincides with the composition 
of the two 
bijections $\pi \mapsto \tr(\pi) \mapsto t_\pi$, where $\tr(\pi):\cH(G)\to \C$ is the trace functional, and
where $t_\pi$ is characterized by the equalities
\[
\tr(\pi)(a_\lambda) =
\tr(t_\pi | V_\lambda),
\]
for every $\lambda\in \Lambda^+$.

(ii) The bijection between $\chi$ and $t_\pi$ modulo $W$-conjugation is via the three 
identifications 
\[
 \chi \in 
\Hom(T(F)/T(\ko),\C^\times)
\simeq
\Hom(X_*(T),\C^\times) =
\Hom(X^*(\widehat T),\C^\times)
=
\widehat T(\C) \ni t_\pi.
\]

(iii) The bijection $\chi\mapsto \tr(\pi)$ is characterized via the Satake 
isomorphism 
by the equalities
\begin{equation}\label{specS}
 \tr(\pi)(f) = \sum_{s\in T(F)/T(\ko)}  \cS(f)(s) \chi(s), \quad
f\in \cH(G),
\end{equation}
\end{proposition}
\index{$\Lambda^+$, Weyl cone in the cocharacter lattice of $G$}

\begin{proof}
The bijection $\pi\mapsto \tr(\pi)$ in (i) is standard and follows from that $(G(F),K)$ is a Gelfand 
pair. By the second orthogonality relation of characters of the group $\widehat G(\C)$, we 
have 
that the values of $\tr(t_\pi | V_\lambda)$ for varying $\lambda\in \Lambda^+$ characterize the 
element $t_\pi$ up to $\widehat G(\C)$-conjugation in $\widehat G(\C)$, hence up to $W$-conjugation 
in $\widehat T(\C)$. This shows that the identities in (i) characterise the map $\tr(\pi)\mapsto 
t_\pi$ uniquely. We shall verify below that $t_\pi$ exists and the map is a bijective.

The identifications in (ii) have been given before the proposition.
The Satake isomorphism induces 
\[
\chi \in \Hom(T(F)/T(\ko),\C^\times)/W  \subset \operatorname{Spec}(\cC_c(T(F)/T(\ko))^W) \xrightarrow[]{ 
\operatorname{Spec}(\cS)}
 \operatorname{Spec}(\cH(G)) \ni \tr(\pi),
\]
where we identify $\Hom(T(F)/T(\ko),\C^\times)/W$ with the closed points of 
$\operatorname{Spec}(\cC_c(T(F)/T(\ko))^W)$ which are algebra functionals
$\cC_c(T(F)/T(\ko))^W\to \C$. This yields the bijection in (iii) between $\chi$ and $\tr(\pi)$
via~\eqref{specS}.

The rest of the proposition amounts to the following commuting triangle of bijections:
\begin{equation*}\label{diagSatake}
 \begin{tikzcd}
\chi \ar[rr,mapsto,"\operatorname{Spec}(S)"]  
& & \tr(\pi) \ar[dl,dashrightarrow,mapsto]
\\ & t_\pi \ar[ul,leftrightarrow] &
\end{tikzcd}
\end{equation*}
Indeed, we have verified above that the middle map, and the 
lower-left map 
are bijections. As a final step, it remains to show that the element $t_\pi$ obtained by 
following the 
 inverse bijections $\tr(\pi) \mapsfrom \chi \mapsfrom t_\pi$ so as to make 
the triangle commute satisfies the equalities $\tr(\pi)(a_\lambda)= 
\tr(t_\pi|V_\lambda)$ in (i).

In view of~\eqref{specS}, we have for every $\lambda\in \Lambda^+$,
\[
 \tr(\pi)(a_\lambda) = \sum_{s\in T(F)/T(\ko)} \cS(a_\lambda)(s)  \chi(s).
\]
Under the identifications
\[
s\in T(F)/T(\ko) \simeq X_*(T) = X^*(T) \ni \mu,
\]
which are dual to those of (ii), 
we have the equality $\chi(s) = \mu(t_\pi)$.
Moreover, 
\[\dim \operatorname{Hom}_{\widehat T(\C)}(\mu,V_\lambda) = \cS(a_\lambda)(s)
\]
by the
definition~\eqref{Salambda} of 
$a_\lambda$.
We obtain that the latter integral is equal to 
\[
 \sum_{\mu \in X_*(T) = X^*(\widehat T)} \dim 
\operatorname{Hom}_{\widehat T(\C)}
(\mu,V_\lambda) \cdot
\mu(t_\pi) 
= \tr(t_\pi | V_\lambda),
\]
which concludes the proof of the claim.
\end{proof}

\begin{definition}\label{d:trlambda}
\index{$\tr_\lambda(\pi)$, trace of $\lambda$-Hecke operator}
For $\lambda\in \Lambda^+$, and an irreducible $K$-unramified representation 
$\pi$, define
\[
\tr_\lambda(\pi) :=  \tr(\pi)(a_\lambda) = 
\tr(t_\pi | V_\lambda).
\]
\end{definition}

The \emph{unramified principal series}
$\operatorname{Ind}^{G(F)}_{B(F)}(\delta^{\frac12} \chi)$ contains a unique non-zero $K$-fixed vector 
$v^\circ$ 
given in 
the induced model 
\[
\{v:G(F)\to \C,\quad v(tug) = \delta^{\frac12}(t) \chi(t)v(g) ,\quad t\in T(F),\ u\in U(F),\ g\in G(F) \}
\]
by the formula
\[
 v^\circ(tuk) :=  \delta^{\frac12}(t)\chi(t), \quad t\in T(F),\ u\in U(F),\ k\in K,
\]
which is justified 
by the Iwasawa 
decomposition $G=B(F)K$
 because $\delta^{\frac12}\chi$ is trivial on $T(\mathfrak o) = T(F)\cap K = B(F) \cap K$.
Every $f\in \lH(G)$ acts on the vector $v^\circ$ by the scalar
\[
 \sum_{s\in T(F)/T(\mathfrak o)} \cS(f)(s)\chi(s) = \int_{T(F)} \cS(f)(s)\chi(s)ds ,
\] 
as can be seen from the following calculation. For every $t\in T(F)$,
\[
 \int_{G(F)} v^\circ(t g) f(g) dg = 
 \int_{B(F)} v^\circ(t b) f(b)  d_{\rm left}b
\]
\begin{equation}\label{calculatev0}
 = \delta^{\frac12}(t) \chi(t) \int_{T(F)} \delta^{\frac12}(s)\chi(s)  \int_{U(F)} f(su) du ds
 = v^{\circ}(t) \int_{T(F)} \cS(f)(s)\chi(s) ds,
\end{equation}
where $d_{\rm left}(su) = dsdu$ is the left Haar measure on $B(F)$ that gives $B(\mathfrak o)$ 
volume one, and $d(bk)=d_{\rm left}b 
dk$ is the Haar measure on $G(F)$ that gives $K$ volume one.

The functor of $K$-fixed vectors $(\pi,V)\leadsto V^K$  is exact from the category of 
admissible 
$G(F)$-representations to the category of finite-dimensional $\cH(G)$-modules as follows 
from the existence of the projection $\int_K \pi(k) dk:V \twoheadrightarrow V^K$.
Since the unramified principal series $\operatorname{Ind}^{G(F)}_{B(F)}(\delta^{\frac12} \chi)$ has 
finite length, it has a unique irreducible $K$-unramified $G(F)$-subquotient (in its 
Jordan--H\"older 
decomposition).

\begin{proposition}\label{p:Casselman}
For every unramified character $\chi:T(F)/T(\mathfrak o)\to \C$, the irreducible $K$-unramified 
$G(F)$-subquotient of the unramified principal series
$\operatorname{Ind}^{G(F)}_{B(F)}(\delta^{\frac12} \chi)$ corresponds with $\chi$ under the 
bijection of Proposition~\ref{p:Satake}. This completes the following 
commutative diagram of bijections:
\begin{equation*}\label{diagSatake2}
 \begin{tikzcd}
& \pi \ar[dr,leftrightarrow] & 
\\ \chi \ar[rr,mapsto,"\operatorname{Spec}(S)"]  \ar[ur,dashrightarrow,mapsto]
& & \tr(\pi) \ar[dl,mapsto,"\ref{p:Satake}(i)"]
\\ & t_\pi \ar[ul,leftrightarrow,"\ref{p:Satake}(ii)"] &
\end{tikzcd}
\end{equation*}
\end{proposition}

\begin{proof}
Let $(\pi,V)$ be the irreducible $K$-unramified $G(F)$-subquotient of 
$\operatorname{Ind}^{G(F)}_{B(F)}(\delta^{\frac12} \chi)$, and write $V_2 \hookrightarrow V_1 \twoheadrightarrow 
V$ 
for two 
$G(F)$-subrepresentations $V_1,V_2$ of $\operatorname{Ind}^{G(F)}_{B(F)}(\delta^{\frac12} \chi)$.
We verify that the equalities~\eqref{specS} which characterize the middle arrow of the 
diagram are satisfied.
Since $V_1^K \to V^K$ is surjective, $\dim V_1^K \le 1$ and $\dim V^K=1$, we deduce that it is a 
bijection (in fact an $\cH(G)$-isomorphism).
In particular $\dim V_1^K = 1$ 
and the $K$-fixed vector $v^\circ\in \operatorname{Ind}^{G(F)}_{B(F)}(\delta^{\frac12} \chi)^K$  
necessarily 
belongs to $V_1^K$.
In the above calculation~\eqref{calculatev0}, we have found the action of $f\in \cH(G)$ on 
$V_1^K = \C \cdot v^\circ$ is given by the right-hand side of~\eqref{specS}. On the other hand, 
the action of  
$f\in \cH(G)$ on $V^K$ is via the scalar $\tr(\pi)(f)$. Since $V_1^K \isom  V^K$, this 
verifies~\eqref{specS}.
\end{proof}

A smooth representation of $G(F)$ is said to be \emph{tempered} if it is unitary and weakly 
contained in the regular representation by translation on $L^2(G(F))$.
An irreducible smooth unitary representation of $G(F)$ is tempered if and only if its 
matrix coefficients belong to $L^{2+\epsilon}(G^{\rm der}(F))$ for 
every $\epsilon >0$ (this follows from~\cite{CHH:tempered}).
Denote by $\Xi(g) := \int_K \delta^{\frac 12}(kg) dk$ the Harish-Chandra function,
where $\delta$ is inflated to $G(F)=B(F)K$ using the Iwasawa decomposition.

\begin{proposition}\label{p:tempered}
Let $\pi$ be an irreducible $K$-unramified representation of $G(F)$, and let $\chi$, $t_\pi$, 
$\tr(\pi)$ be as in Proposition~\ref{p:Satake}. The following six properties are 
equivalent:
\begin{enumerate}[(i)]
\item $\pi$ is tempered,
\item $\pi$ is unitary and $|\tr(\pi)(f)| \le \int_G |f(g)| \Xi(g) dg$ for every $f\in \mathcal H(G)$,
\item $\chi$ is unitary,
\item $t_\pi$ is a compact element, i.e., $t_\pi$ belongs to the maximal compact subgroup of 
$\widehat T(\C)$,
\item $|\tr_\lambda(\pi)|\le \dim V_\lambda$ for every $\lambda \in \Lambda^+$,
\item there exists $C >0$ such that $|\tr_\lambda(\pi)| \le C \cdot \dim 
V_\lambda$ for every $\lambda\in \Lambda^+$.
\end{enumerate}
\end{proposition}

\begin{proof} The result is implicit in early work of 
Langlands. We couldn't locate a proof in the literature, hence we provide one.

(i) $\Leftrightarrow$ (ii). Let $v^\circ$  be a $K$-fixed vector of $\pi$ with $\langle v^\circ,v^\circ \rangle = 1$.
Then
\[
 \tr(\pi)(f) = \int_{G(F)} f(g) \langle \pi(g) v^\circ,  v^\circ \rangle dg,\quad f\in \mathcal H(G).
\]
If $\pi$ is tempered, then \cite[Thm.2]{CHH:tempered} says that the matrix coefficient is 
bounded 
by
$|\langle \pi(g) v^\circ,  v^\circ \rangle| \le \Xi(g)$, which implies (ii).
Conversely, (ii) implies the inequality $|\langle \pi(g) v^\circ,  v^\circ \rangle| \le \Xi(g)$, and 
since $\Xi\in
L^{2+\epsilon}(G^{\rm der}(F))$ for every $\epsilon>0$, we have that 
\cite[Thm.1]{CHH:tempered} implies that $\pi$ is tempered.

(ii) $\Leftrightarrow$ (iii). Since Proposition~\ref{p:Casselman} says that $\pi$ is a 
$G(F)$-subquotient of 
the unramified 
principal series 
$\operatorname{Ind}^{G(F)}_{B(F)}(\delta^{\frac12} \chi)$, we have that (iii) 
implies that $\operatorname{Ind}^{G(F)}_{B(F)}(\delta^{\frac12} \chi)$ is unitary which in turn implies 
that $\pi$ is unitary.
Applying~\eqref{specS}, we find
\[
 \tr(\pi)(f) = \int_{T(F)} \chi(s) \mathcal S(f)(s) ds
= \int_{T(F)} \chi(s) \delta^{\frac 12}(s) \int_{U(F)}  f(su) du ds
\]
\[
 = \int_{B(F)} \chi(b)f(b) \delta^{\frac 12}(b) d_{\rm left}b 
=  \int_{B(F)} \chi(b)  
\int_{K} f(bk) 
\delta^{\frac12}(b) dk d_{\rm left}b
\]
This shows that (iiii) implies the inequalities in (ii).
The converse follows from Macdonald's formula for the spherical function.

(iii) $\Leftrightarrow$ (iv).
We may identify the maximal compact subgroup of $\widehat T(\C)$ with $\Hom(X^*(\widehat 
T),S^1)$, which in turn can be identified following Proposition~\ref{p:Satake}(ii) with 
$\Hom(T(F)/T(\mathfrak o),S^1)$, the group of unitary unramified characters of $T(F)/T(\mathfrak o)$.

The equivalences (iv) $\Leftrightarrow$ (v) and (iv) $\Leftrightarrow$ (vi)  follow from $\tr_\lambda(\pi) 
= \tr(t_\pi | V_\lambda)$ in Definition~\ref{d:trlambda}.
\end{proof}

\begin{remark}
Property (ii) is related to Harish-Chandra's definition of temperedness of an 
admissible representation as having the property that its trace character extends to a 
continuous distribution on the Schwartz space.
\end{remark}

We have been using consistently the \emph{unitary normalization} of the character $\chi$, of 
the
Satake transform, and of the Satake parameter $t_\pi$.
There is also an algebraic normalization, which is that
$q^{\langle \lambda ,\rho \rangle} a_\lambda$ corresponds to the trace function of the 
IC-sheaf of the closure of the cell of the affine Grassmannian associated to $\lambda$. 
\begin{example}
For the trivial representation $\mathbf{1}$, we have
$\tr_\lambda(\mathbf{1}) =  
\tr(\mathbf{1})(a_\lambda)$. 
The Satake parameter $t_{\mathbf{1}}$ is equal to the principal semisimple element 
$\rho(q)\in \widehat T(\C)$, where $\rho$ is seen as a cocharacter 
$X_*(\widehat T)_\C$.
In particular, we obtain
\[
\sum_{x\in K\backslash G /K}
a_\lambda(x)
=\tr(\mathbf{1})(a_\lambda)
=\tr(\rho(q)|V_\lambda)
=
q^{\langle \lambda,\rho \rangle}
\left(
1+ O(q^{-\frac12})
\right).
\]
\index{$d(\lambda)=\langle 
\lambda, 
2\rho 
\rangle = \dim \operatorname{Gr}_\lambda$, degree of $\lambda$-Hecke operator}
 We conclude that $\tr_\lambda(\mathbf{1}) = q^{\frac{d(\lambda)}{2}}
\left(
1+ O(q^{-\frac12})
\right)
$,
where 
\[
d(\lambda):= \langle 
\lambda, 
2\rho 
\rangle = \dim \operatorname{Gr}_\lambda \in \Z_{\ge 0},
\] 
which we interpret as the \emph{degree} 
of the Hecke operator of coweight 
$\lambda\in \Lambda^+$.
\end{example}

\subsection{Base change}
Notation is as in the previous subsection, and we consider the degree $n$ extension  
$k'=\F_{q^n}$ of $k=\F_q$.
There is a base change algebra homomorphism $b:\cH(G_{k'}) \to \cH(G)$, see 
e.g.~\cite{Kottwitz:base-change}.
For any $K$-unramified irreducible representation $\pi$ of $G(k\pseries{t})$, there 
corresponds a unique $K'$-unramified irreducible representation $\Pi$ of 
$G(k'\pseries{t})$ such that
$\tr(\Pi)(f) = \tr(\pi)(b(f))$ for every $f\in \cH(G_{k'})$.
Indeed the corresponding Satake parameters satisfy the 
relation $t_\Pi = t_\pi^n$.
In particular the representation $\pi$ is tempered if and only if the base change 
representation $\Pi$ 
is tempered.
We can identify the positive Weyl chamber $\Lambda^+ \subset X_*(T)$ for the groups $G$ and 
$G_{k'}$. We have then the relation,
\[
\tr_\lambda(\Pi) = 
\tr(t^n_\pi | V_\lambda),
\]
which will be used often in relation to taking the limit as $n\to \infty$.

\subsection{Character sheaves}
\index{$\cL$, character sheaf}
\begin{defi} For a connected algebraic group $H$, say a \emph{character sheaf} on $H$ is 
a rank one lisse sheaf $\mathcal L$ with an isomorphism  $\mathcal L \boxtimes 
\mathcal L \cong m^* \mathcal L $ for $m: H \times H \to 
H$ the multiplication map. \end{defi}

\begin{remark} Given a character sheaf $\lL$, we have an isomorphism $\mathcal L_e = 
\mathcal L_e \otimes \mathcal L_e$, hence an isomorphism $\Ql = \mathcal L_e$. 

Let $m_3: H \times H \times H \to H$ be the multiplication of three elements. Because $m_3 = m \circ ( m \times id) = m \circ (id \times m)$, the isomorphism $\mathcal L \boxtimes 
\mathcal L \cong m^* \mathcal L $ induces two different isomorphism $\mathcal L \boxtimes \mathcal L \boxtimes \mathcal L \cong m_3^* \mathcal L$. These 
two isomorphisms are necessarily equal, because they are maps between lisse sheaves on a 
connected scheme and are equal on the identity point.
\end{remark}

For convenience, we give here many important facts about character sheaves, almost all of 
which are surely well-known.

\begin{lemma}\label{trace-function-character-sheaf} Let $H$ be an algebraic group over a 
finite field $\mathbb F_q$. The trace function of a character sheaf is a one-dimensional 
character of $H(\mathbb F_q)$. \end{lemma}

Throughout this paper, the \emph{trace function} of a sheaf $\mathcal F$ will be the function that takes a point $x$ to the trace of the \emph{geometric} Frobenius on the stalk of $\mathcal F$ at $x$.

\begin{proof} Let $\mathcal L$ be a character sheaf and let $\chi$ be the trace function 
of $\mathcal L$ on $H(\mathbb F_q)$. Then by the definition of a character sheaf, for 
$x,y \in H(\mathbb F_q)$, $\chi(xy)=\chi(x)\chi(y)$. Moreover because $\mathcal L$ is a 
rank one lisse sheaf, $\chi$ is nonzero. Hence it is an homomorphism to $\Ql^\times$ and 
thus a character. \end{proof}

\begin{remark} Not every character of $H(\mathbb F_q)$ necessarily arises from a 
character sheaf.  Consider the group of matrices of the form \[\begin{pmatrix} 1 & a & b 
\\ 0 & 1 & a^p \\ 0 & 0 & 1 \end{pmatrix}\] under matrix multiplication. Any character 
sheaf, restricted to the subgroup $H'$ defined by $a=0$, is a lisse character sheaf on $H' \cong \mathbb 
A^1$. By evaluating the character sheaf on a commutator, one can see that this sheaf is 
necessarily trivial when pulled back along the map $(x,y) \to (x^py - xy^p)$ whose 
generic fiber is geometrically irreducible, and hence the sheaf is trivial when 
restricted to $H'(\mathbb F_p)$. However, not all characters of $H(\mathbb F_p)$ are trivial 
on $H'(\mathbb F_p)$. \end{remark}

 Let $\sigma$ be the arithmetic Frobenius automorphism of $H( \overline{\mathbb F}_q)$. The Lang 
 isogeny is the covering $H \to H$ sending $g$ to $ \sigma(g) g^{-1}$, which is finite 
 \'{e}tale Galois with automorphism group $H(\mathbb F_q)$.

 \begin{lemma}\label{characters-come-from-Lang} Let $H$ be an algebraic group over a finite 
 field $\mathbb F_q$, $\mathcal L$ a character sheaf on $H$, and $\chi$ its trace function. 
 Then the pullback of $\mathcal L$ along the Lang isogeny is trivial, and as a 
 representation of the fundamental group, $\mathcal L$ is equal to the composition of the 
 map 
 $\pi_1( H_{\mathbb F_q})\to H(\mathbb F_q) $ with the character $\chi^{-1}: H(\F_q) \to 
 \overline{\mathbb Q}_\ell^\times$. \end{lemma}
 
 \begin{proof} For the first fact, observe that the pullback of $\mathcal L$ along the 
 Lang isogeny is $\sigma^* \mathcal L \otimes \mathcal L^{-1} = \overline{\mathbb 
 Q}_\ell$ as $\mathcal L$ is defined over $\mathbb F_q$ and hence invariant under 
 $\sigma$. It follows that the monodromy representation of $\mathcal L$ factors through 
 $H(\mathbb F_q)$. By examining the Frobenius elements at points of $H(\mathbb F_q)$, we 
 obtain $\chi$ --- the inverse is obtained because of the difference between arithmetic and 
 geometric Frobenius. \end{proof}

\begin{lemma}\label{character-sheaf-uniqueness} Let $H$ be an algebraic group over a 
finite field $\mathbb F_q$. Every one-dimensional character of $H(\mathbb F_q)$ arises 
from at most one character sheaf.

The orders of the arithmetic monodromy group of the character sheaf, the geometric 
monodromy group of the character sheaf, and the character all agree.  \end{lemma}

\begin{proof} These statements follow immediately from Lemma 
\ref{characters-come-from-Lang}. For the second, it is sufficient to observe that the 
image of the geometric fundamental group inside $H(\mathbb F_q)$ is also $H(\mathbb 
F_q)$, because the total space $H$ of the Lang isogeny is geometrically connected. 
\end{proof}

To check that a character arises from a character sheaf, we will mainly use the following 
lemma:

\begin{lemma}\label{character-sheaf-construction}

	\begin{enumerate}[(i)]

\item Let $H$ be an abelian algebraic group over $\mathbb F_q$. Every one-dimensional 
character of $H(\mathbb F_q)$ arises from a unique character sheaf. The trace function 
on $H(\mathbb F_{q^n})$ of this sheaf is the composition of the original character 
with the norm map.

\item Let $f: H_1 \to H_2$ be an algebraic group homomorphism and let $\mathcal L$ be a 
character sheaf on $H_2$. Then $f^* \mathcal L$ is a character sheaf on $H_1$ whose trace 
function is the composition of the trace function of $\mathcal L$ with $h$. 

\end{enumerate}

Hence every character of the $\mathbb F_q$-points of an algebraic group that factors 
through a homomorphism to an abelian algebraic group arises from a unique character 
sheaf. \end{lemma}

\begin{proof} For assertion (i), one uses the construction of Lemma 
\ref{characters-come-from-Lang} to construct a sheaf from a character, and then checks 
immediately the necessary isomorphism to make it a character sheaf.

Assertion (ii) is a direct calculation.
\end{proof}

When performing harmonic analysis calculations with character sheaves, it is helpful to 
have a description of character sheaves directly in terms of points. This is provided, 
based on central extensions, by the following lemmas:

\index{central extension with Frobenius action}
\begin{lemma}\label{sheaf-central}Let $\tilde{H}$ be a central extension $1 \to \Ql^\times \to \tilde{H}  \to H(\overline{\mathbb F}_q) 
\to 1$ with an action of $\sigma$ such that both maps involved are equivariant. 

Then there exists a unique character sheaf $\mathcal L$ on $H$ whose trace function over 
$\mathbb F_{q^n}$ is given by  $g \mapsto \sigma^n ( \tilde{g} ) \tilde{g}^{-1}$ for 
$\tilde{g}$ any lift of 
$g$ from $H(\overline{\mathbb F}_q)$ to $\tilde{H}$.

Furthermore, every character sheaf arises from a central extension in this way.\end{lemma}

\begin{proof} For the purposes of this proof, it is simpler to define the trace function 
using the arithmetic Frobenius, and then we invert to get the true trace function.

Given a central extension $\tilde{H}$, we form the associated character $\chi: 
H(\mathbb F_q) \to \Ql^\times$, $g \mapsto \sigma(\tilde{g}) \tilde{g}^{-1}$. It is easy to check 
that this is actually a group homomorphism.  %
We compose the Lang isogeny  homomorphism $\pi_1(H_{\mathbb F_q}) \to H(\mathbb F_q)$ with $\chi$ to produce a homomorphism $\pi_1 (H) \to \Ql^\times$ and hence a rank 
one sheaf $\mathcal L_\chi$, as in Lemma \ref{characters-come-from-Lang}.

Let us check that the trace function of $\mathcal L_\chi$ over $\mathbb F_{q^n}$ is given by $ 
g\mapsto \sigma^n(\tilde{g}) \tilde{g}^{-1}$.  Let $g$ be an element of $H(\mathbb 
F_{q^n})$ and let $\sigma(h) h^{-1}=g$.  By definition, the trace function of $\mathcal L_\chi$ at $g$ is defined as $\chi(a)$ for the unique $a \in H(\mathbb F_{q})$ such that $\sigma^n(h) = h a$. (Such $a$ exists because $\sigma ( \sigma^n(h) ) (\sigma^n(h))^{-1} = \sigma^n(g) =g$.)  In other words, the 
trace of $\mathcal L_\chi$ at $g$ is $\chi( h^{-1} \sigma^n(h))$. Choose $\tilde{h}$ a lift of $h$ and let 
$\tilde{g}= \sigma(\tilde{h}) \tilde{h}^{-1}$, so that \[\chi ( h^{-1} \sigma^n(h)) = 
\sigma( \tilde{h}^{-1} \sigma^n(\tilde{h}) ) \left( \tilde{h}^{-1} 
\sigma^n(\tilde{h})\right)^{-1}  = \sigma(\tilde{h})^{-1} \sigma^{n+1}(\tilde{h}) 
\sigma^n(\tilde{h})^{-1} \tilde{h} \] \[= 
\sigma^{n+1}(\tilde{h}) \sigma^n(\tilde{h})^{-1} \tilde{h}\sigma(\tilde{h})^{-1} = \sigma^n(\tilde{g}) \tilde{g}^{-1},\]
where we use the fact that we are working with an element of the center and 
hence may freely conjugate it by any element.

Second, let us check that the trace function of $\mathcal L_\chi$ over $\mathbb F_{q^n}$ is actually a 
character. This follows because \[\sigma^n(\tilde{g}_1\tilde{g}_2) (\tilde{g}_1 
\tilde{g}_2)^{-1} = \sigma^n (\tilde{g}_1 ) \sigma^{n}(\tilde{g}_2) \tilde{g}_2^{-1} 
\tilde{g}_1^{-1} = \sigma^n (\tilde{g}_1 )  \tilde{g}_1^{-1}  \sigma^{n}(\tilde{g}_2) 
\tilde{g}_2^{-1}  \] where we use that $\sigma^{n}(\tilde{g}_2) \tilde{g}_2^{-1} 
$ is central.

It now follows by the Chebotarev density theorem that $\mathcal L_\chi$ admits an isomorphism 
$\mathcal L_\chi \boxtimes \mathcal L_\chi \cong m^* \mathcal L_\chi$ because these two sheaves have the 
same trace function over every finite field. The uniqueness follows from Lemma 
\ref{character-sheaf-uniqueness}.

Conversely, given a character sheaf $\mathcal L$, define $\tilde{H}$ to be 
the set of pairs of a point  $x \in {H} (\overline{\mathbb F}_q)$ and a nonzero section 
of $\mathcal L_x$. Multiplication is given by $(x,s_x)(y,s_y) = (xy, s_x \otimes s_y)$ 
where we use the isomorphism $\mathcal L_x \otimes \mathcal L_y  = \mathcal L_{xy}$ 
induced by taking stalks in the isomorphism $\mathcal L \boxtimes \mathcal L= m^* 
\mathcal L$ that is part of the definition of a character sheaf.  Associativity for this 
multiplication follows from associativity for the isomorphism. To find units and 
inverses, it is sufficient to find them in the stalk over the identity of $H$, where they 
are obvious. 

By definition, the trace function of $\mathcal L$ at $x$ is the trace of Frobenius on $\mathcal L_x$, which because 
$\mathcal L_x$ is one-dimensional is the eigenvalue of Frobenius on the $\mathcal L_x$, which can be 
calculated as $\sigma^n(s_x) s_x^{-1}$ for $s_x$ a section of $\mathcal L_x$, which is equal 
to $\sigma^n(x,s_x)  (x,s_x)^{-1}$ for $(x,s_x)$ a lift of $x$.\end{proof}

\begin{lemma}\label{all-factorizable} \begin{enumerate} \item For $H_1, H_2$ two 
algebraic groups, any character sheaf on $H_1 \times H_2$ is $\mathcal L_1 \boxtimes 
\mathcal L_2$ for $\mathcal L_1$ and $\mathcal L_2$ character sheaves on $H_1$ and $H_2$. 
\item For $H$ an algebraic group over $\mathbb F_{q^n}$, any character sheaf on 
$\operatorname{Res}_{\mathbb F_{q^n}}^{\mathbb F_q} H$ is the Weil restriction of a 
character sheaf on $H$. \end{enumerate}
\end{lemma}

\begin{proof} \begin{enumerate}

\item Let $\mathcal L$ be the character sheaf, let $\mathcal L_1$ be its pullback to 
$H_1$, and let $\mathcal L_2$ be its pullback to $H_2$. Then $\mathcal L_1 \boxtimes 
\mathcal L_2$ and $\mathcal L$ have the same trace function, hence are equal. 

\item Let $\mathcal L$ be the character sheaf, let $\mathcal L'$ be its pullback to 
$(\operatorname{Res}_{\mathbb F_{q^n}}^{\mathbb F_q} H)_{\mathbb F_{q^n}}$ and then to 
$H$, embedded diagonally. Then $\mathcal L$ and $\operatorname{Res}_{\mathbb 
F_{q^n}}^{\mathbb F_q} \mathcal L'$ have the same trace function and thus are equal. 
\qedhere
\end{enumerate}
\end{proof} 

\subsection{Weil Restrictions}

\begin{notation}\index{$G\lWR R \rWR$, Weil restriction of base change $G_R$} We work 
with 
the convention that, for an algebraic 
group $G$ over $k$ and a finite-dimensional ring $R$ over $k$,  $G \lWR R \rWR$ is the 
algebraic group whose $S$-points for a ring $S$ over $k$ are the $R \otimes_k S$ points 
of $G$. Equivalently, $G \lWR R \rWR$ is the Weil restriction $\operatorname{Res}_k G_R$ 
from $R$ to $k$ of the base-change $G_R$. \end{notation}

\begin{example} If we view $k^n$ as a ring by pointwise multiplication, then $G \lWR k^n 
\rWR = G^n$.  More generally, $G \lWR R_1 \times R_2 \rWR = G \lWR R_1 \rWR \times G \lWR 
R_2 \rWR$. For another generalization, if $k'$ is a separable $k$-algebra of degree $n$, 
and $\overline{k}$ is the algebraic closure of $k$, then $\left( G \lWR k' \rWR 
\right)_{\overline{k}} = G_{\overline{k}}^n$. \end{example}

\begin{example} $G \lWR k[t]/t^2 \rWR $ is an extension of $G$ by the Lie algebra 
$\mathfrak g$ of $G$, where $\mathfrak g$ is viewed as an additive group scheme. More 
generally, $G \lWR k[t]/t^n \rWR $ is an $n-1$-fold iterated extension of $G$ by 
$\mathfrak g$. \end{example}

By definition, we have $G \lWR R \rWR (k) = G(R)$, which we will use several times. This 
method of constructing a scheme whose $k$-points are $G(R)$ has many good properties. For 
us, the most important is that it is stable under base change, i.e., for any field $k'$ 
over $k$, 
$G_{k'} \lWR R \otimes_k k' \rWR = \left( G \lWR R \rWR \right)_{k'}$.

\subsection{Sheaves on Stacks} 

We always denote Verdier duality by $D$.

\begin{lemma}\label{ext-relation} Let $Y$ be a stack of finite type over an algebraically 
closed field and let $K_1$ and $K_2$ be bounded complexes of $\ell$-adic sheaves on $Y$.
\begin{enumerate} 
\item $H^i_c (Y, DK_1 \otimes K_2)$ is naturally dual to $\operatorname{Ext}^{-i}_Y ( 
K_2, K_1)$;
\item If $K_1$ and $K_2$ are perverse, then $H^i_c (Y, DK_1 \otimes K_2)$ vanishes for 
$i> 0$;
\item If $K_1$ and $K_2$ are perverse and semisimple, then $H^0_c(Y, DK_1 \otimes K_2) = 
\Hom(K_1,K_2)$.
\end{enumerate}
\end{lemma}

\begin{proof} For part 1, by the definition of cohomology with compact supports 
\cite[\S9.1]{LO08},
\[ H^i_c (Y, DK_1 \otimes K_2) = \left(  H^{-i} ( Y, D ( DK_1 \otimes K_2) 
)\right)^{\vee}. \]
By \cite[Prop.6.0.12 and Thm.7.3.1]{LO08},
\[ H^{-i} ( Y, D ( DK_{1} \otimes K_{2}) ) = H^{-i} (Y, \mHom (K_{2}, K_{1})),\]
which in turn is equal to 
$ \operatorname{Ext}^{-i}_{Y}(K_{2},K_{1} ),$
by definition of $\operatorname{Ext}$, see \cite[Rem.5.0.11]{LO08}.

Part 2 follows because perverse sheaves are the heart of a t-structure by  
\cite[Thm.5.1]{LO09} and so their $\operatorname{Ext}^{-i}$ vanishes for $i>0$.

Part 3 follows because for semisimple perverse sheaves $\operatorname{Ext}^0 ( K_{2}, 
K_{1}) = \Hom(K_{2}, K_{1})$ is dual to $\Hom (K_{1}, K_{2})$.
\end{proof}

\begin{lemma}\label{cohbound-weak} Let $\iota: \overline{\mathbb Q}_\ell \to \mathbb C$ 
be an embedding. Let $Y$ be an Artin stack of finite type over $\mathbb F_q$ with affine 
stabilizers and let $K_1$ and $K_2$ be bounded complexes of $\ell$-adic sheaves on $Y$, 
$\iota$-pure of weights $w_1$ and $w_2$.
Then for any $j \in \mathbb Z$,  \[ \sum_{i =-\infty}^j   (-1)^i \tr
\left(
\Frob_{q^e} , \iota 
(  H^i_c (Y_{\overline{\mathbb F}_q},  DK_1 \otimes K_{2} 
)) 
\right)= 
O\left(\left(q^e\right)^{ \frac{j + w_2 -w_1 }{2}}\right),\]
where the constant in the big $O$ is independent of $e$ but may depend on $(Y,K_1,K_2)$.  
\end{lemma}

\begin{proof}%

The tensor product $DK_1 \otimes K_{2}$ is necessarily mixed of weight $\leq w_2-w_1$ so 
by \cite[Thm.1.4]{Sun12L}, $H^i_c (Y_{\overline{\mathbb F}_q},  DK_1 \otimes K_{2} 
)$ is mixed of weight $\leq i+ w_2-w_1$.

Let $|\Frob_{q}|$ be the operator that acts on generalized eigenspaces of $\Frob_{q}$ 
with eigenvalue the absolute value of the corresponding eigenvalue of $\Frob_{q} $. 
Then we have 
\[ \left|  \sum_{i =-\infty}^{j-1}   (-1)^i \tr\left(\Frob_{q^e} , \iota (H^i_c (Y_{\overline{\mathbb 
F}_q},  DK_1 \otimes K_{2} )  ) \right) \right| \]
\[ \leq \sum_{i =-\infty}^{j-1}   \tr\left(|\Frob_{q}|^e,\iota ( H^i_c (Y_{\overline{\mathbb F}_q},  
DK_1 \otimes K_{2} ) )\right)   .\]
Then because all eigenvalues of $\Frob_{q}$ are $\leq q^{ \frac{j+ w_2 -w_1-1}{2}  }$, 
for any $0<s \leq   e$, we have
\begin{align*}
&\sum_{i =-\infty}^{j-1}   \tr(|\Frob_{q}|^e,  \iota ( H^i_c (Y_{\overline{\mathbb F}_q},  DK_1 
\otimes K_{2} )   ))\\
\leq  q^{ (e-s)  \frac{j+ w_2-w_1-1}{2}  }& \left( \sum_{i =-\infty}^{j-1}   
\tr(|\Frob_{q}|^s, \iota (H^i_c (Y_{\overline{\mathbb F}_q},  DK_1 \otimes K_{2} ) ) )  \right) 
\end{align*} 
(by \cite[Thm.4.2(i)]{Sun12L})
\[  \leq q^{ (e-s)  \frac{j+ w_2-w_1-1}{2}  }  O_s(1) .\]

Note that for $s$ sufficiently small we have
 \[q^{ (e-s)  \frac{j
+w_2-w_1-1 }{2}  }   < q^{ e\frac{j+ w_2-w_1}{2} } = \left(q^e\right) ^{  \frac{j+ 
w_2-w_1}{2}} \] because if $j +w_2 -w_1 -1 \geq 0$ this holds for all nonnegative $s$ and if $j+w_2-w_1-1 < 0$ this holds for all $s<   \frac{e}{ - (j+w_2-w_1 -1)}$.

Thus we have \[ \left|  \sum_{i =-\infty}^{j-1}   (-1)^i \tr\left(\Frob_{q^e} , \iota (H^i_c (Y_{\overline{\mathbb 
F}_q},  DK_1 \otimes K_{2} )  ) \right) \right|  = O\left( \left(q^e\right) ^{  \frac{j+ 
w_2-w_1}{2}}\right).\]

The remaining term satisfies  \[(-1)^j \tr(\Frob_{q^e} , \iota (H^j_c (Y_{\overline{\mathbb 
F}_q},  DK_1 \otimes K_{2} ) )  ) = O\left(\left(q^e\right) ^{ \frac{j + 
w_2-w_1}{2}}\right) 
\] where the constant in the big $O$ is the dimension of $H^j_c (Y_{\overline{\mathbb 
F}_q},  DK_1 \otimes K_{2} ) $.  Thus, the desired bound holds for both terms.
\end{proof}

\subsection{Linear recursive sequences and tensor power trick}
The following is a variant of Gelfand's formula $\lim\limits_{n\to \infty} 
||t^n||^{\frac1n}$ for the spectral radius of an endomorphism $t$.
\begin{lemma}[{\cite[\S3]{Deligne:Weil-I}},~\cite{Bombieri-Katz}]
\label{l:recursive-sequence}
Let $V$ be a finite-dimensional complex vector space, and $t\in \End(V)$. Then 
\[
\rho:= 
\limsup\limits_{n\to \infty} |\tr(t^n | V)|^{\frac1n}
\]
is the spectral radius of $t$, and
\[
|\tr(t^n | V )| \le \dim V \cdot \rho^n,\quad \text{for every $n\ge 0$}.
\]
\end{lemma}

\begin{proof}
Let $\lambda_1,\ldots,\lambda_{\dim(V)}$ denote the eigenvalues of $t$, so that
$\tr(t^n | V ) = 
\sum_{i} \lambda_i^n$.
The power series
\[
\sum^\infty_{n=1}
\tr(t^n|V) \frac{z^n}{n}
=
- \log \det(1-zt|V)
=
- \sum_i \log (1-\lambda_i z)
\]
has radius of convergence equal to $\rho^{-1}$ by the Cauchy--Hadamard theorem (note that 
$n^{\frac1n} \to 1$ as $n\to \infty$). Since it cannot be extended to an holomorphic 
function past the singularities at $z=\lambda_i^{-1}$, we deduce that $\rho$ is equal to 
$\max_i 
|\lambda_i|$, the spectral radius of $t$.
This establishes the first assertion, and then the inequality of the second assertion 
follows.
\end{proof}

\section{Compactly induced representations}\label{s:compactly-induced}

This section is concerned with first developing some preliminary material, leading up to the key definition of mgs representations, followed by giving some basic properties of the definition, then providing some examples and non-examples, and finally describing some additional useful properties. 

We begin, in \S\ref{s:vanishing-Jacquet}, with some purely representation-theoretic computations. In particular, we give in Corollary \ref{c:vanishing-Jmod} a concrete criterion on a subgroup $J \subset G(F) $ and a character $\chi$ such that every irreducible smooth representation of $G(F)$ containing a nonzero  $(J, \chi)$-invariant vector is supercuspidal. 

In \S\ref{s:geometric-version}, we define a ``monomial datum" as a geometric version of 
$(J,\chi)$, and say a datum is ``geometrically supercuspidal" if it satisfies a geometric 
version of this concrete criterion. These geometric analogues contain the classical 
versions in the sense that we can extract from a monomial datum a subgroup $J$ and 
character $\chi$, and they do satisfy the concrete condition if the original monomial datum 
is geometrically supercuspidal.

In \S\ref{s:mgsr}, we define the notion of a mgs representation as an irreducible smooth representation containing a $(J,\chi)$-invariant vector, where $(J,\chi)$ arise from a geometrically supercuspidal monomial datum in this way. (In particular, mgs representations are always supercuspidal.)

The calculations in \S\ref{s:vanishing-Jacquet} involve the compact induction $\cind^{G(F)}_J \chi$, but they do not require us to show that $\cind^{G(F)}_J \chi$ is itself an irreducible supercuspidal representation. Thus they have a different approach than works which aim to construct supercuspidal representations as inductions, where showing that the induced representation is irreducible is of the highest importance. On the other hand, in \S\ref{s:intertwining}, we give a way to check that a monomial datum is geometrically supercuspidal, which does involve showing that $\cind^{G(F)}_J \chi$ is irreducible and supercuspidal, and in \S\ref{sub:admissibility}, we show that $\cind^{G(F)}_J \chi$ is at worst a finite direct sum of supercuspidal representations under mild additional assumptions.

Our examples of mgs representations come in \S\ref{sub:epipelagic} and \S\ref{sub:adler}. These examples arise from existing constructions of supercuspidal representations, such as epipelagic representations and toral supercuspidals. 

In \S\ref{sub:non-examples}, we give examples of monomial data that are not geometrically supercuspidal. 

In \S\ref{sub:preservation}, we check that mgs representations are preserved under certain natural operations.

\subsection{Vanishing of Jacquet modules}\label{s:vanishing-Jacquet}
Let $G(F)$ be a reductive group over a non-archimedean local field $F$.
Let $P$ be a parabolic subgroup with Levi decomposition $P=MN$. The \emph{Jacquet module} 
$(\pi_N,V_N)$ of a smooth representation $(\pi,V)$ of $G(F)$ is the $N$-coinvariants 
of $V$, regarded as an $M$-module. This is an exact functor.
\begin{lemma}\label{l:vanishing-Jacquet}
	Let $\chi$ be a character of an open-compact subgroup $J$ and $P=MN$ a parabolic 
	subgroup.
The following properties are equivalent:
\begin{enumerate}[(i)]
\item The Jacquet module of $N$-coinvariants of the induced representation $\cind^{G(F)}_J 
\chi$ 
vanishes; 
\item for every $g\in {G(F)}$, the restriction of $\chi$ to $gNg^{-1}\cap J$ is 
non-trivial;
\item for every $g_1,g_2\in {G(F)}$, $\int_N f_\chi(g_1 n g_2)dn=0$, where 
\[
f_\chi(g) := 
\begin{cases}
\chi(g),& \text{if $g\in J$},\\
0,&\text{if $g\not\in J$}.
\end{cases}
\]
\end{enumerate}
\end{lemma}

\begin{proof}
We first show the direction (i) $\implies$ (ii).
	We view $\cind_{J}^{{G(F)}} 
	\chi$ as the space of smooth compactly supported functions $f$ on ${G(F)}$ satisfying 
	$f(gh)=f(g)\chi(h)$ for $h\in J$.
	 Since the functional $f \mapsto \int_{n\in 
		N} f(ng^{-1}) $ factors through the Jacquet module of $N$-coinvariants of  
		$\cind_{J}^{{G(F)}} 
	\chi$, it vanishes. 
	Take $f$ in this space to be the function supported on the left coset $g^{-1} J$ 
		such that 
	$f(g^{-1}h)=\chi(h)$ for  $h \in J$.
Then 
\[0=\int_{n\in 
		N} f(ng^{-1}) = \int_{n\in N \cap g^{-1}Jg} \chi(gng^{-1}) =  \int_{ h\in gN 
		g^{-1} \cap 
		J} \chi(h),\]
where the integrations are with respect to Haar measures.
	This implies that the restriction of $\chi$ to the subgroup $gNg^{-1}\cap J$ is 
	non-trivial.

For the direction (ii) $\implies$ (i), observe that a linear basis of $\cind_{J}^{{G(F)}} 
	\chi$ consists of
	the above functions $f_g: gh \mapsto \chi(h)$ supported on the left cosets $gJ$
 of $J$ in ${G(F)}$ for varying $g\in G(F)/J$. 
For $ h \in g N g^{-1} \cap 
	J$, the right translation of $f_g$ by $h$ is equal to $\chi(h) f_g$, and the right 
	translation of $f_g$ by $h$ is equal to the left translation of $f_g$ by an element 
	of $N$, which implies that the images of $f_g$ and $\chi(h) f_g$ in the Jacquet module 
	of 
	$N$-coinvariants are equal. Since property (ii) says that $\chi$ is nontrivial when 
	restricted to $g 
	N 
	g^{-1} \cap J$, this implies that the image of $f_g$ in the Jacquet module is zero. 
Since 
$\cind_{J}^{{G(F)}} 
	\chi$ projects onto the Jacquet module, and the $f_g$'s form a linear basis, we deduce 
	property (i), i.e,  that the Jacquet module vanishes.

For the implication (ii) $\Rightarrow$ (iii), suppose that $g_1 n_0 g_2\in J$ for some 
$n_0\in N$.
Then the condition $g_1 n g_2 \in J$ is equivalent to $g_2^{-1} n^{-1}_0 n g_2 \in J$.
Therefore
\[
\int_{n\in N} f_\chi(g_1 n g_2) =
\chi(g_1 n_0 g_2)
\int_{h\in g_2^{-1} N g_2\cap J}
\chi(h) = 0.
\]
The implication (iii) $\Rightarrow$ (ii) follows by taking $g_1=g$ and $g_2=g^{-1}$.
\end{proof}

\begin{lemma}\label{l:Frob-reciprocity}
The following properties of a smooth irreducible representation $(\pi,V)$ of $G(F)$ are 
equivalent:
\begin{enumerate}[(i)]
\item it has a non-zero $(J,\chi)$-invariant vector; 
\item it is a quotient of $\cind^{G(F)}_J \chi$.
\end{enumerate}
If one of these conditions holds and the Jacquet module of $N$-coinvariants of 
$\cind_J^{G(F)} \chi$ 
vanishes, 
then the Jacquet module of $N$-coinvariants of $\pi$ also vanishes: $V_N=0$.
\end{lemma}
\begin{proof}
The equivalence of
$\Hom_J(\chi,\pi) = 0$ and
$\Hom(\cind^{G(F)}_J \chi, \pi) = 0$
 is a form of Frobenius reciprocity~\cite[Thm.3.2.4]{Cas}.
The second assertion is consequence of the exactness of the Jacquet functor.
\end{proof}

Recall that an admissible representation $(\pi,V)$ is 
\emph{supercuspidal} if $V_N=0$ for every proper parabolic subgroup $P=MN$ of $G(F)$.
It is equivalent~\cite[Thm.5.3.1]{Cas} to that all the matrix coefficients of $(\pi,V)$ have 
compact 
support mod center.
If $(\pi,V)$ is irreducible, then it is sufficient to verify that one nonzero matrix 
coefficient has compact support mod center.
We deduce from Lemma~\ref{l:vanishing-Jacquet} and Lemma~\ref{l:Frob-reciprocity} the 
following which will be used often.

\begin{corollary}\label{c:vanishing-Jmod}
Let $\chi$ be a character of an open-compact subgroup $J$ of $G(F)$.
 The following four properties are equivalent:
\begin{enumerate}[(i)]
\item $\cind_J^{G(F)} \chi$ has vanishing Jacquet module of $N$-coinvariants for every proper 
parabolic subgroup $P=MN$;
\item the restriction of $\chi$ to $N\cap J$ is non-trivial for every proper 
parabolic subgroup $P=MN$ of $G(F)$;
\item  $f_\chi$ is a cuspidal function on $G(F)$;
\item every irreducible smooth 
representation of $G(F)$ with a non-zero $(J,\chi)$-invariant vector is supercuspidal.
\end{enumerate}
\end{corollary}

\begin{remark}
It is proved in~\cite{Bushnell:c-ind} that the following properties on the induced 
representation $\cind^{G(F)}_J \chi$ are 
equivalent: 
\begin{enumerate}[(i')]
\item it is admissible; 
\item it is supercuspidal; 
\item it is a finite direct sum of irreducible supercuspidals. 
\end{enumerate}
These properties (i')-(iii') are stronger that the properties (i)-(iv) of 
Corollary~\ref{c:vanishing-Jmod}, because (ii') $\implies$ (i), or because (iii') $\implies$ (iv).
\end{remark}

\subsection{Geometric version}\label{s:geometric-version}
Let $G$ be a reductive group over a 
finite field $\kappa$, $m$ a natural number, $H$ a connected subgroup of $G \lWR \kappa 
[t]/t^m 
\rWR$,  and $\mathcal L$ a character sheaf on $H$.
We call the quadruple $(G,m,H,\lL)$ a \emph{monomial datum}.

Let $J$ be the inverse image of $H(\kappa)$ in $G(\kappa \bseries{t})$ and let $\chi$ 
be the character induced by $\mathcal L$ on $H(\kappa)$ (see 
Lemma~\ref{trace-function-character-sheaf}), pulled back to $J$. The situation is 
described by the diagram
\begin{equation}\label{J}
\begin{tikzcd}
U_{m}(G(\kappa \bseries{t})) \arrow[r,hook] & J \arrow[r,two heads] \arrow[d,hook] & 
H(\kappa) 
\arrow[d,hook]\\
& G(\kappa \bseries{t} ) \arrow[r,two heads] & G(\kappa[t]/t^{m})
\end{tikzcd}
\end{equation}
 where $U_m ( G (\kappa\bseries{t}))$ is the subgroup of 
$G(\kappa\bseries{t})$ consisting of elements congruent to $1$ modulo 
$t^m$.\index{$U_m(G(\kappa\bseries{t}))$, principal congruence subgroup} In this 
diagram, the square is Cartesian and the sequence $U_m ( G (\kappa\bseries{t})) \to J 
\to H(\kappa)$ is short exact.

This datum defines a monomial representation 
$\operatorname{c-ind}_{J}^{G(\kappa\pseries{t})}
	\chi$. The following definition gives the geometric version of the property 
	that all of the Jacquet modules of $\cind_{J}^{G(\kappa\pseries{t})}
	\chi$ vanish: 

\begin{defi}\label{def:geometric-ind-data}
We say that the monomial datum $(G,m,H,\mathcal L)$ is \emph{geometrically supercuspidal} if 
for 
every proper parabolic subgroup $P=MN$ of $G_{\overline\kappa}$, and every $g \in G 
(\overline\kappa[t]/t^m)$, the restriction of $\mathcal L_{\overline\kappa}$ to the 
identity component of the
intersection 
$g N 
\lWR 
\overline\kappa [t]/ t^{m} \rWR g^{-1} \cap H_{\overline\kappa}$ is non-trivial. 
\index{geometric supercuspidal datum}
\end{defi}

The next Lemma~\ref{checking-supercuspidality} will imply 
a close relationship between this geometric property and the previous vanishing 
property of the Jacquet modules. 
For any finite field extension $\kappa'$ 
of $\kappa$, the datum $(G,m,H,\cL)$ is geometrically supercuspidal if and only 
$(G_{\kappa'},m,H_{\kappa'},\cL_{\kappa'})$ is geometrically supercuspidal.
Let 
	$J_{\kappa'}$ be the inverse image of $H(\kappa')$ in $G ( \kappa' \bseries{t})$ as 
	in the diagram~\eqref{J}. 
	Let 
	$\chi_{\kappa'}$ 
	be 
	the character induced by $\mathcal L$ on $H(\kappa')$, pulled back to $J_{\kappa'}$.

\begin{lemma}\label{checking-supercuspidality} The following properties are equivalent:
\begin{enumerate}[(i)]
\item for every finite extension $\kappa'$ of $\kappa$, the 
	induction 
	$\cind_{J_{\kappa'}}^{G(\kappa'\pseries{t})} 
	\chi_{\kappa'}$ has vanishing Jacquet modules;
\item for every finite extension 
$\kappa'$ of $\kappa$, 
every proper parabolic subgroup $P=MN$ of $G_{{\kappa'}\pseries{t}}$,
the restriction of $\chi_{\kappa'}$ to $N(\kappa'\pseries{t})\cap J_{\kappa'}$ is 
non-trivial;
\item $(G, m,H,\mathcal L)$ is geometrically supercuspidal;
\item for every field extension $\kappa'$ of $\kappa$, any proper parabolic subgroup $P=MN$ 
of $G_{\kappa'}$, and any $g \in G 
(\kappa'[t]/t^m)$, the restriction of $\mathcal L_{\kappa'}$ to the intersection of $g N 
\lWR \kappa' [t]/ t^{m} \rWR g^{-1}$ with $H_{\kappa'}$ is not geometrically isomorphic 
to a constant sheaf.
\end{enumerate}
\end{lemma}

\begin{proof} The equivalence between (i) and (ii) follows from Lemma 
\ref{l:vanishing-Jacquet}. The implication (iv) $\implies$ (iii) follows by taking 
$\kappa'=\overline \kappa$. 

The direction (iii) $\implies$ (ii) is straightforward. It uses the fact that any quasi-split reductive group with 
a Borel subgroup, the Galois group of the base field acts on its Dynkin diagram, and 
parabolic subgroups are classified up to conjugacy by Galois-invariant subsets of the 
roots. Let $B$ be a Borel of $G_{\kappa'}$ and $B_{\kappa'\pseries{t}} \subseteq G_{\kappa'\pseries{t}}$ its 
pullback. Let $P$ be a parabolic subgroup of $G_{\kappa'\pseries{t}}$. Then $P$ is conjugate to a 
parabolic $P'$ containing $B_{\kappa'\pseries{t}}$. Let $S$ be the set of simple roots of 
$B_{\kappa'\pseries{t}}$ contained in the Levi of $P'$. Then $S$ is invariant under $\Gal( 
\kappa'\pseries{t})$. Because the action of $\Gal( \kappa'\pseries{t})$ on the simple roots of 
$B_{\kappa'\pseries{t}}$ factors through the action of $\Gal(\kappa')$ on the simple roots of $B'$, $S$ 
is also invariant under $\Gal(\kappa')$, so it corresponds to a parabolic subgroup $P_0$ of 
$G_{\kappa"} $ containing $B$. Because $P_{0, \kappa'\pseries{t}}$ contains $B_{\kappa'\pseries{t}}$, and 
has the same set $S$ of simple roots in its Levi, we have $P_{0, \kappa'\pseries{t}} = P'$ and 
thus $P_{0, \kappa'\pseries{t}}$ is conjugate to $P$. 
See~\cite[Ex.7.2.3]{Conrad:reductive-gp-schemes}.

Because of the Iwasawa decomposition 
$G(\kappa'\pseries{t})=P_0(\kappa'\pseries{t})G(\kappa'\bseries{t})$, we have that $P$ 
is $G(\kappa'\bseries{t})$-conjugate to $P_{0,\kappa' \pseries{t}} $. Thus $P = g P_0 g^{-1}$ for some $g \in G(\kappa'\bseries{t})$, and so we have $N = g N_{0, \kappa' \pseries{t}} g^{-1}$ for $N_0$ the unipotent radical of $P_0$. Hence $N (\kappa'\pseries{t}) = g N_0(\kappa' \pseries{t}) g^{-1}$. Finally, note that the restriction of 
$\chi_{\kappa'}$ to $gN_0(\kappa'\pseries{t})g^{-1}\cap J_{\kappa'}$ is the trace function 
over $\kappa'$ of the restriction of $\mathcal L$ to $gN_0 \lWR 
\kappa' [t]/ t^{m} \rWR g^{-1}\cap H$. 
Since this restriction is non-trivial, Lemma \ref{character-sheaf-uniqueness} implies 
that its trace function is a
non-trivial character.

So it remains to prove the converse (ii) $\implies$ (iii) $\implies$ (iv).

To verify (iii) $\implies$ (iv), let us first check that, given a morphism $f: Y\to X$ of 
schemes of finite 
type over 
	a field and a lisse sheaf $\mathcal F$ on $Y$, the property that $\mathcal F$ 
	restricted to a 
	fiber 
	of $f$ is constant defines a constructible subset of $X$. By Noetherian induction, it 
	is 
	sufficient 
	to solve the problem after restricting to any open subset of $X$. By \cite[Thm. 
	Finitude, 
	Th\'{e}or\`{e}me 1.9(2)]{sga4h}, there exists an open subset of $X$ such that for 
	each point 
	$x$ in 
	that subset,  \[\left( f_* \mathcal F\right)_x= H^0(Y_x,\mathcal F).\] Restrict to 
	that open 
	subset. Because the image of each irreducible component of $Y$ is constructible, we 
	can choose 
	a 
	smaller open subset of $X$ which is contained in the image of each irreducible 
	component of 
	$Y$ with 
	dense 
	image and  does not intersect the image of any irreducible component of $Y$ without dense image.
	After 
	base-changing to this open subset, each irreducible component of $Y$ maps 
	surjectively onto $X$ 
	(because the irreducible components without dense image no longer exist).  Now we 
	prove the 
	result 
	in this case. At any point $x$, if there is a section of $H^0(Y_x,\mathcal F)$ that 
	gives an 
	isomorphism between $\mathcal F$ and the constant sheaf, then the corresponding 
	section of $f_* 
	\mathcal F$ extends to some neighborhood, which gives an extension of the section of 
	$\mathcal 
	F$ 
	to the inverse image of that neighborhood, where because $\mathcal F$ is lisse it 
	must be an 
	isomorphism on every connected component of $Y$ that intersects that fiber. By 
	construction, every connected component of $Y$ intersects the fiber over $x$, so the 
	map is an isomorphism on the inverse image of the neighborhood of $x$. Hence the set 
	where 
	$\mathcal F$ is isomorphic to the constant sheaf is open, thus constructible, 
	verifying the claim.
	
	Consider the family of schemes $g N \lWR \kappa [t]/t^m  \rWR g^{-1} \cap H$ 
	parameterized by $ 
	g\in G \lWR \kappa [t]/t^m \rWR$. Let $\mathcal F$ be the pullback of $\mathcal L$ 
	to this 
	family. The set in $G \lWR \kappa [t]/t^m \rWR$ where $\mathcal F$ is 
	geometrically trivial on 
	the fiber is constructible. Property (iv) is equivalent to the claim 
	that this set 
	does not contain any point defined over any field extension of $\kappa$.  Because 
	this set is 
	constructible, it is sufficient to check this for every point defined over 
	$\overline\kappa$, which is exactly geometric supercuspidality. This establishes the direction (iii) $\implies$ (iv).

	We now establish (ii) $\implies$ (iii).
Fix a point $g\in G(\overline{\kappa} [t]/t^m)$.	 There exist some finite field 
extension $\kappa^*$ of $\kappa$ such that $g$ is defined over $\kappa^*$ and every 
	connected 
	component of $g N \lWR \kappa [t]/t^m \rWR g^{-1} \cap H$ is defined over $\kappa^*$. 
	If the 
	character sheaf $\mathcal L$ is geometrically trivial on $g N \lWR \kappa [t]/t^m 
	\rWR g^{-1} 
	\cap H$, then its trace function is necessarily constant on each connected component 
	of $g N 
	\lWR \kappa^*[t]/t^m \rWR g^{-1} \cap H$, and hence it corresponds to a character of 
	the 
	component group $\pi_0 \left(g N \lWR \kappa^* [t]/t^m \rWR g^{-1} \cap H \right)$. 
	Thus the 
	eigenvalue of Frobenius at each point is a root of unity of order dividing the order 
	of the 
	component group. We can pass to a further finite field extension $\kappa'/\kappa^*$ 
	that 
	trivializes the 
	eigenvalues of Frobenius at each point. Over this field extension, the corresponding 
	character 
	$\chi_{\kappa'}$ must be trivial when restricted to \[ \left( g N \lWR \mathbb \kappa 
	[t]/t^m  
	\rWR g^{-1} \cap H \right) \left(\kappa'\right)= g N(\kappa' [t]/t^m)g^{-1} \cap 
	H(\kappa'). \qedhere \]
\end{proof}

We deduce from Lemma \ref{checking-supercuspidality} that, in order to establish that 
a monomial datum $(G,m,H,\mathcal L)$ 
is 
geometrically supercuspidal, it suffices to 
verify 
the vanishing of all the Jacquet modules of all the representations compactly induced 
from (the inflation of) the characters $\chi_{\kappa'}$ 
of $H(\kappa')$, for all finite field extension $\kappa'/\kappa$.
This will 
enable us to apply standard techniques from representation theory of reductive groups 
over local fields to verify geometric supercuspidality.
Indeed, we will see examples of $(G,m,H,\mathcal L)$ satisfying these properties later in 
this section.

The following lemma shows that $H$ always lies in a maximal unipotent subgroup. This is useful because a maximal unipotent subgroup of $G (\kappa \pseries{t})$ sometimes lies in two subgroups, both isomorphic to $G(\kappa \bseries {t})$, but not conjugate to each other. The lemma allows us to transfer geometrically supercuspidal monomial data between the two subgroups.

\begin{lemma}\label{supercuspidal-parahoric} If $(G , m, H, \mathcal L)$ is 
geometrically supercuspidal, then $H$ is unipotent mod center. \end{lemma}

\begin{proof} We will prove the contrapositive. Assume that $H$ is not unipotent modulo 
the center of $G\langle \kappa[t]/t^m \rangle$; we 
will show 
that $(G, m , H, \mathcal L)$ is not geometrically supercuspidal. A smooth connected 
algebraic group fails to be unipotent if and only if it admits a nontrivial homomorphism 
from $\mathbb G_m$ (possibly after extending the base field $\kappa$, which we may freely do).  Thus $H$ admits a homomorphism $\alpha:\mathbb G_m \to H$ that 
doesn't factor through the center of $G\langle \kappa[t]/t^m \rangle$.

The natural map $\kappa \to \kappa [t]/t^m$ defines a map $G \to G \lWR \kappa[t]/t^m \rWR$. Let us check that $G$ is a maximal reductive subgroup of $G \lWR \kappa[t]/t^m \rWR $. To do this, observe that the projection $G \lWR \kappa[t]/t^m \rWR  \twoheadrightarrow  G$ has reductive image and unipotent kernel, and because the composition $G \to  G \lWR \kappa[t]/t^m \rWR \twoheadrightarrow G$ is an isomorphism, $G$ is a maximal reductive subgroup.

It follows that every reductive subgroup of $ G \lWR \kappa[t]/t^m \rWR$ is conjugate to a subgroup of $G$. In particular, the image of $\alpha$ is conjugate to a a subgroup of $G$. Because the definition of geometrically supercuspidal is invariant under conjugacy, we may assume without loss of generality that the image of $\alpha$ is a subgroup of $G$. In other words, we may assume that
the composition $\G_m \xrightarrow{\alpha} H \subset G \lWR \kappa[t]/t^m \rWR$
factors through a 
nontrivial homomorphism $\G_m \xrightarrow{\beta} G \to G \lWR \kappa[t]/t^m \rWR$.

Let $T$ be a maximal $\kappa$-split torus of $G$ containing the image of $\beta$. Let $P$ be 
the parabolic 
subgroup of $G$ containing $T$ and every root subgroup of $G$ on which $\beta$ acts by 
conjugating with 
eigenvalue a nonnegative power of the parameter $\mathbb G_m \xrightarrow{\operatorname{id}} \mathbb 
G_m$. 
Let $N$ be 
the maximal 
unipotent 
subgroup of $P$. Then $\beta$ acts on each root subgroup of $N$ with eigenvalue a positive 
power of $\operatorname{id}$. 
Let 
$H' :=H \cap N \lWR \kappa[t]/t^m \rWR$. Then $H'$ is an iterated extension of copies of 
$\mathbb G_a$, on each of which $\beta$ acts by conjugation by a nonzero power of 
$\operatorname{id}$. In other 
words, $H'$ admits a $\beta$-invariant filtration $\{1\} = H'_0 \subseteq H'_1 \subseteq 
\dots \subseteq
H'_m 
= H'$. Let $i$ be the largest natural number such that $\mathcal L$ is geometrically 
trivial on $H'_i$.   Then $\mathcal L$ descends to $H' / H'_i$ and is nontrivial on 
$H'_{i+1}/H'_i$. Because $\mathcal L$ is a character sheaf on $H$, it is 
conjugacy-invariant. Hence it is invariant by the conjugacy action of $\beta$. Thus its 
restriction to $H'_{i+1}$ followed by descent to $H'_{i+1} / H'_i$ is invariant under the 
action 
of 
$\beta$, which is scaling by some nonzero power of $\operatorname{id}$. But there is no nontrivial 
lisse sheaf 
on 
$\mathbb G_a$ which is invariant by scaling by a nonzero power.  So in fact $i=m$, 
and $\mathcal L$ is trivial on $H'$, so $(G, m , H, \mathcal L)$ is not geometrically 
supercuspidal.
\end{proof}

\subsection{Intertwining}\label{s:intertwining}
Let $G$ be a reductive group over a finite field 
$\kappa$, $m$ a natural number, $H$ a connected subgroup of $G \lWR \kappa [t]/t^m \rWR$ 
containing 
the 
center,  and $\mathcal 
L$ a character sheaf on $H$. We can check that $(G,m, H, \mathcal L)$ is geometrically 
supercuspidal using a geometric analogue of the standard method, based on intertwining 
sets. 

The \emph{intertwining set} of $\mathcal{L}$ is the set of $g\in G(\overline{\kappa}\pseries{t})$ 
such that 
$\mathcal{L}\simeq \mathcal{L}^g$ on $H_{\overline{\kappa}}\cap H_{\overline{\kappa}}^g$.
\begin{lemma}
	If the intertwining set is equal to $J_{\overline{\kappa}}$, then  
	$(G,m,H,\mathcal{L})$ is 
	geometrically supercuspidal.
\end{lemma}

\begin{proof}
	Applying Lemma~\ref{checking-supercuspidality}, it suffices to verify the 
	vanishing of all the 
	Jacquet modules of the induced representation $\cind^{G(\kappa' \pseries{t})}_{J_{\kappa'}} 
	\chi_{\kappa'}$ for every finite extension 
	$\kappa'/\kappa$.
By assumption, an element $g\in G(\kappa' \pseries{t})$ interwines $\chi_{\kappa'}$ in the sense that
$\chi_{\kappa'}(h) = \chi_{\kappa'}(g h g^{-1})$ for every $h\in J_{\kappa'} \cap g^{-1} J_{\kappa'} g$ if and 
only if $g\in J_{\kappa'}$.
	The vanishing of all the Jacquet modules of $\cind^{G(\kappa' \pseries{t})}_{J_{\kappa'}} 
	\chi_{\kappa'}$ then follows, in the stronger form of the irreducibility and cuspidality of  $\cind^{G(\kappa' \pseries{t})}_{J_{\kappa'}} 
	\chi_{\kappa'}$,
from the argument of \cite[\S3.11.4]{BH}. See also \cite[Prop.2.4]{Bushnell:c-ind}.
\end{proof}

\begin{lemma}
	Suppose that there is another subgroup $K$ containing $H$ as a normal subgroup, and 
	that the 
	intertwining set is equal to the set of $g$ whose reduction modulo $t^m$ is in $K(\overline 
	\kappa)$ and such that 
	$\mathcal{L}\simeq 
	\mathcal{L}^g$.
	Then $(G,m,H,\mathcal{L})$ is 
	geometrically supercuspidal.
\end{lemma}
\begin{proof}
		We again apply Lemma~\ref{checking-supercuspidality}, and verify the vanishing of 
		the 
	Jacquet module for every finite extension $\kappa'$.
	This follows from~\cite[Lem.2.2]{Reeder-Yu:epipelagic}.
\end{proof}

\subsection{Monomial geometric supercuspidal representations}\label{s:mgsr}

Let $G$ be a reductive group over a finite field $\kappa$. 

\begin{defi}\label{def:mgs-rep}
We say 
that an irreducible smooth representation of $G(\kappa\pseries{t})$ is \emph{mgs} if 
there exists a natural number $m$, a connected subgroup $H$ of $G \lWR \kappa[t]/t^m 
\rWR$, and a character sheaf $\mathcal L$ on $H$, such that
\begin{enumerate}
\item  $(G,m,H,\mathcal L)$ is 
geometrically supercuspidal;
\item  $\pi$ is a quotient of 
$\cind_J^{G(\kappa\pseries{t})} \chi$ where $J$ is the inverse image of 
$H(\kappa)$ in $G (\kappa\bseries{t})$ and $\chi$ is the trace function of $\mathcal 
L$ on 
$H(\kappa)$, pulled back to $J$. 
\end{enumerate}
Furthermore, in this setting, we say that $(G, m, H, \mathcal L)$ is an \emph{mgs datum} 
for $\pi$.
By Lemma~\ref{l:Frob-reciprocity}, condition (2) is equivalent to that the representation 
has a non-zero $(J,\chi)$-invariant vector.
\index{mgs, monomial geometric supercuspidal} \index{mgs datum}
\end{defi}

\begin{lemma}\label{mgs-galois} Let $\pi$ be an mgs representation of $G( 
\kappa\pseries{t})$. Then the pullback of $\pi$ by any automorphism of the field 
$\kappa\pseries{t}$ is an mgs representation. \end{lemma}

\begin{proof} Any such automorphism is a composition of an automorphism of $\kappa$ with 
a change of variables that sends $t$ to a power series with leading term a constant 
multiple of $t$. Automorphisms of $\kappa$ act in a natural way on the mgs datum $(G,m, 
H, \mathcal L)$.  Changes of variables in $t$ act in a natural way on $G \lWR \kappa 
[t]/t^m \rWR$ and hence act in a natural way on $H$ and $\mathcal L$. Both of these 
automorphisms agree with the action of the field automorphism on the induced 
representation, hence preserve the vanishing property of Jacquet modules vanishing, and 
also therefore agree with the pullback of $\chi$. \end{proof}

\begin{lemma}\label{mgs-aut} Let $\pi$ be an mgs representation of 
$G(\kappa\pseries{t})$. Then the pullback of $\pi$ by any automorphism of the group 
$G$ defined over $\kappa\pseries{t}$ is an mgs representation. \end{lemma}

\begin{proof} Let $(G, m, H, \mathcal L)$ be an mgs datum for $\pi$.  By Lemma 
\ref{supercuspidal-parahoric}, $H$ is unipotent. In particular, its image inside $G$ 
is solvable, and hence contained in a Borel subgroup $B$. The inverse image of $B$ in 
$G(\kappa\bseries{t})$ is a minimal parahoric subgroup $I$ of $G(\kappa\pseries{t})$. (This 
follows from the explicit description of the parahoric subgroup in terms of roots. If we 
take an apartment corresponding to the inverse image of a torus of $G$ and perturb  
the hyperspecial point associated to $G(\kappa\bseries{t})$ in a generic direction, 
producing a point in the interior of a chamber whose associated subgroup is a minimal 
parahoric, we see that the parahoric subgroup is the inverse image of some Borel, and 
because all Borels are conjugate all such subgroups are minimal parahoric.)  Because 
all minimal parahoric subgroups are conjugate \cite[\S9]{Landvogt:LNM:compactification-BT}, 
every automorphism of $G_{\kappa\pseries{t}}$ can 
be 
expressed as an inner automorphism composed with an automorphism $\sigma$ such that 
$\sigma(I)=I$. Conjugation by an element of 
$G(\kappa\pseries{t})$ produces a representation isomorphic to 
$\pi$, so it suffices to show that the pullback of $\pi$ by $\sigma$ is mgs. 

Expressing $\sigma$ in the coordinates of $G$, let $\delta$ be the highest power of 
$t^{-1}$ that appears. Then for any $g \in I$, $\sigma(g)\in I$ 
and $\sigma(g)$ modulo $t^m$ depends only on $g$ modulo $t^{m+\delta}$. Hence $\sigma$ 
defines a map $\overline{\sigma}$ from the subset of $G \lWR \kappa[t]/t^{m+\delta} \rWR $ congruent to $B$ 
modulo $t$ to the subset of $G \lWR \kappa[t]/t^{m} \rWR $ congruent to $B$ modulo $t$. 
Because $\sigma$ acts as an automorphism of $I$,  $\overline{\sigma}$ is surjective. 

Consider the datum $(G, m+\delta, \overline{\sigma}^{-1}(H), \overline{\sigma}^* \mathcal L)$. Let $J'$ be 
the inverse image of $\overline{\sigma}^{-1}(H) (\kappa)$ in $G ( \kappa\bseries{t})$ and let 
$\chi'$ be the pullback of the trace function of $\overline{\sigma}^* \mathcal L$ to $J'$. Then $J'= 
\sigma^{-1}(J)$ and $\chi'= \chi \circ \sigma$, so 
$\cind_{J'}^{G(\kappa\pseries{t})} \chi'$ is the pullback of 
$\cind_{J}^{G(\kappa\pseries{t})} \chi$ by $\sigma$ and hence $\pi 
\circ \sigma$ is a quotient of it.

Similarly, over any finite field extension $\kappa'$ of $\kappa$, 
$\cind_{J'_{\kappa'} }^{G(\kappa'\pseries{t})} \chi'_{\kappa'}$ is the 
pullback of $\cind_{J_{\kappa'}'}^{G(\kappa'\pseries{t})} 
\chi_{\kappa'}$, hence has vanishing Jacquet module for every parabolic subgroup, so by 
Lemma \ref{checking-supercuspidality} it is geometrically supercuspidal. 

Hence $\pi \circ \sigma$ is an mgs representation. 
\end{proof}

Any unramified reductive group over an equal characteristic local field $F$
necessarily descends to a 
reductive group $G$ over the residue field $\kappa$ (Lemma~\ref{l:unramified}). 
Combined with the previous two 
lemmas, that allows us to 
give an intrinsic definition of mgs representations of an unramified group over $F$.
Namely $\pi$ is mgs if for some (equivalently any) uniformizer $t$ of $F$, and 
for reductive group $G$ over $\kappa$ and some (equivalently any) isomorphism with
$G_{\kappa\pseries{t}}$, the representation $\pi$ is a quotient 
of $\cind^{G(F)}_J \chi$ for some geometrically supercuspidal datum $(G,m,H,\cL)$.

\begin{remark}
We can make a similar definition over a mixed characteristic local field $F$, and for a 
general reductive group $G$ over $F$ as follows.
Let $\mathfrak o_F$ be its ring of integers, $\varpi$ an uniformizer, and 
$\kappa$ its residue field. Let $\cG$ be a 
smooth group scheme over $\mathfrak o_F$ whose generic fiber is isomorphic to $G$.  Let 
$\cG_m$ 
be the algebraic group over $\kappa$ whose $R$-points for a ring $R$ over $\kappa$ are 
the 
$W_m(R) \otimes_{W(\mathbb \kappa)} \mathfrak o_F$-points of $\cG$, where $W$ is the Witt 
vectors functor and $W_m(R)$ is the ring of truncated Witt vectors modulo $p^m$. (Here 
the Witt 
vectors are 
defined 
using universal polynomials over an imperfect ring). A \emph{monomial datum} consists of a 
connected closed subgroup $H$ 
of $\cG_m$ and a character sheaf $\mathcal L$ on $H$. The datum is \emph{geometrically 
supercuspidal} if for every proper parabolic 
subgroup $P\subset G$ with maximal unipotent $N$, closure $\cN$ in $\cG$, and associated 
$\kappa$-group $\cN_m$, 
and every $g \in \cG_m(\overline \kappa)$, the 
restriction of 
$\mathcal L_{\overline \kappa}$ to the identity component of $H_{\overline \kappa} \cap g 
\cN_{m,\overline{\kappa}} 
g^{-1}$ is nontrivial. Let $J$ be the inverse image of 
$H(\kappa) \subseteq \cG_m(\kappa) = \cG 
(\mathfrak o_F / \varpi^m) $ in $\cG(\mathfrak o_F)$ and let $\chi$ be the pullback of 
the 
trace function of $\mathcal L$ from $H(\kappa)$ to $J$. We say that an 
irreducible smooth representation of $G(F)$ is \emph{mgs} if is has a non-zero 
$(J,\chi)$-invariant vector, or equivalently if it appears as a quotient of 
$\cind_J^{G(F)} \chi$.
\end{remark}

\subsection{Moy--Prasad types and epipelagic representations}\label{sub:epipelagic}
Let $G$ be a quasi-split reductive group over $\kappa$, and $F=\kappa\pseries{t}$. Let 
$x$ 
be a 
point in the Bruhat--Tits building of $G(F)$, and let $r>0$ be a positive real. Let 
$G(F)_{x,r}$ and $G(F)_{x,r+}$ refer as usual to Moy--Prasad subgroups of $G(F)_x$. Then 
$G(F)_{x,r}/ 
G(F)_{x,r+}$ is a vector space over $\kappa$. Let $\chi$ be a character of 
$G(F)_{x,r}$ that factors through this vector space.

\begin{lemma} $G(F)_{x,r}$ is conjugate to a subgroup of $G (\kappa \bseries{t})$. 
\end{lemma}

\begin{proof} It is contained in a minimal parahoric subgroup (e.g. the one associated to any 
adjacent chamber of the Bruhat--Tits building), and we may conjugate it to a minimal 
parahoric 
subgroup inside $G(\kappa \bseries{t})$ \cite[\S9]{Landvogt:LNM:compactification-BT}. \end{proof}

\begin{lemma}\label{moy-prasad-construction} Suppose that $G(F)_{x,r} 
\subseteq G (\kappa \bseries{t})$. Then there exists a natural number $m$, algebraic subgroups $H 
\subseteq G 
\lWR \kappa[t]/t^m \rWR$, $H' \subseteq H$ such that the inverse image of $H(\kappa)$ in 
$G(\kappa\bseries{t})$ is $G(F)_{x,r}$ and the inverse image of $H'(\kappa)$ is 
$G(F)_{x,r+}$. Furthermore, for any finite field extension $\kappa'$ of $\kappa$, the 
inverse image of $H(\kappa')$ is $G(\kappa'\pseries{t})_{x,r}$ and the inverse image of 
$H'(\kappa')$ is $G(\kappa'\pseries{t})_{x,r+}$.

Finally, $H/H'$ is isomorphic to a vector space as an algebraic group. \end{lemma}

The conditions on the rational points uniquely characterize the groups $H$ and $H'$. 

\begin{proof} For some $m$, 
$G(F)_{x,r+}$ contains the subgroup of elements congruent to the identity modulo $t^m$, 
so 
that $G(F)_{x,r}$ and $G(F)_{x,r+}$ are the inverse images of their projections to $G( 
\kappa[t]/t^m)$.

 It is clear from the definition of $G(F)_{x,r}$ and $G(F)_{x,r+}$ that these 
 projections are 
 algebraic subgroups --- the Moy--Prasad subgroups are defined as the subgroups 
 generated by 
 certain additive and multiplicative groups, and we can simply take the algebraic 
 subgroup generated by these groups.
 
 Furthermore, because $r>0$, all the involved subgroups are additive, and their 
 commutators lie in $G(F)_{x,r+}$, so the $H/H'$ is a vector space. 
 \end{proof}
 
 Any character $\chi$ of $G(F)_{x,r}$ trivial on $G(F)_{x,r+}$ defines a character of 
 $G(F)_{x,r}/G(F)_{x,r+}= H(\kappa)/H'(\kappa) = H/H'(\kappa)$ and hence,  
 by Lemma \ref{character-sheaf-construction}, a character sheaf $\mathcal L$ on $H$. By 
 construction, this datum $(G,m,  H, \mathcal L)$ satisfies $J = G(F)_{x,r}$ and 
 $\chi=\chi$. Hence if $(G,m,H,\mathcal L)$ is geometrically supercuspidal, any 
irreducible representation containing a vector on which $G(F)_{x,r}$ acts through the character 
 $\chi$ is mgs. 
 
 A concrete description of when this occurs is provided by Lemma 
 \ref{checking-supercuspidality}. 
 
 We give here a different condition, inspired by the construction of epipelagic 
 representations of Reeder and Yu \cite{Reeder-Yu:epipelagic}. 
 
 \begin{lemma} Let $H$ and $H'$ be the subgroups of Lemma \ref{moy-prasad-construction}. 
 Let $\lambda: H/ H' \to \mathbb G_a$ be a linear map, let $pr: H \to H/H'$ be the 
 projection, let $\mathcal \psi$ be an additive character of $\kappa$, and let $\chi 
 = \psi \circ \lambda \circ pr$ be the trace function of the character sheaf $pr^* 
 \lambda^* 
 \mathcal L_\psi$.
 
 Then $(G,m,H , pr^* \lambda^* \mathcal L_\psi)$ is geometrically supercuspidal if and 
 only if $\lambda$ is GIT-stable for the action of $G(F)_{x,0}/ G(F)_{x,0+}$ on  
 $(H/H')^\vee$. \end{lemma}
 
 \begin{proof}By conjugation, we may assume that $G(F)_{x,0}$ contains the standard 
 minimal parahoric subgroup $I$ (the inverse image in $G (\kappa\bseries{t})$ of a fixed Borel subgroup 
 of the quasi-split group $G(\kappa)$) and hence that $x$ lies in the apartment of the 
 standard maximal torus $T$. Let $P$ be a standard parabolic, and consider a conjugate $g P 
 g^{-1}$. Because $G/P$ is proper and $H^1 ( \operatorname{Spec} \kappa \bseries{t},P)$ is trivial, the natural map $G (\kappa \pseries{t} )/ P (\kappa \bseries{t} ) \to G(\kappa \pseries{t} )/P (\kappa \pseries{t})$ is a bijection, and so we may assume $g \in G (\kappa\bseries{t})$. By further multiplying on the right by an element of $P$, we may assume that $g$ mod $t$ is an element of the Borel times an element of the Weyl group, so $g= g_0w$ where $g_0\in I$ and $w$ lies in the Weyl group.

 Because $P$ is a standard parabolic subgroup, there is some cocharacter $\alpha: \mathbb 
 G_m \to T$ of the standard maximal torus $T$ such that the unipotent subgroup $N$ of $P$ 
 consists of those roots which have positive eigenvalue under $\alpha$. Then $w N w^{-1}$ 
 consists of those roots which have a positive eigenvalue under $w \alpha w^{-1}$. Hence 
 $w N w^{-1} \cap G(F)_{x,r} / (w N w^{-1}  \cap G(F)_{x,r+}) =w N w^{-1} \cap H /( w N 
 w^{-1}  
 \cap H')  $ is generated by the elements of $H/H'$ which have a positive eigenvalue 
 under $w \alpha w^{-1}$, as $H/H'$ has a basis consisting of roots. Thus the projection 
 onto $H/H'$ of $g N g^{-1}$ is generated by the elements which have a positive 
 eigenvalue under $g_0 w \alpha w^{-1} g_0^{-1} $, which is a cocharacter of 
 $G(F)_{x,0}$. Therefore the set of linear forms on $H/H'$ that vanish on that projection is 
 the subspace generated by the elements which have a nonnegative eigenvalue under $g 
 \alpha g^{-1}$. If $\lambda$ is GIT-stable, then by the Hilbert--Mumford criterion it 
 does not lie in this space, so it is nontrivial on the image, hence the pullback of 
 $\mathcal L_\psi$ under $\lambda$ is nontrivial on this image, as desired.
 
 For the converse, if $\lambda$ is not stable, we have a cocharacter of $G(F)_{x,0}$ 
 such that $\lambda$ is a sum of linear forms on $H/H'$ that are eigenvectors of this 
 cocharacter with nonnegative eigenvalue. Hence $\lambda$ vanishes on all elements of 
 $H/H'$ that have positive eigenvalue under the cocharacter. Now let $P$ be the parabolic 
 subgroup of $G$ generated by the maximal torus and all the roots that have nonnegative 
 eigenvalue under this character. Then all elements of $N$ have positive eigenvalue, so 
 $\lambda$ vanishes on $H \cap N$ and therefore $(G,m,H , pr^* \lambda^* \mathcal 
 L_\psi)$  is not mgs. \end{proof}
 
 \begin{corollary} If $G$ is unramifed semisimple, then the epipelagic supercuspidal 
 representations constructed in \cite{Reeder-Yu:epipelagic} are mgs. \end{corollary}
 
 \begin{proof} They are by definition summands of 
 $\cind_{G(F)_{x,r}}^{G(F)} \chi$ for $r$ the minimum positive value and 
 $\chi$ a GIT-stable character of $G(F)_{x,r}/ G(F)_{x,r}^+$.  \end{proof}

\begin{example} We review the simplest example of an epipelagic representation, which 
is also the simplest example of an mgs representation. Let $G =SL_2$ and let $x$ be the 
midpoint of an edge between two vertices of the Bruhat--Tits tree. Let $\mathfrak o = 
\kappa\bseries{t}$ and $\mathfrak p = t\kappa\bseries{t}$. Then $G(F)_{x,0+} = G(F)_{x,1/2}$ is the 
subgroup of matrices of the form $ \begin{pmatrix} 1 +\mathfrak p & \mathfrak o \\ \mathfrak p &1 + 
\mathfrak p\\ \end{pmatrix} $ and $G(F)_{x,1/2+}$ is the subgroup of matrices of the form  $ 
\begin{pmatrix} 1 + \mathfrak p & \mathfrak p \\ \mathfrak p^2&1+ \mathfrak p\\ \end{pmatrix} $, so the 
quotient is isomorphic to $\kappa^2$, given by extracting the leading terms of the top-right 
and bottom-left matrix entries.

Furthermore $G(F)_{x,0}$ is the subgroup of matrices of the form $ \begin{pmatrix} 
\mathfrak o & \mathfrak o \\ \mathfrak p & \mathfrak o\\ \end{pmatrix} $, and so 
$G(F)_{x,0}/ G(F)_{x,0+}$ consists of the cosets $  \begin{pmatrix} a+\mathfrak p & 
\mathfrak o\\ \mathfrak p & a^{-1}+\mathfrak p \\ \end{pmatrix} \in G(F)_{x,0}/ 
G(F)_{x,0+}$ for $a\in \kappa^\times$. The action of such a coset is by multiplication by 
$a^2$ on the top-right entry and $a^{-2}$ on the bottom-left entry, so the stable 
characters are exactly the characters nontrivial on the top-right and bottom-left 
entries. 

The associated mgs datum has $m=2$,  $H$ the four-dimensional subgroup of matrices in 
$SL_2 \lWR \kappa[t]/t^2 \rWR$ congruent mod $t$ to an upper-triangular unipotent matrix, 
and $\mathcal L$ the unique character sheaf on $H$ whose trace function is any fixed character nontrivial on the top-right and bottom-left entries. 
\end{example}

\subsection{Adler datum and toral representations}\label{sub:adler}

We now describe a special case of the construction of \cite{adler-unrefined} that 
produces mgs representations.  To that end, we borrow some notation from 
\cite{adler-unrefined}. Let $G$ be an unramified semisimple group over $F = 
\kappa\pseries{t}$ satisfying 
\cite[Hypothesis 2.1.1]{adler-unrefined}. This allows us to take a $G$-equivariant 
symmetric bilinear form on the Lie 
algebra $\mathfrak g$ of $G$, so that there is an induced isomorphism between $\mathfrak g$ and 
its dual.

 Let $T$ be a maximal $F$-torus of $G$ that splits over a tamely ramified extension $E$ 
 of  $F$ but such that $T/Z(G)$ has no nontrivial map to $\mathbb G_m$ defined over any 
 unramified extension of $F$. Let $X$ be an element of the Lie algebra of $T$. Assume 
 that there is a positive rational number $r$ such that the valuation of $d\alpha(X)$ for 
 every root $\alpha$ of $T$ defined over $E$ is equal to $r$.

Let $x$ be the unique point of the Bruhat--Tits building of $G$ that belongs to the 
apartment of $T$ inside the Bruhat--Tits building of $G(E)$. Let $G(F)_{x,r}, 
G(F)_{x,r+}, 
\mathfrak g_{x,r}, \mathfrak g_{x,r+}$ be the corresponding Moy--Prasad subgroups of $G$ 
and $\mathfrak g$. Then because $r>0$, we may identify $G(F)_{x,r} / G(F)_{x,r+} = 
\mathfrak 
g_{x,r}/ \mathfrak g_{x,r+}$ \cite[(1.5.2)]{adler-unrefined}.  Using the bilinear form, 
we may view $X$ as a character 
of $\mathfrak g_{x,r}/\mathfrak g_{x,r+}$, defining a character $\chi$ of $G(F)_{x,r}$.

\begin{proposition}\label{adler-mgs} Any irreducible representation $\pi$ of $G(F)$ that contains 
$(G(F)_{x,r}, \chi)$ is mgs. \end{proposition}

\begin{proof}  We use the datum $(G,m,H ,\psi)$ constructed in the previous section. It 
remains to check that this datum is geometrically supercuspidal, which we do using 
Lemma 
\ref{checking-supercuspidality}.
 
 It is sufficient to show that, after base-changing to a finite extension of $\kappa$, 
 the Jacquet modules of this induced representation vanish. Because all our assumptions 
 are stable under base change of $\kappa$, it in fact suffices to show that, for all $F, 
 G, T, X$ satisfying these assumptions, the Jacquet module of 
 $\cind_{G(F)_{x,r}}^{G(F)} \chi$ vanishes. 
 
 Adler defines $M$ to be the centralizer of $X$ in $G$. In our case, that is simply 
 $T(F)$, because our assumptions imply that $d\alpha(X) \neq 0$ for every root $\alpha$ 
 of $T$. Because $T$ is anisotropic, $M$ is compact. Adler defines $J$ as $\phi_x( 
 \mathfrak m_{x,r} \oplus \mathfrak m_{x,(r/2)}^{\perp})$, where $\mathfrak m_{x,r}$ and 
 $\mathfrak m_{x,(r/2)}^{\perp}$ are the Moy--Prasad subspaces of the Lie algebra of $T$ 
 and its orthogonal complement in the Lie algebra of $G$ respectively, and $\phi_x$ is an 
 approximate exponential map. For our purposes, it is most significant that $J$ is 
 compact and contains $G(F)_{x,r}$, and is normalized by $M$, so $MJ$ is compact and 
 contains $G(F)_{x,r}$ as an open subgroup.
 
 Thus $\cind_{G(F)_{x,r}}^{MJ} \chi$ is 
 a sum of irreducible representations $\sigma$ of $MJ$, each containing 
 $(G(F)_{x,r},\chi)$ 
 by Frobenius reciprocity. The induced representations $\cind_{G(F)_{x,r}}^{G(F)} \chi =  
 \cind_{MJ}^{G(F)} \cind_{G(F)_{x,r}}^{MJ} \chi$ is the sum of $ 
 \cind_{MJ}^{G(F)} 
 \sigma$, and by the discussion at the beginning of \cite[\S2.5]{adler-unrefined}, these 
 are supercuspidal, so it is a sum of 
 supercuspidal representations, hence has vanishing Jacquet modules.
 \end{proof}

Proposition \ref{adler-mgs} shows that all the representations produced by the construction of 
Adler in the case where $G$ is unramified and semisimple and the centralizer $M$ of $X$ 
is not just anisotropic over the base field but over all unramified extensions are mgs. (To see 
this, we must observe that $M$ anisotropic over unramified extensions implies that $M$ is 
a torus, as all groups become quasi-split over some unramified extension, and hence 
equals $T$. If $M=T$, then $d\alpha(X)\neq 0$ for any root $\alpha$ of $T$. This 
condition, plus the stronger anisotropic condition for $T$, are our only points of 
departure from the setup of~\cite{adler-unrefined}.)

\subsection{Non-examples}\label{sub:non-examples}

We discuss some examples of data $(G,m,H,\mathcal L)$ that are not geometrically supercuspidal and so do not lead to mgs representations.

\begin{example} If $\mathcal L$ is trivial then $(G, m, H, \mathcal L)$ cannot be mgs 
unless $G$ is a torus, as there will always be at least one proper parabolic subgroup. In 
particular, we can simply take $H$ to be trivial.
\end{example}

\begin{example} If the order of the monodromy group of $\mathcal L$, which, by Lemma 
\ref{character-sheaf-uniqueness}, is equal to the order of the associated character, is 
prime to $p$, then its pullback to the intersection with any unipotent subgroup will have 
order prime to $p$, but the order of the unipotent subgroup is a power of $p$, so the 
character sheaf is trivial on that intersection. Thus $(G, m ,H, \mathcal L)$ is not mgs 
unless $G$ is a torus.

For instance, we can take $m= 1$, $H$ a Borel subgroup of $G$, and $\mathcal L$ the pullback of a character sheaf on the maximal torus. It is possible in this case for $\cind_{J}^{G(\kappa\bseries{t})} \chi$ to be irreducible (the inflation of an irreducible principle series representation of $G(\kappa)$) but the Jacquet module of the induced representation is nonvanishing.
\end{example}

\begin{example} We provide an example of a $(G,m , H, \mathcal L)$ which is not mgs even 
though the Jacquet modules of the induced representation are trivial. Let $G=\GL(2)$, 
$m=2$, 
and $H$ be the subgroup of elements congruent to $1$ mod $t$, which is isomorphic to the 
Lie algebra of $G$, i.e., the vector space of $2 \times 2$ matrices.  Consider the linear 
function $A \mapsto \tr(AB)$ on the Lie algebra of $G$, where $B$ is a non-scalar element 
of a 
non-split Cartan of $M_2(\kappa)$. View $H$ as the Lie algebra of $G$ and pull back an 
Artin--Schreier sheaf $\mathcal 
L_\psi$ to $H$ under this linear function. Then for any parabolic subgroup $P$, $g N g^{-1} \cap H$ is a 
one-dimensional vector space of nilpotent matrices, so the character is trivial when 
pulled back to that subgroup if and only if the trace of $B$ times the nilpotent matrix 
vanishes, which happens if and only if $B$ is contained in the associated Borel.  Over 
$\kappa$, this is impossible, so the Jacquet module vanishes, and the induced 
representation is a sum of supercuspidals. However, over a quadratic extension of $\kappa$, there are two Borels 
containing $B$, so $(G,m , H, \mathcal L)$ is not geometrically supercuspidal. 

\end{example}

\subsection{Preservation of mgs}\label{sub:preservation}
We show that mgs representations are preserved under some natural 
operations on algebraic groups. For this subsection and~\S\ref{s:character-datum} only, 
we denote 
groups over the local 
field $F$ by the roman letter $G$, and groups over the residue field $\kappa$ by the bold 
letter $\bG$.

\begin{lemma}\label{l:mgs-restriction}
Let $f: G_1 \to G_2$ be a homomorphism of unramified reductive groups over an equal 
characteristic local field $F$ whose kernel is a torus and whose image is a normal 
subgroup with quotient a torus. Let $\pi_2$ be an mgs representation of $G_2(F)$. Then 
any irreducible quotient $\pi_1$ of $\pi_2 \circ f$ is mgs.  \end{lemma}

\begin{proof} Let $F = \kappa\pseries{t}$. We may choose descents $\bG_1$ and $\bG_2$ 
of $G_1 $ and $G_2$ to $\kappa$ such that $f$ is defined over $\kappa $, because $G_1$ 
and $G_2$ have the same Bruhat--Tits buildings and the same hyperspecial points.

Let $(\bG_2, m, H, \mathcal L)$ be mgs datum for $\pi_2$.  Let $J_2$ be the subgroup 
defined 
by this datum and $\chi$ the character.  Let $v$ be a vector in $\pi_2$ which 
transforms under the subgroup $J_2$ by the character $\chi_2$. Because $\pi_2$ is irreducible, there must be some $g \in G_2(F)$ such that $g v$ remains nonzero in the quotient $\pi_1$. This vector $gv$ transforms under the subgroup $ 
f^{-1} (g J_2 g^{-1})$ by $\chi_2 \circ f$. Because conjugation by $g$ is an outer 
automorphism of $G_1$, and geometric supercuspidality is preserved by automorphisms (Lemma 
\ref{mgs-aut}), we may assume $\pi_1$ contains a vector that transforms under the 
subgroup $f^{-1}(J_2)$ by the character $\chi_2\circ f$.

We have a map $f : \bG_1 \lWR \kappa[t]/t^m \rWR \to  \bG_2  \lWR \kappa[t]/t^m \rWR$. 
It suffices to show that $(G_1, m , f^{-1}(H), f^* \mathcal L)$ is mgs datum for 
$\pi_1$. Let $J_1$ be the subgroup defined by this datum and let $\chi_1$ be the 
character. We have $J_1 = f^{-1}(J)$ and $\chi_1 = \chi_2 \circ f$, so it remains to show 
that $(\bG_1,m ,f^{-1}(H), f^* \mathcal L)$ is geometrically supercuspidal. Let $P_1$ be 
a parabolic subgroup of $\bG_1$. Then $P_1$ is the inverse image under $f$ of a 
parabolic subgroup $P_2$ of $\bG_2$. Moreover, for $N_1$ and $N_2$ the maximal 
unipotent 
subgroups of $P_1$ and $P_2$, $f: N_1 \lWR \kappa[t]/t^m \rWR \to N_2 \lWR \kappa[t]/t^m \rWR$ is an 
isomorphism, because the kernel of $f$ is a torus and does not intersect the unipotent 
subgroups, while the cokernel of $f$ is a torus and so the image of the unipotent 
subgroup in it is trivial. So for any $g$ in $\bG_1 \lWR \kappa[t]/t^m \rWR$,   $f: g N_1 
g^{-1} \cap f^{-1}(H) \to f(g) N_2 f(g^{-1}) \cap H$ is an isomorphism, and since the 
pullback of $\mathcal L$ to $ f(g) N_2 f(g^{-1}) \cap H$ is nontrivial, the pullback of $\mathcal L$ to 
$g N_1 g^{-1} \cap f^{-1}(H)$ is nontrivial.
\end{proof}

\begin{lemma}\label{l:mgs-product}
Let $G_1$ and $G_2$ be unramified reductive groups over an equal characteristic local 
field $F$. Let $\pi= \pi_1 \boxtimes \pi_2$ be a mgs representation of $G_1(F)\times 
G_2(F)$, where $\pi_1$ is a representation of $G_1(F)$ and $\pi_2$ is a 
representation of $G_2(F)$. Then $\pi_1$ and $\pi_2$ are mgs representations of 
$G_1(F)$ and $G_2(F)$ respectively. \end{lemma}

\begin{proof} Let $F = \kappa\pseries{t}$. Choose descents $\bG_1$ and $\bG_2$ and 
isomorphisms $\bG_{1,F} = G_1, \bG_{2,F}=G_2$. Let $G = G_1 \times G_2$, and $\bG=\bG_1 
\times \bG_2$.

 Choose an mgs datum $(\bG, m, H, \mathcal L)$ for $\pi$. Let $H_1 = H \cap \bG_1 \lWR 
 \kappa[t]/t^m \rWR$ and let $H_2 = H \cap \bG_2 \lWR \kappa[t]/t^m \rWR$. Let $\mathcal 
 L_1$ be the pullback of $\mathcal L$ to $H_1$ and let $\mathcal L_2$ be the pullback of 
 $\mathcal L$ to $H_2$. 

To show that $(\bG_1, m,H_1,\mathcal L_1)$ and $(\bG_2, m, H_2, \mathcal L_2)$ are 
geometrically supercuspidal, observe that for any parabolic subgroup $P_1$ of $\bG_1$ 
with maximal unipotent subgroup $N_1$, $P_1 \times \bG_2$ is a parabolic subgroup of 
$\bG=\bG_1 \times \bG_2$ with maximal unipotent subgroup $N_1 \times e$, and for any 
$(g_1,g_2) \in \bG_1 \lWR \kappa[t]/t^m \rWR \times \bG_2 \lWR \kappa[t]/t^m \rWR$, \[ H 
\cap 
(g_1,g_2)(N_1 \times e) (g_1,g_2)^{-1} = H_1 \cap g_1 N_1 g_1^{-1} \] so the pullback of 
$\mathcal L_1$ to $g_1 N_1 g_1^{-1}$ is geometrically nontrivial. The same argument works 
symmetrically for $\bG_2$.

Letting $J, J_1,J_2, \chi,\chi_1,\chi_2$ be the subgroups and characters associated to 
the various data, we have $J_1 \times J_2 \subseteq J$ and $\chi_1 \times \chi_2$ is the 
restriction of $\chi$ to $J_1 \times J_2$, so there is a surjection

\[ \cind_{J_1}^{ G_1(F)} \chi_1 \boxtimes \cind_{J_2}^{ 
G_2(F)} \chi_2 = \cind_{J_1 \times J_2}^{ G_1(F) \times G_2(F)}( 
\chi_1 \times \chi_2 ) \to \cind_J^{G_1(F) \times G_2(F)} \chi \to \pi 
= \pi_1 \boxtimes \pi_2\] and thus surjections $\cind_{J_1}^{ G_1(F)} \to \pi_1$ and 
$\cind_{J_2}^{ G_2(F)} \chi_2  \to\pi_2$, as desired.
\end{proof}

\begin{lemma}\label{l:mgs-extension}
 Let $E/F$ be an unramified extension of local fields. Let $G$ be an unramified 
 reductive group over $E$. Let $\pi$ be a representation of $G(E)$. Then $\pi$ is an 
 mgs representation of $G$ over $E$ if $\pi$ is an mgs representation of the 
 $F$-points of the Weil restriction of $G$ from $E$ to $F$. \end{lemma}

\begin{proof} We may take $F =\kappa\pseries{t}$ and let $E= \kappa'\pseries{t}$. Let 
$\bG'$ be a group over $\kappa'$ with $\bG'_{E}= G$. Let $\bG$ be the Weil 
restriction of $\bG$ from $\kappa'$ to $\kappa$. Then $\bG_{F} $ is the Weil 
restriction of $G$ from $E$ to $F$. Let $(\bG, m, H ,\mathcal L)$ be an mgs datum for 
$\pi$. 

For $R$ a ring over $\kappa'$, by definition
\[\bG_{\kappa'} (R) =  \bG' (R \otimes_\kappa \kappa') = \bG'( \prod_{ \sigma \in \operatorname{Gal} ( \kappa' / \kappa) } R ) = \prod_{ \sigma \in \operatorname{Gal} ( \kappa' / \kappa) }  \bG'(R) .\] This defines an isomorphism  $\bG_{\kappa'}  \equiv  \bG'^{[\kappa':\kappa]}$. Let $i$ be the map $\bG' \to \bG_{\kappa'} $ defined as the inclusion of the factor corresponding to the identity element of the Galois group under this isomorphism.

 Let $H'= i^{-1}(H)$ and let $\mathcal L'$ be the restriction of $\mathcal L$ to $H$. Then to 
check that $ ( \bG', m, H', \mathcal L')$ is geometrically supercuspidal, fix $P'$ a parabolic subgroup 
of $\bG'$, and let $P$ be the product of $P'$ on the factor corresponding to the identity element with $\bG'_{\kappa}$ on all the 
other factors, so that $N = i(N')$, and thus for any $g 
\in \bG' \lWR \kappa'[t]/t^m \rWR$, $ i(g) N i(g)^{-1} = i (  g N' g^{-1})$. Hence 
because $\mathcal L$ is nontrivial on $i(g) N i(g)^{-1} \cap H$, the restriction of $\mathcal 
L$ is nontrivial on $g N' g^{-1} \cap H'$.

Next observe that for $g' \in \bG' ( \kappa' [t]/t^m) $, since $\bG' ( \kappa' [t]/t^m) = \bG(\kappa[t]/t^m)$ by the definition of Weil restriction, there is a corresponding element $g$ in $\bG(\kappa[t]/t^m)$, which we may pull back to $\bG ( \kappa'[t]/t^m)$. We can calculate \[g = \prod_{ \sigma \in   \operatorname{Gal} ( \kappa' / \kappa) } \sigma( i(g'))\] where $\sigma$ acts on $G ( \kappa'[t]/t^m)$ by its action on $\kappa'$, not on the group scheme $G$ over $\kappa$. If $g' \in H'(\kappa) $ then $ i ( g) \in H(\kappa') $. Because $H$ is defined over $\kappa$, this means  $\sigma ( i(g))\in H(\kappa')$ for all automorphisms $\sigma$. Hence $g \in H(\kappa')$, and then because $g \in \bG(\kappa[t]/t^m)$, we finally have $g \in H(\kappa)$. 

It follows that $J'$ is a subgroup of $J$ when both are viewed as subgroups of $G(E)$. Furthermore $\chi'$ is the restriction of $\chi$ to $J'$, so since 
$\pi$ contains a vector transforming under the character $\chi$ of the subgroup $J$, it 
contains a vector transforming under the character $\chi'$ of the subgroup $J'$.
\end{proof}

\subsection{Admissibility}\label{sub:admissibility} In \S\ref{s:vanishing-Jacquet}, we 
discussed the vanishing of 
Jacquet modules of certain induced representations $\cind^{G(F)}_J \chi$, but did not 
otherwise describe their structure. We now 
present a lemma giving a condition for such an induced representation  (and some more 
general ones) to be admissible, in which case it follows that it is a 
finite direct sums of supercuspidals. 
The lemma has some similarities with~\cite[\S{III.2}]{Harish-Chandra:analysis-padic-gps}.

\begin{lemma}Let $F = \kappa \pseries{t}$, $G$ a semisimple group over $\kappa$, $J$ a 
compact open subgroup of $G (F)$, and $\sigma$ a smooth 
finite-dimensional representation of $J$.
Suppose for any proper parabolic subgroup $P$ of $G(F)$, with unipotent radical $N$, the 
restricted representation $\sigma |_{ J \cap N }$ does not contain the trivial 
representation.
Then $\cind^{G(F)}_J(\sigma)$ is a finite direct sum of supercuspidal representations. 
\end{lemma}

\begin{remark}
The semisimplicity condition is necessary because $\cind^{G(F)}_J \sigma$ is never admissible 
if the center of ${G(F)}$ is non-compact. For 
example, all unramified characters of $\F_q\pseries{t}^\times$ appear as quotient of 
$\cind^{\F_q\pseries{t}^\times}_{\F_q\bseries{t}^\times} 1$.
\end{remark}

\begin{proof} By~\cite[Thm.1, (ii) $\implies$ (iv)]{Bushnell:c-ind}, the assertion 
follows if we prove that $\cind^G_J(\sigma)$ is admissible.
(The same proof works for an unramified semisimple group $G$ over a local field $F$ of 
characteristic zero.)

Let  $U_m$ be the principal congruence subgroup of $G(\kappa \bseries{t})$ consisting of elements 
congruent to 
$1$ mod $t^m$. 
To prove that $\cind^G_J(\sigma)$ is admissible, it is sufficient to prove that the subspace of $U_m$-invariant vectors is finite-dimensional for every integer $m$.
It suffices to prove that there are only finitely many double cosets $U_mgJ$ such that

\medskip

\noindent (C)\quad  $\sigma$ restricted to $g^{-1} U_m g \cap J$ contains the trivial 
representation.

\medskip

By the Cartan decomposition, we write $g=k'  \mu(t) k$ with $k, k' \in G( \kappa\bseries{t}) $ with $\mu$ a dominant cocharacter of $G$.
We have
$g^{-1} U_m g   =  k^{-1} \mu^{-1}(t) U_m \mu(t) k$  
(because $U_m$ is normalized by $k'$).

It is sufficient to prove that there are only finitely many possibilities for $\mu$ such that 
there is $g \in G(\kappa\bseries{t}) \mu G(\kappa\bseries{t})$ satisfying the condition (C), as $U_m$ and $J$ are finite index in 
$G(\kappa\bseries{t})$.
 
We shall show that (C) implies that $\langle \mu, \alpha \rangle < m$ for any simple root $\alpha$. 
This defines a finite subset of the cocharacter lattice. 

Suppose for contradiction that $\langle \mu, \alpha  \rangle > m$ for some simple root $\alpha$. Let $N$ be the maximal unipotent of the maximal parabolic associated to $\alpha$. 

Then $\mu^{-1}(t) U_m \mu(t) $ contains $N \cap G (\kappa\bseries{t})$. To check this, it 
is sufficient to check that for any element $u \in N \cap G(\kappa\bseries{t} ) $,  the 
matrix coefficients of $\mu(t) u \mu(t)^{-1}$ in any representation are congruent to the matrix coefficients of the
identity matrix mod $t^m$. We fix a representation, and choose a basis for that representation consisting of eigenvectors for the maximal torus $T$. For any $i,j$ in the index set of this basis, the function $u \mapsto u_{ij}$ that sends an element of $N$ to its $ij$ matrix coefficient is a polynomial function on $N$ which is equivariant according to some character $\chi_{ij}$ of $T$.

If $u \mapsto u_{ij}$ is the constant function on $N$, then \[ (\mu(t) u \mu(t)^{-1})_{ij} = e_ {ij} \] where $e \in N$ is the identity matrix, and so certainly $(\mu(t) u \mu(t)^{-1})_{ij} $ is congruent mod $t^m$ to $ e_{ij}$.

If $u \mapsto u_{ij}$  is nonconstant, then the character $\chi_{ij}$ is a nonempty product of characters $\chi_{\alpha'}$ associated to roots $\alpha'$ of $N$.

Each root $\alpha'$ of $N$ is a positive root, hence a sum of simple roots. Because all the simple roots of $G$ other than $\alpha$ are roots of the Levi of $P$, any sum of them lies in the root lattice of the Levi, and hence is not a root of $N$, so $\alpha'$ is a sum of simple roots, at least one of which is $\alpha$. Because $\mu$ is dominant and thus its pairings with the simple roots other than $\alpha$ are nonnegative, we have $\langle \alpha' ,\mu \rangle \geq \langle \alpha, \mu \rangle \geq m$ by assumption. It follows that $\chi_{\alpha'} ( \mu(t)) = t^{ \langle \alpha', \mu \rangle}$ is divisible by $t^m$.

Because $\chi_{ij}$ is a nonempty product of characters $\chi_{\alpha'}$, $\chi_{ij}(\mu(t) )$ is divisible by $t^m$. Because $u \mapsto u_{ij}$ is $\chi_{ij}$-equivariant, we have
\[ ( \mu(t) u \mu(t)^{-1} )_{ij} = \chi_{ij}(\mu(t) )  u_{ij}.\] 
Because $u \in G (\kappa\bseries{t})$,  $u_{ij}\in \kappa \bseries{t}$, so $( \mu(t) u \mu(t)^{-1} )_{ij}\in t^m \kappa \bseries{t}$. Because the identity matrix $e$ commutes with $T$, we have $e_{ij}=0$, so $( \mu(t) u \mu(t)^{-1} )_{ij}$  is congruent to $e_{ij}$ mod $t^m$.

As we have checked both cases, it follows that $\mu(t) u \mu(t)^{-1}$ is congruent as a matrix to the identity matrix mod $t^m$.
So after conjugation by $k$, we obtain that 
$g^{-1} U_m g$ contains $k^{-1} N k \cap G(\kappa\bseries{t})$, thus $k^{-1} N k \cap J$.
By assumption, the restriction of $\sigma$ to $k^{-1} N k \cap J$ does not contain the 
trivial representation.
A fortiori, the restriction of $\sigma$ to $g^{-1} U_m g \cap J$ does not contain the 
trivial representation, hence (C) is not satisfied. \end{proof}

\section{The base change transfer for mgs matrix coefficients}\label{s:kottwitz}
In~\cite{Kottwitz:base-change}, Kottwitz proves the base change fundamental lemma 
for unramified extensions at not just the unit elements of Hecke algebras but the 
characteristic functions of quite general compact open subgroups. In this section, we 
prove the analogous statement for one-dimensional characters of these compact open 
subgroups. This result should be useful in any attempt to describe how mgs 
representations behave under base change using the trace formula --- in particular, in a 
proof of the conjecture we make in Section \ref{s:BC} --- and may have other applications.

There is no direct way to base change an arbitrary compact open subgroup $J$ and a 
one-dimensional character $\chi$ of it from a field to an unramified extension. On the 
other hand, it is easy to base change the monomial datum $(G,m,H, \mathcal L)$ mentioned 
earlier, and this datum can be used to define a subgroup $J$ and a character $\chi$. The 
fact that the fundamental lemma holds in this setting can be motivated by the geometric 
Langlands philosophy: because the induced representations defined over two different 
fields from the datum $(G,m,H,\mathcal L)$ correspond to the same geometric object, 
i.e., the 
category of $(H,\mathcal L)$-equivariant sheaves on the loop group $G\pseries{t}$, they should 
have the same geometric Langlands parameter, so automorphic base change should take one 
to the other, which suggests that the fundamental lemma should hold.

However, in the proof of the fundamental lemma, the geometric description is not 
necessary. We have isolated the datum needed for a compact open subgroup of a group 
over a 
local field and a character to both have well-defined base changes to an arbitrary 
unramified extension. Our results hold in this setting, and work equally well over equal 
characteristic and mixed characteristic local fields. They may be of general interest.

\subsection{Character datum}\label{s:character-datum}
Let $F$ be a  non-archimedean local 
field, let $L$ be the completion of its maximal 
unramified extension, let $\sigma$ be the Frobenius of $F$ acting on $L$, and let $G$ 
be a connected reductive group over $F$. 
\begin{definition} \index{$\operatorname{Lang}_l$}  A \emph{character datum} on $G(F)$ 
consists of a bounded 
open $\sigma$-invariant subgroup $J_L$ of $G(L)$ and a central extension of topological 
groups with an action of $\sigma$ \[1 \to \mathbb C^\times \to \tilde{J}_L \to J_L \to 
1.\] We take the discrete topology and the trivial $\sigma$ action on $\mathbb 
C^\times$.  
\end{definition}

For $E\subset L$ a degree $l$ unramified extension of $F$, 
the subgroup $G(E)$ consists in the $\sigma^l$-invariant elements of of $G(L)$.
Given a character datum on $G(F)$, define the subgroup 
$J_{E}$ to be 
the $\sigma^l$-invariant subset of $J_L$ and define $\chi_{E}: J_{E} \to \mathbb 
C^\times$ to take a $\sigma^l$-invariant element $g$ to $ \sigma^l(\tilde{g}) 
\tilde{g}^{-1}$, where $\tilde{g}$ is a lift of $g$ from $J_L$ to $\tilde{J}_L$. In 
particular, 
in the $l=1$ case, the character $\chi_{F}$ sends $g \in J_{F}$ to $\sigma(\tilde{g}) 
\tilde{g}^{-1}$. Note that $J_E$ and $\chi_E$ are invariant under $\sigma$ and hence 
independent of the choice of isomorphism of $E$ with the $\sigma^l$-invariant subfield of 
$L$. 

\begin{definition}
Given an integer $l\ge 1$,
we say that the character datum satisfies the axiom $\operatorname{Lang}_l$ if the map $g 
\mapsto 
\sigma^l(g) g^{-1}$ from $J_L$ to itself is surjective.
\index{character datum = central extension $\widetilde{J}_L$ with $\sigma$-action}
\end{definition}

Let $(\bG,m, H, \mathcal L)$ be a monomial datum, that is a group $\bG$ over a finite 
field 
$\kappa$, a natural number $m$, a connected 
algebraic subgroup $H$ of $\bG\lWR \kappa[t]/t^m \rWR$, and a character sheaf $\mathcal 
L$ 
on $H$, we can define a character datum on $G( \kappa\pseries{t})$. Take $J_L$ to be the 
elements of $\bG( \overline{\kappa}\bseries{t})$ congruent mod $t^m$ to elements of $H 
(\overline{\kappa})$. Lemma~\ref{sheaf-central} defines a central extension of 
$H(\overline{\kappa})$ by $\Ql^\times$ with an action of $\sigma$ associated to $\mathcal 
L$. By applying an embedding $\iota$ of $\Ql$ into $\mathbb C$, and pulling back from 
$H(\overline{\kappa})$ to $J_L$, we obtain a central extension $1 \to \mathbb C^\times 
\to \tilde{J}_L \to J_L \to 
1.$

\begin{lemma}\label{character-data-from-geometry} When we obtain a character datum from 
$(\bG,m, H, \mathcal L)$ in this way, the following holds:
\begin{enumerate}
\item The axiom $\operatorname{Lang}_l$ is satisfied for every integer $l\ge 1$.

\item For $\kappa'$ a finite extension of $\kappa$, and $E = \kappa'\pseries{t}$, the 
character $\chi_E$ is equal to $\iota \circ \chi_{\kappa'}$, the trace function of 
$\mathcal L_{\kappa'}$, pulled-back 
from $H(\kappa')$ to $J_E=J_{\kappa'}$.

\end{enumerate}
\end{lemma}

\begin{proof}
(1)
By Lang's theorem \cite[\S4.4.17]{SpringerLinearAlgebraicGroups}, the map $g \to 
\sigma^l(g) g^{-1}$ from $H(\overline{\kappa})$ to itself is surjective for all $l$, and 
by iteratively lifting solutions to the equation $\sigma^l(g)g^{-1} = h$, the same map is 
surjective on $J_L$, so the axiom $\operatorname{Lang}_l$ is satisfied for all $l$.

(2) This follows by comparing the definition with Lemma~\ref{sheaf-central}.
\end{proof}

\begin{remark}
Character data have many of the nice geometric 
properties of monomial data, in particular those needed to prove the base change fundamental lemma below, without bringing any geometry into the definition.
A character datum does not necessarily come from an algebraic subgroup, even if one 
assumes the axioms $\operatorname{Lang}_1$ and $\operatorname{Lang}_l$. For instance, consider the group of all 
matrices in $\SL_2(\F_q\bseries{t})$ that are unipotent upper triangular $\pmod{t}$ and 
whose upper-right entry $\pmod{t}$ 
lies in an extension of $\F_{q^l}$ of degree a power of $p$, where $p$ is the characteristic of $\mathbb F_q$.  Then for any $a$ in 
$\F_{q^{lp^r}}$, the action of $\Frob_{q^{lp^r}}$ on solutions in $\overline{\mathbb F_q}$ of 
$x^q-x=a$ and $x^{q^l}-x=a$ is by translation, hence has order at most $p$, so both these equations have solutions in $\F_{q^{lp^{r+1}}}$. Because of this, the axioms $\operatorname{Lang}_1$ and $\operatorname{Lang}_l$ are satisfied for this group, by taking an upper unipotent solution mod $t$ and lifting to a $t$-adic solution.
\end{remark}

\subsection{Matching of orbital integrals}
Assume that $G_{der}$ is simply connected. 
Let $l\ge 1$, and $\tilde{J}_L\to J_{L}$ be character datum on $G(F)$ 
satisfying $\operatorname{Lang}_1$ and $\operatorname{Lang}_l$. We keep the other 
notation from the definition of character datum. 

Let $\kkE$ be an unramified extension of $\kkF$ of degree $l$, embedded as the fixed 
points of 
$\sigma^l$ in $\kkL$. Let $\theta$ be an automorphism of $\kkE$, with $\kkE^\theta=\kkF$.

Let $f$ on $G(\kkF)$ be equal to $\chi$ on $\kottK_F$ and $0$ elsewhere. Let $f_{\kkE}$ 
on 
$G(\kkE)$ equal $\chi_{\kkE}$ on $\kottK_{\kkE}$ and $0$ elsewhere.  We have the orbital 
integral
\[ O_\gamma(f) = \int_{G_\gamma(\kkF) \backslash G(\kkF)}  f( g^{-1} \gamma g) dg/dt \]
for $G_{\gamma}$ the centralizer of $\gamma$ in $G$, $dg$ the Haar measure on $g$ that 
gives $\kottK_F$ measure one, and $dt$ any fixed Haar measure on $G_\gamma(\kkF)$.

Similarly, we define \[O_{\delta\theta} (f_{\kkE} ) = \int _{I_{\delta\theta}(\kkF) 
\backslash G(\kkE) } f_{\kkE} ( g^{-1} \delta \theta(g) ) d g_{\kkE} / d\mu \] where 
$I_{\delta\theta}$ is the algebraic subgroup of $\operatorname{Res}^E_{F} G$ consisting of $h$ satisfying the equation $ h  = \delta \theta(h) \delta^{-1}  $, $g_{\kkE}$ is the Haar measure 
on 
$G$ such that $\kottK_{\kkE}$ has total mass one, and $d\mu$ is a Haar measure on 
$I_{\delta\theta}(F)$. 
We shall assume that these integrals 
converge absolutely.

Let $j$ be an integer such that $\theta=\sigma^j$ as automorphisms of $\kkE$, and let 
$a,b$ be integers with $bl-aj=1$.

Kottwitz's argument~\cite{Kottwitz:base-change} relies on the system in  
$(\gamma,\delta,c)$ of two equations 
\begin{equation}\label{keq}
\begin{cases}
c \gamma^a \sigma^l c^{-1} = \sigma^l,\\
c \gamma^b \sigma^j c^{-1} = \delta \sigma^j,
\end{cases}
\end{equation}
 valued in the semidirect product of $G(\kkL)$ with the free abelian group on $\sigma$.

\begin{lemma}\label{kottwitz-commutator-identity} Suppose that $\gamma \in \kottK_F$, 
$\delta \in \kottK_{\kkE}$, $c \in \kottK_{\kkL}$ satisfy the system~\eqref{keq}.
Then $\chi( \gamma) = \chi_{\kkE}( \delta)$.\end{lemma}

\begin{proof} Choose lifts  $\tilde{\gamma}$ and $\tilde{\delta}$ to $\tilde{G}_{\kkL}$. 
We will perform calculations in the semidirect product of $\tilde{J}_{\kkL}$ with the 
free abelian group on $\sigma$. We have
 \[  \chi(\gamma) = \sigma( \tilde{\gamma} ) \tilde{\gamma}^{-1} = [ 
 \sigma, \tilde{\gamma} ]. \]

Because $\gamma$ and $\sigma$ commute, $\tilde{\gamma}$ and $\sigma$ commute modulo 
center, so because $bl-aj=1$,
\[ [ \sigma, \tilde{\gamma}]= [ \tilde{\gamma}^{a} \sigma^l, 
\tilde{\gamma}^{b}\sigma^j].\]
Then because this commutator is central, it commutes with $c$, and thus
\[  [ \tilde{\gamma}^{a} \sigma^l, \tilde{\gamma}^{b}\sigma^j] = c  [ \tilde{\gamma}^{a} 
\sigma^l, \tilde{\gamma}^{b}\sigma^j] c^{-1} =  [ c \tilde{\gamma}^{a} \sigma^l c^{-1} , 
c \tilde{\gamma}^{b}\sigma^j c^{-1} ].\]
Finally, because this commutator is independent of the choice of lift to a central 
extension, 
\[ [ c \tilde{\gamma}^{a} \sigma^l c^{-1} , c \tilde{\gamma}^{b}\sigma^j c^{-1} ] = [ 
\sigma^l , \tilde{\delta} \sigma^j] = [ \sigma^l, \tilde{\delta} ] = \sigma^l 
(\tilde{\delta}) \tilde{\delta}^{-1} = \chi_{\kkE} (\delta).\qedhere
\]
\end{proof}

\begin{lemma}\label{kottwitz-twist-identity} Suppose that $\gamma \in G(\kkF) $, $\delta 
\in G(\kkE) $, $c \in G(\kkL)$ satisfy \eqref{keq}, and also satisfy $x^{-1} \gamma x \in 
\kottK_F$, $y^{-1} \delta \theta(y) \in \kottK_{\kkE}$, $y^{-1} c x \in \kottK_{\kkL}$.

Then $\chi_{\kkE} ( y^{-1} \delta\theta(y)) = \chi ( x^{-1} \gamma x)$. \end{lemma}

\begin{proof} This follows by applying Lemma~\ref{kottwitz-commutator-identity} to 
$x^{-1} \gamma x $, $y^{-1} \delta \theta(y)$, $y^{-1} c x $, which can be immediately 
seen to satisfy the system of equations~\eqref{keq}. \end{proof}

The remainder of the argument closely follows~\cite{Kottwitz:base-change}. We repeat the 
arguments in our setting for clarity, and because Kottwitz works in mixed characteristic 
only and we need equal characteristic.

\begin{lemma}\label{kottwitz-orbital-identity} Suppose that $\gamma \in G(\kkF) $, 
$\delta \in G(\kkE) $, $c \in G(\kkL)$ satisfy \eqref{keq}.
Conjugation by $c$ defines an isomorphism from $G_\gamma$ to $I_{\delta \theta}$, and 
we have 
\[ O_{\delta \theta} (f_{\kkE}) = O_{\gamma}(f),\]
where we use this isomorphism to match 
the Haar measures on $G_\gamma$ and $I_{\delta \theta}$. 
\end{lemma}

\begin{proof} 
We break the integral $\int _{I_{\delta\theta}(\kkF) \backslash G(\kkE) } f_{\kkE} ( 
g^{-1} \delta \theta(g) ) d g_{\kkE} / d\mu$ into a sum over double cosets $y \in 
I_{\delta\theta}(\kkF) \backslash G(\kkE)  / \kottK_{\kkE}$.  For each double coset, we 
claim that $f_{\kkE}$ is constant. This is because $f_{\kkE}$ vanishes outside 
$\kottK_{\kkE}$, a set which is invariant under twisted $\kottK_{\kkE}$-conjugation, and 
is a $\theta$-invariant character on $\kottK_{\kkE}$ which is also invariant under 
twisted $\kottK_{\kkE}$ conjugation. This follows from the fact that for $\tilde{k}$ a 
lift of $k$, and $g \in \kottK_{\kkE}$, $\tilde{g}^{-1}\tilde{k} \sigma^j(\tilde{g})$ is 
a lift of $g^{-1} k \theta(g)$, and we have \[\sigma^l ( 
\tilde{g}^{-1}\tilde{k}\sigma^j(\tilde{g}) ) = \sigma^l (\tilde{g})^{-1} 
\sigma^l(\tilde{k} )\sigma^j ( \sigma^l(\tilde{g}))  =   \tilde{g}^{-1} \chi_{\kkE}(g)^{-1}
\tilde{k} \chi_{\kkE} (k) \sigma^j ( \tilde{g} \chi_{\kkE}(g))  = \tilde{g}^{-1}
\tilde{k} \sigma^j (\tilde{g} )   \chi_{\kkE}(k).\]

Hence we can express the integral as a sum over $y \in I_{\delta\theta}(\kkF) \backslash 
G(\kkE)  / \kottK_{\kkE}$ such that $ y^{-1} \delta \theta(y) \in \kottK_{\kkE}$ of  
$\chi_{\kkE} (  y^{-1} \delta \theta(y) )$ times the measure of $I_{\delta\theta}(\kkF) 
\backslash I_{\delta\theta}(\kkF) y \kottK_{\kkE}$.  

Similarly, in the $l=1$ case, the integral is the sum over $x \in G_\gamma( \kkF) 
\backslash G(\kkF) / \kottK_F$ such that $x^{-1} \gamma x \in \kottK_F$ of $\chi ( x^{-1} 
\gamma x)$ times the measure of $G_{\gamma} (\kkF) \backslash G_{\gamma} (\kkF) x$.

Using the axiom $\operatorname{Lang}_l$, one can view $G(\kkE)  / \kottK_{\kkE}$ as the 
$\sigma^l$-fixed points in $G(\kkL)/ \kottK_{\kkL}$, and the set with $ y^{-1} \delta 
\theta(y) \in \kottK_{\kkE}$ as the $\delta\sigma^j$ -fixed points among those. 
Similarly, by $\operatorname{Lang}_1$, $ G(\kkF) / \kottK_F$ is the set of $\sigma$-fixed 
points in $G(\kkL)/ \kottK_{\kkL}$, and the subset of $x$ with  $x^{-1} \gamma x \in 
\kottK_F$  is the $\gamma$-fixed points. Now \eqref{keq} implies precisely that the map 
that sends $x$ to $y=cx$ gives a bijection between the points fixed by $\gamma$ and 
$\sigma$ and the points fixed by $\sigma^l$ and $\delta \sigma^j$. Furthermore, the 
points of $G(\kkL)$ fixed by conjugation by $\gamma$ and $\sigma$ are precisely 
$G_\gamma(\kkF) $, and the points fixed by $\delta \sigma^j$ and $\sigma^l$ are precisely 
$I_{\delta\theta}(\kkF)$, so this gives a bijection between the double cosets 
$I_{\delta\theta}(\kkF) \backslash G(\kkE)  / \kottK_{\kkE}$ and $G_\gamma( \kkF) 
\backslash G(\kkF) / \kottK_F$. 

By construction, for $x$ and $y$ paired by this bijection, we have $y =cx \in G(\kkL)/ 
\kottK_{\kkL}$, so $y^{-1} cx \in \kottK_{\kkL}$, thus by Lemma 
\ref{kottwitz-twist-identity}, $\chi_{\kkE} ( y^{-1} \delta\theta(y)) = \chi ( x^{-1} 
\gamma x)$.

It remains to check that, for $x$ and $y$ paired by this bijection, the measure 
of $I_{\delta\theta}(\kkF) \backslash I_{\delta\theta}(\kkF) y \kottK_{\kkE}$ equals the 
measure of $G_{\gamma} (\kkF) \backslash G_{\gamma} (\kkF) x \kottK_{\kkF}$. To do this, observe that we have fixed measures so that $\kottK_{\kkE}$ and $\kottK_{\kkF}$ have measure $1$, so that the measure of $I_{\delta\theta}(\kkF) \backslash I_{\delta\theta}(\kkF) y \kottK_{\kkE}$ is equal to the inverse of the measure of the stabilizer of $y \kottK_{\kkE}$ in $I_{\delta\theta}$, and $G_{\gamma} (\kkF) \backslash G_{\gamma} (\kkF) x \kottK_{\kkF}$ is equal to the inverse of the measure of the stabilizer of $x \kottK_{\kkF}$ in $G_{\gamma}(\kkF)$. We can equivalently view these stabilizers as the stabilizers of the points $x$ and $y$ in $G(\kkL)/ \kottK_{\kkL}$. Thus, because $y=cx$, these stabilizers are sent to each other by the isomorphism between $G_{\gamma}(\kkF)$ and $I_{\delta\theta}(\kkF)$ defined by conjugation by $c$, which by assumption is a measure-preserving isomorphism, so these measures are equal.

Hence the sums are equal and the orbital integrals are equal. \end{proof}

\subsection{Stable orbital integrals}

We say $\gamma, \gamma' \in G(F)$ are \emph{stably conjugate} if they are conjugate as elements of $G(\overline{F})$.

An \emph{inner twisting} between two algebraic groups is an isomorphism defined over the 
separable closure of the base field, which is Galois-invariant up to compositions with 
inner automorphisms, and where we take two inner twistings to be equivalent if they are 
equal up to composition with an inner automorphism 
\cite[p.68]{Laumon:book:cohomology-Drinfeld}. Given an inner twisting between two 
groups, there 
is a natural transfer, explained in \emph{loc.~cit}, of Haar measures from one group to 
Haar measures 
on the other via 
the Lie algebras.

In particular, if $\gamma$ and $\gamma'$ are stably conjugate, then there is a canonical 
inner twisting (i.e., canonical isomorphism over the separable closure of the base field, 
up to conjugacy) between their centralizers $G_\gamma$ and $G_{\gamma'}$. This enables us 
to define, after fixing a Haar measure on $G_{\gamma}$, the stable orbital integral
\[ SO_{\gamma}(f) = \sum_{\gamma'} e(G_{\gamma'}) O_{\gamma'}(f)\]
where $\gamma'$ traverses a system of conjugacy classes of elements stably conjugate to 
$\gamma$, and $e(G_{\gamma'})$ is the sign defined by Kottwitz.

Less obviously, for $\delta \in G_E$, let $\mathcal N \delta = \delta \theta(\delta) \theta^{2} (\delta) \dots \theta^{l-1}(\delta)$ be the norm of $\delta$.  If $\mathcal N \delta$ is stably conjugate to $\gamma$ then there is a canonical inner twisting $I_{\delta\theta} \to G_\gamma$. Indeed, 

\begin{lemma}\label{canonical-inner-twisting} Let $I= \operatorname{Res}^E_{F} G$ be the Weil restriction, and $I_E$ its base change to $E$.

Let $p$ be the projection $I_E \to G_E$ defined, using the fact that $R$-points of $\operatorname{Res}^E_{F} G$ are $R \otimes_F E$-points of $G$ for any ring $R$, by the map $G( R \otimes_F E) \to G(R)$ for an $E$-algebra $R$ induced by the multiplication map $R \otimes_F E \to R$.

For $d \in G(F^{s})$ such that $d^{-1} \mathcal N \delta d = \gamma $, the map $ g \mapsto    d^{-1}  p(g) d$ from $I_{\delta\theta, F^{s}} \to G_{\gamma, F^s}$ is an isomorphism.

This defines an isomorphism $I_{\delta\theta} \to G_{\gamma}$ which depends only on $\gamma, \delta$, up to conjugation by $G_\gamma$.  \end{lemma}

The proof is the same as \cite[Lem.5.8]{kottwitz-rational-conjugacy} and \cite[\S{I}, 
p.115]{Laumon:book:cohomology-Drinfeld}, though neither reference is in the exact 
context 
we work in.

\begin{proof} We use the fact that $I_E \cong  G^{ \Gal(E/F)}$. Under this isomorphism, the action of $\delta$ is by 
translation, and the map $p$ is projection onto one of the factors. (This follows from 
the fact that $E \otimes_F E = E^{\Gal(E/F)}$, with the action of $\theta$ by 
translation, and the multiplication map to $E$ is projection onto one of the factors). 

Thus the action of $\delta\theta$ on $I$ is by translation by $\theta \in \Gal(E/F)$ and then conjugation by $\delta$. So a fixed point of this action is determined by a tuple of $l$ elements of $G$, each of which when conjugated by $\delta$ becomes equal to the next one. Such a tuple is determined by its value in one copy of $G$, and an element of $G$ extends to a tuple if and only if it returns to itself when conjugated and translated $l$ times, which is equivalent to commuting with $\mathcal N \delta$. This shows that the projection $p$ defines an isomorphism $I_{\delta \theta} \cong G_{ \mathcal N \delta}$ over $L$, and then conjugating by $d$ gives a further isomorphism onto $G_{\gamma}$.

Because any $d'$ satisfying the same equation as $d$, for instance a Galois conjugate of $d$, is equal to $d$ times an element of $G_\gamma$, this map depends only on $\delta, \gamma$ up to conjugation by $G_\gamma$. \end{proof}

Using this isomorphism $I_{\delta\theta} \to G_{\gamma}$ to transfer a fixed Haar measure on $G_{\gamma}$, we can define the stable twisted orbital integral
 \[ SO_{\delta\theta}(f_{\kkE}) = \sum_{\delta'} e(I_{\delta'\theta}) 
 O_{\delta'\theta}(f_{\kkE})\] where $\delta'$ traverse a system of representatives for the 
 twisted conjugacy classes inside the stable twisted conjugacy class of $\delta$. (The transfer of Haar measure depends only on $\gamma, \delta$ because the Haar measure on $G_{\gamma}$ is invariant under conjugation.)
 
 We will now show an identity of stable twisted orbital integrals, continuing to follow~\cite{Kottwitz:base-change}.

\begin{lemma}\label{kottwitz-norm-facts} For each $\delta \in G(\kkE)$, there is at most 
one $\gamma \in G(\kkF)$ up to conjugacy satisfying \eqref{keq}, and always at least one 
if $O_{\delta\theta} (f_{\kkE} ) \neq 0$.  Similarly, for each $\gamma \in G(\kkF)$, 
there is at most one $\delta \in G(\kkE)$ up to $\theta$-conjugacy satisfying 
\eqref{keq}, and always at least one if $O_{\gamma} (f) \neq 0$. 

Finally, $\delta$ and $\gamma$ satisfying \eqref{keq} have $\mathcal N \delta = c \gamma c^{-1}$.

\end{lemma}

\begin{proof}
Fix $\gamma$. The identity $c \gamma^a \sigma^l c^{-1} = \sigma^l$ implies \[ c^{-1} \sigma^l (c) = \gamma^a,\] which uniquely determines $c$ up to left multiplication by something $\sigma^l$-invariant. In other words, this determines $c$ up to left-multiplication by an element of $G(\kkE)$.  For any choice of $c$, the identity $c \gamma^b \sigma^j c^{-1} = \delta \sigma^j$ determines $\delta$, and multiplying $c$ on the left by $G(F)$ is equivalent to conjugating $\delta \sigma^j$ by an element of $G(E)$ and thus is equivalent to $\theta$-conjugating $\delta$ by an element of $G(E)$. So for each $\gamma$, there is at most one $\delta$ up to $\theta$-conjugacy.

For there to exist at least one $\delta$ satisfying \eqref{keq}, it suffices that the equation $c^{-1} \sigma^l 
(c) = \gamma^a$ has a solution, for which by the axiom $\operatorname{Lang}_l$ it suffices that 
$\gamma$ is conjugate to an element of $\kottK_{\kkF}$, which is implied by the 
nonvanishing of $O_{\gamma}(f)$. Moreover, any $\delta$ satisfying \eqref{keq} lies in $G(E)$ because the two
equations together imply that $\delta$ commutes with $\sigma^l$.

For the opposite direction, we change the equations slightly. Because $\gamma$ and $\sigma$ commute with each other, and $\sigma^l$ and $\delta \sigma^j$ commute with each other, we can invert the two-by-two matrix to obtain the equivalent equations 
\[(\delta \sigma^j)^l \sigma^{- jl } = c \gamma c^{-1} \] \[ (\delta \sigma^j)^{-a} \sigma^{bl} = c \sigma c^{-1}  \] Fixing $\delta$, the second equation determines $c \sigma{c}^{-1}$, hence determines $c$ up to right multiplication by an element of $G(F)$. Examining the first equation, we see it determines $\gamma$ after fixing $\delta,c$, and right multiplying $c$ by an element of $G(F)$ has the effect of conjugating $\gamma$ by an element of $G(F)$. 

For $\gamma$ to exist, it suffices that there exists a $c$ with $ c \sigma(c)^{-1} = 
(\delta \sigma^j)^{-a}  \sigma^{bl-1}$, for which by the axiom $\operatorname{Lang}_1$ it suffices that  
$\delta $ is $\theta$-conjugate to an element of $\kottK_{E}$, which is implied 
by the nonvanishing of $O_{\delta \theta}(f_E)$. (Indeed, if $O_{\delta \theta}(f_E)\neq 0$ then there exists $g\in G(E)$ with $u= g^{-1} \delta \theta(g) \in J_E$, so  that $ (\delta \sigma^j)^{-a}  \sigma^{bl-1} = g^{-1}  (u \sigma^j )^a  g  \sigma^{bl-1} = g^{-1}  u \theta( u ) \dots \theta^{a-1} ( u) \sigma^{1-bl} (g)  $ and then applying $\operatorname{Lang}_1$ to $u\theta( u ) \dots \theta^{a-1} ( u)$ and using $\sigma^{1-bl}(g) =\sigma(g)$ we obtain $c$.) Furthermore this implies $\gamma \in 
G(\kkF)$, because it implies $\gamma $ commutes with $\sigma$.

Finally, observe that \[ c \gamma c^{-1} = (\delta \sigma^j)^l \sigma^{- jl }   = \delta 
\theta(\delta)  \theta^{2} (\delta) \dots  \theta^{l-1}(\delta)  = \mathcal N 
\delta.\qedhere
\]
\end{proof}

\begin{lemma}\label{kottwitz-measure-identity}  For any $\delta, \gamma, c$ satisfying \eqref{keq} with $\gamma$ semisimple, the map from $I_{\delta\theta}(F)$ to $G_\gamma(F)$ defined by conjugation by $c$ in fact arises from an isomorphism of group schemes over $F$, which is equivalent to the isomorphism of Lemma \ref{canonical-inner-twisting} in the case $d=c$.  

In particular, the transfer of the Haar measure from $G_\gamma(F)$ to $I_{\delta\theta}(F)$ under this map matches the transfer via the isomorphism of Lemma \ref{canonical-inner-twisting}. \end{lemma}

\begin{proof} The isomorphism $g \mapsto c^{-1} p(g) c$ of Lemma \ref{canonical-inner-twisting} is, by construction, defined over $L$.

To show it descends from $L$ to $F$, we use the fact that $G_{\gamma}$ is reductive and thus, by Lemma \ref{canonical-inner-twisting}, $I_{\delta\theta}$ is reductive, so there exists a scheme parameterizing isomorphisms between $G_\gamma$ and $I_{\delta\theta}$. To check that an $L$-point of this scheme is defined over $F$, it suffices to check that it is stable under the Frobenius $\sigma$. In other words we must check that it commutes with $\sigma$. It suffices to check it commutes with $\sigma^l$ and $\sigma^j$.

Observe that $\sigma^l$ commutes with $p$, and that \[\sigma^l ( c^{-1} g c) = \sigma^l(c)^{-1} \sigma^l (g) \sigma^l(c) = \gamma^{-a}  c^{-1} \sigma^l (g)  c \gamma^{a} = c^{-1} \sigma^l(g) c \] using \eqref{keq} and the fact that $\gamma$ commutes with $c^{-1} \sigma^l(g) c \in G_\gamma$. 

Next observe that \[ \sigma^j ( c^{-1}  p (g)  c) = \sigma^j c^{-1}   p(g) c\sigma^{-j} =  \gamma^{-b} c^{-1} \delta \sigma^j p(g) \sigma^{-j} \delta^{-1} c \gamma^b \] \[  =   \gamma^{-b} c^{-1} p( \delta \theta \sigma^j g \sigma^{-j} \theta \delta^{-1} )  c \gamma^b = \gamma^{-b} c^{-1} p( \sigma^j(c)) c^{-1} = c^{-1} p (\sigma^j(c))c^{-1} \] using \eqref{keq}, the fact that $\sigma^j(g) \in I_{\delta\theta}$ commutes with $\delta\theta$, and the fact that  $c^{-1} p (\sigma^j(c))c^{-1}  \in G_\gamma$ commutes with $\gamma$.
\end{proof}

\begin{theorem}\label{kottwitz-stable} For every semisimple $\gamma \in G(\kkF)$, the 
stable orbital integral 
$SO_{\gamma}(f)$ vanishes unless the stable conjugacy class of $\gamma$ is equal 
to the norm $\mathcal N \delta$ for some $\delta \in G(\kkE)$, in which case it is given 
by $SO_\gamma(f) = SO_{\delta \theta} ( f_{\kkE})$. 

Here we define both stable orbital integrals using the same Haar measure on $G_\gamma$. \end{theorem}

\begin{proof}
For each stable conjugate $\gamma'$ of $\gamma$, if the associated orbital integral is 
nonvanishing, then $\gamma'$ is conjugate to an element of $J_F$. Hence by Lemma 
\ref{kottwitz-norm-facts} there exists a $\delta'$ satisfying Kottwitz's equations, and 
the norm of $\delta'$ is stably conjugate to $\gamma$.

So we may assume that $\gamma$ is stably conjugate to the norm of $\delta$. Now for each 
$\gamma'$ for which the orbital integral is nonvanishing there exists a unique $\delta'$ 
up to $\theta$-conjugacy satisfying \eqref{keq} by 
Lemma~\ref{kottwitz-norm-facts}, and because the norm of $\delta'$ is stably conjugate to the norm of $\delta$, $\delta'$ is
stably $\theta$-conjugate to $\delta$. (To see, this, base change to $E$, so that $I = G^l$ and $\theta$ acts by permutation. Then if two elements of $G^l$ have conjugate norms, we can $\theta$-conjugate one to the other by adjusting each element of the $l$-tuple step-by-step.) By Lemma~\ref{kottwitz-orbital-identity} and Lemma 
\ref{kottwitz-measure-identity}, the orbital integrals and signs of $\gamma'$ and 
$\delta'$ agree. (The signs agree because they depend only on the isomorphism class, and we have an isomorphism between the two groups.) Because each $\gamma'$ corresponds to a unique $\delta'$ up to stable 
$\theta$-conjugation, and by Lemma~\ref{kottwitz-norm-facts} each $\delta'$ with 
nonvanishing orbital integral corresponds to a unique $\gamma'$ up to conjugation, 
the signed sums of orbital integrals over conjugacy classes and $\theta$-conjugacy 
classes agree, so the orbital integrals agree. \end{proof}

The analogue for $\kappa$-orbital integrals should also be possible, by an argument 
analogous to that in~\cite{Kottwitz:base-change}.

\section{Automorphic base change}\label{s:BC}
For every place $y$ of every constant field extension $F_n$ of $F$ of degree $n\ge 1$, we 
will always 
take 
the standard hyperspecial maximal compact 
$G(\ko_y)$ defined by 
the globally split structure of $G$. We say that a representation is unramified when it 
is $G(\ko_y)$-unramified.
Let $\pi$ be an automorphic representation of $G(\A_F)$, and $u\in |X|$ a place such that 
$\pi_u$ is 
mgs. 
In this context, we say that $\pi$ is
base-changeable if the following holds.
\begin{condition}\hypertarget{BC}{}
There exists a finite set of mgs data at $u$, such that for every 
constant field extension $F_n$ of $F$, there exists 
a base change representation $\Pi_n$ of $G(\mathbb A_{F_n})$, which at places 
lying over $u$ is mgs with one of the given mgs data, over the unramified 
places of $\pi$ is unramified and compatible under the Satake isomorphism,  and at all 
other places has depth bounded independently of $n$.
\end{condition}

We make the following conjecture.

\begin{conjecture} Every automorphic representation of $G(\mathbb A_F)$ that is mgs at a 
place $u$ satisfies \condbc{}. \end{conjecture}

This is a standard conjecture on the existence of cyclic base change, analogous to 
results that have been proved over number fields by 
Labesse~\cite[Thm.4.6.2]{Labesse:changement-base}, 
except for the compatibility condition at places lying over $u$, and for the boundedness 
of depth~\cite{Ganapathy-Varma}. Our main evidence that a 
cyclic base change compatible at $u$ should exist is Theorem~\ref{kottwitz-stable}, which 
gives the local transfer identities needed to compare twisted orbital integrals involving 
a test function which detects the mgs condition with usual orbital integrals 
for an analogous test function. Hence the conjecture is 
amenable to endoscopically stabilizing the trace formula and twisted trace formula and 
proving a comparison result between them. Special cases are accessible either by 
establishing stability of a finite set of mgs data at $u$, or by inserting 
stabilizing test functions at an additional place, we hope to do this in the 
sequel~\cite{Sawin-Templier:II-bc}.

\section{Geometric setup}\label{s:geometric-setup}

We now discuss geometric models for a family of automorphic forms with prescribed 
local behavior. Afterwards, we will use these geometric models to bound the traces of 
Hecke operators on this family.

Let $k$ be a field, let $X$ be a curve over $k$, and let $F=k(X)$. When we connect to 
analysis we will assume $k$ finite, but for the purely geometric parts we will not need 
that assumption. Let $G$ be a split semisimple algebraic group over $k$. Let $D$ be an 
effective 
divisor on $X$, which we will often view as a closed subscheme in $X$. We write $D = 
\sum_{x \in D} m_x [x] $ where $m_x$ is the multiplicity of $x$ in $D$.
\index{$D=\sum_{x} m_x [x]$, divisor, level}

\begin{defi}\index{$\BunGD$, moduli of $G$-bundles with $D$-level structure} Let $\BunGD$ 
be the moduli space of $G$-bundles on $X$ with 
a trivialization along $D$ (notation is in analogy with that of principal congruence 
subgroups). 
\end{defi}

We write $|X|$ for the set of closed points of $X$ and $|X-D|$ for the points outside the 
support of $D$. 
\index{$\lvert X \rvert$, set of closed points}
For $x \in |X|$, let $\kappa_x$ be the residue field at $x$. 
\index{$\kappa_x$, residue field}
We fix a local coordinate $t$ of 
$X$ at each closed point $x$, so that $\ko_x=\kappa_x \bseries{t}$ is the complete local 
ring 
at $x$, 
\index{$\ko_x=\kappa_x \bseries{t}$, complete local ring at $x$}
but our constructions will be independent of the choice of coordinate and so this 
is really just a notational convenience.  With this convention, 
$F_x=\kappa_x\pseries{t}$. The ad\`ele ring 
$\mathbb A_F$ is the restricted product $\prod_{x \in |X|}' F_x$.

\begin{notation}\index{$\KK$, compact subgroup} 
Let \[ \KK = \prod_{x \in |X - D|} G( \ko_x)  \times \prod_{x \in 
D} 
U_{m_x} (G(\ko_x)),\]    where  $U_ {m_x} ( G(\kappa_{x} \bseries{t}) )$  
is the subgroup of $G(\kappa_{x} \bseries{t})$ consisting of elements congruent to 
$1$ modulo $t^{m_x}$. 
Then Weil's parameterization lets us write $\BunGD ( k)$ as the adelic double quotient 
$G(F)\backslash G( \mathbb A_F) / \KK$, see Lemma~\ref{weil-parameterization-semisimple} 
below.
\end{notation}

\index{$\mathcal O_D$, ring of global sections} Let $\mathcal O_D$ be the 
ring of global sections of 
the structure sheaf on the scheme $D$, so that $G \lWR \mathcal O_D \rWR$ is the group of 
automorphisms of the trivial $G$-bundle on $D$.

\begin{lemma}We have isomorphisms  \[ \mathcal O_D \simeq \prod_{x \in D} \kappa_{x}[t]/ 
t^{m_x},\quad G \lWR \mathcal O_{D} \rWR \simeq \prod_{x \in D} G \lWR 
\kappa_{x}[t]/t^{m_x} 
\rWR . \]
\end{lemma}

\begin{proof} The first isomorphism follows from viewing $D$ as a disjoint union of 
schemes $m_x [x]$, and choosing local coordinates for each $x$, and the second 
isomorphism follows from the first. \end{proof} 

\begin{defi}\index{factorizable subgroup $H \subseteq G \lWR \mathcal O_D \rWR$} We say that 
an algebraic subgroup $H \subseteq G \lWR \mathcal O_D \rWR$ is \emph{factorizable} if it 
is equal to a product $\prod_{x \in D} \operatorname{Res}_{\kappa_{x}}^k H_x$ where $H_x$ 
is an algebraic subgroup of $G_{\kappa_{x}} \lWR \kappa_{x} [t]/t^{m_x} \rWR$ and 
$\operatorname{Res}_{\kappa_{x}}^k H_x$ is its Weil restriction from $\kappa_{x}$ to $k$, 
making it a subgroup of $G \lWR \kappa_{x} [t]/ t^{m_x} \rWR$ . \end{defi}

\begin{lemma} If $H \subseteq G \lWR \mathcal O_D \rWR$ is factorizable, then for any 
separable field extension $k'$ of $k$, the base change $H_{k'}$ of $H$ from $k$ to $k'$ 
remains factorizable as a subgroup of $G_{k'} \lWR \mathcal O_D \otimes k' \rWR $.  
\end{lemma}

\begin{proof} Write $H=\prod_{x \in D} \operatorname{Res}_{\kappa_{x}}^k H_x$.  Let us check that \begin{equation}\label{factorizable-stability}
 (\operatorname{Res}_{\kappa_x}^k H_x)_{k'} =
\operatorname{Res}^{k'}_{\kappa_x\otimes k'} (H_x)_{\kappa_x\otimes k'} =
\prod_{x'|x} \operatorname{Res}^{k'}_{\kappa_{x'}} (H_x)_{\kappa_{x'}}.
\end{equation} The first equality follows from taking the definition of the Weil restriction and base changing everything from $k$ to $k'$. The second follows from the fact that $k'/k$ is separable and thus $\kappa_x \otimes k' = \prod_{x'|x} \kappa_{x'}$ is a product of fields.

Taking the product of \eqref{factorizable-stability} over $x \in D$, the 
resulting subgroup $H_{k'}$ is factorizable.\end{proof}

Fix a smooth connected factorizable subgroup $H \subseteq G \lWR \mathcal O_D \rWR$ and a 
character sheaf $\mathcal L$ on $H$. By Lemma~\ref{all-factorizable}, $\mathcal L$ splits as a 
product $\boxtimes_{x \in D} \operatorname{Res}_{\kappa_{x}}^k \mathcal L_x$ for 
character sheaves $\mathcal L_x$ on $H_x$. A datum $(G,D,H,\mathcal L)$ will give rise to a 
set of
\emph{monomial local conditions} on an automorphic representation of $G(\mathbb A_F)$ as 
follows.

\begin{notation}\label{def:Jx} Let $J_x$ be the inverse image of $H_x(\kappa_{x} )$ in  
$G ( 
\kappa_{x}\bseries{t})$, which maps to $G ( \kappa_{x}[t]/ t^{m_x}) = G \lWR 
\kappa_{x}[t]/t^{m_x} \rWR (\kappa_{x})$ by the natural projection.  
\end{notation} 

\begin{defi} Let  $\chi_x$ be the character of $H_x(\kappa_{x} )$, and thus of $J_x$, 
induced by $\mathcal L_x$ and let $\chi$ be the character of $H(k)$ induced by $\mathcal 
L$.  \end{defi} \index{$\chi_x$, character of $J_x\subset G(\ko_x)$}

Under these definitions, we have a commutative diagram
\[ 
\begin{tikzcd}
\KK \arrow[r,hook] & \prod\limits_{x \in | X - D|} G( \ko_x)  \times \prod\limits_{x \in 
D} J_x \arrow[r,two heads] \arrow[d,hook] & H(\kappa) 
\arrow[d,hook]\\
& \prod\limits_{x \in |X|} G(\ko_x) \arrow[r,two heads] & G(\mathcal O_D)
\end{tikzcd} 
\] 
where the square is a Cartesian and the top row is a short exact sequence. 

For clarity and concreteness, we explicate the datum $(G, m_x, H_x, \mathcal L_x)$ that 
will 
appear in the proof of the main theorem of the paper. At each place, we will either take 
$H_x$ the trivial group and $\mathcal L$ the trivial sheaf, or we will take $(G, m_x,H_x, 
\mathcal L_x)$ to be geometrically supercuspidal. Examples of the second kind of data 
were provided in Lemma~\ref{sub:epipelagic}.

\begin{remark}\label{space-of-automorphic-forms} Assume that $k$ is a finite field. 
Consider the space of $L^2$-functions on 
$\BunGD(k)= G(F)\backslash G( \mathbb A_F) / \KK$  that are $\chi$-equivariant for the 
natural right action of \[H(k) \subseteq  G \lWR \mathcal O_D \rWR (k) = \prod_{x \in D} 
G( \kappa_{x}[t]/t^{m_x}) =  \prod_{x \in D }  G(\ko_x ) / U_ {m_x} ( G(\ko_x) ) = G 
(\ko_F) / \KK\] on $\BunGD(k)$, where $\ko_F = \prod_{x \in |X|}  \ko_x 
=\prod_{x \in |X|} \kappa_x \bseries{t} $. 
We view this as a space of automorphic forms.

We break this space into eigenspaces under Hecke operators, with irreducible 
subconstituents given by 
automorphic representations of $G(\mathbb A_F)$. All automorphic
representations that appear as subquotients are unramified away from $D$, and at every 
point 
$x 
\in D$ admit a nontrivial map from the compact induction 
$\cind_{J_x}^{G(\kappa_{x}\pseries{t})} \chi_x$.

The dimension of the space associated to an automorphic representation $\pi$ of 
$G(\mathbb A_F)$ is equal to its global multiplicity in $L^2 ( G(F) \backslash G(\mathbb 
A_F))$ times the product over $x$ of the dimension of the $(J_x,\chi_x)$-eigenspace in 
$\pi_x$. \end{remark}

\begin{remark} We compare our datum $(G, D, H, \mathcal L)$ defining a space of 
automorphic forms to the ``geometric automorphic datum" defined by Yun in 
\cite[\S2.6.2]{YunRigidity}. Both are geometric versions of the notion of an automorphic 
representation defined by local conditions, but Yun's is somewhat more general, as we 
have made various restrictions for technical and notational simplicity. 

    We work with semisimple groups, while Yun fixes a central character. 
The group ``$\mathbf K_S$" in~\cite{YunRigidity} carries the same information as our $H$. 
The group ``$\mathbf K_S$" is a pro-algebraic subgroup of $\prod_{x\in S} G \lWR \kappa_x 
\bseries{t}\rWR$, 
whereas $H$ is an 
algebraic subgroup of $G\lWR \cO_D \rWR$. This is only a technical difference: by 
truncating, we 
	avoid working with pro-algebraic groups. More significantly, Yun allows the local 
	subgroups to be contained in any parahoric subgroup, while we allow only the standard 
	hyperspecial subgroup, and he allows them to be arbitrary subgroups of $G \lWR 
	\kappa_x \bseries{t}\rWR$ and not just Weil restrictions from 
	$G_{\kappa_x\bseries{t}}$, which means that his definition is not stable under base 
	field extension (this can be repaired by either specializing to subgroups that are 
	Weil restrictions or generalizing to subgroups of the product of local groups at all 
	places, rather than products of local subgroups). The notation ``$\mathcal K_S$" in 
	\cite{YunRigidity} is our 
	$(\mathcal L_x)_{x\in D}$.
\end{remark}

\begin{remark} Most of our methods apply over an arbitrary base field $k$, and it 
would not 
be surprising if they could be generalized to the derived category of $D$-modules. For 
instance, Theorem~\ref{mainduality-semisimple} could possibly be established for 
$D$-modules, in which case Lemma~\ref{purity-semisimple} would be the statement that a 
$D$-module pushforward is supported in a single degree.  Similarly, the Ramanujan bound 
in a particular case established in~\cite{HNY:Kloosterman} has 
been used in~\cite{Lam-Templier} to prove that certain character $D$-modules were 
concentrated in a single degree.

 If this were done, it might have relevance to the characteristic zero geometric Langlands program. 
However, it is easy to see that the geometric supercuspidality
condition cannot 
be satisfied by any tamely ramified character sheaf, and thus cannot be satisfied at all 
for sheaves
or $D$-modules with regular singularities in characteristic zero. Hence using this 
technique requires dealing with irregular singularities.

 \end{remark}

 \begin{remark}\label{r:eigensheaf} 
We note that this geometric setup can also be used to motivate \condbc{}. Let $\pi$ be an 
automorphic representation generated by some automorphic function 
 on $G(F)\backslash G( \mathbb A_F) / \KK$ which is $\chi$-equivariant for the right 
 action of $H(k)$. Suppose that it is the trace function of a Hecke eigensheaf on 
 $\BunGD$ that 
 is $\mathcal L$-equivariant for the right action of $H$. Then $\pi$ satisfies 
 \condbc{}, except possibly for finitely many extensions. Indeed, over each finite 
 field extension $k'$ of $k$, we can take the trace function of the Hecke eigensheaf over 
 $k'$, which is a Hecke eigenfunction (assuming it is non-zero), and generates one or 
 more automorphic representations with the same Satake parameters at unramified places. 
 Because the Hecke eigenvalues come from the same geometric Langlands parameter as the 
 Hecke eigensheaf, they have matching Satake parameters with $\pi$. Because the 
 automorphic function lies on $\BunGD(k')$, the generated representations have bounded 
 depth, and because it is $(H(k'), \chi_{k'})$-equivariant, the generated 
 representations 
 are compatible with the same mgs data at every mgs place. The only potential problem is 
 if the trace function is identically zero, which can only happen for finitely many field 
 extensions. \end{remark}

\subsection{Moduli Spaces}
As in~\S\ref{sub:Satake}, let $\Lambda^+$ be a Weyl cone in the cocharacter lattice of 
$G$ (which is naturally in 
bijection with a Weyl cone in the character lattice of $\widehat{G})$.

Let $x$ be a point in $X$ and let $U \subseteq X$ be a neighborhood of $x$. Let 
$\alpha_1$ and 
$\alpha_2$ be two $G$-bundles defined over $U$, and let $f: \alpha_1 \to \alpha_2$ 
be an isomorphism over $U -\{x\}$. If we choose trivializations of $\alpha_1$ and 
$\alpha_2$ in a formal neighborhood of $x$, we can represent the restriction of $f$ to the punctured formal neighborhood of $x$ as an element of 
$G ( \kappa_x \pseries{t} ) $. Changing the trivializations corresponds to the left and 
right action of $G(\kappa_x\bseries{t})$ on this element, so the isomorphism $f$ 
defines a double coset in $G(\kappa_x\bseries{t}) \backslash G ( \kappa_x \pseries{t} ) / 
G(\kappa_x\bseries{t}) $. These double cosets are naturally in bijection, under the 
Cartan decomposition, with $\Lambda^+$. We can view this decomposition as coming from the 
affine Grassmannian $G\pseries{t}/G\bseries{t}$, because each double coset in 
$G(\kappa_x\bseries{t}) \backslash G ( \kappa_x \pseries{t} ) / G(\kappa_x\bseries{t}) $ 
is a $G(\kappa_x\bseries{t})$-orbit in the $\kappa_x$-points $G ( \kappa_x \pseries{t} ) 
/ G(\kappa_x\bseries{t}) $ of the affine Grassmannian.  These orbits are the Schubert cells 
of the affine Grassmannian, which again are in bijection with $\Lambda^+$. This geometric 
description makes clear that, in any algebraic family of $G$-bundles $\alpha_1, \alpha_2$ and 
maps 
$f$ between them, the set of points where the double coset associated to $f$ is in a 
particular cell of the affine 
Grassmannian is locally closed and, moreover, the set of points where $f$ is in the 
closure of a particular Schubert cell of the affine Grassmannian is closed. Using these 
closed cells, we will define a Hecke correspondence.
\index{affine Grassmannian, Schubert cells}

Let $\weights$ be a function from $|X|$ to $\Lambda^+$, that sends all but finitely many 
points to the trivial cocharacter and sends all the points of $D$ to the trivial 
cocharacter. Define the support of $\weights$ to be the set of points that $\weights$ 
sends to a nontrivial cocharacter (i.e., the usual definition of the support of a 
function, if we view the trivial cocharacter as the zero element of $\Lambda^+$). \index{$W:\lvert X\rvert \to \Lambda^+$, finitely supported}

\begin{defi}\label{d:Hecke-G(D)} Let $\Hecke_{G(D),\weights}$ be the moduli space of 
pairs $\alpha_1,\alpha_2$ of $G$-bundles with an isomorphism $f:\alpha_1 \to \alpha_2$ 
away 
from the support of $W$, and with a 
trivialization 
\[
t_1:\alpha_1|_D \isom G \times \Spec(\cO_D)
\]
 of the first bundle along 
$D$, such that near each point $x$ of the 
support of $W$, when $f$ is viewed as a point in the  
$G(\kappa_x\pseries{t})$ as above, it projects to a point in the affine Grassmannian that lies in 
the closed 
cell corresponding to $\weights_x$. 
\index{$\Hecke_{G(D),W}$, Hecke moduli space}
\end{defi}

\begin{defi}  We define a map $\Delta^{\weights}: \Hecke_{G(D),\weights} \times H  \to 
\BunGD \times \BunGD$ that sends $(\alpha_1, \alpha_2, f, t_1)$ to $((\alpha_1,t_1), (\alpha_2, h \circ t_1 \circ f|^{-1}_D))$. In other words, 
the left projection is taking the first $G$-bundle  with 
trivialization over $D$, and the right projection is taking the second $G$-bundle 
$\alpha_2$, using 
$f$ to carry over the trivialization $t_1$, and then twisting the trivialization by the 
element $h\in H$.%
\index{$\Delta^W$, Hecke correspondence}
 \end{defi}
 
 We will work with the intersection cohomology complex $IC_{\Hecke_{G(D),\weights}} $ on $\Hecke_{G(D),\weights}$, which by definition is the unique irreducible perverse sheaf isomorphic to $\Ql [ \dim \Hecke_{G(D),\weights}]$ on the open set where $\Hecke_{G(D),\weights}$ is smooth.

\begin{remark} The trace function of $\Delta^{\weights}_! (IC_{\Hecke_{G(D),\weights}} 
\boxtimes \mathcal L)$, which is a function on $\BunGD(k) \times \BunGD(k)$, is the 
kernel for the composition of the Hecke operator associated to $W$ by the Satake 
isomorphism with the averaging operator of the $(H(k),\chi)$-action 
(Lemma~\ref{arithmetic-to-geometry}). Thus it acts as a 
Hecke operator on the space of automorphic forms described in Remark 
\ref{space-of-automorphic-forms}.  
 \end{remark} 

The aim of Section~\ref{s:cleanness} will be to prove the following cleanness 
property of $\Delta^{\weights}$.
\begin{theorem}[=Theorem \ref{mainduality-semisimple}] Assume that $(G,m_u, H_u, \mathcal 
L_u)$ is 
geometrically supercuspidal for some $u \in D$ and $\charr(k)>2$. Then the natural map \[ 
\Delta^{\weights}_! \left( IC_{\Hecke_{G(D),\weights} }  \boxtimes \mathcal L \right) \to 
\Delta^{\weights}_* \left( IC_{\Hecke_{G(D),\weights} }  \boxtimes \mathcal L \right) \] 
is an isomorphism. \end{theorem}

Using this, in Section~\ref{s:properties}, we will prove that $\Delta^{\weights}_! 
(IC_{\Hecke_{G(D),\weights}} \boxtimes \mathcal L)$ is a pure perverse sheaf, which we 
will use in Section~\ref{s:trace-function} to derive numerical consequences.

\begin{remark}\label{schieder-motivation} Let us explain some of the motivation for 
Theorem~\ref{mainduality-semisimple}. As we mentioned before, the trace function of 
$R\Delta^{\weights}_! (IC_{\Hecke_{G(D),\weights}} \boxtimes \mathcal L)$ is a Hecke 
kernel on a particular space of automorphic forms. In particular, in the case when 
$\weights$ is trivial, it is simply the idempotent projector onto this space of 
automorphic forms.

In the case where $G = SL_2$, $D$ is empty, and $\weights$ is trivial, the trace function 
of 
$R\Delta^{\weights}_* (IC_{\Hecke_{G(D),\weights}} \boxtimes \mathcal L)$ was calculated 
by Schieder \cite[Prop.8.15]{Schieder14}. Viewing the trace function as a kernel, 
the induced operator on the space of automorphic forms was calculated by Drinfeld and 
Wang, who found that it acts as the identity on cusp forms \cite[Prop.3.2.2(i),  
Theorem 1.3.4, and Equation 3.2]{DW15}. A similar calculation was done by Wang for 
general groups in \cite[Thm.C.7.2 and Thm.1.4.3]{Wang}.
 If this fact is true for the families of automorphic forms with more general 
 local conditions, then the trace function of $R\Delta^{\weights}_* 
 (IC_{\Hecke_{G(D),\weights}} \boxtimes \mathcal L)$ should equal the trace function of 
 $R\Delta^{\weights}_! (IC_{\Hecke_{G(D),\weights}} \boxtimes \mathcal L)$ as soon as one 
 of the local prescribed conditions ensures that the automorphic forms in the family are 
 cuspidal by 
 mandating that one of the local factors is supercuspidal. If we believe this, then we 
 might conjecture that they should agree as sheaves and not just trace functions as long 
 as the local condition also forces cuspidality over finite field extensions. \end{remark}

\section{Cleanness of the Hecke complex}\label{s:cleanness}

As before, let $X$ be a smooth projective curve over a finite field $k$, $G$ a split 
semisimple algebraic group over $k$, $D$ an effective divisor on $X$, $H$ a smooth 
factorizable 
subgroup of $G \lWR \mathcal O_D \rWR$, and $\mathcal L$ a character sheaf on $H$.

We begin, in \S\ref{defining-compactification}, by constructing a compactification of $\Hecke_{G(D),\weights} \times H$ over $\BunGD \times \BunGD$.  The advantage of having a compactification is that it reduces the cleanness property of $\Delta_W$ that we are trying to prove (in Theorem \ref{mainduality-semisimple}) to the corresponding cleanness statement for the open immersion $j$ of  $\Hecke_{G(D),\weights} \times H$ into its compactification (Theorem \ref{mainextension-semisimple}). We can prove this cleanness statement by working locally with individual points of the compactification. This compactification will also help in proving (in Lemma \ref{schematic-affine}) that $\Delta_W$ is schematic and affine. These facts (Lemma \ref{schematic-affine} and Theorem \ref{mainduality-semisimple}), will be the main results from this section that are relevant to subsequent sections, as they together imply very strong properties (Lemma \ref{purity-semisimple}) of $\Delta^{\weights}_! \left( IC_{\Hecke_{G(D),\weights} }  \boxtimes \mathcal L \right) $.

We define this compactification by giving explicit coordinates for a map of $G$-bundles. 
We do this by using a faithful representation $V$ of $G$. We then allow these 
coordinates to go to infinity. What this means in practice is described in 
\S\ref{lemmas-on-semisimple}, which is devoted to describing the projective closure of the 
affine variety $G \subset \operatorname{End} V$.  We give (in Lemma 
\ref{closure-characterization}) a classification of points on this affine closure, and then 
describe how $G$ acts on them. This will eventually allow us to classify the points of the 
compactification.

For an open immersion $j$, the cleanness can be interpreted as a vanishing of stalks. We begin the proof, in \S\ref{s:vanishing-cusp}, by proving the vanishing of stalks for a special set of points in the compactification, those ``near the cusps", which arise from highly unstable $G$-bundles (Lemma \ref{htvan-semisimple}). These $G$-bundles have extra symmetries, and we use these symmetries to obtain the vanishing. Roughly, we show that these symmetries act trivially on the stalk, and, if the stalk is nontrivial, they act nontrivially on it.

We continue the proof in \S\ref{s:Hecke-correspondences} by showing how to relate the stalks at different points in the compactification. We show that if a stalk vanishes at one point, it vanishes at certain related points. This will enable, in \S\ref{s:7-conclusion}, an inductive proof that the stalk vanishes everywhere in the compactification outside the original $\Hecke_{G(D),\weights} \times H$. We do this by defining a Hecke correspondence between the compactification and itself. Just as, in the classical setting of modular curves, the graph of a Hecke correspondence is itself a modular curve, and therefore admits Hecke correspondences at coprime places, we have a notion of Hecke correspondence for $\Hecke_{G(D),\weights} \times H$, and even its compactification. The most technically difficult part is checking that these Hecke correspondences are smooth (Lemma \ref{smooth-semisimple}). This then enables us to relate the stalks at two corresponding points by smooth base change (Lemma \ref{travelling-semisimple}).

We conclude in  \S\ref{s:7-conclusion} with an induction on the ``height" of a point, which 
we think of as a generalization of the $y$-coordinate of a point on the upper-half plane 
(Definition \ref{defi-height}). The larger this height is, the more a point is near the cusp. 
(The points ``near the cusp" are exactly the points with height over some threshold.) The 
key lemma for this induction is that every point is related by a Hecke correspondence to 
some point of greater height (Lemma \ref{heightgrowth-semisimple}).

 \subsection{A compactification of $\Hecke_{G(D),\weights} \times H$}\label{defining-compactification}
  Let $V$ be a faithful 
 representation of $G$, which we also view as a functor $\alpha \mapsto V(\alpha)$ from 
 $G$-bundles to vector bundles. 
\index{$V$, faithful representation of $G$}
Throughout this section, we will be working geometrically 
 and so we can and will assume that $k$ is algebraically closed.
We ssume that $V$ lifts to the Witt vectors of $k$ and the pairing of any root of $G$ with any weight of $V$ is less than the 
characteristic $p$ of 
 $k$ (this technical condition is used in Lemma \ref{parabolic-construction}, and the existence of a suitable $V$ is checked in Lemma \ref{V-existence}). We fix a maximal torus and a Borel $T\subset B$ inside $G$. 
As in the previous section, let $W$ be a function from $|X|$ to $\Lambda^+$ with 
finite support disjoint from the effective divisor $D=\sum_{x\in |X|} m_x [x]$.
 
 \begin{defi}  For each point $x \in |X|$, consider the composition $\G_m \stackrel{W_x}{\longrightarrow} G \to \GL(V)$ of the 
 representation $V$ with the cocharacter $\weights_x\in 
 \Lambda^+$. This is a representation of $\mathbb G_m$, hence is a sum of one-dimensional representations, which we can express as $\lambda \mapsto \lambda^{e_1},\dots, \lambda^{e_{\dim V}}$ for a tuple of integer weights $e_1,\dots, e_{\dim V}$. Let $\set{\weights}_x = - \min(e_1,\dots, e_{\dim V})$.
 
 Let $\{\weights\}:|X|\to \Z$ be the divisor, whose 
 multiplicity at each point $x\in |X|$ is $\set{W}_x$.\index{$\set{W}_x$, lowest weight attached to the cocharacter $W_x$}
\end{defi}

The support of $\set{W}$ is less than the support of $W$. In particular we have that 
$\{\weights\}$ is disjoint from $D$.

\begin{example}
(i) If $G = \Sp_{2n}$, $V$ is the standard representation, and $\weights_x$ is the cocharacter 
with eigenvalues $\lambda^{w_1}, 
 \dots, \lambda^{w_n}, \lambda^{-w_n}, \dots, \lambda^{-w_1}$ where $w_1,\dots,w_n$ are integers with $w_1 \geq \dots \geq w_n \geq 0$ then $\{\weights\}_x = 
 w_1$. 

(ii) If $G=\SL_n$, $V$ is the adjoint representation, and $\weights_x$ is the cocharacter whose 
eigenvalues on the standard representation are
 $\lambda^{w_1},\dots,\lambda^{w_n}$ for $w_1,\dots,w_n$ integers with $w_1 \geq \dots \geq w_n$ and $\sum_{i=1}^n w_i =0$, then its eigenvalues on the adjoint representation have the form $\lambda^{w_i-w_j}$, so $\{\weights\}_x=w_1-w_n$.
\end{example}

Before compactifying $\Hecke_{G(D),\weights} \times H$, we compactify $G$ by considering 
the projective completion of $\End(V)$:
 \begin{notation} Let $\overline{G}$ be the closure of $G \subseteq \End V \subseteq 
 \mathbb 
 P (\End V\oplus k )$, where we embed $\End V$ into the projective space $\mathbb P(\End 
 V\oplus k)$ by $x \mapsto [x:1]$. (The map $G \to 
 \GL(V) \to \End V$ is an immersion because $V$ is a faithful representation). 
\index{$\overline{G}$, compactification of $G$ inside $\mathbb{P}(\End V \oplus k )$}
\end{notation}

Given two pairs $(\alpha_1,t_1),(\alpha_2,t_2)$ of a $G$-bundle and a trivialization over 
$D$ and a projective section $\varphi \in \mathbb P (\Hom (V(\alpha_1),V(\alpha_2) 
\otimes \mathcal 
O_X(\{\weights\}) )\oplus k )$, because $\Hom (V(\alpha_1),V(\alpha_2) \otimes \mathcal 
O_X(\{\weights\}) )\oplus k$ is the vector space of global sections of 
\[
\mHom 
(V(\alpha_1),V(\alpha_2) \otimes \mathcal O_X(\{\weights\}) )\oplus \mathcal O_X,
\]
 we can 
view $\varphi$ as a nonzero global section of $ \mHom (V(\alpha_1),V(\alpha_2) \otimes 
\mathcal O_X(\{\weights\}) )\oplus \mathcal O_X$, well-defined up to scaling. Locally 
over any 
open set, closed set, or punctured formal neighborhood, that does not intersect the 
support of $\weights$ and where we have a trivialization of $\alpha_1$ and 
$\alpha_2$, 
we obtain a section of $(\End V \oplus k) \otimes \mathcal O_X$ up to scaling.

\begin{defi}
\index{$\overline{\Hecke}_{G(D),H,\weights,V}$, compactification of the Hecke stack}   
Let $\overline{ \Hecke}_{G(D),H,\weights,V}$ be the moduli space of five-tuples consisting of $\alpha_1,t_1,\alpha_2,t_2,\varphi$ where  
$(\alpha_1,t_1),(\alpha_2,t_2)$ are two pairs of a $G$-bundle and a trivialization over 
$D$ and
\[
\varphi \in 
\mathbb P (\Hom (V(\alpha_1),V(\alpha_2) \otimes \mathcal 
O_X(\{\weights\}) )\oplus k)
\]
 such that
\begin{enumerate}
 
 \item Over any point $x$ in the complement of the support of $\weights$, for any 
 trivialization of $\alpha_1$ and $\alpha_2$ over $x$, the induced point of 
 $(\End V \oplus k) \otimes \kappa_x$ lies in the affine cone of $\overline{G}$. (Note that $\overline{G}$ 
 is invariant under the left and right action of $G$, so this does not depend on the 
 choice of trivialization.)

 \item  In a punctured formal neighborhood of any point $x$ in the support of $\weights$, 
 for any trivialization of $\alpha_1$ and $\alpha_2$ over that punctured formal 
 neighborhood, the induced section of $\End V\oplus \mathcal O_X$, when viewed as a point in the 
 formal loop space $(\End V \oplus k ) \pseries{t}$, is in the closure of the set of 
 pairs $(\lambda V(g),\lambda)$ where $\lambda \in \mathbb G_m$ and $g \in G\pseries{t}$ 
 is in the Schubert cell associated to $W_x$.

 \item Over $D$, using the trivializations $t_1$ and $t_2$, the induced element of  
 $(\End V\oplus k) \lWR \mathcal O_D \rWR$ lies in the closure of the set of pairs 
 $(\lambda h, \lambda)$ where $\lambda \in \mathbb G_m$ and 
$h \in H \subseteq G \lWR \mathcal O_D \rWR \subseteq \End 
 V \lWR \mathcal O_D \rWR$. 
Equivalently, using an arbitrary trivialization over $D$,  
 $V(t_2) \circ \varphi|_D \circ V(t_1)^{-1}$ lies in this closure, where 
 $V(t_i):V(\alpha_i)|_D \isom V\langle \cO_D \rangle$ are the associated trivialization.
  
 \end{enumerate}
\end{defi}

For interpreting the last two conditions, remember that a global section of $\mathcal O_X$ is always 
constant over $X$, so forcing the last coordinate to be locally constant over $X$ is not 
any additional restriction.
	Recall from Definition~\ref{d:Hecke-G(D)} that 
$\Hecke_{G(D),\weights}$ is the moduli space of four-tuples $(\alpha_1,\alpha_2,f,t_1)$ consisting of a
pair of $G$-bundles $\alpha_1,\alpha_2$, an isomorphism 
$f: \alpha_1 \to 
	\alpha_2$ away from the support of $\weights$, that near each point in the support of 
	$\weights$ is in the closure of the cell of the affine Grassmannian associated to the 
	corresponding representation, and a trivialization $t_1$ of $\alpha_1$.

To understand $\overline{ \Hecke}_{G(D),H,\weights,V}$ geometrically, it helps to first describe the analogous moduli space without the conditions (1), (2), (3). We can describe this as a projective bundle.
	
\begin{lemma}\label{projectivization-coherent-sheaf} The moduli space of five-tuples 
$((\alpha_1,t_1), (\alpha_2, t_2), \varphi)$ where $(\alpha_1,t_1), (\alpha_2,t_2) \in \BunGD$ and $\varphi \in  
\mathbb P (\Hom (V(\alpha_1),V(\alpha_2) \otimes \mathcal 
O_X(\{\weights\}) )\oplus k)$ is a projective bundle over $\BunGD \times \BunGD$, in the sense of $\operatorname{Proj}$ of the symmetric algebra of a coherent sheaf on $\BunGD \times \BunGD$. \end{lemma}

\begin{proof} The projectivization of a vector space is  $\operatorname{Proj}$ of the symmetric algebra of the dual vector space. So it suffices to check that there is a coherent sheaf on $\BunGD \times \BunGD$ whose fiber at each point is the dual of $\Hom (V(\alpha_1),V(\alpha_2) \otimes \mathcal 
O_X(\{\weights\}) )\oplus k$. By Serre duality, this dual is $H^1 (X, V(\alpha_2)^\vee \otimes \mathcal 
O_X(-\{\weights\})\otimes  V(\alpha_1) \otimes K_X)\oplus k$. Because $H^1$ is the top cohomology group, its value at each point is the fiber of the coherent sheaf $R^ 1 \pi_* (  V(\alpha_2)^\vee \otimes \mathcal 
O_X(-\{\weights\})\otimes  V(\alpha_1) \otimes K_X) $, for $\pi$ the projection $X \times \BunGD \times \BunGD \to \BunGD \times \BunGD$.

Finally, the sum of $H^1$ with $k$ is the fiber of the sum of this coherent sheaf with $\mathcal O_{ \BunGD \times \BunGD}$. 
\end{proof}

\begin{lemma}\label{j-semisimple} 
\index{$j: \Hecke_{G(D),\weights} \times H \to \overline{ \Hecke}_{G(D),H,\weights,V} $}
There is a well-defined map $j: \Hecke_{G(D),\weights} \times H \to \overline{ 
\Hecke}_{G(D),H,\weights,V} $ 
	that sends $(\alpha_1,t_1,\alpha_2,f,h)$
	to 
$
((\alpha_1, t_1), (\alpha_2,  h \circ t_1 \circ f|_D^{-1}), \varphi)
$
 where 
\[
\varphi \in \Hom (V (\alpha_1),V(\alpha_2) \otimes \mathcal O_X(\{\weights\}))  \subseteq 
\mathbb P (\Hom (V(\alpha_1),V(\alpha_2) \otimes \mathcal O_X(\{\weights\}) )\oplus k)
\]
 is $V(f):V(\alpha_1)\to V(\alpha_2)$ tensored with the natural map $\mathcal O_X \to 
 \mathcal 
 O_X(\{\weights\})$. 
 \end{lemma}
 
 \begin{proof} First we show that $\varphi$ is in fact a homomorphism from $V(\alpha_1)$ 
 to $V(\alpha_2) \otimes \mathcal O_X(\{\weights\})$ defined everywhere on $X$. This is 
 clear away from the support of $\weights$, where $f$ is an isomorphism. In a formal 
 neighborhood of each point $x$ in the support of $\weights$, for $f$ whose associated 
 point of $G\pseries{t}$ is in the Schubert cell corresponding to $\weights_x$, the order 
 of the pole of $V(f)$ is at most 
 $\{\weights\}_x$, by definition of $\{\weights\}$. For $f$ whose associated point 
 of $G\pseries{t}$ is in the closure of the Schubert cell, because the pole order is a 
 lower semicontinuous function, the order of the pole is also at most 
 $\{\weights\}_x$, and so it becomes a homomorphism after we tensor with $\mathcal 
 O(\{\weights\})$.
 
 Next we show that $\varphi$ satisfies the local conditions (1), (2), and (3) of the 
 definition of $ \Hom (V (\alpha_1),V(\alpha_2) \otimes \mathcal O_X(\{\weights\})) $. It 
 satisfies condition (1) because $f$ is an isomorphism away from the support of 
 $\weights$, condition (2) because $f$ is in the closure of the correct cell of the 
 affine Grassmannian near points in the support of $\weights$, and condition (3) because 
 over $D$, we have $t_2 \circ f|_D \circ t_1^{-1} = h \in H$. \end{proof}
 
Let $\overline{\Delta}^\weights: \overline{ \Hecke}_{G(D),H,\weights,V} \to 
  \BunGD \times \BunGD$ send $(\alpha_1,t_1,\alpha_2,t_2,\varphi)$ to 
  $((\alpha_1,t_1),(\alpha_2,t_2))$.

\begin{lemma}\label{compactification-semisimple} The map $\overline{\Delta}^\weights$ is 
projective and $\overline{\Delta}^\weights \circ j = \Delta^{\weights}$. \end{lemma}

\begin{proof} The first claim follows immediately from Lemma \ref{projectivization-coherent-sheaf} because the graph of 
$\overline{\Delta}^\weights$ is defined as a subset of a projective bundle consisting of 
triples satisfying three closed conditions, and thus is a closed subset, hence projective.  
The second claim follows because $\overline{\Delta}^\weights \circ j $ sends $(\alpha_1, 
t_1, \alpha_2, f, h)$ to  $((\alpha_1,t_1), (\alpha_2, h\circ t_1 \circ f|D^{-1} ))$ which is precisely 
the 
definition of $\Delta^{\weights}$. \end{proof}
 
 \begin{lemma}\label{jopen-semisimple} $j$ is an open immersion, and its image is the 
 locus in $\overline{ \Hecke}_{G(D),H,\weights,V}$ where $\varphi  \in \Hom 
 (V(\alpha_1),V(\alpha_2) \otimes \mathcal O_X(\{\weights\}) ) \subseteq \mathbb P (\Hom 
 (V(\alpha_1),V(\alpha_2) \otimes \mathcal O_X(\{\weights\}) )\oplus k)$.
  \end{lemma}
 
 \begin{proof} By construction, a point in the image of $j$ has $\varphi$ contained in 
 $\Hom (V(\alpha_1),V(\alpha_2) \otimes \mathcal O_X(\{\weights\}) )$. The subset where $\varphi  \in \Hom 
 (V(\alpha_1),V(\alpha_2) \otimes \mathcal O_X(\{\weights\}) )$ is the inverse image of a standard affine open chart of projective space, and thus is an open subset $U$ of $\overline{ \Hecke}_{G(D),H,\weights,V}$. Hence, to prove that $j$ is an open immersion 
 whose image is $U$, it suffices to find an inverse of $j$ over this $U$.
 
Fix a point $(\alpha_1,t_1,\alpha_2,t_2,\varphi) \in U$ and an open set away from the support of $\weights$ where 
$\alpha_1$ and $\alpha_2$ can be trivialized, so that $ \Hom (V(\alpha_1),V(\alpha_2) 
\otimes \mathcal O_X(\{\weights\}) ) = \End V \otimes \mathcal O_X$. Using this isomorphism, we can view the section $\varphi$ as a map from the curve $X$ to the vector space $\End V$. By definition, its image must lie in  $\End V \cap 
\overline{G} $. Because $\End V \cap \overline{G} = G$, we can view $\varphi$ as a map from $X$ to $G$. Remembering the trivialization, $\varphi$ defines an isomorphism of $G$-bundles $\alpha_1 
\to \alpha_2$.

Because changing the two trivializations acts on $\End V$ by left and right multiplication by $G$, this isomorphism does not depend on the choice of trivialization. Thus it glues to a 
global isomorphism away from $\weights$. Hence we obtain an isomorphism $f: \alpha_1 \to 
\alpha_2$ as $G$-bundles away from $\weights$. By assumption we know that $\varphi$, when viewed as a point in the 
 formal loop space $(\End V \oplus k ) \pseries{t}$, is in the closure of the set of 
 pairs $(\lambda V(g),\lambda)$ where $\lambda \in \mathbb G_m$ and $g \in G\pseries{t}$ 
 is in the Schubert cell associated to $W_x$. Because $\varphi = ( V(f),1)$ and $V$ is 
 faithful, this implies that $f$, when viewed as a point in $G \pseries{t}$, it is in the 
 closure of the Schubert cell associated to $\weights_x$, hence modulo $G 
 \bseries{t}$, it is in the closure of the cell of the affine Grassmannian associated to 
 $\weights_x$.

Over $D$, $t_2 \circ f \circ t^{-1}$ lies in the closure of the set of points $(h\lambda, 
\lambda)$ for $h \in H$. Because the last coordinate is nonzero, we may fix it to equal 
$1$, and thus take $\lambda=1$, so it lies in the closure of $H$ inside $\End V \lWR 
\mathcal O_D \rWR$. Because $H$ is a closed subgroup of $G \lWR \mathcal O_D \rWR$, which 
is closed in $\End V \lWR \mathcal O_D \rWR$,  in fact $t_2 \circ f \circ t_1^{-1}$ lies 
in $H$, so we may take $h$ to be $t_2 \circ f \circ t_1^{-1}$. 

Verifying that this is an inverse is a routine calculation.
 \end{proof} 
 
 \begin{lemma}\label{schematic-affine} The map $\Delta^{\weights}$ is schematic and affine. \end{lemma}
 
 \begin{proof}  By Lemma~\ref{compactification-semisimple} and~\ref{jopen-semisimple}, this map is the composition of the open immersion $j$ with the projective morphism $\overline{\Delta}^{\weights}$. Moreover, this open immersion is the complement of the hyperplane $ \mathbb P (\Hom 
 (V(\alpha_1),V(\alpha_2) \otimes \mathcal O_X(\{\weights\}) ))$ inside  $\mathbb P (\Hom 
 (V(\alpha_1),V(\alpha_2) \otimes \mathcal O_X(\{\weights\}) )\oplus k)$. Thus, by Lemma \ref{projectivization-coherent-sheaf}, 
 $\Delta^{\weights}$ is a hyperplane complement in a projective morphism, so it is 
 affine.\end{proof}

\begin{remark}
Throughout 
Section~\ref{s:cleanness}, we do not need the full formalism of \'{e}tale cohomology on 
stacks. This is because the relevant morphisms are schematic morphisms between Artin 
stacks, so we can define $\Delta^{\weights}_!$ and $\Delta^{\weights}_*$ smooth-locally 
as derived pushforwards with respect to morphisms of schemes.
\end{remark}
  
If $G = \SL_n$ and $V$ is the standard representation, we can classify the points 
of $\overline{ \Hecke}_{G(D),H,\weights,V}$ according to the generic rank of $\varphi$. For each rank, we can consider the maximal parabolic 
subgroup that preserves the kernel of $\varphi$, and its unipotent radical, elements of 
which fix $\varphi$ when acting by composition on the right. Sections of this unipotent 
radical act as local automorphisms of $\overline{\Hecke}_{G, H, \weights, V}$. These 
automorphisms can be used to show the vanishing of $ j_*  (IC_{\Hecke_{G(D),H}} \boxtimes 
\mathcal L)$ at these points. In the general group case, we will replace the study of the 
rank with the orbits in 
$\overline{G}$ of the joint  left and right  action of $G \times G$. We describe these 
orbits using the standard theory of reductive groups in the next subsection.

 \begin{remark} If $G$ is adjoint and $V$ is an irreducible representation whose highest weight is regular, i.e., not fixed by any 
 nontrivial 
 element of the Weyl group, then $\overline{G}$ is isomorphic to 
 ``wonderful compactification" of $G$. We expect in this case that 
 $\overline{\Hecke}_{G(0),H,1,V}$ is very close to the Drinfeld--Lafforgue--Vinberg 
 compactification of $\Bun_G$ as defined by Schieder~\cite{Schieder14}, which is closely connected to the wonderful compactification. Our proof 
 uses heavily the explicit representation $V$ as a form of coordinates, but it seems 
 plausible that a ``coordinate-free" proof of the same result can be obtained using the 
 abstract theory of the wonderful and Drinfeld--Lafforgue--Vinberg compactifications. 
 
However, for our proof, there is no reason to choose $V$ to be the representation 
associated to a regular weight. If we instead choose a representation like the standard 
representation (for $G$ a classical group), the compactification we use, and other 
concepts involved like the height, admit particularly simple descriptions. The reader may 
wish to follow along with the case $G=\Sp_{2g}$ in mind, say. \end{remark}
 
 \subsection{Lemmas on semisimple groups}\label{lemmas-on-semisimple}
Let $G$ be a split semisimple group, $V$ a faithful representation over $k$, and fix a 
split maximal torus $T$ of $G$.

 \begin{lemma}\label{closure-characterization} Any point in $\overline{G} - G  \subseteq 
 \mathbb P(\End V \oplus k)$ can be expressed as $(g_1 e g_2, 0)$ where $g_1, g_2 \in G$ 
 and $e$ is the idempotent projector onto the sum of eigenspaces of $T$ whose weights lie 
 in some proper face of the convex hull of the weights of $V$. \end{lemma}

 \begin{example}
Let us provide some examples of what these idempotent projectors look like:

 (i) Let $G = \SL_n$ and let $V$ be the standard representation. Then the weights of $V$ 
 are $n$ linearly independent vectors, forming the vertices of an ($n-1$)-simplex. Hence 
 any nonempty proper subset of the weights is the set of weights lying in some proper face of the convex hull. Thus any 
 diagonal matrix with all diagonal entries $0$ and $1$, not all $1$ and not all $0$, is 
 such an $e$. 

(ii) Let $G = \Sp_{2g}$ and let $V$ be the standard representation. Then the weights of $V$ 
are the vectors with one entry $\pm1$ and the rest $0$ in $\mathbb Z^g$. The convex 
polytope this forms is a cross-polytope, whose proper faces are all simplices. The weights 
lying in a face form a subset $S$ of these vectors, such that for any $v \in S$, $-v\not\in S$. Thus $e$ is an idempotent projector onto an isotropic subspace, whose kernel contains a maximal isotropic subspace.

(iii) Let $G=G_2$ and let $V$ be the unique seven-dimensional irreducible representation. 
Then the weights of $V$ form the six vertices and center of a hexagon. The proper faces 
consist of either one vertex or two adjacent vertices, so the sum of the eigenspaces is a 
subspace of dimension one or two. These subspaces are isotropic under the $G_2$-invariant 
quadratic form on $V$ and the two-dimensional subspaces are sent to zero by the unique 
$G_2$-equivariant map $\wedge^2 V \to V$ (as the product of their eigenvalues under $T$ 
is not a weight of $V$). 
\end{example}
 
 In all the above examples,  the stabilizer of the  sum of the eigenspaces of $T$ whose 
 weights 
 lie in a proper face is a maximal parabolic subgroup of $G$. We will later prove that 
 it is always a parabolic subgroup, but it need not be maximal --- for instance when 
 $G=\SL_n$ and $V$ is the adjoint representation, it need not be maximal for $n \geq 3$.

 \begin{proof}[Proof of Lemma~\ref{closure-characterization}] First note that any point $x$ of the closure of $G$ is the limit as $t$ goes to $0$ of a 
 $k'\pseries{t}$-valued point of $G$ for some field $k'$. To see this, choose a generic linear subspace $L$ of dimension $1+ \dim \End V - \dim G$ containing $x$. By genericity,  $L \cap \overline{G}$ has dimension $1$ and $L \cap (\overline{G} \setminus G )$ has dimension $0$. The normalization of $L \cap \overline{G}$ is a smooth curve mapping to $\overline{G}$ whose image contains $x$ but all but finitely many points of which map to $G$, and choosing a local coordinate at some point mapping to $x$ gives the desired $k'\bseries{t}$-valued point.

 By the Bruhat decomposition, any 
 such point can be written as $g_1(t) \chi(t) g_2(t)$ where $g_1,g_2$ are 
 $k'\bseries{t}$-valued points of $G$ and $\chi$ is a cocharacter of $T$. Now $\chi(t)$ 
 converges as $t$ goes to $0$ to a point $\chi(0) \in \mathbb P(\End V\oplus k)$, and because 
 the left and right group actions are continuous, $g_1(t) \chi(t) g_2(t)$ converges as 
 $t$ goes to $0$ to $g_1(0) \chi(0) g_2(0)$. 
 
 If $\chi$ is trivial, then $\chi(0)$ is the identity element and this limit is in $G$.
 
 Otherwise, in an eigenbasis, $\chi(t)$ is a diagonal matrix whose entries are integer powers of $t$, where the integer power appearing is a linear function of the weight. The projective coordinates for $\chi(t)$ are the entries of this matrix plus an additional $1$. Because $\chi$ is nontrivial, not all these exponents are $0$, 
 and because $G$ is semisimple, the sum of the exponents vanishes, so some are negative 
 and some are positive. To calculate the limit in projective space as $t$ goes to $0$, we first divide each coordinate by the minimal power of $t$ that appears, making each coordinate a nonnegative power of $t$, and then set $t=0$, making each coordinate $1$ or $0$. The $1$s occur exactly on the diagonal entries corresponding to eigenspaces with minimal exponent. Thus $\chi(0)$ is the idempotent projector $e$ onto the sum of eigenspaces of $T$ with minimal exponent. These are the eigenspaces where some nontrivial linear function of the weights is minimized, i.e. some proper face of the convex hull of the weights.
 
 The last coefficient of 
 $\chi(0)$ is $0$, so 
 multiplying on the left by $g_1(0)$ and the right by $g_2(0)$ we obtain $(g_1 e g_2, 0)$.\end{proof}

 Fix a proper face of the convex hull of the weights of $V$, and take the idempotent 
 projector $e$, so that $\im(e)$ is the sum of 
 the $T$-eigenspaces whose weights lie on that face and $ \ker(e)$ is the sum of the 
 $T$-eigenspaces whose weights do not lie on that face.

 Associated to a point in the $\overline{G}-G$ is a natural parabolic subgroup, the stabilizer of $\ker(e)$ (as we will see below, in Lemma \ref{parabolic-construction}). A key useful property is that its unipotent radical acts trivially on $e$ (Lemma \ref{levi-quotient}). In \S\ref{s:vanishing-cusp}, we will define a height function so that, at points of large height, there are many global automorphisms of the $G$-bundle $\alpha_1$ that lie in the unipotent radical. We will then exploit these extra symmetries. Thus, we will define our height function using this particular parabolic subgroup.
 
\begin{lemma}\label{parabolic-construction}
The stabilizer of $\ker(e)$ is a parabolic subgroup of $G$, and this 
  stabilizer remains smooth after lifting $G$ and $V$ to the Witt vectors of $k$.  
\end{lemma}
 
 \begin{proof} Let $S$ be the stabilizer of $\ker(e)$, viewed as a group scheme over the Witt vectors $W(k)$. We first check that $S_{W(k)[1/p]} $ is a parabolic 
 subgroup. It suffices to check that $S_{W(k)[1/p]} $ is proper and contains a Borel subgroup.
 
 By Lemma \ref{closure-characterization}, there exists a linear form $\omega$ on the weight space such that $e$ is the idempotent projector onto the eigenspaces of weights that maximize $\omega$.
 
The stabilizer $S_{W(k)[1/p]} $ is 
 proper because the weights of $\im(e)$ are the weights maximzing $\omega$, so the sum of $\omega$ over the weights 
 of $\im(e)$ is positive, and thus the sum of $\omega$ over the weights of $\ker(e)$ is negative, 
 which is 
 impossible if $\ker(e)$ is a representation of $G$. Thus $\ker(e)$ is not $G$-stable and 
 so its 
 stabilizer $S_{W(k)[1/p]} $ is proper.
  
 To show that $S_{W(k)[1/p]}$ contains a Borel, note that the linear form $\omega$ is in some Weyl chamber of the dual to weight space. With regards 
 to the ordering induced by that Weyl chamber, $\omega$ takes nonnegative values 
 on all the simple roots, hence takes nonnegative values on all the positive roots.  
 Hence the set of weights of $V$ where $\omega$ takes its maximal value is closed 
 under addition of positive roots, and the complement of this set is closed under 
 addition of negative roots.  Therefore $\ker(e)$ is closed under the lowering operators 
 and 
 thus stable under the opposite Borel.
 
To show that $S$ is smooth over $W(k)$, and thus remains 
parabolic in characteristic $p$, it suffices to check that the cotangent space of $S$ at the identity is $p$-torsion 
free, in other words that every element of the Lie algebra of $S$ 
in 
characteristic $p$ is the reduction mod $p$ of an element in the Lie algebra of $S$ in 
characteristic zero. Because $\ker(e)$ is $T$-invariant, the Lie algebra of $S$ is 
a sum of $T$-eigenspaces, and so it is sufficient to check this for raising operators 
associated to roots. Let $J^+$ be the raising operator associated to a root and let $J^-$ 
be the lowering operator associated to the opposite root.

Suppose that $J^+$ does not 
stabilize $\ker(e)$ in characteristic zero but does in characteristic $p$. Because $J^+$ 
does not stabilize, it raises $\omega$, so $J^+\im(e)=0$,  and $J^-$ lowers $\omega$ so $J^- 
\operatorname{Im}(e) \subseteq \ker(e)$. Thus in characteristic $p$, $J^+ J^- \im(e) \subseteq J^+ 
\ker 
(e) \subseteq \ker(e)$, and $J^- J^+ \im(e) \subseteq J^- 0 = 0$, so $[J^+ , J^- ]  
\im(e) 
\subseteq \ker(e)$. Now $[J^+, J^-]$ is an element of the Lie algebra of the maximal 
torus. More precisely, $[J^+, J^-]$ is the coroot corresponding to $J^+$, so $\im(e)$ is a sum of eigenspaces of this 
coroot, and thus all the eigenvalues must be $0$ mod $p$. Because the eigenvalues are pairings of the coroot corresponding to $J^+$ with weights of $V$, and hence are integers 
at most $p$, they must be zero, by our assumption from the beginning of this section. Because $J^+ \im(e)=0$, all eigenvalues of $[J^+, J^-]$ 
on 
$\im(e)$ are highest weights of their corresponding representations, so  all irreducible 
representations of the $\mathfrak sl_2$ generated by $J^+, J^-$, and $[J^+,J^-]$ other 
than those contained in $\ker(e)$ have highest weight zero, hence are trivial, hence have 
$J^+$ vanish on them, which contradicts the assumption that $J^+$ does not stabilize 
$\ker(e)$ in characteristic zero.
 \end{proof}

 \begin{lemma}\label{levi-quotient} Let $P$ be the stabilizer of $\ker(e)$ and let $M$ be 
 its Levi subgroup.  The action of $P$ on $ V / \ker(e)$ factors through the projection 
 $P 
 \to M$. %
 \end{lemma}
 
 \begin{proof} The set of weights in a proper face is the locus where some linear form $\omega$ on 
 the weight lattice takes its maximal value among the weights of $V$.
 
Because the subspace $\ker(e)$ is stable under the maximal torus, $P$ contains, 
and hence is normalized by, the maximal torus, so the Lie algebra of $P$ is generated by some 
subset of the raising and lowering operators corresponding to roots. The maximal unipotent subgroup of $P$ is generated by the operators corresponding to some further subset of 
the roots.
 
If the raising or lowering operator corresponding to some root $\alpha$ acts nontrivially on 
$V/ \ker(e)$, then there must be two weights maximizing $\omega$ that differ by $\alpha$, so we must have $\omega(\alpha)=0$, and thus the operator corresponding to 
$-\alpha$ is 
also in the stabilizer $P$, and hence the unipotent element corresponding to $\alpha$ is in some 
$SL_2$-triple and thus is not in the maximal unipotent subgroup of $P$.

Because no generator of the maximal unipotent subgroup of $P$ acts nontrivially on $V/ \ker (e)$,  the whole 
unipotent subgroup acts trivially, and so the action factors through $M$.\end{proof}

This statement will be useful to prove the smoothness of the Hecke correspondence later:

 \begin{lemma}\label{neighborhood-parabolic-structure} Let $e$ be the idempotent 
 projector onto the $T$-eigenspaces in some proper face of the convex hull of the 
 weights of $V$. Let $P$ be the parabolic subgroup of $G$ consisting of elements 
 stabilizing $\ker(e)$. Then
 \begin{enumerate}[(i)]
 \item   The natural map $\pi: G \to P \backslash G$ extends to a map $\pi'$ from an open 
 subset $U$ of $\overline{G}$ to  $P \backslash G$, such that $(e,0) \in U$ and $\pi'(e) = P 
 \in P\backslash G$. 
\item Let $\overline{P}$ be the projective closure of $P$ inside $\mathbb P(\End V\oplus k)$. 
Any element of $U$ sent to the identity under $\pi'$ lies in $\overline{P}$. 
\end{enumerate}\end{lemma}

 \begin{proof} 
(i) Let $U \subseteq \overline{G} $ be the open subset consisting of $(x,\lambda) \in 
 \overline{G}$  where $\rank(ex) = \rank(e)$. There is a map $k$ from $U$ to the 
 Grassmannian  $\operatorname{Gr}( \dim \ker(e),\dim V) $ that sends $x$ to $\ker(ex)$.  
 Such a map is invariant under the left action of $P$, which by definition preserves 
 $\ker(e)$, so we have a commutative diagram
 
 \[\begin{tikzcd} G \arrow[r, "\pi"] \arrow[d] & P \backslash G \arrow[d,"i"] \\ U \arrow[r, "k" ] & \operatorname{Gr}( \dim \ker(e),\dim V) \\ \end{tikzcd} \]
 
 Because $P$ is the schematic stabilizer of the kernel of $e$, $i$ is an embedding, and because $P$ 
 is parabolic, $P \backslash G$ is proper, and so $i$ is a closed immersion. In particular, the image of $i$ is closed. Because $G$ is dense in $U$, the image of $k$ is contained in the image of $i$, so we can factor $k = i \circ \pi'$ for a unique map $\pi': U \to P \backslash G$. By commutativity, this extends $\pi$.
 
Because $e$ is idempotent, $\rank(e^2)=\rank(e)$, so by the definition of $U$, $(e,0)\in U$.
Furthermore $i \circ \pi'(e) = \ker(e^2) = \ker(e) = i (1)$, so because $i$ is injective, 
$\pi'(e)=1$. 
   
(ii) There is a map $m:  \overline{P} \times G \to \overline{G}$ defined by the embedding 
 $\overline{P} \subseteq \overline{G}$ and the right action of $G$ on $\overline{G}$.  Because $m$ is stable under the action of $p \in P$ on $\overline{P} \times G$ that sends $(x,g)$ to $(xp, p^{-1} g)$,  $m$ descends to a map $\gamma: P 
 \backslash (\overline{P} \times G ) \to \overline{G}$. Now $P \backslash (\overline{P} \times G) $ 
 is an $\overline{P}$-bundle on $P\backslash G$ and both of these are proper, so 
 $P \backslash (\overline{P} \times G) $ is proper.  Because the map to $\overline{G}$ is proper and 
 has dense image, it is surjective.  Let $U$ be the open subset of $\overline{G}$ on which 
 the map $\pi': U \to P \backslash G$ is defined. Then $\gamma^{-1}(U)$ admits two maps 
 to $P \backslash G$, the first given by $\pi' \circ \gamma$ and the second by projection to the second 
 factor, which agree on the dense subset $P \backslash (P \times G) = G$ and hence are 
 equal as $P \backslash G$ is separated.  Hence every point  that is sent to the identity 
 must be an element of $( \overline{P} \times G) / P$ with the second factor in $P$, in 
 other words an element of $\overline{P}$, as desired.
 \end{proof}
 
\begin{lemma}\label{V-existence} Let $G$ be a split semisimple algebraic group over a 
field $k$ of characteristic $p$. If $p>2$, then there exists a faithful representation $V$ of $G$ 
defined over $\mathbb Z$ such that the pairing of any weight of $V$ with any coroot 
of $G$ is less than $p$. 

If $p=2$, there exists such a representation if each nontrivial normal subgroup of $G$ has nontrivial 
center. \end{lemma}

\begin{proof} In fact, we will construct $V$ where all the pairings are at most $2$. We will construct it over $\mathbb Z$ as a sum of highest weight representations, and 
reduce modulo $p$.  It is sufficient to show that, for each character $\chi$ of the center $Z(G)$ of 
$G$, there exists such a representation $V_\chi$ whose central character is $\chi$ and 
whose kernel is contained in $Z(G)$. Then taking $V = \sum_\chi V_\chi$, the kernel of $V$ will be the 
intersection inside $Z(G)$ of the kernels of all characters of $Z(G)$, and hence be trivial.

For this statement, we can assume that $G$ is simply-connected, as any representation of 
the universal cover of $G$ with a central character pulled back from $G$ is in fact a 
representation of $G$ with the same central character.

For $G =G_1 \otimes G_2$, $V_1$ a representation of $G_1$ satisfying the condition on weights, and $V_2$ a representation of $G_2$ satisfying the condition on weights, $V_1 \otimes V_2$ satisifies the condition on weights. The same is true for the condition on kernels. Thus, expressing $G$ as a product of simple groups, we may assume that $G$ is simple. 

We will now check that, for each simply-connected simple $G$, and each character $\chi$, there exists such a representation $V_\chi$.

For the trivial character, the adjoint representation satisfies the pairing condition 
if and only if the Dynkin diagram has no edges of multiplicity greater than $2$, so we 
can use the adjoint representation for any group except $G_2$. Because the center of 
$G_2$ is trivial, we can use the seven-dimensional standard representation for $G_2$. 

It remains to handle the nontrivial characters. For any simple group, there exists a 
unique minuscule representation for each central character, and for any nontrivial 
character, the minuscule representation satisfies both conditions. Indeed, because it is 
not the trivial representation, its kernel is contained in the center, and because the 
Weyl group acts transitively on the weights (the definition of minuscule) the weights lie 
on a sphere, and so no three are collinear. But any weight whose pairing with a coroot is 
$k$ lies in a $k+1$-dimensional representation of the $SL_2$ containing the dual root, 
hence lies in a series of $k+1$ weights in a line, so we must have $ k \leq 1$. 

In the $p=2$ case, we can take $V$ to be the sum of all minuscule representations of $G$, 
which 
necessarily have all pairings $\leq 1$.  Because, for each character of the center of $G$, there exists a minuscule representation with that central character, $V$ is a faithful representation of the center of $G$. If $V$ were not a faithful representation of $G$, then some nontrivial normal subgroup of $G$ would act trivially on $V$. By assumption, this subgroup has nontrivial center, which also acts trivially, contradicting the faithfulness restricted to the center.  Thus $V$ is faithful.

\end{proof}

\subsection{Vanishing near the cusp}\label{s:vanishing-cusp}

Fix a point $(\alpha_1,t_1,\alpha_2,t_2,\varphi)$ of 
$\overline{\Hecke}_{G(D),H,\weights,V}$ not in the image of $j$. We will define, using this data, a parabolic subgroup $P$ of $G$ and a group scheme $\cP_{\alpha_1,\varphi}$ over $X$, locally isomorphic to $P$.

Let $U$ be an open subset of $X$ on which $\alpha_1$ and $\alpha_2$ are trivialization. Over $U$, we have $\Hom 
 (V(\alpha_1),V(\alpha_2) \otimes \mathcal O_X(\{\weights\}) )\oplus \cO_X \cong (\End V \oplus k) \otimes \mathcal O_X$. Restricting $\varphi$, which is a section of $\Hom 
 (V(\alpha_1),V(\alpha_2) \otimes \mathcal O_X(\{\weights\}) )\oplus \cO_X$ well-defined up to scaling, to $U$, we obtain a map $U \to  (\End V \oplus k)$ up to scaling, and hence a map $U \to \mathbb P (\End V \oplus k)$. By the definition of $\overline{\Hecke}_{G(D),H,\weights,V}$, this map has image in $\overline{G}$, and so we obtain a map $U \to \overline{G}$.  By Lemma \ref{jopen-semisimple}, the last coordinate of $\varphi$ vanishes, so this map has image contained in $\overline{G} \setminus G$. 
 
 By Lemma \ref{closure-characterization}, $\overline{G}\setminus G$ is a finite union of locally closed $G\times G$-orbits of the form $ G(e,0)G$. Because this union is finite, one must contain the image of an open subset $X_0 \subseteq U$. Let $e$ be this idempotent projector and let $P$ be the stabilizer of $\operatorname{\ker e}$, which by Lemma~\ref{parabolic-construction} is a 
 parabolic subgroup. Let $N$ be the unipotent radical of $P$.
 
 We note that $e$, and thus $P$, is independent of the choice of trivialization, since changing the trivializations would have the effect of multiplying on the left and right by maps $U \to G$, which preserves all $G \times G$ orbits.  
 
\begin{defi} The Grassmannian of $\dim (\ker e)$-dimensional subspaces of $V(\alpha_1)$ forms 
a fiber bundle over $X$. Over $U'$, $\ker \varphi$ defines a section of this bundle. Because the 
Grassmannian is proper, this extends to a section $s_\varphi$ over $X$. Let $\cP_{\alpha_1,\varphi} 
\subseteq \Aut(\alpha_1)$ be the group scheme over $X$ of automorphisms preserving $s_\varphi$. 
\end{defi} 
 \index{$\cP_{\alpha_1,\varphi}$, group scheme over $X$ locally conjugate to $P$}

We can check that $\cP_{\alpha_1,\varphi}$ is locally conjugate to $P$. The $G$-orbit of $\ker e$ 
inside the Grassmannian of $\dim (\ker e)$-dimensional subspaces of $V$ is isomorphic to $P 
\backslash G$, thus proper, hence closed. For any subset $U^*$ of $X$ on which $\alpha_1$ is 
trivialized, the image of $U^* \cap X_0$ under the section $\ker \varphi$ lies in $G \cdot (\ker e)$, 
so because $G \cdot (\ker e)$ is closed, the image of $s_{\varphi}$ is contained in $G \cdot (\ker e)$, 
and thus the stabilizer of $s_{\varphi}$ at any point is conjugate to the stabilizer $P$ of $\ker 
e$. 

\begin{defi} Let $\cN_{\alpha_1,\varphi}$ be the unipotent radical of $\cP_{\alpha_1,\varphi}$.   \index{$\cN_{\alpha_1,\varphi}$ } \end{defi} 

We can observe that the closed subset of the Grassmannian of subspaces of $V(\alpha_1)$ 
which locally under a trivialization is $G \cdot (\ker e)$ is 
\[( ( G \cdot (\ker e)) \times \alpha_1)/G = (P \backslash G \times \alpha_1)/G = P \backslash \alpha_1 \] and so the section $s_\varphi$ defines a reduction of $\alpha_1$ from a $G$-bundle to a $P$-bundle. Then $\mathcal P_{\alpha_1,\varphi}$ is the twist of $P$ by this $P$-bundle under the conjugation action of $P$, and similarly $\mathcal N_{\alpha_1,\varphi}$ is the twist of $N$.

\begin{lemma}\label{varphi-preserved} Let $\sigma $ be a section of $\cN_{\alpha_1, 
\varphi}$, viewed as an automorphism of $V ( \alpha_1)$. Then \[\varphi \circ \sigma = 
\varphi.\] \end{lemma}

\begin{proof}  Because this equation is a closed condition, it suffices to check this 
over $X_0$, and to work locally. In particular, we may trivialize $\alpha_1$ and 
$\alpha_2$. Using that trivialization, from Lemma~\ref{closure-characterization}, $\varphi$ can be 
expressed as $g_1 e g_2$.  From the definition of $\cP_{\alpha_1,\varphi}$ and 
$\cN_{\alpha_1,\varphi}$, we see that $\cP_{\alpha_1, \varphi} =  g_2^{-1} P g_2$ and 
$\cN_{\alpha_1, \varphi} =  g_2^{-1} N g_2$. So it suffices to check that for $\sigma \in 
N$,   $ e \sigma = e$. Elements of $N$ certainly lie in $P$ and thus preserve the kernel 
of $e$, so to check $e \sigma = e$ it suffices to check that they act trivially on the 
quotient by this kernel, which is done in Lemma~\ref{levi-quotient}. \end{proof}

The global sections of $\cN_{\alpha_1, 
\varphi}$ will be crucial in our vanishing argument. We next give a reasonable criterion for there to be sufficiently many global sections of $\cN_{\alpha_1, 
\varphi}$, by writing it as an iterated extension of vector bundles, and then assuming those 
vector bundles have no low-degree quotient line bundle. The Riemann--Roch theorem then 
implies that there are enough sections in a precise sense --- see Lemma 
\ref{section-semisimple}.

Because $G$ is split, we may assume that $P$ is defined over $\mathbb Z$. 

\begin{defi} Let $N_{\mathbb Q} = N_{0,\mathbb Q} \supseteq N_{1,\mathbb Q} \supseteq N_{2,\mathbb Q} \supseteq \dots  \supseteq  N_{r,\mathbb Q} = 1$ be the derived series of $N_{\mathbb Q}$.  Let $N_{0,\mathbb Z} \supseteq N_{1,\mathbb Z} \supseteq N_{2,\mathbb Z}  \supseteq \dots  \supseteq  N_{r ,\mathbb Z}=1$ be their schematic closure in $N_{\mathbb Z}$, and let $N_0 \supseteq N_1 \supseteq N_2  \supseteq \dots  \supseteq  N_r =1$ be their reductions mod $p$.  \end{defi}\index{ $N_{i}$, filtration of a unipotent radical} 

\begin{lemma} For all $i$, $N_i$ is a smooth connected $P$-invariant subgroup of $N$, and 
$N_i/ N_{i+1}$ is isomorphic to a vector space (i.e., a power of $\mathbb G_a$), where 
the action of $P$ on $N_i / N_{i+1}$ is by vector space automorphisms. \end{lemma}

\begin{proof}We can verify all these facts by the theory of root groups.

Let $U$ be a maximal unipotent subgroup of $G$, defined over $\mathbb Z$, containing $N$. For 
each root $\alpha$ of $U$, there is a root group $U_\alpha$, a subgroup isomorphic to $\mathbb 
G_a$ over $\mathbb Z$, which in characteristic zero is the exponential of that 
root~\cite[Thm.4.1.4 and Def.4.2.3]{Conrad:reductive-gp-schemes}. (In general the root 
group may be a line bundle, but over $\mathbb Z$ the only line bundle is $\mathbb G_a$.) 
Moreover, $U$ is isomorphic as a scheme to the product of these root groups, with the 
isomorphism given by multiplication in the group law, for any fixed ordering of the roots 
\cite[Thm.5.1.13]{Conrad:reductive-gp-schemes}.

Choose an ordering where the roots not in $N_{\mathbb Q}$ are first, then the roots in $N_{0,\mathbb Q}$ but not in $N_{1,\mathbb Q}$, and so on, and use the induced isomorphism to a product of copies of $\mathbb G_a$ as coordinates on $U$. In this ordering, each of the closed subsets $N_{i,\mathbb Q}$ is defined by the vanishing of an initial segment of the coordinates. Hence their schematic closures, and the reductions mod $p$, are defined by the same equations. In particular, they are smooth and connected. The fact that these closed subsets are $P$-invariant, and are subgroups, can be expressed by algebraic equations and hence holds in the reduction mod $p$ because it holds over $\mathbb Q$. 

Because the commutator of two roots in $N_{i,\mathbb Q}$ necessarily lies in $N_{i+1,\mathbb Q}$, the group law on $N_{i,\mathbb Q}/ N_{i+1,\mathbb Q}$ is simply given by addition in our fixed coordinates, and thus the action of $P$ is linear in these coordinates. Because these are both closed conditions, they also hold modulo $p$. \end{proof}

\begin{defi}  Let $\cN_{\alpha_1,\varphi,i}$ be the subgroup of $\cN_{\alpha_1,\varphi}$ that is locally $P$-conjugate to $N_i$, which is well-defined since $N_i$ is a $P$-invariant subgroup of $N$.  \end{defi}

\begin{lemma}\label{N-filtration}   The quotient $\cN_{\alpha_1,\varphi,i}/ 
\cN_{\alpha_1,\varphi,i+1}$ is a vector bundle on $X$. \end{lemma}

\begin{proof} This follows from the fact that $N_i/ N_{i+1}$ is a vector space and $P$ acts by vector space automorphisms. \end{proof}

\begin{defi}\label{defi-height} Let the \emph{height} of $(\alpha_1,t_1, \alpha_2, 
t_2,\varphi)$ be minus the smallest degree of a line bundle which occurs as a quotient of 
any of the vector bundles $\cN_{\alpha_1,\varphi}$. \end{defi} \index{height of a point 
on $\overline{\Hecke}$}

We will first, in Lemma \ref{symmetry-semisimple}, see how the stalk of  $j_* 
  (IC_{\Hecke_{G(D),\weights} } \boxtimes \mathcal L )$ changes upon composing $t_1$ by an element $h \in H$. In Lemma \ref{htvan-semisimple}, we will contrast this with the fact that the stalk is invariant under composing $t_1$ by a global section of $\cN_{\alpha_1,\varphi}$, restricted to $H$, to show that the stalk vanishes.

\begin{lemma}\label{symmetry-semisimple} 
  Fix $(\alpha_1,\alpha_2,t_1,t_2, \varphi)$ in $\overline{\Hecke}_{G(D),H,\weights,V}$. 
  Consider the map $c: H \to  \overline{\Hecke}_{G(D),H,\weights,V}$ that sends $h \in H$ 
  to $(V_1,V_2,   h \circ t_1 ,  t_2,  \varphi)$. The pullback  $c^* j_* 
  (IC_{\Hecke_{G(D),\weights} } \boxtimes \mathcal L )$ is isomorphic to the tensor product of the stalk of 
  $j_*  (IC_{\Hecke_{G(D),\weights} } \boxtimes \mathcal L )$ at 
  $(\alpha_1,\alpha_2,t_1,t_2,  \varphi)$ with $\mathcal L^{-1} $. \end{lemma}
  
\begin{proof} Consider the commutative diagram

\[ \begin{tikzcd}
\Hecke_{G(D),\weights} \times H \arrow[r,"j" ] & \overline{\Hecke}_{G(D),H,\weights,V} \\
\Hecke_{G(D),\weights} \times H \times H \arrow[u,"a"]\arrow[r,"j \times id"]& 
\overline{\Hecke}_{G(D),H,\weights,V} \times H \arrow[u,"b"] &  H\arrow[l,"d"] 
\arrow[ul,"c",swap] \\
    \end{tikzcd} \]
where the vertical map $b$ sends $((\alpha_1,\alpha_2,t_1,t_2, \varphi),h)$ to 
$(\alpha_1,\alpha_2, h \circ t_1,  t_2, \varphi)$ , the vertical map $a$ sends $((\alpha_1,\alpha_2, f, t_1), (h_1,h_2))$ to $((\alpha_1,\alpha_2, f, h_2 \circ t_1), h_1h_2^{-1})$, the arrow $d$ sends 
$h\in H$ to $\left((V_1,V_2,  t_1 ,t_2,  \varphi), h\right)$, and so $c$ sends 
$h\in H$ to  $(V_1,V_2,   t_1 , h \circ t_2,  \varphi)$.

We have \[c^* j_*  (IC_{\Hecke_{G(D),\weights} } \boxtimes \mathcal L )=d^*b^* j_* 
(IC_{\Hecke_{G(D),\weights} } \boxtimes \mathcal L).\] 

We have 
\[ \hspace{-.5in} b^* j_* (IC_{\Hecke_{G(D),\weights} } \boxtimes \mathcal L) =  (j\times id)_* a^* (IC_{\Hecke_{G(D),\weights} } \boxtimes \mathcal L)= (j \times id)_*  (  IC_{\Hecke_{G(D),\weights}} \boxtimes \mathcal L \boxtimes \mathcal L^ {-1} ) = j_*  (  IC_{\Hecke_{G(D),\weights} }\boxtimes \mathcal L)  \boxtimes \mathcal 
L^{-1} \] with the first identity by smooth base change, because the left square is 
Cartesian, the third identity is by the K\"{u}nneth formula, and the second identity requires some thought:  By the character sheaf property of $\mathcal L$, the pullback of $\mathcal L$ along the map $(h_1,h_2 ) \to (h_1 h_2^{-1})$ is $\mathcal L \boxtimes \mathcal L^{-1}$. The pullbacks of $IC_{\Hecke_{G(D),\weights}}$ along the morphisms sending $((\alpha_1,\alpha_2, f, t_1), (h_1,h_2))$ to $(\alpha_1,\alpha_2, f, h_2 \circ t_1)$ and $(\alpha_1,\alpha_2, f, t_1)$ are equal since these are both smooth morphisms of the same relative dimension. Thus $a^* (IC_{\Hecke_{G(D),\weights} } \boxtimes \mathcal L)=IC_{\Hecke_{G(D),\weights}} \boxtimes \mathcal L \boxtimes \mathcal L^ {-1}$.

Finally, $d^* ( j_*  (  IC_{\Hecke_{G(D),\weights} }\boxtimes \mathcal L)  \boxtimes \mathcal 
L^{-1})$ is the tensor product of the stalk of 
  $j_*  (IC_{\Hecke_{G(D),\weights} } \boxtimes \mathcal L )$ at 
  $(\alpha_1,\alpha_2,t_1,t_2,  \varphi)$ with $\mathcal L^{-1} $.\end{proof}

\begin{lemma}\label{section-semisimple} Let $\beta$ be a $P$-bundle on $X$. Let 
$\cP_\beta$ 
be the associated twisted form of $P$ and $\cN_\beta$ its unipotent radical. Assume that 
all vector bundles in the canonical filtration of $\cN_\beta$ have no nontrivial 
quotients 
of degree at most $2g-2+|D|$. Then there is a section of $\cN_{\beta}$ over 
$\operatorname{Res}_k^D  (\cN_{\beta}|_D) \times X$, whose restriction to 
$\operatorname{Res}_k^D  (\cN_{\beta}|_D) \times D$ is the canonical section.\end{lemma}

\begin{proof} Let $i: D \to X$ be the immersion, so that $\Gamma(D, i^* \cN_\beta) = 
\Gamma ( X, i_* i^* \cN_\beta)$. First we will show that the map $\Gamma(X, \cN_{\beta}) 
\to \Gamma(X, i_* i^* \cN_{\beta}) $ is 
surjective. The cokernel is contained in the $H^1$ of $X$ with coefficients in the kernel 
of the natural map $\cN_\beta \to i_* i^* \cN_\beta$. The kernel of the natural map 
$\cN_\beta \to i_* i^* \cN_\beta$ has a filtration, induced by pulling back the 
filtration of $\cN_{\beta}$, whose associated graded objects are $\left( \cN_{i,\beta}/ 
\cN_{i+1,\beta} \right) 
\otimes \mathcal O(-D)$. By the assumption on height, $\left( \cN_{i,\beta}/ 
\cN_{i+1,\beta} 
\right) \otimes \mathcal O(-D)$ has no line bundle quotients of degree $2g-2$, thus 
admits no nontrivial maps to the canonical bundle, hence has vanishing $H^1$, so the 
kernel has vanishing $H^1$ as well, and the map is surjective.

Moreover, the $H^1$ of the kernel will still vanish when base changed by any affine 
scheme, as these are flat over the base field, and so the natural map $\Gamma(X \times Y, 
\cN_{\beta}) \to \Gamma(X \times Y, i_* i^* \cN_{\beta}) $ is surjective for any affine $Y$. We 
take $Y$ to be the Weil 
restriction $\operatorname{Res}_k^D  (\cN_{\beta}|_D)$ of $\cN_{\beta}$ from $D$ to $k$, 
over which there is a canonical  element of $\Gamma(D, \cN_{\beta})$. This gives a 
section 
of $\cN_{\beta}$ over $\operatorname{Res}_k^D  (\cN_{\beta}|_D) \times X$, whose 
restriction 
to $\operatorname{Res}_k^D  (\cN_{\beta}|_D) \times D$ is the canonical section.\end{proof}

\begin{lemma}\label{htvan-semisimple} Assume that some $(G,m_u, H_u , \mathcal L_u)$ is 
geometrically supercuspidal. Then the stalk of $j_* (IC_{\Hecke_{G(D),\weights} } 
\boxtimes \mathcal L )$ vanishes at points whose height is greater than $2g-2 + |D|$. 
\end{lemma}

\begin{proof} 

Consider a point $(\alpha_1,\alpha_2,t_1,t_2, \varphi)$ in 
$\overline{\Hecke}_{G(D),H,\weights,V}$ of height greater than $2g-2+|D|$. Let $\beta$ be 
the associated $P$-bundle, so that $\cP_{\beta} = \cP_{\alpha_1,\varphi}$. By Lemma 
\ref{section-semisimple}, there is a section $s$ of $\cN_{\alpha_1,\varphi}$ over 
$\operatorname{Res}_k^D  (\cN_{\alpha_1,\varphi}|_D) \times X$, whose restriction to 
$\operatorname{Res}_k^D  (\cN_{\alpha_1,\varphi}|_D) \times D$ is the canonical section.

Now consider the map $\tau$ from $\operatorname{Res}_k^D  (\cN_{\alpha_1,\varphi}|_D)$ to 
$\overline{\Hecke}_{G(D),H,\weights,V}$ that sends $g \in \operatorname{Res}_k^D  
(\cN_{\alpha_1,\varphi}|_D)$ to $(\alpha_1,\alpha_2, t_1 \circ s(g)|_D,t_2, \varphi)$. Because the 
restriction of $s$ to $D$ is the canonical section, $s(g)|_D$ is the section of 
$\cN_{\alpha_1,\varphi}$ over $D$ induced by $g$. This map 
is 
actually equal to the constant map by a diagram
\[ \begin{tikzcd}  \alpha_1 \arrow[r, "\varphi"] \arrow[d,"s(g)"] & \alpha_2 
\arrow[d,"id"]\\
\alpha_1 \arrow[r,"\varphi"] & \alpha_2  \\ \end{tikzcd} \]
which commutes by Lemma~\ref{varphi-preserved} because $s(g) \in \cN_{\alpha_1,\varphi}$.

Because $\tau$ is the constant map,
\[ \tau^* j_* (IC_{\Hecke_{G(D),\weights} } \boxtimes \mathcal L ) = \mathbb Q_\ell \boxtimes \left( j_* (IC_{\Hecke_{G(D),\weights} } 
\boxtimes \mathcal L ) \right) _{ (\alpha_1,\alpha_2,t_1,t_2, \varphi)} \] where ${}_{(\alpha_1,\alpha_2,t_1,t_2, \varphi)} $denotes the stalk at $(\alpha_1,\alpha_2,t_1,t_2, \varphi)$.

Now $t_1 \circ s(g)|_D=  (t_1 s(g)|_D t_1^{-1}) \circ t_1$.   Because $\alpha_1$ admits a trivialization over $N$, $\cN_{\alpha_1,\varphi}$ is 
conjugate over $D$ to $N$, and so $\operatorname{Res}_k^D  (\cN_{\alpha_1,\varphi} )$ is 
isomorphic to  $N \lWR \mathcal O_D \rWR$, in such a way that the embedding $g \mapsto 
(t_1 g t_1^{-1} )$ into $ G \lWR \mathcal O_D \rWR$ is conjugate to the standard 
embedding. 

Now consider the pullback of $j_* (IC_{\Hecke_{G(D),\weights} } \boxtimes \mathcal L )$ 
along 
the map that sends $h$ to $(\alpha_1,\alpha_2, t_1 \circ h,t_2, \varphi)$ for $h$ in the 
intersection of $H$ with this conjugate copy of $N \lWR \mathcal O_D \rWR$. This pullback 
is $ \mathbb Q_\ell \otimes  \left( j_* (IC_{\Hecke_{G(D),\weights} } 
\boxtimes \mathcal L ) \right) _{ (\alpha_1,\alpha_2,t_1,t_2, \varphi)}$. On the other hand, from Lemma~\ref{symmetry-semisimple}, we know 
that this same pullback is $\mathcal L^{-1} \otimes  \left( j_* (IC_{\Hecke_{G(D),\weights} } 
\boxtimes \mathcal L ) \right) _{ (\alpha_1,\alpha_2,t_1,t_2, \varphi)}$.  From the definition of 
geometric supercuspidal, we know that even restricting to a further intersection with 
$H_x$, 
the pullback of $\mathcal L^{-1}$ is not a geometrically constant sheaf, and so its 
tensor product with no nonzero vector space is geometrically constant, and hence the 
stalk $\left( j_* (IC_{\Hecke_{G(D),\weights} } 
\boxtimes \mathcal L ) \right) _{ (\alpha_1,\alpha_2,t_1,t_2, \varphi)}$ vanishes, as desired.
\end{proof}

 \subsection{Hecke Correspondences}\label{s:Hecke-correspondences}

 We will use the following space to compare the stalks of $j_* 
 (IC_{\Hecke_{G(D),\weights} } \boxtimes \mathcal L )$ at different points:
 
\begin{defi} Fix a geometric point $Q \in X$ that is neither in $D$ nor the support of 
$\weights$ and a cocharacter $\mu$ in the Weyl cone of $G$.  Let $ \Hecke_{Q,\mu} \left( 
\overline{\Hecke}_{G(D),H,\weights,V} \right)$ be the moduli space of quadruples consisting of two points 
$(\alpha_1,\alpha_2,t_1,t_2,\varphi)$ and $(\alpha_3,\alpha_4,t_3,t_4,\varphi')$ in 
$\overline{\Hecke}_{G(D),H,\weights,V} $ and isomorphisms $m_1: \alpha_3 \to \alpha_1 $ 
and 
$m_2: \alpha_4 \to \alpha_2$ away from $Q$, such that $ t_1 \circ m_1 |_D=t_3$, $t_2 \circ 
m_2 |_D= t_4$, $\varphi \circ V(m_1) = V(m_2) \circ \varphi'$, and such that $m_1$ and 
$m_2$, expressed as points in $G \pseries{t}$ via local coordinates at $Q$, are in $G \bseries{t} 
\mu(t) G \bseries{t}$.  (Note that here we use a Schubert cell and not its closure.)

Let  $pr_{12}$ and $pr_{34}:  \Hecke_{Q,\mu} \left( \overline{\Hecke}_{G(D),H,\weights,V} 
\right) \to \overline{\Hecke}_{G(D),H,\weights,V}$ be the maps induced by 
$(\alpha_1,\alpha_2,t_1,t_2,\varphi)$ and $(\alpha_3,\alpha_4,t_1,t_2,\varphi')$ 
respectively.
 \end{defi}

 Let $(\alpha_1,t_1,\alpha_2,t_2,\varphi)$ be a point of 
 $\overline{\Hecke}_{G(D),H,\weights,V}$ not in the image of $j$. As we did at the 
 beginning of the previous subsection, we can choose some open set $X_0$ where $\varphi$ 
 locally takes the form $g_1 e g_2$ for the idempotent projector $e$ onto the space of 
 $T$-eigenvalues of some 
 proper face of the convex hull of the weights of $V$. Equivalently, we can trivialize  
 $\alpha_1$ and $\alpha_2$ over $X_0$, using Lemma \ref{trivialization-generic}, so that  $\varphi$ in the induced 
 coordinates is an idempotent projector $e$. Fix such trivializations.
 
 Let $Q$ be a point in $X_0$ that does not 
 lie in $D$. Let $P$ be the stabilizer of the kernel of $e$. Let $\mu: \mathbb G_m \to T$ 
 be a 
 cocharacter such that the eigenvalue of $g \to \mu(\lambda)^{-1} g \mu(\lambda)$ is 
 a nonnegative power of $\lambda$ on roots in $P$ and is negative on roots not in $P$ 
 (which exists by~\cite[Prop.2.2.9]{CGP-pseudoreductive}).
 
 In this subsection, we will show how to choose a point of  $\Hecke_{Q,\mu} \left( 
 \overline{\Hecke}_{G(D),H,\weights,V} \right)$ whose image under $pr_{12}$ is 
 $(\alpha_1,\alpha_2,t_1, t_2,\varphi)$, whose image under $pr_{34} $ has greater height than 
 $(\alpha_1,\alpha_2, t_1, t_2,\varphi)$, and such that the stalks of the pullbacks of 
 $j_*  (IC_{\Hecke_{G(D),\weights} } \boxtimes \mathcal L )$ on its image under $pr_{12}$ 
 and its image under $pr_{34}$ are isomorphic. This is precisely what we will need to 
 inductively show that the stalk vanishes in the proof of 
 Theorem~\ref{mainextension-semisimple} in the next subsection. 
 
The key step in comparing the stalks is to show that the maps $pr_{12}$ and $pr_{34}$ are 
smooth, as it allows us to use the smooth base change theorem. This can be checked by 
comparing sections of the relevant stalks over the local ring, which can be reduced by a 
Beauville--Laszlo argument to a purely algebraic calculation, which we handle first:

\begin{lemma}\label{smooth-semisimple-algebraic} Let $R$ be a Henselian local ring, with 
maximal ideal $\mathfrak m$. Let $M \in \End V (R\bseries{t} )  $ be a matrix and $s\in 
R\bseries{t}$ an element such that $(M,s)$ are the projective coordinates of an 
$R\bseries{t}$-point of $\overline{G}$. Assume that $(M,s)$ is congruent to $(e,0)$ 
modulo $\mathfrak m$. Let $g_a$ and $g_b$ be elements of $G(R\bseries{t})$ such that $g_a 
\mu(t)^{-1} g_b$ is congruent to $\mu(t)$ mod $\mathfrak m$, where the cocharacter $\mu$ is as 
above.
 
	Then there exist elements $g_c$ and $g_d$ in $G(R\bseries{t})$, such that $g_c \mu(t) g_d$ 
	is 
	congruent to $\mu(t)$ mod $\mathfrak m$, and such that \[\left( g_a \mu(t)^{-1} 
	g_b\right)   M \left( g_c \mu(t) g_d\right) \] belongs to $\End V(R\bseries{t}) \subseteq 
	\End V(R\pseries{t})$. 
 
	Moreover the products $g_c \mu(t) g_d$ for all $g_c$ and $g_d$ satisfying these two 
	conditions lie in a single orbit under the right action of $G(R\bseries{t})$.

 \end{lemma}
 
 \begin{proof} We make a series of reductions.
 
 First note that we may assume $R$ is Noetherian. This is because the problem only 
 depends on finitely many entries of $M, g_a, g_b$ --- those entries that are nonvanishing 
 mod a power of $t$ equal to the sum of the highest negative power of $t$ appearing in 
 entries of $\mu(t^{-1})$ and $\mu(t)$. Hence the problem is defined over a Henselization 
 of a finitely generated subring of $R$, which is Noetherian. For the uniqueness 
 statement, because \[ \left( G(R\bseries{t}) \mu(t) G(R\bseries{t}) \right)/ 
 G(R\bseries{t}) \] is 
 represented by a scheme of finite type  --- more specifically, a Schubert cell of the 
 affine 
 Grassmannian ---  we may check uniqueness in the Henselization of another finitely 
 generated subring of $R$, that generated by the finitely many entries of $M,g_a,g_b$ 
 plus the coordinates in this Schubert cell of two different possible values of $g_c,g_d$.
 
 Next we will show that, by possibly changing $g_a$, we may assume that $g_b$ is congruent to $1$ mod $\mathfrak m$. This is 
 because the map
\index{$G\bseries{t}$, $G\pseries{t}$, formal loop group}
 \[G\bseries{t} \to G\bseries{t} \backslash \left( G\bseries{t} \mu(t)^{-1} G\bseries{t} 
 \right) \] that sends $g$ to $ G\bseries{t}\mu(t)^{-1} g$ (equivalently to 
 $G\bseries{t}g_a \mu(t)^{-1} g$) is smooth at the identity, and so we can lift 
 $G\bseries{t} g_a \mu(t)^{-1} g_b$, which is congruent to $\mu(t)^{-1}$ mod $\mathfrak m$, to 
 an $R$-point of $G\bseries{t}$ congruent to $1$ mod $\mathfrak m$.
 
 Now because $\overline{G}$ is stable under left-multiplication by $G$, we may replace 
 $M$ by $g_b M$ and so assume $g_b=1$.  Because left-multiplication by $g_a$ does not 
 affect integrality, we may assume $g_a=1$.
 
 Now applying Lemma~\ref{neighborhood-parabolic-structure}(1), from $(M,s) \in 
 \overline{G}(R\bseries{t})$ we obtain a point $\pi'(M,s) \in \left( P \backslash 
 G\right) (R\bseries{t})$ congruent to $P$ mod $\mathfrak m$. Since the map $G \to P \backslash G$ is smooth, and the point $\pi'(M,s) \in (P\backslash G) (R\bseries{t})$ lifts mod $\mathfrak m$ to the point $1 \in G ( R \bseries{t}/\mathfrak m)$, it follows that $\pi'(M,s)$ lifts to some $\sigma \in G( R \bseries{t})$. We 
 can multiply $M$ on the right by $\sigma^{-1}$ without affecting the existence of 
 $g_c,g_d$ or their uniqueness, because we can always multiply $g_c$ on the left by 
 $\sigma$ to cancel it. So we may assume that $\pi'(M,s) $ is the identity, and hence by 
 Lemma~\ref{neighborhood-parabolic-structure}(2) that $(M,s)$ lies in 
 $\overline{P}(R\bseries{t})$.

 This implies the existence of a solution. In fact we can take $g_c=g_d=1$, so it 
 suffices to check that $\mu(t)^{-1} M \mu(t)$ is integral. By construction, all the 
 nonzero entries of elements of the Lie algebra of $P$ are multiplied by a nonnegative 
 power of $t$ when conjugated by $\mu(t)$. In characteristic zero, this implies that all 
 the nonzero entries of elements of $P$ are multiplied by a nonnegative power of $t$ when 
 conjugated by $\mu(t)$, as these are exponentials of the Lie algebra elements. Because 
 $V$ lifts to characteristic zero,  the same thing is true for the nonzero 
 entries in the characteristic $p$ representation, and thus the same thing is true for 
 elements of the closure $\overline{P}$ of $P$, including $(M,s)$. So indeed $\mu(t)^{-1} 
 M \mu(t)$ is integral, as desired.

 The argument for uniqueness is more subtle. It suffices to show that, for $M,g_a,g_b$ in 
 this special form, all solutions $g_c \mu(t) g_d$ map to the point $ \mu(t) 
 G(R\bseries{t})$ of the Schubert cell  \[ \left( G(R\bseries{t}) \mu(t) G(R\bseries{t}) 
 \right)/ G(R\bseries{t}) .\] By induction, it is 
 sufficient to assume that the solution maps to this point modulo $\mathfrak 
 m^n$ for some $n\geq 1$ and show that it also maps to this point modulo $\mathfrak m^{n+1}$. Here, to ensure that the map $R \to \hat{R}$ is injective, we use the Noetherian hypothesis. Because the map 
 $G(R\bseries{t}) \to \left( G(R\bseries{t}) \mu(t) G(R\bseries{t}) \right)/ 
 G(R\bseries{t}) $ sending $g$ to $g \mu(t) G(R \bseries{t})$ is smooth, and because $g_c \mu(t) g_d G(R \bseries{t})$ is 
 congruent to $\mu(t) G(R \bseries{t}) $ modulo $\mathfrak m^n$, we may assume $g_c$ is congruent to $1$ 
 modulo $\mathfrak m^{n}$. Then modulo $\mathfrak m^{n+1}$, $g_c$ is $1+\tau $ for some $\tau  \in \mathfrak m^n 
 \mathfrak g(R\bseries{t})$, where $\mathfrak g$ is the Lie algebra of $G$. Then we can 
 write \[  \mu(t^{-1}) M g_c  \mu(t) g_d = \mu(t^{-1} ) M (1+\tau) \mu(t) g_d = 
 \mu(t^{-1}) M \mu(t) g_d +  \mu(t^{-1}) M  \tau \mu(t)g_d.\]  We know that $ \mu(t^{-1}) 
 M \mu(t) g_d $ is integral, so this implies that $\mu(t^{-1}) M  \tau \mu(t)g_d$ is 
 integral, which, inverting $g_d$, implies that $\mu(t)^{-1} M \tau  \mu(t) $ is integral.  
 Because $\tau$ is divisible by $\mathfrak m^{n}$ and $M$ is congruent to $e$ modulo 
 $\mathfrak m$, modulo $\mathfrak m^{n+1}$ we have \[ \mu(t)^{-1} M \tau \mu(t) = 
 \mu(t)^{-1} e \tau \mu(t)  = e \mu(t)^{-1} \tau \mu(t).\] Thus $ e \mu(t)^{-1} \tau \mu(t)$ is integral.  If we write $\mu(t)^{-1} \tau \mu(t)$ as  $\sum_i v_i t^i$ for $i \in \mathbb Z$, then we have $e v_i =0$ for $i<0$. For $i<0$, since $e v_i =0$, $\operatorname{Im}(v_i) \subseteq \ker (e)$, so $v_i ( \ker(e)) \subseteq \ker(e)$, thus $v_i$ lies in the Lie algebra of $P$ because $P$ is by definition the stabilizer of $\operatorname{ker} (e)$.
 
Fix a basis of $\mathfrak g$ consisting of roots of the maximal torus and an arbitrary basis for the Lie algebra of the maximal torus. In such a basis, the Lie algebra $\operatorname{Lie}(P)$ is the span of a subset of the basis vectors, consisting of the roots in $P$ and the maximal torus. Thus, because $v_i \in \operatorname{Lie}(P)$ for $i<0$, if we express $\mu(t)^{-1} \tau \mu(t)$ as a $R\pseries{t}$-linear combination of the basis vectors, the coefficients of every basis vector not in $\operatorname{Lie}(P)$ will be integral. However, because the eigenvalues of conjugation by $\mu(t)$ on $\operatorname{Lie}(P)$ are nonnegative powers of $t$, the coefficient of every basis vector in in $P$ of $\mu(t)^{-1} \tau \mu(t)$ will be integral. So all coefficients are integral, and  thus $\mu(t)^{-1} \tau \mu(t)$ is integral.
 Finally, because $g_c 
 \mu(t) g_d \equiv  1 \mu(t) (1 + \mu(t)^{-1} \tau \mu(t)) g_d \mod m^{n+1},$ this shows that  $g_c \mu(t) g_d$ maps to the point $ \mu(t) 
 G(R\bseries{t})$ of the Schubert cell  $ \left( G(R\bseries{t}) \mu(t) G(R\bseries{t}) 
 \right)/ G(R\bseries{t}) $ modulo $\mathfrak m^{n+1}$, as desired.
 \end{proof}

 We will define a special point of $ \Hecke_{Q,\mu} \left( 
 \overline{\Hecke}_{G(D),H,\weights,V} \right) $ where the smoothness of $pr_{12}$ and 
 $pr_{34}$ is as easy as possible to check. Recall that we have already fixed 
 trivalizations of $\alpha_1$ and $\alpha_2$ on the open set $X_0$, and thus on a formal 
 neighborhood of $Q$.  Let $m_1: \alpha_3 \to \alpha_1$ and 
 $m_2: \alpha_4 \to \alpha_2$ be the unique modifications of $\alpha_1$ and $\alpha_2$ 
 respectively that are isomorphisms away from $Q$ and that in a formal neighborhood of 
 $Q$ are locally isomorphic to the map $\mu(t)$. (This uniquely characterizes them by 
 Beauville--Laszlo.) Let $t_3 = t_1 \circ m_1$ and $t_4 = t_2 \circ m_2$ be the 
 trivializations. Let $\varphi': V(\alpha_3) \to V(\alpha_4)$ be the map that, away from $Q$, 
 is $\varphi$, and in a formal neighborhood of $Q$, is $e$. Let $y=\left( 
 (\alpha_1,\alpha_2,t_1,t_2,\varphi),(\alpha_3,\alpha_4,t_1,t_2,\varphi'),m_1,m_2\right)$.
  Because $e$ commutes with $\mu(t)$, $\varphi \circ V(m_1) = V(m_2) \circ \varphi'$ and 
 so $y$ is a point of $ \Hecke_{Q,\mu} \left( \overline{\Hecke}_{G(D),H,\weights,V} 
 \right) $. 
 
 We can translate Lemma~\ref{smooth-semisimple-algebraic} into a geometric lifting lemma:

 \begin{lemma}\label{smooth-semisimple-local} Let $R$ be a Henselian local ring with 
 maximal ideal $\mathfrak m$. Let  $(\alpha_1^*,\alpha_2^*,t_1^*,t_2^*,\varphi^*)$ be an 
 $R$-point of of $\overline{\Hecke}_{G(D),H,\weights,V}$ that modulo the maximal ideal of 
 $R$ 
 is  $(\alpha_1,\alpha_2,t_1,t_2,\varphi)$.  Let $\alpha_4^*$ be a $G$-bundle on $X_R$ 
 and let $m_2^*$ be an isomorphism: $m_2^* : \alpha_4^* \to \alpha_2^*$ away from $Q$ 
 that expressed in local coordinates over a formal neighborhood of $Q$ lies in $G\bseries{t} \mu(t) G\bseries{t}$ and such 
 that $(\alpha_4^*, m_2^*)$ mod $\mathfrak m$ is isomorphic to $(\alpha_4, m_2)$. 
 
 Then there exists a unique triple of a $G$-bundle $\alpha_3^*$ on $X_R$, isomorphism 
 $m_1^*: \alpha_3^* \to \alpha_1^*$ away from $Q$ that in a formal neighborhood of $Q$ lies in $G\bseries{t} \mu(t) G\bseries{t}$, and $\varphi'^{*} \in \mathbb 
 P(\Hom_X(V(\alpha_3), V(\alpha_4) +k)$ such that $\varphi^* \circ V(m_1^*) = V(m_2^*) 
 \circ \varphi'^{*}$, that is congruent to $(\alpha_3, m_1,\varphi')$ modulo $\mathfrak 
 m$ 
 up to isomorphism.
 
 \end{lemma} 
 
 \begin{proof} Fix trivializations of $\alpha_1^*, \alpha_2^*,\alpha_4^*$ over the formal 
 neighborhood of $Q$ that agree modulo $\mathfrak m$ with the trivializations of 
 $\alpha_1$ and $\alpha_2$ we have chosen and with the trivialization of $\alpha_4$ in 
 which $m_2$ is $\mu(t)$.
 
	 By Beauville--Laszlo, the data of $\alpha_3^*$ is equivalent to the data of a 
	 $G$-bundle over a formal neighborhood of $Q$, a $G$-bundle over the complement of 
	 $Q$, and an isomorphism between the two over the punctured formal neighborhood. 
	 Because $m_1^*$ is an isomorphism over the complement of $Q$, we can take the 
	 $G$-bundle over the complement of $Q$ to be $\alpha_1^*$, so the data of 
	 $(\alpha_3^*, m_1^*)$ is simply a $G$-bundle over a formal neighborhood of $Q$ with 
	 an isomorphism to $\alpha_1^*$ over the punctured formal neighborhood.  Because we 
	 have a trivialization of $\alpha_1^*$, this data is equivalent to an element of 
	 $G(R\pseries{t})$ modulo the right action of $G(R\bseries{t})$. We can view this 
	 element as $m_1^*$ because it is the isomorphism from $\alpha_3^*$ to $\alpha_1^*$ 
	 in formal coordinates.
 
 The map $\varphi^{'*}$ is uniquely determined by the other data, as we must have 
 $V(m_2^*)^{-1}  \circ  \varphi^*  \circ V(m_1^*)= \varphi^{'*}$. However, this formula 
 may not define any $\varphi^{'*}$, as it defines a section of $\mHom(V(\alpha_3), 
 V(\alpha_4)) + \mathcal O_X$ away from $Q$ that may have a pole of $Q$.
 
 If we express $\varphi^*$ in our trivialization over the punctured formal neighborhood 
 as $(M,s)$, then by assumption $(M,s)$ are the projective coordinates of an 
 $R\bseries{t}$-point of $\overline{G}$ and are congruent to $(e,0)$ mod $\mathfrak m$.  
 
  If we view $m_2^*$ over the punctured formal neighborhood of $Q$ as an element of 
	$G(R\pseries{t})$, by assumption on $m_2$, it can be expressed as $g_b^{-1} \mu(t) 
	g_a^{-1}$ 
	for $g_a,g_b \in G(R\bseries{t})$ and it is congruent to $\mu(t)$ modulo $\mathfrak m$.
  
	Then the possible values of $(\alpha_1^* ,m_1^*)$ are parameterized by those elements 
	of $G(R\pseries{t})$ that are of the form $g_c \mu(t) g_d$ for $g_c, g_d \in G(R\bseries{t})$, that are congruent to 
	$\mu(t)$ modulo $\mathfrak m$, and such that $g_a \mu(t)^{-1} g_b M  g_c \mu(t) g_d$ 
	is integral, up to the right action of elements of $G(R\bseries{t})$ that are 
	congruent to $1$ modulo $\mathfrak m$. By Lemma~\ref{smooth-semisimple-algebraic}, 
	there is a unique such element up to equivalence.  
 \end{proof}
 We can now prove the desired smoothness statement: 
  \begin{lemma}\label{smooth-semisimple} Both $pr_{12}$ and $pr_{34}$ are smooth at $y$.  
  \end{lemma}

 \begin{proof}We can factor $pr_{12}$ as the composition of first, the map $p'$ that projects 
 onto a point $(\alpha_1,\alpha_2,t_1,t_2,\varphi)$ of 
 $\overline{\Hecke}_{G(D),H,\weights,V}$ 
 with a $G$-bundle $\alpha_4$ and isomorphism $m_2 : \alpha_4 \to \alpha_2$ such that 
 $m_2$ near $Q$ is in the cell of the affine Grassmannian corresponding to $\mu$, with, 
 second,
 the map that forgets $\alpha_4$ and $m_2$. This second map is a locally trivial 
 fibration by the cell of the affine Grassmannian associated to $\mu$ and hence is 
 smooth.
 
 Thus it is sufficient to show that the first projection $p'$ is \'{e}tale at $y$. To do this 
 we 
 may ignore the trivializations $t_3, t_4$ as these are uniquely determined by the other 
 data. The projection $p'$ is then defined by adding $\alpha_3, m_1, \varphi'$. Then $p'$ is 
 schematic of finite type, since the data of the pair $(\alpha_3,m_1)$ is equivalent to a section of a 
 locally trivial fibration by 
 the cell of the affine Grassmannian associated to $\mu$, and then $\varphi'$ is a 
 section of a projective bundle satisfying a closed condition, so $p'$ is 
 represented by a closed subset of a projective bundle on a fibration by a variety.  To 
 check that $p'$ is \'{e}tale at the point $y$, we use the fact that each $R$-point of the 
 base for a Henselian local ring $R$ congruent mod $\mathfrak m$ to the image of $y$ has 
 a unique lift to an $R$-point of the total space congruent mod $\mathfrak m$ to $y$, 
 which is Lemma~\ref{smooth-semisimple-local}. This implies that there is a section of $p'$ over 
 the \'{e}tale local ring at the $p'(y)$, and that this section is equal over the \'{e}tale 
 local ring at $y$ to the identity, which implies the natural map from the \'{e}tale local ring at $p'(y)$ to the \'{e}tale local ring at $y$ is 
 an isomorphism and so the map is \'{e}tale.
   
 Finally, we can deduce the $pr_{34}$ case from the $pr_{12}$ case by symmetry, taking 
 the dual of $V$ and so reversing all the arrows. Note that the assumption on the weights 
 of $V$ is preserved by duality.  
  \end{proof}
  
  Using smoothness, we can prove an isomorphism of stalks, which will be a key ingredient in our induction argument:
 
 \begin{lemma}\label{travelling-semisimple} The stalks of $pr_{12}^*  j_*  
 (IC_{\Hecke_{G(D),\weights} } \boxtimes \mathcal L )$ and $pr_{34}^* j_*  
 (IC_{\Hecke_{G(D),\weights} } \boxtimes \mathcal L )$ at $y$ are isomorphic. \end{lemma}
 
 \begin{proof} By Lemma~\ref{jopen-semisimple}, the image of $j$ inside 
 $\overline{\Hecke}_{G(D),H,\weights,V}$ consists of those 
 $(\alpha_1,\alpha_2,t_1,t_2,\varphi)$ where the last coordinate of $\varphi$ is nonzero. 
 For a point of   $\Hecke_{Q,\mu} \left( \overline{\Hecke}_{G(D),H,\weights,V} \right) 
 $,  the equation $\varphi \circ V(m_1)= V(m_2) \circ \varphi'$ ensures that the last coordinate of $\varphi$ is nonzero if and only if the last coordinate of $\varphi'$ is nonzero. Let  
 $\Hecke_{Q,\mu} \left(\Hecke_{G(D),\weights} \times H \right) $ be the open subset where the last coordinates of $\varphi$ and $\varphi'$ are nonzero, $j'$ its inclusion into $\Hecke_{Q,\mu} \left( 
 \overline{\Hecke}_{G(D),H,\weights,V} \right) $, and $pr_{12}'$ and $pr_{34}'$ the 
 projections onto $\Hecke_{G(D),\weights} \times H$. This gives a commutative diagram:
 \[ \begin{tikzcd}
 \overline{\Hecke}_{G(D),H,\weights,V} & \Hecke_{Q,\mu} \left( 
 \overline{\Hecke}_{G(D),H,\weights,V} \right) \arrow[l, "pr_{12}"] \arrow[r,"pr_{34}"] & 
 \overline{\Hecke}_{G(D),H,\weights,V} \\
  \Hecke_{G(D),\weights} \times H \arrow[u,"j"] &\Hecke_{Q,\mu} 
  \left(\Hecke_{G(D),\weights} \times H \right)  \arrow[l, "pr_{12}'"] 
  \arrow[r,"pr_{34}'"]  \arrow[u,"j'"]& \Hecke_{G(D),\weights} \times H  \arrow[u,"j"]\\
  \end{tikzcd}
  \]

  To show the isomorphism, observe that in a neighborhood of $y$, $pr_{12}^*  j_*  
  (IC_{\Hecke_{G(D),\weights} } \boxtimes \mathcal L ) = j'_* pr_{12}^{'*}( 
  IC_{\Hecke_{G(D),\weights} } \boxtimes \mathcal L )$ by smooth base change and Lemma 
 ~\ref{smooth-semisimple}. So it suffices to show that $pr_{12}^{'*}( 
  IC_{\Hecke_{G(D),\weights} } \boxtimes \mathcal L )= pr_{34}^{'*}( 
  IC_{\Hecke_{G(D),\weights} } \boxtimes \mathcal L )$.  Let $p_c : \Hecke_{G(D),\weights} \times H \to \Hecke_{G(D),\weights}$ and $p_h: \Hecke_{G(D),\weights \times H} \to H$ be the projections.  We have $IC_{\Hecke_{G(D),\weights} } \boxtimes \mathcal L = p_c^* IC_{\Hecke_{G(D),\weights} } \otimes  p_h^* \mathcal L$ so \[ pr_{12}^{'*}( 
  IC_{\Hecke_{G(D),\weights} } \boxtimes \mathcal L ) = pr_{12}^{'*} p_c^* IC_{\Hecke_{G(D),\weights} } \otimes pr_{12}^{*'} p_h^* \mathcal L,\] and similarly for $pr_{34}$. Hence it suffices to show that   \[ pr_{12}^{'*} p_c^* IC_{\Hecke_{G(D),\weights} } = pr_{34}^{'*} p_c^* IC_{\Hecke_{G(D),\weights} } \]  and \[ pr_{12}^{'*} p_h^* \mathcal L = pr_{34}^{'*} p_h^* \mathcal L.\]

The map $pr_{12'}$ is smooth by Lemma~\ref{smooth-semisimple}, and $p_c$ is smooth because $H$ is. Thus, $ pr_{12}^{'*} p_c^* IC_{\Hecke_{G(D),\weights} } $ is simply a shift of $IC_{\Hecke_{Q,\mu} 
  \left(\Hecke_{G(D),\weights} \times H \right) }$. The same argument works for for $pr_{34}$, which gives the first desired identity.
  
  The second desired identity follows from $p_h \circ pr_{12}'  = p_h \circ pr_{34}'$, which can be expressed also as the commutativity of the extended diagram \[ \begin{tikzcd}
 \overline{\Hecke}_{G(D),H,\weights,V} & \Hecke_{Q,\mu} \left( 
 \overline{\Hecke}_{G(D),H,\weights,V} \right) \arrow[l, "pr_{12}"] \arrow[r,"pr_{34}"] & 
 \overline{\Hecke}_{G(D),H,\weights,V} \\
  \Hecke_{G(D),\weights} \times H \arrow[u,"j"] \arrow[dr, "p_h"] &\Hecke_{Q,\mu} 
  \left(\Hecke_{G(D),\weights} \times H \right)  \arrow[l, "pr_{12}'"] 
  \arrow[r,"pr_{34}'"]  \arrow[u,"j'"]& \Hecke_{G(D),\weights} \times H  \arrow[u,"j"] 
  \arrow[dl, "p_h"]\\
&  H & \\
  \end{tikzcd}
  \]
  
  If $(\alpha_1,\alpha_2, t_1,t_2,\varphi)$ is in the image under $j$ of some point $((\alpha_1,t_1,\alpha_2,f) , h) \in \Hecke_{G(D),\weights \times H}$, then $\varphi=V(f)$ for some isomorphism $f$ of $G$-bundles $\alpha_1\to\alpha_2$, and $t_2 = h \circ t_1 \circ f^{-1}|_D$, so $h = t_2 \circ f |_D\circ t_1^{-1}$. Similarly if $\varphi' = V(f')$ then we have $h'= t_4 \circ f'|_D \circ t_3^{-1}$.  To check that the diagram commutes, we must check $h=h'$.  Because $V$ is faithful, the  identity $V(m_1) \circ \varphi' = \varphi \circ V(m_2)$ implies $m_2 \circ f' = f \circ 
  m_1$.  Thus we have
  \[ t_2 \circ f|_D \circ t_1^{-1} = t_2 \circ f |_D\circ m_1 |_D\circ t_3^{-1} = t_2 \circ m_2|_D 
  \circ f' |_D\circ t_3^{-1} = t_4 \circ f'|_D \circ t_3^{-1} \]
showing that the diagram commutes and completing the proof.  
 \end{proof}
 
 The final ingredient in our induction is a lemma that checks that the height grows:
 
 \begin{lemma}\label{heightgrowth-semisimple} For $y=( 
 (\alpha_1,\alpha_2,t_1,t_2,\varphi),(\alpha_3,\alpha_4,t_3,t_4,\varphi'),m_1,m_2)$ 
 defined as before, the height of $(\alpha_3,\alpha_4,t_1,t_2,\varphi')$ is strictly 
 greater than the height of $(\alpha_1,\alpha_2,t_1,t_2,\varphi)$. \end{lemma}
 
 \begin{proof}  Consider the natural isomorphism $\cN_{\alpha_1,\varphi} \to 
 \cN_{\alpha_3,\varphi'}$ away from $Q$ that is induced by the isomorphism $m_1$.   This 
 isomorphism respects the canonical filtration of $N$ by vector spaces. Hence 
 it defines an isomorphism from the associated graded vector bundles of 
 $\cN_{\alpha_1,\varphi}$ to the associated graded vector bundles of 
 $\cN_{\alpha_3,\varphi'}$. We will show that each map of vector bundles appearing this 
 way extends to a map of vector bundles over all of $X$ that vanishes over the fiber of 
 $Q$.  
 
 To do this, it is sufficient to calculate in a neighborhood of $Q$. Over that 
 neighborhood, we can assume that $\varphi$ and $\varphi'$ are both simply the map $e$, 
 so that $\cN_{\alpha_1,\varphi}$ and $\cN_{\alpha_3,\varphi'}$ are each $N$, and the 
 induced 
 map is the homomorphism $g \to m_1^{-1} \circ g \circ m_1 = \mu(t)^{-1} g \mu(t)$. So it 
 is sufficient to show that the eigenvalues of $\mu(t)$ acting by conjugation on the 
 associated graded module of the canonical filtration of $N$ are all positive powers of $t$. Because the associated graded is 
 also the associated graded of the Lie algebra of a filtration on the Lie algebra of $N$, it is sufficient to show that all 
 the eigenvalues of $\mu(t)$ on the Lie algebra of $N$ are positive powers of $t$. To do 
 this, observe that for any root in the Lie algebra of $N$, its dual root is not in the 
 Lie algebra of $P$, so the eigenvalue of $\mu(t)$ on it is a negative power of $t$.

Given a map $V_1\to V_2$ that is an isomorphism away from a point $Q$ and vanishes $Q$ , any line bundle $L$ that appears as a 
quotient of $V_2$  admits a nontrivial map from $V_1$ which vanishes at a point, and so $L_1(-Q)$ admits a nontrivial map from $V_1$, and thus some line bundle which maps to $L_1(-Q)$ and thus has degree $< \deg L_1$ must appear as a quotient of $V_1$. It follows that the height of $(\alpha_1,\alpha_2,t_1,t_2,\varphi)$ 
is less than the height of $(\alpha_3,\alpha_4,t_1,t_2,\varphi')$. 
    \end{proof}
        
\subsection{Conclusion}\label{s:7-conclusion}

    \begin{theorem} \label{mainextension-semisimple} Assume that $V$ lifts to the Witt vectors of $k$ and that the pairing of any weight of $V$ with any coroot of $G$ is less than $p$.
    
    Assume that $(G,m_u, H_u, 
    \mathcal L_u)$ is geometrically supercuspidal for some $u \in D$ and $\charr(k)>2$. Then the natural 
    map \[ j_!  (IC_{\Hecke_{G(D),\weights} } \boxtimes \mathcal L )\to j_*  
    (IC_{\Hecke_{G(D),\weights} } \boxtimes \mathcal L ) \] is an isomorphism.
    \end{theorem}

\begin{proof}

We check the isomorphism on stalks at each point.
By Lemma~\ref{jopen-semisimple}, $j$ is an open immersion, and thus the isomorphism holds 
for points in the image of $j$. At points outside the image of $j$, it is sufficient to 
prove that the stalk of $j_*  (IC_{\Hecke_{G(D),\weights} } \boxtimes \mathcal L ) $ 
vanishes. We do this by induction on the height. The base case when the height is greater 
than $2g-2 + |D|$ is handled by Lemma~\ref{htvan-semisimple}.

For the induction step, we assume it is true for height $>h$ and prove it for height $h$. 
Given a point  $(\alpha_1,\alpha_2,t_1,t_2,\varphi)$ of height $h$, we have defined a 
point $y$ of $ \Hecke_{Q,\mu} \left( \overline{\Hecke}_{G(D),H,\weights,V} \right)$. By 
Lemma~\ref{travelling-semisimple}, the stalk at $pr_{12}(y)$ is equal to the stalk at 
$pr_{34}(y)$. By Lemma~\ref{heightgrowth-semisimple}, the height of $p_{34}(y)$ is 
greater than $h$, so by our induction hypothesis the stalk vanishes, and then the stalk at 
$(\alpha_1,\alpha_2,t_1,t_2,\varphi)$  vanishes, completing the induction step.
\end{proof}

\begin{theorem}\label{mainduality-semisimple} Assume that $(G,m_u, H_u, \mathcal L_u)$ is 
geometrically supercuspidal for some $u \in D$ and $\charr(k)>2$. Then the natural map \[ 
\Delta^{\weights}_! \left( IC_{\Hecke_{G(D),\weights} }  \boxtimes \mathcal L \right) \to 
\Delta^{\weights}_* \left( IC_{\Hecke_{G(D),\weights} }  \boxtimes \mathcal L \right) \] 
is an isomorphism. \end{theorem}

\begin{proof}  By Lemma~\ref{V-existence}, there exists a representation $V$ satisfying the condition of Theorem \ref{mainextension-semisimple}.

We have observed that $\overline{\Delta}^\weights \circ j = \Delta^{\weights}$ and that 
$\overline{\Delta}^\weights$ is proper. We thus have
\begin{align*}
 \Delta^{\weights}_! \left( IC_{\Hecke_{G(D),\weights} }  \boxtimes \mathcal L \right) = 
 \overline{\Delta}^{\weights}_* j_! \left( IC_{\Hecke_{G(D),\weights} }  \boxtimes 
 \mathcal L \right) 
&= \overline{\Delta}^{\weights}_* j_* \left( IC_{\Hecke_{G(D),\weights} }  \boxtimes 
\mathcal L \right) \\
&= \Delta^{\weights}_* \left( IC_{\Hecke_{G(D),\weights} }  \boxtimes \mathcal L \right).
\qedhere
\end{align*}
 \end{proof}

In fact, this result also holds in characteristic $2$ if $G$ has no nontrivial normal subgroup with trivial center by a similar proof, using the second part of Lemma~\ref{V-existence}.

\section{Properties of the Hecke complex}\label{s:properties}
Let $X$ be a smooth projective curve over $k$, $G$ a split semisimple group over $k$,  $D$ an 
effective divisor on $X$, $H$ a smooth connected factorizable subgroup of $G \lWR 
\mathcal O_D \rWR$, and $\mathcal L$ a character sheaf on $H$.

For $\weights: |X| \to \Lambda^+$ a function with finite support, supported away from $D$, let 
\[K_{\weights} := \Delta^{W}_! \left( IC_{\Hecke_{G(D),\weights}}   
\boxtimes \mathcal L \right) [ \dim H].\]
\index{$K_W$, Hecke complex}

We will use Theorem~\ref{mainduality-semisimple}, and other tools, to show important 
properties of $K_{\weights}$. In \S\ref{sub:purity-perversity-ss} we will show 
it is a pure perverse sheaf. In \S\ref{sub:vanishing-ss} we will describe its 
support. In \S\ref{sub:trace-fn-ss} we will calculate its trace function. 

\subsection{Purity and Perversity}\label{sub:purity-perversity-ss}

\begin{notation} Let $d(\weights) := \sum_{x \in |X|}2 (\deg x) \langle \weights_x, \rho 
\rangle $ where $\rho$ is half the sum of the positive roots of the maximal torus of $G$. 
\end{notation}
\index{$d(W)$, total sum of degrees}

\begin{lemma}\label{dimension-count} 

\begin{enumerate}[(i)]

\item The dimension of $\Bun_{G(D)}$ is $( \dim G) (g+|D|-1)$.

\item The dimension of $\Hecke_{G(D),\weights}$ is $(\dim G) (g+|D|-1) + d(\weights)$

\end{enumerate}

\end{lemma}

\begin{proof} \begin{enumerate}[(i)]

\item $\Bun_{G(D)}$ is a $G\lWR \mathcal O_D \rWR$-torsor on $\Bun_{G}$. The dimension of 
$\Bun_{G}$ is $( \dim G) (g-1)$ and the dimension of $G  \lWR \mathcal O_D \rWR$ is $( 
\dim G) |D|$.

\item $\Hecke_{G(D),\weights}$ is a fiber bundle over $\Bun_{G(D)}$ in the \'{e}tale topology. The fiber over each 
point of $\Bun_{G(D)}$ is equal to the product over $x$ in the support of $\weights$ of the Weil 
restriction from 
$\kappa_x$ to $k$ of the closure of the Schubert cell of the affine Grassmannian associate 
with 
$\weights_x$. The dimension of this fiber is the sum over $x$ of $\deg x$ times 
the dimension of this cell. The dimension of the cell is $2 \langle \weights_x, \rho 
\rangle$ so the sum is $d(\weights)$. \qedhere
\end{enumerate}
\end{proof}

We refer the reader to \cite{bbd,KiehlWeissauer} for the foundations of the theory of perverse sheaves in characteristic $p$ and \cite{LO09} for the generalization to stacks. 

\begin{lemma}\label{purity-semisimple}  Assume that $(G, m_u, H_u, \mathcal L_u)$ is 
geometrically supercuspidal for some $u\in D$ and $\charr(k)>2$. Then the complex 
$K_{\weights}$ is perverse, pure of weight $(\dim G)(g+|D|-1) + d(\weights) + \dim H $, 
and 
geometrically semisimple. \end{lemma}

\begin{proof} The intersection cohomology complex $ IC_{\Hecke_{G(D),\weights}}$ is defined as 
the intermediate extension of $\mathbb Q_\ell [  \dim {\Hecke_{G(D),\weights}}] $ from the smooth 
locus of $\Hecke_{G(D),\weights}$ to the whole space, and thus is perverse by 
\cite[Thm.4.3(ii)]{bbd}. Because $\mathbb Q_\ell [ \dim IC_{\Hecke_{G(D),\weights}}]$ is pure of weight 
$ \dim {\Hecke_{G(D),\weights}}$ on the smooth locus \cite[\S5.1.8]{bbd}, and the intermediate 
extension preserves purity \cite[Cor.5.3.2]{bbd},
$IC_{\Hecke_{G(D),\weights}}$ is pure of weight $\dim {\Hecke_{G(D),\weights}}= ( \dim G) (g+|D|-1) + 
d(\weights)$ (by Lemma~\ref{dimension-count}).

Because $\mathcal L$ is lisse on a smooth variety of dimension $\dim H$, 
$\mathcal L [\dim H ]$ is perverse. By Lemma~\ref{schematic-affine}, $\Delta^{\weights}$ 
is schematic and affine. Thus by Artin's theorem, $\Delta^{\weights}_* \left( 
IC_{\Hecke_{G(D),\weights}} \boxtimes \mathcal L [\dim H] \right)$ is semiperverse \cite[Thm.4.1.1]{bbd} and 
$K_{\weights} = \Delta^{\weights}_! \left( IC_{\Hecke_{G(D),\weights}} \boxtimes \mathcal 
L [\dim H] \right)$ is cosemiperverse \cite[Cor.4.1.2]{bbd}. Because they are equal by Theorem 
\ref{mainduality-semisimple}, they are each perverse. (We can apply these results for schemes because perversity is a smooth-local condition, so we may check it locally, and Artin stacks are smooth-locally modeled by schemes.)  

By Lemma~\ref{character-sheaf-uniqueness}, $\mathcal L$ has arithmetic monodromy of finite order, so every Frobenius eigenvalue of $\mathcal L$ has finite order, and hence has absolute value $1$, so $\mathcal L$ is pure of weight $0$. Thus its shift $\mathcal L [\dim H]$ is pure of weight $\dim H$, so the exterior 
product $IC_{\Hecke_{G(D),\weights}} \boxtimes \mathcal L [\dim H]$ is pure of weight 
$(\dim G)  (g+|D|-1) + d(\weights) + \dim H $. Hence by Deligne's theorem (which we may 
apply because $\Delta^{\weights}$ is schematic), $K_{\weights} =\Delta^{\weights}_! 
\left( IC_{\Hecke_{G(D),\weights}} \boxtimes \mathcal L [\dim H] \right)$ is mixed of 
weight $\leq(\dim G)  (g+|D|-1) + d(\weights) + \dim H$ and  $\Delta^{\weights}_* \left( 
IC_{\Hecke_{G(D),\weights}} \boxtimes \mathcal L [\dim H] \right)$ is mixed of weight 
$\geq (\dim G)  (g+|D|-1) + d(\weights) + \dim H$ \cite[Stabilities 5.1.14(i,i*)]{bbd}. Because they are equal by Theorem 
\ref{mainduality-semisimple}, they are each pure of weight $(\dim G)  (g+|D|-1) + 
d(\weights) + \dim H$.

The geometric semisimplicity of a pure perverse sheaf on an Artin stack with affine stabilizers follows from 
\cite[Thm.1.2]{Sun12D}
\end{proof}

\begin{lemma}\label{complex-dual}   Assume that $(G, m_u, H_u, \mathcal L_u)$ is 
geometrically supercuspidal for some $u\in D$ and $\charr(k)>2$. Then the Verdier dual of 
$K_{\weights} $ is the analogue of $K_{\weights}$ defined with the dual character sheaf 
$\mathcal L^\vee$, twisted by $\Ql ( (\dim G) (g+|D|-1) + d(\weights) + \dim H ) $.

\end{lemma} \begin{proof}We have \[ DK _{\weights} = D \Delta^{W}_! \left( IC_{\Hecke_{G(D),\weights}}   
\boxtimes \mathcal L  [ \dim H]\right)=\Delta^{W}_* D \left( IC_{\Hecke_{G(D),\weights}}   
\boxtimes \mathcal L [ \dim H] \right)  \] \[= \Delta^{W}_*  \left( DIC_{\Hecke_{G(D),\weights}}   
\boxtimes D (\mathcal L[ \dim H] ) \right)  .\]

Now $D( \mathcal L [ \dim H] )=  \mathcal L^{\vee} (\dim H) [\dim H]$  and $D IC_{\Hecke_{G(D),\weights}} = IC_{\Hecke_{G(D),\weights}} ( \dim \Hecke_{G(D),\weights} ) = IC_{\Hecke_{G(D),\weights}}(  \dim G (g+|D|-1) + d(\weights) )$ so \[ DK_{\weights} =\Delta^{W}_*  \left( IC_{\Hecke_{G(D),\weights}}   
\boxtimes  \mathcal L \right) ( (\dim G) (g+|D|-1) + d(\weights) + \dim H ) [ \dim H]  \] \[=\Delta^{W}_!  \left( 
IC_{\Hecke_{G(D),\weights}}   
\boxtimes  \mathcal L \right) ( (\dim G) (g+|D|-1) + d(\weights) + \dim H )) [ \dim H].
\qedhere\]   
\end{proof}

\subsection{Vanishing Properties}\label{sub:vanishing-ss}

The following definition is one way of generalizing to the ramified case the very 
unstable bundles of Frenkel--Gaitsgory--Vilonen \cite[\S3.2]{FGV02}. 

\begin{defi} Let $P$ be a parabolic subgroup of $G$ with maximal unipotent subgroup $N$. 
To a $P$-bundle on $X$, we attach a form of $N$ twisted by the conjugation action of $P$ 
on $N$, which admits a natural filtration into vector bundles (see 
Lemma~\ref{N-filtration}). We say that a \emph{$P$-bundle is 
very unstable} if none of these vector bundles admit a nontrivial map to $K_X(D)$.  
We say that a \emph{$G$-bundle is very unstable} if it admits a reduction to a very 
unstable 
$P$-bundle for some maximal parabolic subgroup $P$ of $G$. 
\end{defi}
\index{very unstable $G$-bundle}

This  definition makes sense for $G$-bundles on $X$ defined over any field, and in 
particular an algebraically closed field. The utility of this definition is that it allows us to prove that the stalk of $K_{\weights}$ vanishes:

\begin{lemma}\label{very-unstable-semisimple} Assume that $(G, m_u, H_u, \mathcal L_u)$ 
is geometrically supercuspidal for some $u\in D$ and $\charr(k)>2$. Then the stalk of 
$K_{\weights}$ at a geometric point $((\alpha_1,t_1),(\alpha_2,t_2))$ of $\Bun_{G(D)} 
\times \Bun_{G(D)}$ vanishes if $V_1$ or $V_2$ is very unstable, as does the stalk 
of its Verdier dual. \end{lemma}

\begin{proof} By Lemma~\ref{complex-dual}, and because geometric supercuspidality is preserved by duality, we can reduce 
to 
the case of $K_{\weights}$. By switching $\alpha_1$ and $\alpha_2$ and replacing 
$\weights$ by the conjugate of $-\weights$ under the longest element of the Weyl group, 
we can reduce to the case where $\alpha_1$ is very unstable.

By proper base change, the stalk of $K_{\weights}$ at $( (\alpha_1,t_1),(\alpha_2,t_2) )$ is the cohomology with 
compact supports of the fiber of $\Delta^\weights$ over $(\alpha_1,t_1),(\alpha_2,t_2)$ 
with coefficients in $IC_{\Hecke{G(D)}} \boxtimes \mathcal L$.  This fiber consists 
of isomorphisms $\varphi: \alpha_1 \to \alpha_2$ away from the support of $\weights$, 
satisfying local conditions at points in the support of $\weights$, such that $t_2 \circ 
\varphi|_D \circ t_1 \in H$. 

Let $\beta$ be a reduction of $\alpha_1$ to a very unstable $P$-bundle. By Lemma 
\ref{section-semisimple}, there is a section over $\Res_k^D (\cN_\beta|D) \times X$ of 
$\cN_\beta$, and therefore a section $s$ of the automorphism group of $\alpha_1$, that 
restricted to $D$ is the canonical section. Let $S$ be the subgroup of $\sigma \in 
\Res_k^D (\cN_\beta|D) \times X$ such that $ t_1^{-1} \circ \sigma \circ t_1 \in H$.  
Then 
$S$ acts on this fiber by sending $\varphi$ to $\varphi \circ s(\sigma)$, which satisfies 
\[t_2^{-1} \circ \varphi|_D \circ s(\sigma)|_D \circ t_1 = t_2^{-1} \circ \varphi|_D 
\circ \sigma  \circ t_1 \in H\] by assumption. This action preserves 
$IC_{{\Hecke{G(D)}}}$, because it is canonical, but acts on $\mathcal L$ by tensoring 
with $\mathcal L( t_1^{-1} \sigma t_1)$. Hence the action of the automorphism on the 
cohomology is by tensoring with $\mathcal L( t_1^{-1} \sigma t_1)$, which is nontrivial 
by the geometrically supercuspidal assumption, so the cohomology is equal to itself 
tensored with a nontrivial local system, hence the cohomology vanishes, as 
desired.\end{proof}

Next, we will describe an explicit open set of $\BunGD$ whose complement consists entirely of very unstable $G$-bundles. We will be able to restrict attention to this open subset, which has many useful properties (most crucially, it is quasicompact), for most calculations.

First, it is necessary to prove a version of the main theorem of reduction theory that is 
uniform in $q$. In the work~\cite{FGV02}, the role of this lemma is played by some 
calculations with the Harder--Narasimhan filtration.
See also~\cite{Harder:Chevalley-fn-fields}.
Recall that $G$ is split.
\begin{lemma}\label{siegel-set}  Every $G$-bundle on $X$ admits a reduction to a $B$-bundle 
whose induced $T$-bundle, composed with the character associated to any simple positive 
root to produce a line bundle, has degree $\geq -2g$. \end{lemma}

We use the convention that in the $SL_2$ triple where the upper-right nilpotent is the 
given positive root, the associated cocharacter is $t \mapsto \begin{pmatrix} t & 0 \\ 0 
& t^{-1} \end{pmatrix}$.

\begin{proof} Fix a $G$-bundle $\beta$. First we check $\beta$ admits a reduction to a $B$-bundle. To prove this, 
note that $\beta$ admits a trivialization over the generic point, hence a $B$-reduction over 
the generic point, which extends to the whole curve because the associated $G/B$-bundle 
is proper.

Next we define a height on the set of $B$-reductions of $\beta$. Observe that the associated 
$G/B$-bundle (i.e., $\beta$ modulo the right action of $B$) is a projective scheme $Y$ 
over $X$. Given a character $\chi_0$ of $T$, which induces a character of $B$, we can form the associated 
line bundle $L_{\chi_0}$ on $Y$ by composing the universal $B$-bundle with the inverse character $B \to \mathbb G_m$. Fix a character $\chi_0$ of $T$ that is in the interior of 
the Weyl chamber of $B$, so that it is positive on all the positive coroots. Then the 
associated line bundle $L_{\chi_0}$ is ample. (We use the inverse character so that dominant weights will correspond to ample line bundles.) 

Any $B$-reduction, consisting of a $B$-bundle  $\alpha \subseteq \beta$, defines a section $s: X\to Y$.  The Weil height of $s$ according to $L_{\chi_0}$ is defined to be the degree of $s^* L_{\chi_0}$. This is manifestly an integer and is bounded below.  Hence it takes a minimum 
value. For $s$ the section associated to a $B$-bundle $\alpha$,  $s^* L_{\chi_0}$ is the inverse of the composition of $\alpha$ with $\chi_0$ , so the height of $s$ is 
minus the degree of $\chi_0(\alpha)$. 
 Choose a $B$-reduction $\alpha_1$ whose height attains the minimum value. We will show that the composition of $\alpha_1$ with every simple root character has degree $\geq -2g$, giving the desired conclusion.

Fix a simple root. Let $\chi$ be the associated character of $B$ and let $P$ be the 
associated parabolic. Then the quotient of the Levi subgroup of $P$ by its center is a 
split adjoint-form group of rank one, hence is isomorphic to $\PGL_2$. We have a 
commutative diagram with Cartesian 
square.

\[ \begin{tikzcd}
\mathbb G_m & B(\PGL_2)\arrow[r] \arrow[l,swap,"\lambda_1/\lambda_2"]& \PGL_2 \\
& B \arrow[u] \arrow[ul,"\chi"] \arrow[r] & P\arrow[u]\\
\end{tikzcd}\]

By functoriality,  $\alpha_1$  defines a $P$-bundle $P(\alpha_1)$ and hence a $\PGL_2$-bundle 
$\PGL_2(\alpha_1)$, which we can view as 
a rank two vector bundle $V$ on $X$, up to a twist by a line bundle. After twisting, we 
may assume that $V$ has degree $2g-1$ or $2g$. By Riemann--Roch, $H^0(X,V)$ has 
dimension 
$\geq (2g-1) + 2-2g =1$, so it has a global section, and hence $V$ can be written as the 
extension by a line bundle $L_1$ of degree $\geq 0$ of another line bundle $L_2$, which 
necessarily has degree $\leq 2g$.  This gives a reduction $E$ of $\PGL_2(\alpha_1) $ to 
$B(\PGL_2)$. Let $\alpha_2$ be the fiber product \[ P(\alpha_1) \times_{ \PGL_2(\alpha_1) } E.\] Then 
because \[ B = P \times_{\PGL_2} B(\PGL_2),\] $\alpha_2$ is a $B$-bundle. Furthermore, $\alpha_2$ 
agrees with $\alpha_1$ when projected to $P$, and hence $\alpha_2$ is another 
$B$-reduction of $\beta$. Finally, $\alpha_2$ agrees with $E$ when projected to $B(\PGL_2)$.

We can express $\chi_0$ as a sum of some 
character that factors through $P$ with a positive multiple of $\chi$. This is because 
the characters that factor through $P$ form a wall of the Weyl chamber, to which $\chi$ 
is perpendicular, and pointing towards the interior of the Weyl cone.  Observe that the 
degree of $\chi(\alpha_2)$ is equal to the degree of $L_1$ minus the degree of $L_2$, 
which is at least $-2g$ by construction. So if $\chi(\alpha_1)<-2g$, then 
$\chi(\alpha_2)> \chi(\alpha_1)$ and thus $\chi_0(\alpha_2) > \chi_0(\alpha_1)$, which contradicts the assumption that the height  $- \chi_0(\alpha)$ is 
minimized by $\alpha_1$.\end{proof}

We are now ready to define our key open subset $U$:
Let $V$ be a faithful representation of $G$. Let $r$ be the maximum number of simple 
 roots that can be added to form a positive root of $G$ and 
let $k$ be the maximum $\ell^1$-norm of any weight of $V$, measured in a basis of simple 
roots of $G$.  Let $\epsilon$ be $1$ if $r=1$ and $D$ is empty and $0$ otherwise. Let $L$ 
be a line bundle on $X$ of degree at least $k( 2rg+\deg D + \epsilon)+2g-1 $.
\begin{defi} \label{U-qc-condition}
Let $U$ consist of $(\alpha,t) \in \Bun_{G(D)}$ such that $H^1( X , V(\alpha) \otimes L (-Q))$ 
vanishes for each point $Q$ in $X$. \end{defi} 
\index{$U$, open quasicompact subset of $\Bun_{G(D)}$}

\begin{lemma}\label{quasicompact-semisimple} \begin{enumerate}[(i)]

\item $U$ is an open subset of $\Bun_{G(D)}$.

\item $U$ is quasicompact.

\item $U$ is the quotient of a smooth scheme of finite type by a reductive algebraic 
group of finite type. 

\item Every vector bundle in the complement of $U$ inside $\Bun_{G(D)}$ is very 
unstable.

\item The stalk of $K_{\weights}$ vanishes on $\Bun_{G(D)} \times \Bun_{G(D)}$  outside 
$U \times U$.
\end{enumerate} \end{lemma}

\begin{proof}

To prove assertion (i), observe that $U$ is the complement of the projection from 
$\Bun_{G(D)} \times X$ to $\Bun_{G(D)}$ 
of the locus where $H^1( X , V(\alpha) \otimes L (-Q)) \neq 0$. By the semicontinuity 
theorem, this locus is closed, and $X$ is proper, hence universally closed, so the 
projection is closed as well.

Assertion (ii) follows from assertion (iii). To prove assertion (iii), observe that a 
$G$-bundle $\alpha$ satisfies the condition from Definition \ref{U-qc-condition} if and only if 
$V(\alpha) \otimes L$ is 
globally generated and satisfies $H^1(X, V(\alpha) \otimes L)=0$. In this case, $H^0(X, 
V(\alpha) \otimes L)$ is a $(\dim V) ( \deg L+ 1-g)$-dimensional vector space. 

Thus, let $\mathcal M_3$ be the moduli space of triples of a $G$-bundle $\alpha $ satisfying the 
condition of Definition \ref{U-qc-condition}, a trivialization of $\alpha$ over $D$, and a basis 
for $H^0(X, V(\alpha) \otimes L)$. Then $U$ is the quotient of $\mathcal M_3$ by $\GL_{(\dim V) 
(\deg L +1-g)}$. (In particular, $\mathcal M_3$ is a $\GL_{(\dim V) 
(\deg L +1-g)}$-torsor over $U$, hence $\mathcal M_3$ is smooth.) Thus, to prove (iii), it suffices to check that $\mathcal M_3$ is a scheme of finite type.

Let $\mathcal M_2$ be the moduli space of pairs of a $G$-bundle $\alpha$ satisfying the condition of Definition \ref{U-qc-condition} and a basis for $H^0(X, V(\alpha) \otimes L)$. Then $\mathcal M_3$ is a $G \lWR \mathcal O_D\rWR$-torsor over $\mathcal M_2$, so it suffices to show that $\mathcal M_2$ is a scheme of finite type.

Given a point of $\mathcal M_2$, and in particular a basis for  $H^0(X, 
V(\alpha) \otimes L)$, we obtain a map from $X$ to the 
Grassmannian $\operatorname{Gr}$ of rank $\dim V$ quotients of a fixed $(\dim V) ( \deg L+ 1-g)$-dimensional 
vector space, where the map has degree $(\dim V) ( \deg L)$. Let $\mathcal M_1$ be the moduli 
space of maps $f$ from $X$ to $\operatorname{Gr}$ with degree $(\dim V) (\deg L)$. Let 
${V}_{\textrm{taut}}$ be the tautological bundle on $\operatorname{Gr}$.  Then $\mathcal M_1$ is a 
scheme of finite type, and $\mathcal M_2$ maps to $\mathcal M_1$. The image of this map is 
contained in the open subset $\mathcal M_1'$  of $\mathcal M_1$ where  $H^1 ( X, f^* 
V_{\textrm{taut}} )=0$ and the natural map  $H^0(\operatorname{Gr},V_{\textrm{taut}})  \to H^0(X, 
f^* V_{\textrm{taut}})$ is an isomorphism. The fiber of the map $\mathcal M_2 \to \mathcal M_1'$ 
parameterizes reductions of the structure group of $f^* V_{\textrm{taut}} \otimes L^{-1}$ to 
$G$.  Thus $\mathcal M_2 \to \mathcal M_1$ is a schematic morphism of finite type 
\cite[Cor.3.2.4]{WangModuli}, and so $\mathcal M_2$, and finally $\mathcal M_3$, are schematic of 
finite type. 

To prove assertion (iv), let $\alpha$ be a $G$-bundle outside $U$. Then for some point $Q$, we 
have $H^1(X, V(\alpha) \otimes L (-Q))\neq 0$. Hence by Serre duality we have $H^0 ( X, K_X 
\otimes V(\alpha)^\vee \otimes  L^\vee (Q)) \neq 0$, so 
$V(\alpha)$ admits a nontrivial map to the line bundle $K_X \otimes L^\vee (Q)$, which has degree 
at most $-k(2rg+\deg D + \epsilon)$.  Let $\gamma_1,\dots, \gamma_n$ be the simple roots of $B$. Choose a $B$-reduction of $\alpha$ as in Lemma 
\ref{siegel-set}, and let $\beta$ be the induced $T$-bundle, where $T$ is the maximal torus of $T$. Using Lemma~\ref{siegel-set}, we have chosen $\beta$ so that  \begin{equation}\label{gamma-lower-bound} \deg (\gamma_1(\beta)), \dots, \deg(\gamma_n(\beta)) \geq -2g.\end{equation} 

As a representation of $B$, $V$ admits a filtration by one-dimensional 
characters. The filtration of $V(\alpha)$ induced by this $B$-reduction is a filtration by line bundles, each arising by $\beta$ from a one-dimensional character of $T$. Because $V(\alpha)$ admits a nontrivial map to a line bundle of degree $\leq -k(2rg+\deg D + \epsilon)$, at least one of these line bundles has degree  $\leq- k(2rg+\deg D+ \epsilon)$. The degree of the line bundle associated to a character of $T$ is a linear form $\omega$  on the weight space. If we had \[ | \deg (\gamma_i(\beta)) | < 2 rg+ \deg D + \epsilon\]  for each root $\gamma_i$, then the linear form $\omega$ would have absolute value $< 2 rg+ \deg D + \epsilon$ on each basis vector, hence have absolute value $< k (2 rg+ \deg D + \epsilon)$ on each vector with $\ell^1$ norm at most $k$, so by the definition of $k$ have absolute value $< k (2rg+ \deg D + \epsilon)$ on each weight of $T$, giving a contradiction. Thus, for some $i$, we must have \[ |\deg(\gamma_i(\beta))| \geq 2 rg+ \deg D + \epsilon > 2g.\]  Combined with \eqref{gamma-lower-bound}, this implies that \[ \deg(\gamma_i(\beta) ) > 2 rg+ \deg D + \epsilon.\]

 Let $P$ be the parabolic subgroup defined by the set of all the 
 roots other than $\gamma_i$. Let $N$ be the unipotent radical of $P$. Then $N$ is an iterated extension, as an 
algebraic group, of one-dimensional representations of $B$, each a 
character of $B$ corresponding to a positive root in the unipotent radical of $P$ and thus to the sum of at most $r$ positive roots, at least one of which is $\gamma_i$. Because each of the other roots has degree $\geq -2g$ and $\gamma_i$ has degree 
$\geq 2rg + \deg D$, the product has degree at least $2g + \deg D$ and so does not admit 
a nontrivial map to $K_X(D)$. Hence none of the $N_i$'s do either, and the bundle is very 
unstable.

Assertion (v) follows from assertion (iv) and Lemma~\ref{very-unstable-semisimple}.\qedhere
\end{proof}

\section{The trace function of the Hecke complex}\label{s:trace-function}

We maintain the assumptions and notation of Section~\ref{s:properties}.

\subsection{Calculation of the trace function}\label{sub:trace-fn-ss}

To describe the trace function of $K_\weights$ explicitly, we will first give an explicit 
description of the points of $\BunGD(\mathbb F_q)$, that the trace function is a function on, in 
terms of adelic double cosets. This is a variant of the classical Weil parameterization. 
It will be helpful for later to give an adelic description of the automorphisms of a point 
of $\BunGD$, which we do in Lemma \ref{weil-automorphism-semisimple}. The trace function 
of $K_\weights$ can be calculated as a sum. We will describe the set to be summed over in 
Lemma \ref{adelization-semisimple}, and define the function to be summed in Definition 
\ref{def:test-function}, culminating in a description of the trace function in Lemma 
\ref{trace-function-semisimple}.

Recall some of our earlier notation: $\KK = \prod_{x \in |X - D| } G( 
\ko_x)  \times \prod_{x \in D} U_{m_x} (G(\ko_x))$, where $\ko_x= 
\kappa_x\bseries{t}$ and
$U_{m_x} ( G(\ko_x) )$  is the subgroup of $ G(\ko_x) $ 
consisting of elements congruent to $1$ modulo $t^{m_x}$.
\begin{lemma}\label{weil-parameterization-semisimple}   There is a bijection between 
 $ G(F) \backslash G (\mathbb A_F) / \KK $ and $\Bun_{G(D)}(k)$.

Moreover, this bijection arises from a bijection between $G(\mathbb A_F)$ and the set of 
tuples $(\alpha,z_\eta,(z_x)_{x\in |X|})$ of a $G$-bundle $\alpha$ and a 
trivialization 
$z_\eta: \alpha |_\eta \isom G_\eta$ of $\alpha$ over the generic point and a 
trivialization 
$z_x: \alpha |_{ \ko_x} \isom G_{\ko_x}$ for each closed 
point 
$x\in |X|$.
Explicitly, the bijection sends $(\alpha,z_\eta,(z_x)_{x\in |X|})$ to the tuple \[( 
z_\eta |_{\kappa_x\pseries{t}} \circ z_x^{-1}|_{\kappa_x\pseries{t}})_{x \in |X|} \in 
\sideset{}{'}\prod_{x \in |X|} G (\kappa_x\pseries{t}) = G (\mathbb A_F)\] of transition 
maps defined over the punctured formal neighborhood of $x$. 
Forgetting $z_\eta$ corresponds to quotienting out by $G(F)$ on the left, and keeping 
from $(z_x)_{x\in |X|}$ only the trivialization $z_x$ modulo $t^{m_x}$ for $x\in D$ 
corresponds to 
quotienting by $\KK$ on the right.  Here, the trivialization $z_x$ modulo $t^{m_x}$ for $x\in D$ matches the trivialization of $\alpha$ over $D$ that comes with a point of $\BunGD(k)$.
\end{lemma}

\begin{proof} This is the standard definition of the Weil parameterization. By Lemma~\ref{trivialization-generic}, for any $G$-bundle there in fact exists a trivialization over the generic point, and because there are no nontrivial torsors of connected algebraic groups over finite fields, there exists a trivialization over a formal neighborhood of every closed point.

One then checks that this map sends the set of all possible trivializations to a double coset in $ G(F) \backslash G (\mathbb A_F) / \KK $  and 
that each double coset arises from a unique isomorphism class of $G$-bundles.
\end{proof} 

Recall that $J_x$ is the inverse image of $H_x(\kappa_{x})$ under the map 
$G(\ko_x)\twoheadrightarrow G(\kappa_x)$.

\begin{defi} For $g \in G(\mathbb A_F)$, let $\Aut_D (g) $ be the subgroup of $\gamma \in 
G(F)$ such that $g^{-1} \gamma g \in \KK$. Let $\Aut_{D,H}(g)$ be the subgroup of 
$\gamma 
\in G(F)$ such that 
\[
g^{-1} \gamma g \in  \prod_{x \in |X - D|} G( \ko_x)  \times 
\prod_{x \in D} J_x.
\]
We have that $\Aut_D (g)$ is a normal subgroup of $\Aut_{D,H}(g)$.
\end{defi} 
\index{$\Aut_D (g)$, $\Aut_{D,H}(g)$, automorphism groups}

There is an action of $H(k)$ on $\Bun_{G(D)}(k)$ where $h \in H(k)$ acts by fixing the $G$-bundle $\alpha$ and composing the trivialization $t_D$ of $\alpha$ over $D$ with $H$. 

For the action of $H(k)$ on $\Bun_{G(D)}(k)$, we say that the stabilizer in $H(k)$ of a point $ (\alpha, t_D)$ consists of all elements $H(k)$ that send $(\alpha,t_D)$ to a point isomorphic to $(\alpha, t_D)$. Equivalently, this is the stabilizer of the isomorphism class of $(\alpha,t_D)$ for the induced action of $H(k)$ on the set of isomorphism classes. The analogous definition works for any group action of a groupoid.

\begin{lemma}\label{weil-automorphism-semisimple} Let $g$ be an element of $G(\mathbb 
A_F)$, and $(\alpha,t)$ be the point of $\Bun_{G(D)}(k)$ corresponding to the double 
coset of $g$. Then 

\begin{enumerate}[(i)]

\item The automorphism group of $(\alpha,t_D)$ is $\Aut_D(g)$.

\item Under the identification 
$H(k) = \prod_{x\in D} H_x(\kappa_x) = \prod_{x\in D} J_x  / \prod_{x \in D} U_{m_x} 
(G(\ko_x))$,
the action of $H(k)$ on $\Bun_{G(D)}(k)$ is intertwined with 
the action of $\prod_{x\in D} J_x$ by right 
multiplication on $G(F) \backslash G(\mathbb A_F) / \KK$. 

\item The stabilizer in $H(k)$ of a point $(\alpha,t_D)$ is 
$\Aut_{D,H}(g)/\Aut_D(g)$.
\end{enumerate}

 \end{lemma}

\begin{proof}\begin{enumerate}[(i)]

\item Any automorphism of $(\alpha,t_D)$, when restricted to the generic point by the 
trivialization $t_\eta$, defines an element $\gamma \in G(F)$. Conversely, any element 
$\gamma\in G(F)$ defines an automorphism of $\alpha$ over the generic point. The 
condition that the 
automorphism extends to a place $x$ is precisely the condition that $g_x^{-1} \gamma 
g_x$ is in $G(\ko_x)$. For $x \in D$, the condition that the automorphism commute with 
the trivialization $t$ is the condition that $g_x^{-1} \gamma g_x \in U_{m_x} 
(G(\ko_x))$.

\item For $h \in H(k)$ write $h=(h_x)_{x\in D}$ under the identification $H(k)=\prod_{x\in D} 
H_x(\kappa_x)$.  The double coset corresponding to a $G$-bundle with a trivialization over $D$ arises, by Lemma~\ref{weil-parameterization-semisimple}, from all choices of a trivialization $z_\eta$ over the generic point and $z_x$ in a formal neighborhood over each point $x$, such that for $x \in D$, $z_x$ is congruent mod $t^{m_x}$ to the trivialization over $D$. Thus, the action of $h$ on the trivialization over $D$ is equivalent to composing $z_x$ with an element of $J_x$ in the inverse image of $h_x$. This is equivalent to multiplying $g_x$ by an element of $J_x$ in the inverse image of $h_x$.

\item $\Aut_D(g)$ is the kernel of the natural map from $\Aut_{D,H}(g)$ to $H(k)$ given 
by projection $\gamma \mapsto g_x^{-1} \gamma g_x$ from $J_x$ to $H_x(\kappa_x)$. The elements in the 
image are 
exactly those elements of $H(k)$ that can be lifted to elements  in $\prod_{x\in D} 
J_x$ 
whose action by right multiplication fixes the double coset of $ g$, i.e., the stabilizer 
in $H(k)$ of $(\alpha,t_D)$.\qedhere
\end{enumerate}
\end{proof}

From now on, let $k = \mathbb F_q$. We need a lemma 
about the compatibility of the geometric and classical Satake isomorphisms, which is 
well-known. This is implicit in the 1982 combinatorial formulas of Lusztig and Kato, 
whose relationship to the IC sheaf is the generalization to the affine Grassmannians of 
the calculations by Kazhdan--Lusztig of the trace of Frobenius on the IC-sheaves of 
the closure of Schubert cells 
in a complete flag variety.
 Our proof is an elaboration of a sketch by Richarz and Zhu~\cite[p.449]{RicharzZhu15}, 
 and we provide some details since we were not able to find a more 
detailed exposition in the literature.
\begin{lemma}\label{satake-compatibilty}  Let $\lambda \in \Lambda^+$ be a coweight of 
$G$. Let $IC_\lambda$ be the IC-sheaf of the closure  of the cell of the affine 
Grassmannian $\Gr_G=G\pseries{t} / G\bseries{t}$ associated to $\lambda$. The trace of 
Frobenius 
on the stalk $IC_{\lambda, x}$ of $IC_{\lambda}$ at a point $x \in \Gr_G(\F_q) = G ( 
\mathbb 
F_q\pseries{t})/ G(\mathbb F_q\bseries{t})$ is equal to the 
value at $G(\mathbb F_q\bseries{t}) xG (\mathbb F_q \bseries{t})$ of the  
function $a_\lambda\in \mathcal H(G)$ associated to the representation of $\widehat{G}$ with 
highest 
weight $\lambda$ 
by 
the Satake isomorphism, times $q^{\langle \lambda,\rho \rangle}$. 
\end{lemma}

\begin{proof} Consider the function $f_\lambda \colon \Gr_G(\F_q) \to \overline{\mathbb Q}_\ell \cong \mathbb C$ 
defined by the stalks of $IC_\lambda$ times $q^{-\langle \lambda, \rho \rangle}$, i.e., 
the stalks of the twist $IC_{\lambda} ( \langle \lambda,\rho\rangle)$.  Because the 
Schubert cell is 
left $G(\mathbb F_q\bseries{t})$-invariant, $IC_{\lambda}$ is left $G(\mathbb 
F_q\bseries{t})$-invariant, and so $f_\lambda$ descends to a function on $G(\mathbb 
F_q\bseries{t}) \backslash 
G ( 
\mathbb 
F_q\pseries{t})/ G(\mathbb F_q\bseries{t})$. Because the Satake 
transform is an isomorphism, in order to verify that it coincides with $a_\lambda$, it 
suffices to check that the Satake transform of  $f_\lambda$ is the character of the 
representation of $\widehat{G}$ with highest weight $\lambda$. 

For $\mu:\mathbb G_m \to T $ a cocharacter, let $[\mu]\in X^*(\widehat T)$ be the 
associated character of the dual torus. Then by definition, the Satake transform of 
$f_\lambda$  is given by 
\[ \sum_{\mu: \mathbb G_m \to T} [\mu] \cdot  q^{ - \langle \mu, \rho \rangle} \int_{ h\in N(\F_q\pseries{t})} 
f_\lambda(h \mu(t)) dh  \] where $N$ is 
the unipotent radical of a Borel, and we take the Haar measure $dh$ on $N(\mathbb 
F_q\pseries{t})$  where $N(\mathbb 
F_q\bseries{t})$ has measure one. For \[g \in N(\mathbb F_q\bseries{t})  \mu(t) G( \mathbb F_q \bseries{t}) 
,\] 
consider the total measure assigned 
 by $dh$ to 
\[   g 
 G (\mathbb F_q\bseries{t}) \mu(t)^{-1}  \cap N (\F_q \pseries{t} ).\] Because the Haar measure $dh$ is left $N(\F_q\bseries{t})$-invariant, this equals the measure of \[ \mu(t) G(\mathbb F_q\bseries{t}) \mu(t)^{-1}   \cap N(\F_q \pseries{t}) = \mu(t)  \left(  N (\mathbb F_q\pseries{t})  \cap G( 
\mathbb 
F_q\bseries{t}) \right) \mu(t)^{-1} ,\] which by definition is the index of  $N(\mathbb F_q\pseries{t})  \cap G( \mathbb 
F_q\bseries{t})$ inside $\mu(t)  \left(  N (\mathbb F_q\pseries{t})  \cap G( 
\mathbb F_q\bseries{t}) \right) \mu(t)^{-1}$. By viewing $N$ as an iterated extension of root groups, and observing that the action of $\mu(t)$ on the root group associated to $\alpha$ is scaling by $q^{ \langle \mu,\alpha\rangle}$, we can see that this index is $q^{  2 \langle \mu, \rho \rangle}$, where $\rho$ as usual is half the sum of the positive roots.

On the other hand, if \[  g\notin N(\mathbb F_q\bseries{t})  \mu(t) G( \mathbb F_q \bseries{t}) \] then the total measure assigned to $g 
 G (\mathbb F_q\bseries{t}) $ by $dh$ is zero.

Thus,  \[ \int_{ h\in N(\F_q\pseries{t})} f_\lambda(h\mu(t)) dh = \sum_{ g \in N(\F_q\pseries{t}) \mu(t) G 
(\mathbb 
F_q\bseries{t})/G(\mathbb 
F_q\bseries{t})} f_\lambda(g)   q^{ 2 \langle \mu, \rho \rangle} \] so the Satake transform of $f_\lambda$ 
is \[ \sum_{\mu: \mathbb G_m \to T} [\mu] \cdot  q^{   \langle \mu, \rho \rangle}  \sum_{ g \in N(\F_q\pseries{t}) 
\mu(t) G (\mathbb 
F_q\bseries{t})/G(\mathbb 
F_q\bseries{t})} f_\lambda(g)  .\] 

The subset $N(\F_q\bseries{t}) \mu(t) G (\mathbb F_q\bseries{t})/G(\mathbb 
F_q\bseries{t}) \subseteq \Gr_G(\F_q)$ is the set of $\mathbb F_q$-points of the locally 
closed 
subscheme $S_{\mu}$ of the affine Grassmannian defined by Mirkovi\'{c}-Vilonen 
\cite[\S5.3.5]{ZhuIntro}, see also~\cite[\S3.2]{Baumann-Riche}. Hence the sum of the trace 
function $f_\lambda$ 
 of $IC_{\lambda}( \langle 
\lambda,\rho\rangle) $ over this set is the trace of Frobenius on $H^* _c ( S_{\mu,\overline{\mathbb F}_q}, IC_{\lambda}( \langle 
\lambda,\rho\rangle)) $. By \cite[Thm.5.3.9(2)]{ZhuIntro} 
and~\cite[Prop.10.1]{Baumann-Riche}, all eigenvalues of 
Frobenius on this cohomology group are equal to $q^{ -\langle \mu,\rho \rangle}$ and 
occur in degree $\langle 2 \rho, \mu \rangle$, so the trace of Frobenius is $q^{ -\langle 
\mu,\rho \rangle} \dim H^{\langle 2\rho, \mu \rangle}  _c ( S_{\mu,\overline{\mathbb F}_q}, IC_{\lambda}( \langle 
\lambda,\rho\rangle)) $. Thus, the Satake transform of $f_\lambda$  is 
\begin{equation}\label{satake-character-sum}\sum_{\mu: \mathbb G_m \to T} [\mu] \cdot \dim H^{\langle 
2\rho, \mu \rangle}  _c ( S_{\mu,\overline{\mathbb F}_q}, IC_{\lambda}( \langle 
\lambda,\rho\rangle)).\end{equation} By 
\cite[Thm.5.3.9(3) and Lem.5.3.17]{ZhuIntro}, this cohomology group is isomorphic to the 
$\widehat{T}$-eigenspace with character $[\mu]$ in the representation of $\widehat{G}$ with 
highest weight $\lambda$. This means the multiplicity of $[\mu]$ in the sum 
\eqref{satake-character-sum} is the multiplicity of $[\mu]$ in the representation $V_\lambda$ of 
$\widehat{G}$ 
with highest weight $\lambda$, so \eqref{satake-character-sum} is the character $\tr(V_\lambda)$ 
of that 
representation, as desired.
\end{proof}

To state and prove a bijection between isomorphisms $\varphi$ satisfying a list of conditions and $\gamma \in G(F)$ satisfying a different list of conditions, it is helpful to name these conditions. These will be used only in the following Lemma \ref{adelization-semisimple}.

Let $g_1, g_2$ be two elements of $G (\mathbb 
A_F)$, and let $(\alpha_1,t_1)$, $(\alpha_2,t_2)$ be the corresponding points of $\Bun_{G(D)} 
(k)$.

We say an isomorphism $\varphi:\alpha_1\to \alpha_2$ away from the support of $\weights$ satisfies condition (C-$\varphi$) if $\varphi$, expressed 
as an element of $G\pseries{t}$ by local coordinates near each point $x$ in the support of 
$\weights$, is in the closed cell of the affine Grassmannian associated to $\weights_x$, and if $t_2 \circ \varphi|_D \circ t_1^{-1}$ is contained in $H$.

In other words, $\varphi$ satisfies (C-$\varphi$) if and only if $( (\alpha_1,t_1),(\alpha_2,t_2),\varphi)$ is a point of $ 
\Hecke_{G(D),\weights} \times H$.

We say $\gamma \in G(F)$ satisfies condition (C-$\gamma$) if $g_2^{-1} \gamma g_1$ is in $G(\ko_x)$ at all 
points outside the support of $\weights$ and the support of $D$, is in 
the closure of the cell of the Bruhat decomposition of $G(F_x)$ associated to 
$\weights_x$ for each point $x$ in the support of $\weights$, and lies in 
$J_x$ for each point $x\in D$.

\begin{lemma}\label{adelization-semisimple} 
Let $g_1, g_2$ be two elements of $G (\mathbb 
A_F)$, and let $(\alpha_1,t_1)$, $(\alpha_2,t_2)$ be the corresponding points of $\Bun_{G(D)} 
(k)$.

There is a bijection between the set of isomorphisms  $\varphi: \alpha_1 \to \alpha_2 $ away 
from the support of $\weights$ satisfying the above condition (C-$\varphi$) and $\gamma\in G(F)$ 
satisfying the above condition (C-$\gamma$) such that, if $\varphi$ and $\gamma$ correspond under 
this bijection, we have the two identities:

\begin{enumerate}

\item  $t_2 \circ \varphi|_D \circ t_1^{-1} \in H(k) $ equals the product over $x \in D$ of the projection of 
the local component of $ g_2 \gamma g_1^{-1}$ from $J_x$ to $H_x(\kappa_{x})$.

\item The trace function at $IC_{  \Hecke_{G(D),\weights} }$ at  $(  \alpha_1, \alpha_2, \varphi, t_1)\in  \Hecke_{G(D),\weights} $ equals  $\prod_{x \in \weights} f^W_x ( g_2 \gamma g_1^{-1})$, where $f^W_x$ is the function on $G ( F_x)$ associated by the Satake isomorphism to the character of the representation of $\widehat{G}$ whose highest weight corresponds to $\weights_x$. 

\end{enumerate}

  \end{lemma}
  For interpreting the identities (1) and (2) above, it is helpful to note that the projection $\Hecke_{G(D),\weights} \times H\to H$ sends $((\alpha_1, t_1), (\alpha_2, t_2), \varphi  ) $ to $ t_2 \circ \varphi|_D \circ t_1^{-1}$ and the projection $\Hecke_{G(D),\weights} \times H \to \Hecke_{G(D),\weights} $ sends $((\alpha_1, t_1), (\alpha_2, t_2), \varphi  ) $ to $(\alpha_1,\alpha_2, \varphi,t_1)$.
  
  \begin{proof}  
  
  Let $t_{\eta,1}, t_{x,1}$, $t_{\eta,2}$, $t_{x,2}$ be the trivializations of $\alpha_1$ 
  and $\alpha_2$ at the generic point and in formal neigborhoods respectively.  Then because $t_{\eta,1}$ and $t_{\eta,2}$ are isomorphisms, there is 
  a bijection between isomorphisms $\varphi_\eta: \alpha_1\to\alpha_2$ over the generic points 
  and the elements $ t_{\eta,2} \circ \varphi_\eta \circ t_{\eta,1}^{-1}$ of $G(F)$. Let 
  $\gamma = t_{\eta,2} \circ \varphi_\eta \circ t_{\eta,1}^{-1}$.
  
  We define our bijection to send $\varphi$ to $\gamma$. The inverse map defines $\varphi_\eta$ over the generic point as $t_{\eta,2}^{-1} \circ \gamma \circ t_{\eta,1}$ and then extends $\varphi_\eta$ uniquely to an isomorphism $\varphi$ away from the support of $\weights$.
  
  To show this gives a bijection, it suffices to  check that the extension $\varphi$ of $\varphi_\eta$ exists and satisfies condition (C-$\varphi$) if and only if $\gamma$ satisfies condition (C-$\gamma$).
  To check this, first note that, restricted to the 
  punctured formal neighborhood of $x$,\[t_{x,2} \circ \varphi \circ t_{x,1}^{-1} = 
  t_{x,2} \circ t_{\eta,2}^{-1} \circ \gamma \circ t_{\eta,1} \circ t_{x,1}^{-1}  =  
  g_{2,x}^{-1} \gamma g_{1,x}\] is the local component  of $ g_2 \gamma g_1^{-1}$ at $x$.
  
  Now we check that the restriction (C-$\varphi$) places on $\varphi$ at a point $x$ is equivalent to a corresponding restriction (C-$\gamma$) places on the local component of $g_2 \gamma g_1^{-1}$ at the same point $x$: 
  
  \begin{itemize}
  
  \item For $x$ not in the support of $\weights$, the condition that $\varphi_\eta$ extends to an an isomorphism in a neighborhood of $x$ is equivalent to the condition that  $  g_2^{-1} \gamma g_1$ lies  in $G(\ko_x)$.  (If $x \in D$, this is implied by the stronger condition that $  g_2^{-1} \gamma g_1$ lies in $J_x$). 

	\item  Let $x$ be in the support of $\weights$. The condition that, expressed in local 
	coordinates at $x$, $\varphi$ is in the closure of the cell in the affine Grassmannian 
	associated to $\weights_x$ is equivalent to the condition that $ g_2^{-1}  \gamma g_1$ lies in the 
	closure of the cell of the Bruhat decomposition of $G (F_x)$ associated to 
	$\weights_x$. 
	
	\item The fact that $t_2 \circ \varphi|_D \circ t_1^{-1}$ lies in $H$  is 
	equivalent to the condition that $  g_2^{-1}  \gamma g_1$ is in $H$ modulo $D$, or 
	equivalently modulo $t^{m_x}$ for each $x$ in $D$, which is precisely the definition 
	of $J_x$. \end{itemize}
	
	Combining these equivalences at all points $x$, we see that (C-$\gamma$) is equivalent to (C-$\varphi$), together with the claim that the extension $\varphi$ of $\varphi_\eta$ exists, and so the map that sends $\varphi$ to $\gamma$ is a bijection.

The identity (1) follows from the fact that, for $x\in D$, $ g_{2,x}^{-1} \gamma g_{1,x}= t_{x,2} \circ \varphi \circ t_{x,1}^{-1} $ and thus is congruent to $t_2 \circ \varphi \circ t_1^{-1}$ modulo $t^{m_x}$. 

The identity (2) follows from Lemma~\ref{satake-compatibilty}.\end{proof}

\index{$f^W_x$, test functions}    
\begin{defi}\label{def:test-function} For $x$ a closed point of $X$, let $f_x^{\weights}$ 
on $G(F_x) $ equal:
\begin{itemize}
\item If $x$ is not contained in $D$ or the support of $\weights$, the characteristic 
 function of $G (\ko_x )$. 
 
\item If $x$ is contained in the support of $\weights$, the function associated by the 
 Satake isomorphism to the character of the representation of $\widehat{G}$ whose highest 
 weight corresponds to $\weights_x$, times $q^{ \deg x  \langle\weights_x , \rho\rangle}$.
 
\item If $x$ is contained in $D$, the function that vanishes outside of $J_x$ and is 
equal to $\chi_x$ on $J_x$. 
\end{itemize}
\end{defi}

 \begin{lemma}\label{trace-function-semisimple}  Let $g_1, g_2$ be two elements of $G 
 (\mathbb A_F)$. Let $(\alpha_1,t_1)$ and $(\alpha_2,t_2)$ be the points of $\Bun_{G(D)} (k)$ 
 corresponding to the double cosets of $g_1$ and $g_2$ respectively. Then the trace of 
 $\Frob_{k}$ on the stalk of $K_{\weights}$ at $((\alpha_1,t_1),(\alpha_2,t_2))$ is 
 
 \[  \sum_{\gamma \in G(F)} \prod_{x \in |X|}  f_x^{\weights}( g_2^{-1} \gamma g_1)   . \]
 
 \end{lemma}

 \begin{proof} By the Lefschetz formula, the trace is the sum of the trace function of 
 $IC_{\Hecke_{G(D), \weights} }  \boxtimes \mathcal L$ over 
 $(\Delta^{\weights})^{-1} ((\alpha_1,t_1),(\alpha_2,t_2))$, where $(\Delta_W)^{-1}$ denotes the inverse image. (The fiber  $(\Delta^{\weights})^{-1} ((\alpha_1,t_1),(\alpha_2,t_2))$ is an affine 
 scheme of finite type, so we do not need to apply the Lefschetz formula for stacks.)
  
By Definition~\ref{d:Hecke-G(D)},  $(\Delta^{\weights})^{-1} ((\alpha_1,t_1),(\alpha_2,t_2))$ consists of isomorphisms $\varphi: \alpha_1 \to \alpha_2$ away from the 
support of $\weights$, that expressed as elements of $G\pseries{t}$ by local 
coordinates near each point $x$ in the support of $\weights$ are in the closed cell of 
the affine Grassmannian associated to $\weights_x$, such that $t_2 \circ \varphi|_D \circ 
t_1^{-1}$ is contained in $H$.
 
 By Lemma~\ref{adelization-semisimple}, such maps $\varphi$ are in bijection with 
 $\gamma$ in $G(F)$ such that $g_2^{-1} \gamma g_1$ is in $G (\mathcal O_{F_v})$ at all 
 places outside the support of $\weights$ and the support of $D$, is in the closure of 
 the cell of the Bruhat decomposition of $G(F_x)$ for each place $x$ associated to 
 $\weights_x$ for each point $x$ in the support of $W$, and lies in $J_x$ for each point 
 $x$ of $D$.

Furthermore, the trace function of $IC_{\Hecke_{G(D), \weights}} \boxtimes \mathcal L$ 
is equal to the product of the trace function of $IC_{\Hecke_{G(D), \weights}}$ and 
the trace function of $\mathcal L$. The trace function of $IC_{\Hecke_{G(D), \weights}}$ 
is the product over the places lying in the support of $\weights$ of the function 
associated to the corresponding representation of $\widehat{G}$ in the Satake isomorphism 
times $q^{ \langle W_x, \rho \rangle}$ by 
Lemma~\ref{adelization-semisimple}. The trace function of $\mathcal L$ is a character of 
$H(k)$, which by definition is $\prod_{x \in D} \chi_x$. 

Examining, we see that the trace of the point associated to an element $\gamma$ is 
precisely $ \prod_{x \in |X|}  f_x^{\weights}( g_2^{-1} \gamma g_1)$. Summing over 
$\gamma$, we obtain the stated sum. \end{proof}

\begin{defi}\label{def:global-function} For $g_1,g_2 \in G(\mathbb A_F)$, let \[\FrK_{\weights} (g_1,g_2 ) =   
\sum_{\gamma \in G(F)} \prod_{x \in |X|}  f_x^{\weights}( g_2^{-1} \gamma g_1)     \] be 
the trace function of $K_{\weights}$. \end{defi}
\index{$\FrK_\weights$, automorphic kernel, trace function of $K_W$}

\subsection{Cohomological interpretation of the trace}

We can interpret the inner product of two functions $\FrK_{\weights_1}$, $\FrK_{\weights_2}$ cohomologically. Using this cohomological interpretation, we will get a very strong bound, in Theorem \ref{geometric-average-bound}. We will later express this inner product as the trace of a Hecke operator on our space of automorphic forms (in Proposition \ref{p:spectral-trace}), and therefore obtain bounds for traces of Hecke operators.

\begin{lemma}\label{arithmetic-to-geometry}  Assume that $p>2$ and some $(G, m_u, H_u, 
\mathcal L_u)$ is geometrically supercuspidal. Then we have 

\[ \sum_{g_1, g_2 \in G(F) \backslash G (\mathbb A_F) / \KK}  \frac{  \overline{ 
\FrK_{\weights_1} (g_1,g_2)} \FrK_{\weights_2}( g_1,g_2) }{ |\Aut_D (g_1) | | 
\Aut_D(g_2)|}\]

\[=  {q}^{ (\dim G)  (g+|D|-1) + d(\weights_1) + \dim H }  \sum_{i \in \mathbb Z}   
(-1)^i \tr(\Frob_{q} ,  H^i_c ( U_{\overline{k}} \times U_{\overline{k}},  
DK_{\weights_1} \otimes K_{\weights_2} ) ) \]
where the sum on the left is finitely supported and the sum on the right is absolutely 
convergent.

\end{lemma}

\begin{proof} By  \cite[Thm.4.2(i)]{Sun12L}, $ \sum_{i \in \mathbb Z}   
(-1)^i \tr(\Frob_{q} ,  H^i_c ( U_{\overline{k}} \times U_{\overline{k}},  
DK_{\weights_1} \otimes K_{\weights_2} )) $ is absolutely convergent and 
 the Lefschetz formula for algebraic stacks \cite[Thm.4.2(ii)]{Sun12L} reads
\[  \sum_{i \in \mathbb Z}   
(-1)^i \tr(\Frob_{q} ,  H^i_c ( U_{\overline{k}} \times U_{\overline{k}},  
DK_{\weights_1} \otimes K_{\weights_2} )  = \sum_{ ( (\alpha_1,t_1),(\alpha_2,t_2) ) \in U(k) \times U(k) }  \frac{ \tr ( \operatorname{Frob}_q, ( DK_{\weights_1} \otimes K_{\weights_2} )_{((\alpha_1,t_1),(\alpha_2,t_2))} }{ |\Aut( (\alpha_1,t_1), (\alpha_2,t_2))| } .\]

By Lemma~\ref{quasicompact-semisimple}, $( DK_{\weights_1} \otimes K_{\weights_2} )$ vanishes 
outside $U(k) \times U(k)$ and so the above is equal to
\[ \sum_{ ( (\alpha_1,t_1),(\alpha_2,t_2) ) \in \BunGD(k) \times \BunGD(k) }  \frac{ \tr ( \operatorname{Frob}_q, 
( DK_{\weights_1} \otimes K_{\weights_2} )_{((\alpha_1,t_1),(\alpha_2,t_2))} }{ |\Aut( (\alpha_1,t_1), 
(\alpha_2,t_2))| }. \] Furthermore, this sum is finitely supported. 

We have \[ \tr ( \operatorname{Frob}_q, ( DK_{\weights_1} \otimes K_{\weights_2} 
)_{((\alpha_1,t_1),(\alpha_2,t_2))} = \tr ( \operatorname{Frob}_q, ( 
DK_{\weights_1})_{((\alpha_1,t_1),(\alpha_2,t_2))}
 ) \tr (\operatorname{Frob}_q, (   K_{\weights_2} )_{((\alpha_1,t_1),(\alpha_2,t_2))} )\] 

\[  = {q}^{- (\dim G) (g+|D|-1) + d(\weights_1) - \dim H} \overline{  \tr ( \operatorname{Frob}_q, ( 
K_{\weights_1})_{((\alpha_1,t_1),(\alpha_2,t_2))} )  }  \tr (\operatorname{Frob}_q,   (K_{\weights_2} 
)_{((\alpha_1,t_1),(\alpha_2,t_2))} )  \] \[ =  {q}^{- (\dim G) (g+|D|-1) + d(\weights_1) - \dim H} 
K_{\weights_1}(g_1,g_2) K_{\weights_2} (g_1,g_2), \] where $g_1$ corresponds to $\alpha_1,t_1$ and 
$g_2 $ corresponds to $\alpha_2,t_2$ under the bijection of Lemma~\ref{adelization-semisimple}. 
The first identitiy is straightforward, the second identity uses an application due to Katz 
of a result of Gabber~\cite[Lem.1.8.1(1)]{mmp} and the fact that $K_{\weights_1}$ is pure and 
perverse of weight $(\dim G)  
(g+|D|-1) + d(\weights_1) + \dim H $, and the third identity uses Lemma~\ref{trace-function-semisimple}. 

Finally, we have $|\Aut( (\alpha_1,t_1), (\alpha_2,t_2))| = |\Aut_D (g_1) | | 
\Aut_D(g_2)|$ by Lemma~\ref{weil-automorphism-semisimple}.  Plugging these in, we get the stated formula. 
\end{proof}

Let $n$ be a natural number. Define $F_n = \mathbb F_{q^n}(X)$ and $X_n = X_{\mathbb 
F_{q^n}}$. We can base change the datum $(G,D, H , \mathcal L,\weights_1,\weights_2)$ from 
$\mathbb F_q$ to $\mathbb F_{q^n}$ in the following way: We pull back $G$ from $\mathbb 
F_q$ to $\mathbb F_{q^n}$, we pull back $D$ from $X$ to $X_n$, we compose $\weights_1$ 
and $\weights_2$ with the projection $|X_n| \to |X|$, and we base change $H$ and 
$\mathcal L$ from $G \lWR \mathcal O_D \rWR$ to $(G \lWR \mathcal O_D \rWR)_{\mathbb 
F_{q^n}}$. Let $f_x^{\weights_i,n}$ be the local factors defined by this new datum and 
$\FrK_{\weights_i,n} (g_1,g_2)= \sum_{\gamma \in G(F_n)} \prod_{x \in |X|}  
f_x^{\weights_i,n}( g_2^{-1} \gamma g_1) $ for $g_1,g_2 \in G(\mathbb A_{F_n})$. Let $\KK_n$ 
be defined also in terms of this base-changed datum.

\begin{theorem}\label{geometric-average-bound} Assume that $p>2$ and some $(G, m_u, H_u, \mathcal L_u)$ is geometrically 
supercuspidal. Then
\[ \sum_{g_1, g_2 \in G(F_n) \backslash G (\mathbb A_{F_n}) / \KK_n}  
\frac{ \overline{ \FrK_{\weights_1,n} (g_1,g_2)} \FrK_{\weights_2,n}( g_1,g_2) }{ { 
|\Aut_D (g_1) | | \Aut_D(g_2)|}}
 = O (  (q^n) ^{  (\dim G)  (g+|D|-1) + \frac{d(\weights_1)}{2}+ \frac{d(\weights_2)}{2} 
 + \dim H}).\]

\end{theorem}

\begin{proof} By Lemma~\ref{purity-semisimple}, $K_{\weights_2}$ is pure of weight $w_2 = 
(\dim G) (g+|D|-1) + d(\weights_2) + \dim H $ and $K_{\weights_1}$ is pure of weight 
$w_1= (\dim G) (g+|D|-1) + d(\weights_1) + \dim H $, so $w_2 - w_1  
=d(\weights_2)-d(\weights_1).$

By Lemma~\ref{cohbound-weak}, taking $j=0$, it follows that 
 \[ \sum_{i =-\infty}^0   (-1)^i \tr(\Frob_{q^n} ,  H^i_c ( U_{\overline{k}} \times 
 U_{\overline{k}},  DK_{\weights_1} \otimes K_{\weights_2} ) ) = O \left( \left( 
 q^n\right)^{\frac{d(\weights_2)- d(\weights_1)}{2}}\right) .\] 

Applying Lemma~\ref{ext-relation}(2), this cohomology group vanishes for $i > 0$, so  \[ 
\sum_{i \in \mathbb Z}   (-1)^i \tr(\Frob_{q^n} ,  H^i_c ( U_{\overline{k}} \times 
U_{\overline{k}},  DK_{\weights_1} \otimes K_{\weights_2} ) ) = O \left( 
\left(q^n\right)^{\frac{d(\weights_2)- d(\weights_1)}{2}}\right) .\] 

Then we apply Lemma~\ref{arithmetic-to-geometry} over $\mathbb F_{q^n}$. It is clear that 
base changing all the data in this way is equivalent to base-changing 
$\Hecke_{G(D),\weights} \times H$ and thus to base-changing $K_{\weights_1}, 
K_{\weights_2}$, so we obtain
\[ \sum_{ g_1, g_2 \in G(F_n) \backslash G (\mathbb A_{F_n}) / \KK_n} \overline{ 
\FrK_{\weights_1,n} (g_1,g_n)} \FrK_{\weights_2,n}( g_1,g_2) \]
\[=  \left(q^n\right)^{ (\dim G)  (g+|D|-1) + d(\weights_1) + \dim H }  \sum_{i \in 
\mathbb Z}   (-1)^i \tr(\Frob_{q^n} ,  H^i_c ( U_{\overline{k}} \times U_{\overline{k}},  
DK_{\weights_1} \otimes K_{\weights_2} ) )\]\[ = O \left( \left(q^n\right)^{ (\dim G)  
(g+|D|-1) + d(\weights_1) + \dim H }\left(q^n\right)^{\frac{d(\weights_2)- 
d(\weights_1)}{2} }\right)=  O \left( \left (q^n\right) ^{  (\dim G)  (g+|D|-1) + 
\frac{d(\weights_1)}{2}+ \frac{d(\weights_2)}{2} + \dim H}\right).
\qedhere
\] 
\end{proof}

\subsection{Integrality and Weil numbers} 

Dimensions of spaces of automorphic forms over function fields can often be expressed naturally as sums of Weil numbers. The same is true for the traces of Hecke operators that we study here. In fact, these Weil numbers are algebraic integers. We prove this indirectly, by first proving, in Lemmas \ref{integrality-start-semisimple}, \ref{integrality-kernel-semisimple}, and \ref{integrality-trace-semisimple}, that the traces themselves are algebraic integers, then proving in Lemma \ref{eigenvalues-Weil-numbers} that the eigenvalues of Frobenius acting on the relevant cohomology groups are Weil numbers, but not necessarily integers, and finally combining these, in Lemma \ref{weil-numbers-semisimple}, to express the trace in terms of integral Weil numbers.

Let $m$ be the order of the arithmetic monodromy group of $\mathcal L$, which is equal to 
the order of the character $\chi$ by Lemma~\ref{character-sheaf-uniqueness}. It is also 
stable under finite field extension by Lemma~\ref{character-sheaf-uniqueness}, as the 
arithmetic and geometric monodromy groups are equal.

\begin{lemma}\label{integrality-start-semisimple} For all $x \in |X|$ and all $\weights: 
|X| \to \Lambda^+$, the function $f_x^{\weights}$ takes values in $\mathbb Z[\mu_m]$. 
\end{lemma}

\begin{proof} If $x$ lies in $D$, this follows from the fact that $\chi$ is an eigenvalue 
of Frobenius on $\mathcal L$ and hence is a root of unity in the monodromy group. If $x$ 
does not lie in $D$ or the support of $\weights$, then $f_x$ takes the values zero and one, 
both integers. If $x$ lies in the support of $W$, then the value is a polynomial in $q$ 
by the Kazhdan--Lusztig purity theorem. \end{proof} 

\begin{lemma}\label{integrality-kernel-semisimple} For all $g_1,g_2 \in G(\mathbb 
A_F)$,   $\FrK_W(g_1,g_2)$ is divisible in $\mathbb Z[\mu_m]$ by $|\Aut_{D,H} (g_1) |$ and 
by $|\Aut_{D,H}(g_2)|$. \end{lemma}

\begin{proof} Let $\gamma'$ be an element of  $\Aut_{D,H}(g_1)$.   Then for all $x \in | 
X - D|$, $ g_1^{-1} \gamma' g_1 \in G( \kappa_x \bseries{t})$ and so 
$f_x^{\weights} (g_2^{-1} \gamma \gamma' g_1) = f_x^{\weights} ( g_2^{-1} \gamma g_1)$.  
For $x \in D$,  $ g_1^{-1} \gamma' g_1 \in J_x$ and so 
\[
f_x^{\weights} (g_2^{-1} \gamma 
\gamma' g_1) = f_x^{\weights} ( g_2^{-1} \gamma g_1) \chi_x(g_1^{-1} \gamma' g_1).
\]
 Hence 
right multiplication by $\gamma'$ multiplies $\prod_{x \in |X|}  f_x^{\weights}( g_2^{-1} 
\gamma g_1)$ by $\prod_{x\in D} \chi_x(g_1^{-1} \gamma' g_1)$. It follows that 
$\FrK_W(g_1,g_2) = \FrK_W(g_1,g_2) \prod_{ x \in D} \chi_x(g_1^{-1} \gamma' g_1)$ and hence $\FrK(g_1,g_2)=0$, 
and we are done, unless $\prod_{ x\in D} \chi_x(g_1^{-1} \gamma' g_1) =1 $. So we may assume that  $\prod_{ x\in D} \chi_x(g_1^{-1} \gamma' g_1) =1 $ for all 
$\gamma' \in \Aut_{D,H}(g_1)$

This implies that $\prod_{x \in |X|}  f_x^{\weights}( g_2^{-1} \gamma g_1)$ is invariant under 
right multiplication of $\gamma$ by elements of $\Aut_{D,H}(g_1)$.  We can write 
$\sum_{\gamma \in G(F)} \prod_{x \in |X|}  f_x^{\weights}( g_2^{-1} \gamma g_1)$ as a sum 
over orbits of this right multiplication action. Because the action is by multiplication in a group, its orbits are cosets of $\Aut_{D,H}(g_1)$, 
and so the size of each orbit is $|\Aut_{D,H}(g_1)|$, and by Lemma 
\ref{integrality-start-semisimple}, the sum over each orbit is an element of $\mathbb 
Z[\mu_m]$ times $|\Aut_{D,H}(g_1)|$, so the final (finite) sum is divisible by 
$|\Aut_{D,H}(g_1)|$. 

A symmetrical argument works for $\Aut_{D,H}(g_2)$, using left multiplication instead.\end{proof}

\begin{lemma}\label{integrality-trace-semisimple} The sum \[ \frac{1}{ | H(k)|^2}  \sum_{g_1, g_2 \in G(F) \backslash G (\mathbb A_F) / \KK }\frac{  \overline{ 
\FrK_{\weights_1} (g_1,g_2)} \FrK_{\weights_2}( g_1,g_2) }{ |\Aut_D (g_1) | | 
\Aut_D(g_2)|}\] is an element of $\mathbb Z[\mu_m]$. \end{lemma}

\begin{proof} Break the sum into a sum over orbits under the action of $H(k) \times H(k) $ on 
$ G(F) \backslash G (\mathbb A_F) / \KK \times G(F) \backslash G (\mathbb A_F) / \KK $ by right multiplication. 
It suffices to show that the sum over each orbit, divided by $| H(k)|^2$, lies 
in $\mathbb Z[\mu_m]$. 

Because this action corresponds to right multiplication by $\prod_{x \in D} J_x$, it 
multiplies $\prod_{x \in |X|} f_x^W (g_1^{-1} \gamma g_2)$ by $\prod_{x \in D} 
\chi_x(h)$, so it multiplies $\FrK_{\weights_i} (g_1,g_2)$ by $\prod_{x \in D} \chi_x 
(h)$, which is a root of unity, so it fixes $\overline{ \FrK_{\weights_1} (g_1,g_2)} 
\FrK_{\weights_2}( g_1,g_2)$.   Hence the sum over each orbit is the size of that orbit 
times $\frac{\overline{ {\FrK}_{{\weights}_1} (g_1,g_2)}  {\FrK}_{{\weights}_2}( g_1,g_2) 
}{ |\Aut_D (g_1) | | \Aut_D(g_2)|}$ for some $g_1,g_2$ in that orbit.   By the orbit-stabilizer theorem and Lemma~\ref{weil-automorphism-semisimple}, the size of the orbit is $\frac{ |H(k)|^2 |\Aut_D (g_1) | | \Aut_D(g_2)| }{|\Aut_{D,H} (g_1) | | \Aut_{D,H}(g_2)|}$.  Hence the sum over the 
orbit, divided by $|H(k)|^2$, is   $\frac{\overline{ \FrK_{\weights_1} (g_1,g_2)} 
\FrK_{\weights_2}( g_1,g_2) }{ |\Aut_{D,H} (g_1) | | \Aut_{D,H}(g_2)|}$, which is an 
algebraic integer by Lemma~\ref{integrality-kernel-semisimple}.  
 \end{proof}
 
 We use the convention (following \cite[Def.10.1]{Sun12L}) that Weil $q$-numbers 
 are algebraic numbers whose absolute values are a power of $\sqrt{q}$ independent of the 
 choice of complex embedding, while Weil $q$-integers are algebraic integers with the 
 same property.
 
 \begin{lemma}\label{eigenvalues-Weil-numbers} All the eigenvalues of $\Frob_q$ on $H^i_c ( U_{\overline{k}} \times 
 U_{\overline{k}},  DK_{\weights_1} \otimes K_{\weights_2} ) $ are Weil $q$-numbers. 
 \end{lemma}
 
 \begin{proof} By Lemma~\ref{complex-dual} have 
\[DK_{\weights_1} =  \Delta^{\weights_1}_! \left( 
IC_{\Hecke_{G(D),\weights_1}}   
\boxtimes \mathcal L^{-1}  \right) [  \dim 
H]((\dim G)(g +|D|-1 ) + d(\weights_1) + \dim H) .\]
Then if we form a Cartesian square
\[ \begin{tikzcd} Y\arrow[r,"p_1"] \arrow[d,"p_2"] &  \Hecke_{G(D),\weights_1}\times H 
\arrow[d,"\Delta^{\weights_1}"] \\
\Hecke_{G(D), \weights_2} \times H \arrow[r, "\Delta^{\weights_2}"] &  U \times U\\ 
\end{tikzcd} \]
By the projection formula, proper base change, and the projection formula again 
 \begin{equation*} \begin{split}  H^i_c ( U_{\overline{k}} \times U_{\overline{k}}& ,  
DK_{\weights_1} \otimes K_{\weights_2} )\\
 = H^{i + \dim H} _c(  \Hecke_{G(D), \weights_2} \times H& ,\Delta^{W_2 *}  DK_{\weights_1} \otimes  ( IC_{\Hecke_{G(D), \weights_2}} \boxtimes 
\mathcal L ) ) \\
= H^{i+2 \dim H} _c(\Hecke_{G(D), \weights_2} \times H& , p_{2!}p_1^* ( IC_{\Hecke_{G(D), \weights_1}} 
\boxtimes \mathcal L^{-1} )   \otimes  ( IC_{\Hecke_{G(D), \weights_2}} \boxtimes 
\mathcal L ))\\
= H^{i + 2\dim H}_c ( Y_{\overline{k}}&, p_1^* ( IC_{\Hecke_{G(D), \weights_1}} 
\boxtimes \mathcal L^{-1} ) \otimes p_2^*  ( IC_{\Hecke_{G(D), \weights_2}} \boxtimes 
\mathcal L ) ).\end{split} \end{equation*} 

We can stratify $\Hecke_{G(D)}^{\weights_i}$ into strata, the inverse images of 
Schubert cells, on which the Kazhdan--Lusztig purity theorem implies that 
$IC_{\Hecke_{G(D), \weights_1}}$ is a shift of a Tate twist of a constant sheaf. By excision, it 
suffices to prove that all eigenvalues of Frobenius on the cohomology of the inverse images of these strata are $q$-Weil numbers. We can remove the $\mathcal 
L$ and $\mathcal L^{-1}$  terms by noting that these are summands of the pushforward of 
the constant sheaf along the Lang isogeny (Lemma~\ref{characters-come-from-Lang}), so the whole cohomology group is a summand of 
the cohomology of the inverse image of one of these strata under the Lang isogeny of $H 
\times H$. Because this is an algebraic stack, it follows from \cite[Lem.10.2]{Sun12L} 
that all eigenvalues of Frobenius on its cohomology are $q$-Weil numbers.
 \end{proof} 
 
 \begin{theorem}\label{weil-numbers-semisimple}  There exists a natural number $N$, 
 $q$-Weil integers $\alpha_1,\dots,\alpha_N$ of weight $\le  2(\dim G)  (g+|D|-1) + 
 d(\weights_1)+ d(\weights_2)- 2\dim H$, and signs 
 $\epsilon_1,\dots,\epsilon_N\in \{\pm 1\}$, such that for all $n$, \[ \frac{1}{ |H ( 
 \mathbb F_{q^n} ) |^2} \sum_{g_1, g_2 \in G(F_n) \backslash G (\mathbb A_{F_n}) 
 / \KK_n} \frac{  \overline{ \FrK_{\weights_1,n} (g_1,g_2)} \FrK_{\weights_2,n}( 
 g_1,g_2)} { { |\Aut_D (g_1) | | \Aut_D(g_2)|} }  = \sum_{i=1}^N \epsilon_i \alpha_i^n.\] 
Furthermore, we may arrange such that 
\begin{itemize}
\item
$\alpha_1,\ldots, \alpha_{\dim \Hom_{\overline{\mathbb 
F}_q}(K_{\weights_1}, K_{\weights_2})}$ are $q ^{  (\dim G)  (g+|D|-1) + d(\weights_1)- 
\dim H}$ times the eigenvalues of $\Frob_q$ on $\Hom_{\overline{\mathbb F}_q} 
(K_{\weights_1},K_{\weights_2} )$, which are of weight $d(W_2)-d(W_1)$, 
\item  
$\epsilon_1,\ldots,\epsilon_{\dim \Hom_{\overline{\mathbb F}_q}(K_{\weights_1}, 
K_{\weights_2})}$ are all equal to $1$,  
\item $\alpha_i$ has weight $< 2(\dim G)  
(g+|D|-1) + 
d(\weights_1)+ d(\weights_2) -2 \dim H$ for $i > \dim 
\Hom_{\overline{\mathbb F}_q}(K_{\weights_1}, K_{\weights_2})$.
\end{itemize}
\end{theorem}

\begin{proof}
Let $S_n$ be the left-hand side of the formula.
We apply Lemma~\ref{arithmetic-to-geometry} over $\mathbb F_{q^n}$. It is clear that 
base-changing all the data in this way is equivalent to base-changing 
$\Hecke_{G(D),\weights} \times H$ and thus to base-changing $K_{\weights_1}, 
K_{\weights_2}$, so we obtain
\[
|H ( 
 \mathbb F_{q^n} ) |^2 \cdot S_n
=  \left(q^n\right)^{ (\dim G)  (g+|D|-1) + d(\weights_1) + \dim H }  \sum_{i \in \mathbb 
Z}   (-1)^i \tr(\Frob_{q^n} ,  H^i_c ( U_{\overline{k}} \times U_{\overline{k}},  
DK_{\weights_1} \otimes K_{\weights_2} ) )\]

By Lemma~\ref{purity-semisimple}, $K_{\weights_2}$ is pure of weight $w_1 = (\dim G) 
(g+|D|-1) + d(\weights_2) + \dim H $ and $K_{\weights_1}$ is pure of weight $w_2= (\dim 
G) (g+|D|-1) + d(\weights_1) + \dim H $, so $w_2 - w_1  =d(\weights_2)-d(\weights_1)$. 
Hence the eigenvalues of $\Frob_q$ on $H^i_c ( U_{\overline{k}} \times U_{\overline{k}},  
DK_{\weights_1} \otimes K_{\weights_2} ) $ are Weil numbers of weight $\leq  
d(\weights_2) - d(\weights_1)+i$. 

By Lemma~\ref{ext-relation}(2), this cohomology group vanishes for $i>0$.  Hence we can 
write $|H (\mathbb F_{q^n} ) |^2 \cdot S_n$ as a convergent signed sum 
of $n$th powers of Weil numbers, with the largest possible weight being 
\begin{align*}
&2(\dim G)  
(g+|D|-1) + 2d(\weights_1) + 2\dim H  + d(\weights_2) - d(\weights_1)\\
&=
2(\dim 
G)  (g+|D|-1) + 2\dim H  + d(\weights_1) +   d(\weights_2),
\end{align*} 
and appearing in 
$H^0$.

Now $|H(\mathbb F_{q^n})|$ is a finite signed sum of $n$th powers of Weil numbers, with 
the largest weight $2\dim H$ appearing with multiplicity $1$ and sign $1$, 
because $H$ is smooth and geometrically connected. Hence  $\frac{1}{|H(\mathbb F_{q^n})|^2}$ is a 
convergent signed sum of $n$th powers of Weil numbers, with the largest weight $-4\dim H$ 
appearing with multiplicity $1$ and sign $1$.

Hence the product $S_n=\frac{1}{|H(\mathbb F_{q^n})|^2} \cdot \left(  |H (\mathbb F_{q^n} ) |^2 \cdot S_n\right) $ is also a convergent signed sum of $n$th powers of Weil 
 numbers.  By Lemma~\ref{cohbound-weak} this convergence is uniform in $n$. Thus the 
 generating function $\sum_{n=1}^{\infty}  u^n S_n$ is a signed sum of terms 
 of the form $\frac{ \alpha_i u}{ 1- \alpha_i u} $, with the $\alpha_i$ Weil numbers. In 
 particular, it is a meromorphic function with poles of order $1$ at inverses of Weil 
 numbers $\alpha_i$ and with residues integer multiples of $1/\alpha_i$.

However, it is also a power series with coefficients in the ring of integers of a number field $\mathbb Q(\mu_m)$. A variant due to 
Dwork of a result of E.~Borel implies that it is a rational function \cite[Thm.3, 
p.645]{Dwork-rationality}, so all but finitely many of the $\alpha_i$ occur with zero 
multiplicity, and we have the stated claim, except with $q$-Weil numbers rather than 
$q$-Weil integers. To check they are algebraic integers, it is sufficient to check that they are
$\ell$-adic integers for each prime $\ell$. The $\ell$-adic radius of convergence of this 
rational function is at least one, because all its coefficients are algebraic integers, 
so all its poles have $\ell$-adic norm at least one, and the $\alpha_i$ are the inverses 
of its poles.

The maximum weight of the Weil numbers occurring is 
\begin{align*}
&-4\dim H + 2(\dim G)  
(g+|D|-1) 
+ 2\dim H  + d(\weights_1) +   d(\weights_2)\\
&= 2(\dim G)  (g+|D|-1) - 2\dim 
H  + d(\weights_1) +   d(\weights_2).
\end{align*}
 A Weil number meets that bound only if 
it a Weil number from $H^0$ multiplied by the constant $q^{ (\dim G)  (g+|D|-1) + 
d(\weights_1) + \dim H } $ and then multiplied by $q^{- 2 \dim H}$. By Lemma 
\ref{ext-relation}(3), $H^0$ is isomorphic to $\Hom(K_1,K_2)$. Because $K_1$ and $K_2$ 
are pure, all eigenvalues on $\Hom(K_1,K_2)$ actually have size $q^{\frac{  d(W_1) - 
d(W_2) }{2}}$, so a Weil number meets that bound if and only if it comes from $H^0$ in 
this way. Bringing these numbers to the front of the line we obtain the stated claim.
\end{proof}

\section{$q$-aspect families}\label{s:family}
We continue with the set-up of Sections~\ref{s:properties} and~\ref{s:trace-function}, 
that is $G$ is a 
split semisimple group over $k$, $D$ 
is an effective divisor on $X$, and $\cL$ is a character sheaf on the factorizable subgroup 
$H$ of $G \langle \cO_D 
\rangle$. Suppose $k=\F_q$. 
We shall define the $q$-aspect family 
$\lV=\lV(G,X,D,H,\cL)$. This $q$-aspect family will be crucial in turning our bound for the trace of a Hecke operator into a bound for the individual Hecke eigenvalues.

\index{$\lV_n$, family in the $q$-aspect}
For every $n\ge 1$, let $F_n := F \otimes_{\mathbb F_q} \mathbb 
F_{q^n}$.
We define $\lV_n$ as consisting of automorphic representations $\Pi$ of $G(\mathbb 
A_{F_n})$, that are $G(\ko_y)$-unramified for every $y\in |(X- D)_n|$, and at each 
place $y \in D_n$ lying over a place $x \in D$, 
and with
residue field $\kappa_y/\kappa_x$,
admit a vector on which the preimage $J_y \subset G(\kappa_y\bseries{t})$ of 
$H(\kappa_y)$ 
acts by 
the character $\chi_y$ associated to the sheaf $\mathcal L$.
The automorphic representations are counted with multiplicity, more precisely it is the 
product 
of the automorphic multiplicity of $\Pi$ with the dimension of the space of 
$(J_y,\chi_y)$-invariant 
vectors in $\Pi_y$ for every $y \in D_n$.

\subsection{Spectral expansion of the trace}
For any $n\ge 1$, $\Pi\in \lV_n$, and $y\in |(X - D)_n|$, 
the representation $\Pi_y$ is $G(\ko_y)$-spherical.
Recall from \S\ref{sub:Satake} that to every $G(\ko_y)$-unramified irreducible 
representation 
$\Pi_y$ is attached a Satake parameter 
$t_{\Pi_y}\in \widehat{T}(\C)/W$.
For a dominant weight $\lambda\in \Lambda^+$, we have defined
\[
\tr_\lambda(\Pi_y):= 
 \tr(\Pi_y)(a_\lambda) 
=
\tr(t_{\Pi_y} | V_\lambda ).
\]
For a function $W:|X|\to \Lambda^+$ of finite support $\operatorname{supp}(W)$ disjoint from $D$,
let
\[
\tr_W(\Pi) := \prod_{y\in \operatorname{supp}(W)_n} \tr_{W_y}(\Pi_y).
\]

\begin{proposition}\label{p:spectral-trace} 
 For every $W_1,W_2:|X|\to \Lambda^+$ with 
finite support 
disjoint from $D$, and 
for every $n\ge 1$,
\[
\left| H( \mathbb F_{q^n})\right|^2
\sum_{\Pi \in \lV_n}
\tr_{W_1}(\Pi)
\overline{
\tr_{W_2}(\Pi)}
= \frac{ 1}{\left(q^n\right)^{\frac{d(W_1)+d(W_2)}{2}} } 
\sum_{g_1,g_2\in G(F_n) \backslash G( \mathbb A_{F_n})  / \KK } 
\frac{
 \FrK_{W_1} (g_1,g_2)
\overline{
\FrK_{W_2}(g_1,g_2)}
}
{|\Aut_D(g_1)||\Aut_D(g_2)|}
\]
\end{proposition}
\begin{proof}
Recall Definition~\ref{def:test-function} of the test functions $f^W_y$ for $y\in 
|X_n|$, and let $f:=\prod_{y\in |X_n|} f_y^W$. 
By Definition~\ref{def:global-function}, $\FrK_W(g_1,g_2)$ is the kernel of the convolution operator $*f$ on the vector space 
of all forms. Here, and below, we shall work with the counting measure on $G( \mathbb 
A_{F_n})  / \KK $ when forming convolutions.

Since $f^W_u$ is a cuspidal function for the place $u$, 
the operator $*f$ has image inside the space of cusp forms.
More precisely, consider an orthonormal Hecke basis $\cB_n=\{\varphi\}$ of the space of cuspidal
automorphic forms on $G(F_n) \backslash G( \mathbb A_{F_n})  / \KK $, where the inner product is
\[
\frac{1}{|H(\F_{q^n})|}
\sum_{g \in G(F_n) \backslash G( \mathbb A_{F_n})  / \KK } 
\frac{|\varphi(g)|^2}{|\Aut_D(g)|}.
\] 
Since automorphic representations in $\lV_n$ are counted with multiplicity, this implies 
that we can arrange the basis so that there is an injection $\lV_n \hookrightarrow 
\cB_n$, which we shall denote by $\Pi 
\mapsto \varphi_\Pi$. In other words, $\{\varphi_\Pi\}$ is a basis of the subspace of 
automorphic functions on $\Bun_{G(D)}(\F_q)$ which are $(\prod_{y\in D} J_y,\prod_{y\in 
D} 
\chi_y)$-equivariant. We can also arrange so that $\varphi*f=0$ if $\varphi 
\in \cB_n - 
\cV_n$ (for this consider the case $W=0$, in which case the operator $*f$ is idempotent, 
and its kernel forms an orthogonal complementary subspace).

The convolution operator $*f$ is an integral operator with 
kernel 
\[
\sum\limits_{\varphi\in \cB_n} (\varphi*f)(g_1) \overline{\varphi(g_2)}
=\sum\limits_{\Pi\in \lV_n} (\varphi_\Pi*f)(g_1) \overline{\varphi_\Pi(g_2)}
\]
We can show that 
we have $\varphi_\Pi *f = \left| H( \mathbb F_{q^n})\right| q^{\frac{d(W)}{2}}\tr_W(\Pi) 
\varphi_\Pi$.
 To 
do this, observe that for every $y\in \operatorname{supp}(W)_n$, $*f_y^W$ acts on the 
representation 
$\Pi_y$ 
by scalar multiplication by $\tr_{W_y}(\Pi_y) 
\left(q^n\right) ^{ \langle W_y, \rho \rangle}$, and that for $y\in D$, $*f^W_y$ acts on 
$\varphi_\Pi$ by a volume factor.
Precisely,
\[
\sum_{g_2 \in  G( \mathbb A_{F_n})  / \KK}   
\prod_{y 
\in |X_n|}  f_y^W( g_2^{-1}  g_1) \varphi_\Pi(g_2) = \frac{1 }{ \vol (\KK)} \int_{g_2 \in 
G(\mathbb A_{F_n})} \prod_{y \in |X_n|} f_y^W ( g_2^{-1} g_1) \varphi_\Pi(g_2)\]  
\[ =  
\frac{1 
}{ \vol (\KK)} \prod_{y \in |X_n|} 
\int_{h \in G(F_y)}
f_y^W ( h) 
\varphi_\Pi(g_1 
h^{-1})\]
\[
= \frac{\prod_{y\in D} \vol(J_y)}{\vol(\KK)}
\cdot
\prod_{y\in \operatorname{supp}(W)_n}
\tr_{W_x}(\Pi_y) 
\left(q^n\right) ^{ \langle W_y, \rho \rangle}
\cdot
\varphi_\Pi(g_1),
\]
and the ratio of volumes is equal to $\left| H( \mathbb F_{q^n})\right|$ by 
Definition~\ref{def:Jx}.
We deduce the identity
\[
K_{W}(g_1,g_2) = \left| H( \mathbb F_{q^n})\right| \left(q^n\right) ^{ d(W)/2} 
\sum_{\Pi\in \lV_n}  \tr_{W}(\Pi) \varphi_\Pi(g_1) 
\overline{\varphi_\Pi(g_2)}.
\]
The proposition now follows from orthogonality relations for the orthonormal basis 
$\cB_n$.
\end{proof}

\subsection{Average Ramanujan bound}
We fix a place $v\in |X - D|$.
\begin{theorem}\label{t:average-Ramanujan}
Let $\lambda\in \Lambda^+$ be a dominant weight.
For every integer $n\ge 1$,
\[
\sum_{\Pi\in \lV_n} 
\prod_{w|v}
|\tr_{\lambda}(\Pi_w)|^2
\ll
q^{n(\dim G(g+|D|-1)-\dim H)}.
\]
The multiplicative constant is independent 
of $n$, it depends only on $X, v,\lambda,(G,D,H,\cL)$. 
\end{theorem}

\begin{proof} Let $W:|X|\to \Lambda^+$ be defined by
\[
W_x:= 
\begin{cases}
\lambda,& \text{if $x=v$,}\\
0,& \text{if $x\neq v$.}
\end{cases}
\]
Let $\FrK_W$ 
be the function defined in Definition~\ref{def:global-function}. Then by Proposition~\ref{p:spectral-trace} and Theorem~\ref{geometric-average-bound} 
\[
\sum_{\Pi \in \lV_n}
|\tr_{W}(\Pi)|^2
= \frac{ 1}{(q^n) ^{d(W)} \left| H( \mathbb F_{q^n})\right|^2 } 
\sum_{g_1,g_2\in G(F_n) \backslash G( \mathbb A_{F_n})  / \KK } 
\frac{
 \FrK_{W} (g_1,g_2)
\overline{
\FrK_{W}(g_1,g_2)}
}
{|\Aut_D(g_1)||\Aut_D(g_2)|}\] \[ = \frac{ O \left(  \left(q^n\right)^{ (\dim G) 
(g+|D|-1) + d(W) + \dim H} \right) }{ (q^n) ^{d(W)} \left| H( \mathbb 
F_{q^n})\right|^2} =O 
\left(  \left(q^n\right)^{ (\dim G) (g+|D|-1) - \dim H} \right).
\qedhere
\]
 \end{proof}

\begin{corollary}\label{cor:average-ram-one}
Let $n\ge 1$, $\Pi\in \lV_n$, and let $\lambda$ be a dominant weight of $G$. Then  \[ \prod_{ w \mid 
v}  |\tr_{\lambda}(\Pi_w)|^2
\ll
q^{n(\dim G(g+|D|-1)-\dim H)}\] 
with the constant independent of $n$. 
\end{corollary}

\begin{proof} This follows from Theorem~\ref{t:average-Ramanujan} because the left side is a sum of squares and hence any term is bounded by the whole. \end{proof}

\subsection{Sums of Weil numbers}\label{sub:sums-Weil}
In the course of the proof above we have shown that several spectral quantities are sums 
of Weil numbers.
Such results are of independent interest, and we spell them out in more detail in this 
subsection.

\begin{proposition}\label{p:count-Weil}
There exist $q$-Weil integers $\alpha_i$ of weight $\le 2 (\dim G)(g+|D|-1) - 2\dim H$, such 
that
\[
|\lV_n|=
\sum\nolimits_i \alpha_i^n,\quad
\text{for every $n\ge 1$.}
\]
\end{proposition}
\begin{proof} This follows from
Theorem~\ref{weil-numbers-semisimple} and Proposition~\ref{p:spectral-trace}, taking $W_1=W_2=0$. In this case $d(W_1)=d(W_2)=0$ so the factor of $q^{ d(W_1)/2+ d(W_2)/2}$ may be ignored.
\end{proof}

\begin{proposition}
For every $W:|X|\to \Lambda^+$ of finite support disjoint from $D$,
there exist $q$-Weil integers $\beta_{j}$ of weight $\le 2(\dim G)(g+|D|-1) -2 \dim H + 
d(W)$, such that
\[
q^{n\frac{d(W)}{2}}
\sum_{\Pi\in \lV_n}
\tr_W(\Pi)
=
\sum\nolimits_j \beta_{j}^n,\quad
\text{for every $n\ge 1$.}
\]
\end{proposition}
\begin{proof}This follows from
Theorem~\ref{weil-numbers-semisimple} and Proposition~\ref{p:spectral-trace}, taking $W_1=W$ and $W_2=0$. 
\end{proof}

\subsection{The main theorem}

To prove the main theorem, we shall embed the automorphic representation $\pi$ of 
$G(\A_F)$ in a 
suitable 
automorphic family 
$(\lV_n)_{n\ge 1}$ in 
the $q$-aspect: %
at the place $u$, we shall use the
mgs datum, and at the other ramified places, we shall choose a datum 
with 
trivial character, and with sufficient depth that $\pi$ and its base changes $\Pi_n$ have a 
nonzero invariant vector.

\begin{lemma}\label{moy-prasad-metric} Let $G$ be a reductive group over a local field. 
There is a constant $c$ such that for any two points $x,y$ in the Bruhat--Tits building, 
for all depths $r$, the Moy--Prasad subgroup $G_{x,r}$ contains a conjugate of 
$G_{y,r+c}$. If $G$ is split, we can take $c$ to depend only on the root data of $G$ and 
not on the base field. \end{lemma}

\begin{proof} After conjugation, we may assume that $x$ and $y$ are contained in the same 
apartment. Define a metric on this apartment where the distance $d(x,y)$ is the max over 
all roots of the absolute value of the difference between the evaluations of the linear 
function associated to this root on $x$ and $y$. Then by construction, it is clear that 
$G_{x,r}$ contains $G_{y, r+ d(x,y)}$. Take $c$ to be the supremum over pairs $x,y$ of the 
minimum distance between $x$ and any conjugate of $y$ under the affine Weyl group action. 
Because this action is cocompact, a finite supremum in fact exists. Because the metric on 
the apartment and the affine Weyl group can be defined combinatorially, $c$ depends only 
on the underlying root data.  \end{proof} 

\begin{lemma}\label{family-existence} Let $G$ be a split semisimple algebraic group. Let 
$F = \mathbb F_q(X)$. Let $\pi$ be an automorphic representation of $G(\A_F)$, mgs at a 
place $u$, with \condbc{}. Then there exists an effective divisor $D$ on $X$, a subgroup 
$H 
\subseteq G \lWR \mathcal O_D \rWR$, and a character sheaf $\mathcal L$ on $H$, that is 
geometrically supercuspidal on $U$, and such that for all $n$, the base change $\Pi_n$ of 
$\pi$ to $ F_n$ is contained in the associated family $\mathcal V_n$. \end{lemma}

\begin{proof} By the definition of \condbc{}, there exists mgs datum 
$(G_{\kappa_u},m_u,H_u,\cL_u)$ such that for all $n$, for all places $u'$ of $F_n$ lying 
over $u$ with local field $E_{u'}$, $\Pi_{n,u'}$ is a quotient of 
$\cind^{G(E_{u'})}_{J_{u,E}} 
\chi_{u,E}$.

Let $S$ be the set of ramified places of $\pi$ other than $u$.  Again by the definition 
of \condbc{}, $\Pi_n$ is unramified outside $S \cup \{u\}$, with a bound on the depth 
inside 
$S$. Let $m$ be some integer greater than this bound on the depth plus the constant of 
Lemma~\ref{moy-prasad-metric}. It follows that for all places $x$ lying over a place in 
$S$,  $\Pi_n$ contains a vector invariant under the depth $m$ subgroup of the standard 
hyperspecial maximal compact, which is the subgroup of elements of $G(\kappa_x[[t]])$ 
congruent to $1$ mod $t^m$.

It follows that if we let $D$ be the divisor of multiplicity $m$ at each point of $S$ and 
multiplicity $m_u$ at $u$, $H = H_u$, and $\mathcal L  = \cL_u$, then $\Pi_n \in 
\mathcal V_n$ for all $n$.

Finally, $(G, D, H, \mathcal L)$ is geometrically supercuspidal at $u$ because 
$(G_{\kappa_u}, m_u, H_u, \cL_u)$ is geometrically supercuspidal. \end{proof} 

To improve the bound of Corollary~\ref{cor:average-ram-one} for this family, and obtain 
the main theorem, we use a variant of the tensor power trick, where bounds for large $n$ 
will imply stronger bounds for small $n$.

\begin{theorem}\label{t:tempered}
 Let $G$ be a split semisimple algebraic group. Assume the characteristic 
of $F$ is not $2$. Let $\pi$ be an automorphic representation of $G(\A_F)$, mgs at a 
place $u$, and satisfying \condbc{}. Let $v$ be a place at which $\pi$ is unramified for 
the 
standard hyperspecial maximal compact subgroup $G(\ko_v)$. Then $\pi$ is tempered at $v$. 
\end{theorem}

\begin{proof} 
Let $\lambda\in \Lambda^+$ be a dominant weight.
We apply Corollary~\ref{cor:average-ram-one} to the family produced by 
Lemma~\ref{family-existence} to obtain that 
\[
\prod_{w|v}
|\tr_\lambda(\Pi_{n,w})|^2
\ll 
 (q^n) ^{ 
(\dim G )  (g+ |D| -1 ) - \dim H }. 
\]

Let $n_0:=\operatorname{gcd}(n,[\kappa_v:k])$, and $n_1:=n/n_0$.
All the places $w|v$ have 
isomorphic residue field $\kappa_w$, with $[\kappa_w:\kappa_v]=n_1$, and by the 
definition of base change, they have 
the same Satake parameter. 
So all of the $n_0$ terms in the above product are equal to each other, and we deduce 
\[
|\tr_\lambda(\Pi_{n,w}) |\ll 
(q^{n_1})^{ \left( (\dim G ) (g+ |D| -1 )  - \dim H \right)/2}.
\]

Let $t_{\pi_v}$ be the Satake parameter of $\pi_v$.
Then the Satake parameter of $\Pi_{n,w}$ is equal to $t_{\pi_v}^{n_1}$, hence
$\tr_\lambda(\Pi_{n,w}) = \tr(t_{\pi_v}^{n_1}|V_\lambda)$.
Because all $n_1\ge 1$ arise for some $n$ (specifically for $n=[\kappa_v:k]n_1$),
Lemma~\ref{l:recursive-sequence} implies that we have the improved inequalities
\[
|\tr(t_{\pi_v}^{n_1}|V_\lambda)|
\le
\dim V_\lambda
\cdot
(q^{n_1})^{ \left( (\dim G ) (g+ |D| -1 )  - \dim H \right)/2}.
\]
In particular for $n_1=1$,
\[
|\tr_\lambda(\pi_v)|
\le
\dim V_\lambda
\cdot
q^{ \left( (\dim G ) (g+ |D| -1 )  - \dim H \right)/2}.
\]

Since the inequality holds for every $\lambda\in \Lambda^+$, we deduce by 
Proposition~\ref{p:tempered} 
that in fact 
$|\tr_\lambda(\pi_v)|\le \dim V_\lambda$, and $\pi_v$ is tempered. \end{proof}

\begin{remark} A close analogue of the argument may be found in the 
Bombieri--Stepanov 
proof of the Riemann hypothesis for curves over finite fields. Weil's proof for a curve 
$C$ of genus $g$ over $\mathbb F_q$ immediately proves in one stroke the Riemann bound 
$|\#C(\mathbb F_q)  - q -1| \leq 2 g \sqrt{q}$. The proof of Bombieri--Stepanov, say in the 
special case of a Galois cover of $\mathbb P^1$, involves more steps. One first deduces an 
estimate $\#C(\mathbb 
F_q)  \leq 1 + q + (2g+1)\sqrt{q}$, then by applying this bound to twists of $C$, obtains $\#C(\mathbb F_q)  - q -1 \geq 1 + q - O( (2g+1)) \sqrt{q})$, with a constant depending on the order of the Galois group. To improve the constant 
from $O(2g+1)$ to the correct value $2g$, it is necessary to use the rationality of the 
zeta function. From the estimate for $\#C(\mathbb F_{q^n})$ for $n$ large, one deduces the 
sharp bound for the zeroes of the zeta function and thus a sharp bound for the number of 
points.

Our method closely follows the strategy of the last deduction. Instead of the zeroes of the zeta function, we are attempting to bound the eigenvalues of the Satake parameter. Instead of using the rationality of the zeta function, we use cyclic base change to compare the Satake eigenvalues for the base changed automorphic form to the Satake eigenvalues of the original form. The main difference is that, while the bound 
$(2g+1)\sqrt{q}$ is sufficient for most practical purposes, the constant factor which we 
amplify away is ineffective, and would render the estimate useless in the $\lambda$ 
aspect if not dealt with.
\end{remark}

\begin{remark} We compare our use of the tensor power trick to Rankin's trick. In both 
cases, some special case 
of functoriality is used to amplify a weaker bound into a stronger one. The needed 
functoriality is rather weak in our case, where it is cyclic base change.
However, our argument and Rankin's trick are different in one crucial respect, other than 
the different versions of functoriality applied. Rankin's trick produces an improvement 
in the dependence on $q$ in the bound. Speaking geometrically, we may refer to it as an 
improvement of the weight. In our method, however, the weight is fixed as $q$ varies 
(unsurprising as it arises geometrically as the weight of a cohomology group), and is not 
improved directly. Instead, we pass to the large $q^n$ limit to handle a constant term 
independent of $q$. 
\end{remark}

\subsection{Hecke eigenvalues are Weil numbers}
 We establish the following 
strengthening of the previous Theorem~\ref{t:tempered}. Assumptions are as before.
\begin{theorem}\label{t:individual-sum-Weil}
 For every $\lambda\in 
\Lambda^+$, the trace  
$q^{\langle \lambda,\rho \rangle} \tr_\lambda(\pi_v)$ of the $\lambda$-Hecke operator is 
a sum of length 
$\dim(V_\lambda)$ of $q$-Weil integers 
of weight $\langle \lambda, 2\rho \rangle$. 
\end{theorem}

\begin{proof}
Hecke eigenvalues are algebraic numbers because of the finiteness of the support of 
cuspidal automorphic functions with prescribed local conditions.
Next we will prove that the Hecke eigenvalues have size $q^{\langle \lambda,\rho 
\rangle}$ for every 
embedding of the coefficient field into $\C$. Every embedding comes from another 
automorphic form satisfying the same assumptions, possibly with a different mgs datum. 
Indeed the local mgs condition at $u$ is preserved under $\Aut(\C)$, and also the global 
\condbc. Thus the previous Theorem~\ref{t:tempered} applies.
Finally the integrality follows either
from~\cite[Prop.2.1]{Lafforgue:valuations-padiques}, or from 
Lemma~\ref{integrality-trace-semisimple} by varying $\lambda\in \Lambda^+$. 
\end{proof}

\begin{example}
Consider the rigid automorphic sheaves constructed 
in~\cite{HNY:Kloosterman,Yun:epipelagic}.
The \condbc{} is satisfied because the trace function over each finite extension 
$\F_{q^n}$ defines an automorphic function that generates a corresponding automorphic 
representation (see Remark~\ref{r:eigensheaf}). We have seen in Section~\ref{sub:epipelagic} that 
epipelagic representations are 
mgs. Thus Theorem~\ref{t:tempered} applies, and the temperedness is consistent with the 
results of \emph{loc. cit.}, indeed the construction of $\ell$-adic sheaves on 
$\PP^1_{\backslash \{0,\infty\}}$ 
that 
generalize Kloosterman sums. The conclusion of Theorem~\ref{t:individual-sum-Weil} on 
integrality is also consistent with \emph{loc.~cit.}, precisely, it follows 
from~\cite[(5.8)]{HNY:Kloosterman}, which explicates $\operatorname{Kl}^{V_\lambda}$ as 
an exponential sum, and because each of the Kummer, 
Artin--Schreier, and IC sheaves is integral. This is analogous to 
Lemma~\ref{integrality-trace-semisimple}.
\end{example}

\section{Relationship with Lafforgue--Langlands parameters and Arthur 
parameters}
\label{s:Lafforgue}

In this section, we will describe a potential approach to provide a different proof of 
the main theorem of this paper, using V. Lafforgue's Langlands parameterization, the 
Lafforgue--Genestier semisimplified local Langlands parameterization, and some 
conjectural 
explicit calculations with that parameterization.  We will then express the same 
strategy, or a very similar strategy, in the language of Arthur parameters, and again 
without direct reference to parameters of any kind, using only the notion of 
two representations being in the same $L$-packet. 

The starting point of all three approaches will be a guess about the Langlands parameters 
of mgs representations. We can verify this conjecture in the $\GL_r$ case, where the 
local 
Langlands correspondence is known by results of Laumon--Rapoport--Stuhler, and 
Henniart--Lemaire~\cite{Henniart-Lemaire:changement-base}.

\begin{proposition}\label{p:gln-irreducible} Let $F_u$ be a non-archimedean local field 
and let 
$\pi_u$ be a mgs 
representation of 
$\GL_r(F_u)$. Then its local Langlands parameter $\sigma_u: W_{F_u} \to \GL_r( 
\Ql)$ is irreducible when restricted to the inertia group of $F_u$. 
\end{proposition}

\begin{proof} 
For each unramified extension $F_u'$ of $F_u$, let $\pi'_u$ be the base change 
representation of $\pi_u$.
It follows from \cite[Prop.II.2.9]{Henniart-Lemaire:changement-base}, 
\cite[Prop.II.5.15.2]{Henniart-Lemaire:changement-base}, and the orbital integral 
identity in Theorem~\ref{kottwitz-stable} that $\pi'_u$ is an mgs representation, with 
datum compatible with that of $\pi_u$.
In particular $\pi'_u$ is supercuspidal.

It is established in~\cite[Thm.IV.1.5]{Henniart-Lemaire:changement-base} that the 
Langlands parameter of $\pi'_u$ is the restriction of the 
Langlands parameter $\sigma_u$ to $W_{F_u'}$.
Since $\pi'_u$ is supercuspidal, we have that $\sigma_u$ restricts to an 
irreducible $W_{F_u'}$ representation.

Because $\sigma_u(I_{F_u})$ is a finite group, the action of $\sigma_u(\Frob_u)$ on it by 
conjugation has finite order $m$. Let $F_u'$ be an unramified extension of $F_u$ of 
degree $m$. Then $\sigma_u(W_{F_u'})$ is generated by $\sigma_u(I_{F_u})$ and the $m$th 
power of $\Frob_u$, which commutes with it. Hence $\sigma_u( \Frob_u^m)$ lies in the 
center of $\sigma_u(W_{F_u'})$, which acts irreducibly, so $\sigma_u( \Frob_u^m)$ is 
a scalar, and hence $\sigma_u(I_{F_u})$ acts irreducibly, as desired. \end{proof}

To conjecturally apply this to general groups, and use it to verify Ramanujan, we use the 
work of V. Lafforgue and Genestier--Lafforgue on the Langlands correspondence over 
function fields, which we now review:
Recall that $D$ is an effective divisor on $X$, and $\KK$ is 
the compact subgroup of the adelic points $G(\A_F)$ of the split semisimple $G$ consisting 
at each 
place of local 
sections of the group scheme congruent to the identity modulo $D$.

Lafforgue~\cite{Lafforgue:reductifs-chtoucas} defines a $\cC_c( \KK \backslash G( \mathbb 
A_F) / \KK,\Ql)$-module 
decomposition 
of  $\cC^{\rm cusp}_c( \Bun_{G(D)}(\mathbb F_q), \Ql)$ indexed by continuous semisimple 
representations $\sigma:  \Gal(\overline{F} / F) \to \widehat{G} (\Ql)$, unramified away 
from $D$. 
Since $\pi^{\KK}$ is irreducible and nonzero, it appears inside a module of this decomposition.

Letting $\iota$ be an embedding $\Ql \to \mathbb C$, we say a continuous representation of
$\Gal(\overline{F} / F)$ is \emph{$\iota$-pure} of weight $w$ if for each unramified place $v$, 
the image by $\iota$ of the eigenvalues of $\Frob_{v}$ on the representation are complex 
numbers of norm $|\kappa_v|^{\frac w2}$. We say that a representation is \emph{$\iota$-mixed} if 
it 
has a filtration whose associated graded components are $\iota$-pure of 
increasing weights. All representations $\sigma$ appearing in the above decomposition, 
composed with any representation of $\widehat{G}$, are $\iota$-mixed. (In fact 
by~\cite{Lafforgue:chtoucas} this is known for any representation, but it has a direct 
proof in this case.)

Genestier--Lafforgue~\cite{Genestier-Lafforgue} define for each local representation 
$\pi_u$ a semisimple 
representation $\sigma_{\pi_u}: \Gal( \overline{F_u} / F_u) \to \widehat{G} (\Ql)$, which 
satisfies the following compatibility condition: Whenever $\pi^{\KK}$ appears as an 
irreducible $\cC_c( \KK \backslash G( \mathbb A_F) / \KK,\Ql)$-module inside the summand of \\
$\cC^{\rm cusp}_c( \Bun_{G(D)}(\mathbb F_q), \Ql)$ indexed by a representation $\sigma:  \Gal(\overline{F} 
/ F) \to \widehat{G} (\Ql)$, 
the 
semisimplification of the restriction of $\sigma$ to $\Gal( \overline{F_u} / F_u) $ is 
equal to $\sigma_{\pi_u}$.

The key conjecture, which is expected to generalize Proposition~\ref{p:gln-irreducible}, 
is as follows.
In the case of an epipelagic representation $\pi_u$, it is consistent with the 
conjectures 
of~\cite[\S7.1]{Reeder-Yu:epipelagic}, in which the assertion is expressed in the form 
$\widehat{\mathfrak{g}}^{\sigma_{\pi_u}(I_{F_u})}=0$.
\begin{conjecture}\label{inertial-parameters} For $\pi_u$ a mgs representation, the image 
of the inertia subgroup $I_{F_u}$ of $\Gal( \overline{F_u} / F_u) $ under the parameter 
$\sigma_{\pi_u}$ is not contained in any proper parabolic 
subgroup of $\widehat{G}( \Ql)$. 
\end{conjecture}

It follows from this conjecture that, if $\pi$ is mgs at one place, then $\pi$ is tempered at all unramified places. This follows from the below chain of reasoning, which depends on the Lemmas \ref{parabolic-parameter-restriction},\ref{mixed-parameter}, and \ref{pure-parameter} immediately afterwards.

\begin{itemize}
\item[(1)] Assume that $\pi_u$ is mgs. 
\end{itemize}
Then, under Conjecture~\ref{inertial-parameters}:
\begin{itemize}
\item[(2)] The image of $\Gal( \overline{F_u} / F_u) $ under $\sigma_{\pi_u}$ is not contained 
in 
any proper parabolic subgroup of $\widehat{G}(\Ql)$.
\end{itemize}
Thus we deduce:
\begin{itemize}
\item[(3)] The image of $\Gal(\overline{F} / F)$ under the Lafforgue--Langlands parameter 
$\sigma$ of $\pi$ is not contained in any proper parabolic subgroup of $\widehat{G}(\Ql)$. 
\item[(4)] The composition of the Lafforgue--Langlands parameter $\sigma$ of $\pi$ with 
every 
representation of $\widehat{G} (\Ql)$ is pure of weight $0$. 
\item[(5)] $\pi$ is tempered at every unramified place.
\end{itemize}
Indeed the implication (2) $\implies$ (3) is Lemma~\ref{parabolic-parameter-restriction},
then (3) $\implies$ (4) is Lemma~\ref{mixed-parameter}, and Lemma~\ref{pure-parameter} 
gives (4) $\implies$ (5).

\begin{lemma}\label{parabolic-parameter-restriction} Let $\sigma: \Gal(\overline{F} / F) 
\to \widehat{G}(\Ql)$ be a representation with image contained in a proper parabolic 
subgroup. For any place $u$ of $F$, the image of the semisimplification of the restriction 
of  
$\sigma$ to  $\Gal( \overline{F_u} / F_u) $ is contained in a proper parabolic subgroup. 
\end{lemma}

\begin{proof} That the property of being contained in a parabolic subgroup is stable 
under restriction is obvious. That it is preserved under semisimplification is immediate 
from the definition of semisimplification --- we take a minimal parabolic subgroup 
containing the image of the representation, if any, and then project onto the Levi of that 
parabolic. Furthermore, the semisimplification is independent of which minimal parabolic we take. Thus, as long as some proper parabolic subgroup contains the image, some proper Levi subgroup contains the 
image of the semisimplification. \end{proof} 

\begin{lemma}\label{mixed-parameter} Let $\sigma: \Gal(\overline{F} / F) \to 
\widehat{G}(\Ql)$ be a $\iota$-mixed representation whose image is not contained in a proper 
parabolic 
subgroup. Then for every representation  $V$  of 
$\widehat{G}$, the composite $V(\sigma)$ is pure of weight zero.
\end{lemma}

\begin{proof} Because $V(\sigma)$ is $\iota$-mixed, it has a canonical filtration into 
pure 
representations. The image of $\sigma$ is contained in the stabilizer of this filtration 
inside $\widehat{G}$. We will show that either this stabilizer is a proper parabolic subgroup 
of 
$\widehat{G}$ or $V(\sigma)$ is pure of weight zero.

Let $v$ be a place at which $\sigma$ is unramified and let $T$ be a torus containing the 
semisimplication $\Frob_v^{ss} $ of $\Frob_v$. Then the generalized eigenspaces of $\Frob_v$ are sums of 
eigenspaces of $T$. For $\chi$ a character of $T$, let $\omega(\chi) = \log | \iota (\chi( \Frob_v^{ss} ))|$. Then $\omega$ is a linear function on the weight lattice of $T$. Because each associated graded of the 
weight filtration is pure of increasing weight, the eigenvalues of $\Frob_v$ on each 
associated graded all have the same absolute value, so each associated graded of the weight filtration is a sum of eigenspaces of $T$ where $\omega$ takes a fixed value, and this value of $\omega$ is increasing in the filtration. Thus an element 
preserves the weight filtration if and only if it sends eigenspaces of $T$ to eigenspaces 
of $T$ where $\omega$ takes equal or lower values on their weights.

This is exactly the subgroup of $\widehat{G}$ generated by all roots where $\omega$ takes a nonnegative value on their weights. This subgroup is parabolic unless 
it contains every root, in which case $\omega$ is zero on all roots, which 
because $\widehat{G}$ is semisimple implies it is zero on all characters of $T$, so the 
representation is pure of weight zero.   \end{proof}

\begin{lemma}\label{pure-parameter} Let $\pi$ be a representation of $G(\mathbb A_F)$ 
such that $\pi^{\KK}$ is nonzero and appears inside the summand of 
$\cC_c^{\rm cusp}( 
\Bun_{G(D)}(\mathbb F_q), \Ql)$ indexed by a parameter $\sigma$ such that $V(\sigma)$ is $\iota$-pure 
of weight zero for every representation $V$ of 
$\widehat{G}$.
Then $\pi$ is tempered at all unramified places.
\end{lemma}

\begin{proof} This follows from Proposition~\ref{p:tempered} and the compatibility between the action of 
$\cH(G(F_v),G (\mathfrak o_{v}))$ on the summand of 
$\cC^{\rm cusp}_c( \Bun_{G(D)}(\mathbb F_q), \Ql)$ indexed by $\sigma$ and the conjugacy class 
of 
$\sigma(\Frob_v)$. \end{proof}

We now sketch two, more conjectural, analogues of this argument.

The first is based on 
Arthur parameters, and explains how we expect our main theorem can be related to Arthur's 
conjectures. We can, conditionally on different conjectures, prove that all 
representations $\pi$ mgs at one place are tempered at every unramified place by a 
modified chain of deductions (1) $\implies$ (2) $\implies$ (3) $\implies$ (4') $\implies$ (5), where (4') is as 
follows.

\begin{itemize}

\item[(4')] The image of $\SL_2$ in every global Arthur parameter of $\pi$ is trivial. 
\end{itemize}

The implication (3) $\implies$ (4') depends on the conjectural existence of Arthur parameterizations compatible with Lafforgue's Langlands parameterization. Using this, the proof is similar to the proof of Lemma 
\ref{mixed-parameter}, but with a diagonal element in $\SL_2$ replacing the Frobenius 
element.  The implication (4') $\implies$ (5) is part of Arthur's conjectures on Arthur 
parameters.
It is clear that if the conjectural relationship of Lafforgue--Langlands
parameters with Arthur parameters could be proved, then this argument would be 
essentially the same as the previous argument.

The second analogue avoids mentioning parameters of any kind, except through their 
$L$-packets, and relies on conjectures only in terms of automorphic representations. Conditionally on conjectures, we can prove (1) $\implies$ (5) via a chain of implications (1) $\implies$ (2'') $\implies$ (3'') $\implies$ (4'') $\implies$ (5), where (2''), (3''), (4'') are as follows.

\begin{itemize}

\item[(2'')] All representations of $G(F_u)$ in the $L$-packet containing $\pi_u$ are 
supercuspidal. 

\item[(3'')] All automorphic representations $\pi'$ such that $\pi_v$ 
and 
$\pi'_v$ are in the same $L$-packet for every place $v$ of $F$, are cuspidal.

\item[(4'')] All automorphic representations $\pi'$ such that $\pi_v \simeq \pi'_v$ for all but 
finitely 
many places $v$ of $F$, are cuspidal.

\end{itemize}

The implication (4'') $\implies$ (5) is consequence of the conjecture~\cite{Gan-Gurevich:CAP} 
that non-tempered cuspidal automorphic representations are CAP.
The implications (2'') $\implies$ (3'') $\implies$ (4'') are trivial, and
the implication (1) $\implies$ (2'') is a variant of Conjecture 
\ref{inertial-parameters}.

Our method of proof of the main result is also purely automorphic, and in some respects 
follows this last strategy.
Indeed property (4'') is necessary to construct a spectral set $\mathcal{V}$, prescribed 
by local behavior containing
$\pi_u$, which is obtained by projection from an automorphic kernel $K(x,y)$ 
of compact support. See the related discussion in \S\ref{sub:intro-mgs}.
Properties (2'') and (3'') appear implicitly in \condbc{}, since the theory of base 
change 
and stabilization of trace formulas is related to the notion $L$-packet. 

\begin{remark}
Many of the reverse implications are known or conjectured. In the Arthur parameter 
setting, (4') implies (3), since discrete series representations should have elliptic 
Arthur parameters, meaning that the Weil group and $\SL_2$ aren't both contained in 
the 
same parabolic subgroup. The same statement is true in the Lafforgue--Langlands 
parameter setting, 
conditional on conjectural relationship with
Arthur parameters. In every setting, (5) is known to imply (4) (resp. (4'), (4'')). 
However (3) never implies (2) as cuspidality of an automorphic representation 
cannot imply local supercuspidality of its constituents. Hence it is not possible to 
prove the conjecture that (1) implies (2) as a corollary of our main result. 
\end{remark}

Finally, we include for comparison a proof of a part of a conjecture of 
Clozel~\cite[Conj.4(1)]{Clozel:park-city} 
in the function field case, obtainable unconditionally from the work of V. 
Lafforgue~\cite{Lafforgue:reductifs-chtoucas}, which 
we mentioned in Remark~\ref{rem:unramified} of the introduction. 

\begin{theorem}\label{temperedOneplace} Let $G$ be a split semisimple group over a 
function field $F$ and 
$\pi$ a cuspidal automorphic representation of $G(\mathbb A_F)$. If $\pi$ is tempered at one 
unramified place, then $\pi$ is tempered at all unramified places. \end{theorem}

\begin{proof} 
Choose some compact open subgroup $\KK$ which fixes a nonzero vector $f\in \pi^{\KK}$, where 
$D$ is an 
effective divisor containing  the ramified places of $\pi$. Viewing $\pi$ as a 
subrepresentation of 
$L^2( G(F) \backslash G(\mathbb A_F))$, this vector defines a locally constant function $f$ on 
$G(F) \backslash G(\mathbb 
A_F) / \KK$. Because $\pi$ is cupsidal, $f$ is compactly supported. 
Fix an isomorphism $\iota\colon \overline{\mathbb Q}_\ell \cong \mathbb C$.
Lafforgue's 
theorem~\cite{Lafforgue:reductifs-chtoucas} gives a decomposition of 
 $\cC^{\rm cusp}_c( \Bun_{G(D)}(\mathbb F_q), \Ql)$ indexed by continuous semisimple 
representations $\sigma\colon   \Gal(\overline{F} / F) \to \widehat{G} (\Ql)$. Because $f$ is nonzero, there must exist a parameter $\sigma$ such that the projection of $f$ onto the module indexed by $\sigma$ is nonzero. 

Similarly, we can choose a parameter $\sigma'$ such that the projection of the complex 
conjugate $\overline{f}$ to the space indexed by $\sigma'$ is nonzero.

Now for $v$ an unramified place of $\pi$, and $V$ any representation of $\widehat{G}$, because 
$f$ is 
a $G(\mathfrak o_{v})$-invariant vector in the representation space of $\pi$, it is an 
eigenfunction of the corresponding $V$-Hecke operator, with eigenvalue 
 $\tr (t_{\pi_v}, V)$, where $t_{\pi_v}$ is the Satake parameter of $\pi_v$.
By \cite{Lafforgue:reductifs-chtoucas}, this eigenvalue coincides with 
$\tr(\operatorname{Frob}_v, V(\sigma))$. So we must have
\[  
\iota(\tr(\operatorname{Frob}_v, V(\sigma)))  = \tr (t_{\pi_v}, V).
 \]
 Similarly, we have
$ \iota( \tr(\operatorname{Frob}_v, V(\sigma'))) = \overline{\tr (t_{\pi_v}, V)}$. More strongly, the 
characteristic polynomials of $\Frob_v$ acting on $V(\sigma)$ is sent by $\iota$ to the 
characteristic polynomial of $t_{\pi_v}$ acting on $V$, while the characteristic polynomial 
of $\Frob_v$ acting on $V(\sigma')$ is sent by $\iota$ to the complex 
conjugate polynomial. Thus $V(\sigma) \oplus V(\sigma')$ is $\iota$-real in the sense that its 
characteristic polynomial of Frobenius has real coefficients (at every unramified place, 
under $\iota$).

Now assume $\pi_v$ is tempered for the given unramified place $v$. Then the Satake 
paramater 
$t_{\pi_v}$ is unitary by Proposition~\ref{p:tempered}, so by this previous identity of 
characteristic 
polynomials, all the 
eigenvalues 
of $\operatorname{Frob}_v$ on $V(\sigma)$ are sent by $\iota$ to complex numbers of norm $1$.  
The same is true for their complex conjugates, the images under $\iota$ of the eigenvalues 
of $\operatorname{Frob}_v$ on $V( \sigma')$. We can now apply \cite[Thm.4.1]{KatzNote} to 
$V(\sigma) \oplus V(\sigma')$ --- because it is $\iota$-real and its eigenvalues of Frobenius at one 
place are complex numbers of norm $1$, it follows that its eigenvalues of Frobenius at 
every place are complex numbers of norm $1$. It follows at every other unramifed 
place $w$ that the eigenvalues of the Satake parameter $t_{\pi_w}$ on $V$ have norm $1$. 
\end{proof}

\subsection*{Acknowledgements}
We thank Jean-Pierre Labesse, Vincent Lafforgue, Bau-Ch\^au Ng\^o, and Sug Woo Shin for 
helpful discussions, and Paul Nelson for a careful reading.  We also thank the anonymous 
referees for their many helpful comments.
This article begun while both the authors were in residence at the MSRI, supported by the 
NSF under Grant No. DMS-1440140. The authors received funding from the European Research Council under the European Community's Seventh Framework Programme (FP7/2007-2013) / ERC Grant agreement no. 290766 (AAMOT) to visit IHES.  W.S. was supported by Dr. Max R\"ossler, the Walter Haefner Foundation and the ETH 
Z\"urich Foundation. N.T. was supported by the NSF-CAREER under agreement No. 
DMS-1454893, and by a Simons Fellowship under agreement 500294.

\def\cprime{$'$}
\providecommand{\bysame}{\leavevmode\hbox to3em{\hrulefill}\thinspace}
\providecommand{\MR}{\relax\ifhmode\unskip\space\fi MR }
\providecommand{\MRhref}[2]{%
  \href{http://www.ams.org/mathscinet-getitem?mr=#1}{#2}
}
\providecommand{\href}[2]{#2}

\newgeometry{left=1cm,right=1cm,tmargin=2.5cm, marginpar=1cm}
{\small  \printindex}
\end{document}